\documentclass[12pt]{amsart}

\usepackage{longtable}
\usepackage{float}
\restylefloat{table}

\newcommand{\hide}[1]{}

\usepackage{amssymb}
\usepackage{amsbsy}
\usepackage{amscd}
\usepackage{amsmath}
\usepackage{amsthm}

\numberwithin{equation}{section}

\theoremstyle{plain}
\newtheorem{thm}{Theorem}[section]
\newtheorem{prop}[thm]{Proposition}

\newtheorem{thm-defi}[thm]{Theorem/Definition}
\newtheorem{cor}[thm]{Corollary}
\newtheorem{new-lemma}[thm]{Lemma}

\theoremstyle{definition}

\newtheorem{defi}[thm]{Definition}
\newtheorem{rem}[thm]{Remark}

\newtheorem{example}[thm]{Example}

\newcommand{\C}{{\mathcal C}}
\newcommand{\CH}{{\mbox CH}}

\newcommand{\E}{{\mathcal E}}
\newcommand{\F}{{\mathcal F}}
\newcommand{\G}{{\mathcal G}}

\renewcommand{\H}{{\mathcal H}}
\newcommand{\HH}{{\mathbb H}}

\newcommand{\LB}{{\mathcal L}}

\newcommand{\T}{{\mathcal T}}
\newcommand{\M}{{\mathcal M}}

\newcommand{\U}{{\mathcal U}}

\newcommand{\W}{{\mathcal W}}
\newcommand{\X}{{\mathcal X}}
\newcommand{\Y}{{\mathcal Y}}
\newcommand{\Z}{{\mathcal Z}}
\renewcommand{\P}{{\mathcal P}}
\newcommand{\PP}{{\mathbb P}}

\newcommand{\RealNumbers}{{\mathbb R}}
\newcommand{\Integers}{{\mathbb Z}}
\newcommand{\CC}{{\mathbb C}}
\newcommand{\ComplexNumbers}{{\mathbb C}}
\newcommand{\RationalNumbers}{{\mathbb Q}}
\newcommand{\LieAlg}[1]{{\mathfrak #1}}
\newcommand{\Reflection}{{\mathfrak R}}
\newcommand{\Pin}{{\rm Pin}}
\newcommand{\Spin}{{\rm Spin}}

\newcommand{\fM}{{\mathfrak M}}

\newcommand{\IsomRightArrowOf}[1]{
\stackrel
{\stackrel{#1}{\cong}}
{\rightarrow}
}

\newcommand{\LongIsomRightArrow}{\stackrel{\cong}{\longrightarrow}}

\newcommand{\RightArrowOf}[1]{\stackrel{#1}{\rightarrow}}

\newcommand{\LongRightArrowOf}[1]{\stackrel{#1}{\longrightarrow}}

\newcommand{\StructureSheaf}[1]{{\mathcal O}_{#1}}

\newcommand{\restricted}[2]{#1_{\mid_{#2}}}

\newcommand{\rank}{{\rm rank}}
\newcommand{\coker}{{\rm coker}}
\newcommand{\Pic}{{\rm Pic}}
\newcommand{\Alb}{{\rm Alb}}
\newcommand{\alb}{{\rm alb}}
\newcommand{\RelAlb}{{\mathcal Alb}}
\newcommand{\Sym}{{\rm Sym}}

\newcommand{\Hom}{{\rm Hom}}
\newcommand{\Aut}{{\rm Aut}}
\newcommand{\End}{{\rm End}}

\newcommand{\SheafHom}{{\mathcal H}om}
\newcommand{\SheafEnd}{{\mathcal E}nd}
\newcommand{\SheafExt}{{\mathcal E}xt}

\newcommand{\Wedge}[1]{\stackrel{#1}{\wedge}}

\newcommand{\Alg}{{\rm Alg}}

\newcommand{\tildeGSplus}[1]{\widetilde{G}(S^+)^{even}_{#1}}
\newcommand{\mon}{{\rm mon}}
\newcommand{\Tildemon}{\widetilde{{\rm mon}}}

\input xy
\xyoption{all}

\begin{document}
\title[The monodromy of generalized Kummer varieties]
{The monodromy of generalized Kummer varieties and algebraic cycles on their intermediate Jacobians}
\author{Eyal Markman}
\address{Department of Mathematics and Statistics, 
University of Massachusetts, Amherst, MA 01003, USA}
\email{markman@math.umass.edu}

\date{\today}



\begin{abstract}
We compute the subgroup of the monodromy group of 
a generalized Kummer variety 
associated to equivalences of derived categories of abelian surfaces. 
The result was previously announced in \cite{MM}.
Mongardi showed that the subgroup constructed here is in fact the whole monodromy group
\cite[Theorem 2.3]{mongardi}. 
As an application we prove the Hodge conjecture for the generic abelian fourfold of Weil type with complex multiplication by an arbitrary imaginary quadratic number field $K$, but with trivial discriminant invariant in $\RationalNumbers^*/Nm(K^*)$. The latter result is inspired by a recent observation of O'Grady that the third intermediate Jacobians of  smooth projective varieties of generalized Kummer deformation type form complete families of abelian fourfolds of Weil type.
Finally, we prove the surjectivity of the Abel-Jacobi map from the Chow group $\CH^2(Y)_0$ of co-dimension two algebraic cycles homologous to zero on every projective irreducible holomorphic symplectic manifold $Y$ of Kummer type onto the third intermediate Jacobian of $Y$, as predicted by the generalized Hodge Conjecture.
\end{abstract}

\maketitle

\tableofcontents

\section{Introduction}
%
\subsection{Monodromy of generalized Kummers}
\label{sec-monodromy-intro}

Let $X$ be a complex projective abelian surface, $X^{(n)}$ its $n$-th symmetric product, and $X^{[n]}$ the Hilbert 
scheme of length $n$ zero dimensional subschemes of $X$. Let
 $\pi: X^{[n]}\rightarrow X$ be the composition of the Hilbert-Chow morphism $X^{[n]}\rightarrow X^{(n)}$ and the summation morphism $X^{(n)}\rightarrow X$.  The generalized Kummer variety $K_X(n\!-\!1)$ is
the fiber of $\pi$ over $0\in X$. 
$K_X(n\!-\!1)$ is a smooth, projective, simply connected variety 
of dimension $2n-2$. It admits a holomorphic symplectic
form, unique up to a constant multiple
\cite{beauville-varieties-with-zero-c-1}. 
The morphism $\pi$ is an isotrivial family, every fiber is isomorphic to $K_X(n\!-\!1)$.
The variety $K_X(1)$ is the Kummer $K3$ surface associated to $X$.
Let $s_n\in H^{even}(X,\Integers)$ be the Chern character of the ideal sheaf 
of a length $n$ subscheme of $X$. 
The moduli space $\M(s_n)$ of rank $1$ torsion free sheaves on $X$ with Chern character $s_n$ is isomorphic to $X^{[n]}\times \Pic^0(X)$.

Set
\[
V \ \ := \ \ H^1(X,\Integers) \oplus H^1(X,\Integers)^*.
\]
$V$ has a natural symmetric unimodular bilinear pairing, 
given by 
\begin{equation}
\label{eq-pairing-on-V-introduction}
((a_1,a_2),(b_1,b_2))=b_2(a_1)+a_2(b_1),
\end{equation}
and $H^*(X,\Integers)$ is the spin representation of the arithmetic group $\Spin(V)$
recalled below in equation (\ref{eq-Spin-Pin-and-G}). 
The half-spin representations are $S^+:=H^{even}(X,\Integers)$
and $S^-:=H^{odd}(X,\Integers)$. Each of the half spin representations admits a symmetric integral and unimodular $\Spin(V)$-invariant bilinear paring, recalled below in (\ref{eq-Mukai-pairing}), and $\Spin(V)$, $\Spin(S^+)$, and $\Spin(S^-)$ all embed 
as the same subgroup of $SO(V)\times SO(S^+)\times SO(S^-)$. The latter identification of the three spin groups is a consequence of an integral version of triality for $\Spin(8)$ (Theorem \ref{thm-triality-principle}).
Let $\tilde{\tau}$ act on $H^i(X,\Integers)$ by
$(-1)^{i(i-1)/2}$ and on $V$ by $\tilde{\tau}(a_1,a_2)=(-a_1,a_2)$. Denote by $G(S^+)^{even}$ the subgroup of $GL(V\oplus S^+\oplus S^-)$ generated by $\Spin(V)$ and $\tilde{\tau}$. The group $G(S^+)^{even}$ arrises naturally as one of the  Clifford groups (see Equation (\ref{eq-Spin-Pin-and-G})). We describe next a monodromy representation of $G(S^+)^{even}$ on the
cohomology ring of the moduli space $\M(s_n)$.

\begin{defi}
\label{def-monodromy}
{\rm
Let $Y$ be a smooth projective  variety. 
An automorphism $g$ of the cohomology ring 
$H^*(Y,\Integers)$ is called a {\em monodromy operator}, 
if there exists a 
family $\Y \rightarrow B$ (which may depend on $g$) 
of compact K\"{a}hler manifolds, having $Y$ as a fiber
over a point $b_0\in B$, 
and such that $g$ belongs to the image of $\pi_1(B,b_0)$ under
the monodromy representation. 
The {\em monodromy group} $Mon(Y)$ of $Y$ is the subgroup 
of $GL(H^*(Y,\Integers))$ generated by all the monodromy operators. 
}
\end{defi}

Let $\P$ be an object in the bounded derived category $D^b(X\times X)$ of coherent sheaves on $X\times X$, which is the Fourier-Mukai kernel of an auto-equivalence $\Phi_\P:D^b(X)\rightarrow D^b(X)$ of the derived category of the abelian surface $X$. Then $ch(\P)$,
considered as a correspondence, induces a automorphism of $H^*(X,\Integers)$, 
which is the image of an element $g$  of  $\Spin(V)$ in
$GL(H^*(X,\Integers))$ via the spin representation, by results of Mukai and Orlov (see Section \ref{sec-derived-categories}). 
Let $\E$ be a universal sheaf over $X\times \M(s_n)$. 
Let $\pi_{ij}$ be the projection from $X\times \M(s_n)\times X\times \M(s_n)$ onto the product of the $i$-th and $j$-th factors.
Set
\[
\gamma_g :=
c_{2n+2}\left(
\pi_{24,*}\left[
\pi_{12}^*\E^*\otimes \pi_{34}^*\E\otimes \pi_{13}^*\P
\right][1]
\right),
\]
where the pull back, push-forward, dual $\E^*$, and tensor product are all taken in the derived category and $[1]$ is the shift. One can
express the right hand side above in terms of the cohomology class $ch(\P)$, using the Grothendieck-Riemann-Roch theorem.  
We get the class 
$\gamma_g$ in $H^{4n+4}(\M(s_n)\times \M(s_n),\Integers)$, associated to every element $g$ of $\Spin(V)$,
by replacing $ch(\P)$ by $g$ in the latter cohomological expression.
Considering also the analogue for compositions of equivalences of derives categories and dualization yields a class $\gamma_g$
for every element $g\in G(S^+)^{even}$ (Equation (\ref{eq-gamma-delta})).
Let $G(S^+)^{even}_{s_n}$ be the subgroup of $G(S^+)^{even}$ stabilizing $s_n$. Define $\Spin(V)_{s_n}$ similarly.
Assume $n\geq 3$.

\begin{thm}
\label{thm-introduction-monodromy-representation-mu}
(Theorem \ref{thm-monodromy-representation-mu}).
\begin{enumerate}
\item
The correspondence $\gamma_g$ induces a graded ring automorphism, for every $g \in G(S^+)^{even}_{s_n}$. The resulting map
\[
\mon: G(S^+)^{even}_{s_n} \rightarrow \Aut H^*(\M(s_n),\Integers)
\]
is a group homomorphism and its image is contained in the monodromy group $Mon(\M(s_n))$.
\item
\label{thm-item-universal-sheaf-is-automorphic}
For every $g\in \Spin(V)_{s_n}$ there exists a topological complex line bundle $L_g$ over $\M(s_n)$, such that
\[
(g\otimes \mon_g)(ch(\E))=ch(\E)\pi^*_{\M}ch(L_g),
\]
where $\pi_{\M}:X\times \M(s_n)\rightarrow \M(s_n)$ is the projection.
\end{enumerate}
\end{thm}

Let $\Gamma_X$ be the group  of points of order $n$ on $X$. The translation action of $\Gamma_X$ on $X$ induces a translation action on $X^{[n]}$.
The morphism $\pi:X^{[n]}\rightarrow X$ is invariant with respect to the latter action and so $\Gamma_X$ acts on $K_X(n\!-\!1).$ This induces an embedding of $\Gamma_X$ in $Mon(K_X(n\!-\!1))$.
The action of $\Gamma_X$ on $H^i(K_X(n\!-\!1),\Integers)$ is trivial, for $i=2,3$,
by Lemma \ref{lemma-Gamma-v}(\ref{lemma-item-Gamma-v-embedds-in-Mon}).

\begin{prop}
\label{prop-introduction-overline-mon}
(Proposition \ref{prop-overline-mon})
There exists a unique injective homomorphism 
\begin{equation}
\label{eq-overline-mon}
\overline{\mon}: G(S^+)^{even}_{s_n} \rightarrow Mon(K_X(n\!-\!1))/\Gamma_X,
\end{equation}
such that the restriction homomorphisms $H^i(\M(s_n),\Integers)\rightarrow H^i(K_X(n\!-\!1),\Integers)$, $i=2,3$, 
are $G(S^+)^{even}_{s_n}$-equivariant with respect to the homomorphisms $\mon$ and $\overline{\mon}$.
\end{prop}

\hide{
We construct a faithful linear representation 
\begin{equation}
\label{eq-mon}
\Tildemon \ : \ \tildeGSplus{s_n} \ \ \ \longrightarrow \ \ \
\Aut [H^*(K_X(n\!-\!1),\Integers)]
\end{equation}
of a group $\tildeGSplus{s_n}$, which is an extension 
\begin{equation}
\label{eq-extension-defining-G-V-n}
1\rightarrow \Gamma_X \rightarrow \tildeGSplus{s_n} 
\rightarrow G(S^+)^{even}_{s_n} \rightarrow 1
\end{equation}
(see Proposition \ref{prop-overline-mon} and the exact sequence (\ref{eq-extension-denining-tildeGSplus}) below).
$G(S^+)^{even}_{s_n}$ is the stabilizer of $s_n$ in the even Clifford group $G(S^+)^{even}$ of $S^+$. 
The group $G(S^+)^{even}_{s_n}$ is itself an extension
\[
1\rightarrow \Spin(V)_{s_n} \rightarrow G(S^+)^{even}_{s_n} \rightarrow 
\Integers/2\Integers \rightarrow 1,
\]
where $\Spin(V)_{s_n}$ is the stabilizer of $s_n$ in $\Spin(V)$.
Recall that $V$ and the half-spin representations of 
$\Spin(V)$ are restrictions 
of representations of a larger group $G(V)^{even}$ of 
even invertible elements of the Clifford algebra conjugating $V$ to itself
(see (\ref{eq-Spin-Pin-and-G})). 
An integral version of triality for $\Spin(V)$ implies that $S^+$ and $S^-$ 
are endowed each with a symmetric bilinear form and $\Spin(V)$ maps into
$SO(V)$, $SO(S^+)$ and $SO(S^-)$. Triality identifies 
$\Spin(V)$ with $\Spin(S^+)$ and $\Spin(S^-)$, as subgroups of
$SO(V)\times SO(S^+)\times SO(S^-)$ (Theorem \ref{thm-triality-principle}).
The identification does not extend to 
$G(V)^{even}$, $G(S^+)^{even}$, and $G(S^-)^{even}$. 
Among the latter three, $G(S^+)^{even}$ is the only one which 
preserves the bilinear form on $S^+$. In addition, $G(S^+)^{even}$ 
surjects onto $SO(S^+)$. 
Conjugation in $\tildeGSplus{s_n}$ induces a representation of
$G(S^+)^{even}_{s_n}$ in $GL(\Gamma_X)$, which is given explicitly in
(\ref{eq-homomorphism-from-stabilizer-in-G-S-plus-even-to-GL}) below. 
The class of the extension (\ref{eq-extension-defining-G-V-n}) is yet to be determined.

The action of the group $\Gamma_X$ on $H^*(K_X(n\!-\!1),\Integers)$ 
comes from an inclusion of $\Gamma_X$ in the automorphism group of 
$K_X(n\!-\!1)$. $\Gamma_X$ acts trivially on $H^i(K_X(n\!-\!1),\Integers)$, for $i=2,3$,
by Lemma \ref{lemma-Gamma-v}(\ref{lemma-item-Gamma-v-embedds-in-Mon}).
The appearance of the stabilizer of $s_n$ is natural at the level of
the second cohomology. Both 
$H^2(K_X(n\!-\!1),\Integers)$ and $H^{even}(X,\Integers)$
are endowed with natural symmetric bilinear pairings and 
$H^2(K_X(n\!-\!1),\Integers)$ is naturally isometric to the orthogonal 
complement $s_n^\perp$ of the Chern character $s_n$ in $H^{even}(X,\Integers)$
\cite{yoshioka-abelian-surface}. 
Each automorphism $\Tildemon(g)$, $g\in \tildeGSplus{s_n}$, 
is  a monodromy operator.
}

Let $Y$ be a hyperk\"{a}hler variety deformation equivalent to the
generalized Kummer $K_X(m)$ of an abelian surface, $m\geq 2$. 
We determine the image $Mon^2(Y)$ of the monodromy group in 
$\Aut[H^2(Y,\Integers)]$.
The second cohomology
$H^2(Y,\Integers)$ admits the symmetric bilinear 
Beauville-Bogomolov-Fujiki pairing \cite{beauville-varieties-with-zero-c-1}.
It has signature $(3,-4)$. 
Given an element $u$ of $H^2(Y,\Integers)$ with
$(u,u)$ equal $2$ or $-2$, let $R_u:H^2(Y,\Integers)\rightarrow H^2(Y,\Integers)$ be the reflection in $u$, 
$R_u(w)=w-2\frac{(w,u)}{(u,u)}u$, and set $r_u:=\frac{(u,u)}{-2}R_u$.
Then $r_u$ is the reflection in $u$, when $(u,u)=-2$, and $-r_u$ is the reflection in $u$, when $(u,u)=2$. 
Set 
\begin{equation}
\label{eq-W}
\W \ \ \ := \ \ \ \langle
r_u \ : \ u \in H^2(Y,\Integers) \ \
\mbox{and} \ \ (u,u) = \pm2
\rangle
\end{equation} 
to be the subgroup of $O(H^2(Y,\Integers))$ generated by the 
elements $r_u$. Then $\W$ is a normal subgroup of finite index
in $O(H^2(Y,\Integers))$. 

The lattice $H^2(Y,\Integers)$ is not unimodular.
The residual group $H^2(Y,\Integers)^*/H^2(Y,\Integers)$ is
cyclic of order $\dim(Y)+2$.
The image of $\W$ in the automorphism group of
$H^2(Y,\Integers)^*/H^2(Y,\Integers)$ has order $2$ and is generated by
multiplication by $-1$, by \cite[Lemma 4.10]{markman-monodromy-I}. We get a character 
\begin{equation}
\label{eq-residue-character}
\chi \ : \ \W \ \ \longrightarrow \ \ \{1,-1\}\subset \ComplexNumbers^\times. 
\end{equation}
The character group $\Hom(\W,\ComplexNumbers^\times)$ 
is isomorphic to $\Integers/2\Integers\times \Integers/2\Integers$
and is generated by  $\chi$ and the determinant character 
$\det$ (the proof is identical to that of
Corollary 7.9 in \cite{markman-monodromy-I}). 
Note that $\det(r_u)=(u,u)/2$ and
$\chi(r_u)=-(u,u)/2$. Consequently, their product $\det\cdot \chi$
takes $r_u$ to $-1$, for both $+2$ and $-2$ vectors $u$. 
Let $\W^{\det\cdot \chi}$ be the kernel of $\det\cdot \chi$.

\begin{thm}
\label{thm-Mon-2}
The image $Mon^2(Y)$ in $O(H^2(Y,\Integers))$, 
of the monodromy group $Mon(Y)$, is equal to $\W^{\det\cdot \chi}$. Consequently, the homomorphism 
\[
\overline{\mon}: G(S^+)^{even}_{s_n} \rightarrow Mon(K_X(n\!-\!1))/\Gamma_X,
\]
 given in
(\ref{eq-overline-mon}), is an isomorphism.
\end{thm}

The inclusion $\W^{\det\cdot \chi}\subset Mon^2(Y)$ is proven in Section \ref{subsec-translation-invariant-subring}.
The reversed inclusion was proven by Mongardi in \cite{mongardi}, using general results about the action of the monodromy of an irreducible holomorphic symplectic manifold\footnote{
An {\em irreducible holomorphic symplectic manifold $Y$} is a simply connected compact K\"{a}hler manifold, such that $H^0(Y,\Omega^2_Y)$ is one-dimensional spanned by a nowhere degenerate $2$-form.
} 
on classes of extremal curves, as well as an additional monodromy constraint proven in \cite[Cor. 4.8]{MM}. Proposition \ref{prop-introduction-overline-mon} and Theorem \ref{thm-Mon-2}
yield the short exact sequence
\[
1\rightarrow \Gamma_X \rightarrow Mon(K_X(n\!-\!1))
\rightarrow G(S^+)^{even}_{s_n} \rightarrow 1
\]
the extension class of which is yet to be determined.

Theorem \ref{thm-Mon-2} implies that the image of $Mon^2(Y)$ in
$O(H^2(Y,\Integers))/(-1)$ has index $2^{\rho(n)}$, where
$n=\frac{\dim Y + 2}{2}$ and 
$\rho(n)$ is the Euler number of $n$ (the number of distinct prime divisors 
of $n$), (see \cite[Lemma 4.2]{markman-constraints}). 
Consequently, the Hodge-isometry class of $H^2(Y,\Integers)$ does not determine
the bimeromorphic class of $Y$, for $Y$ with a generic period.
There are $2^{\rho(n)}$ distinct bimeromorphic classes 
of hyperk\"{a}hler varieties, deformation equivalent to the 
generalized Kummer, for each generic weight $2$ Hodge-isometry class, by Verbitsky's Torelli Theorem for irreducible holomorphic symplectic manifolds \cite[Theorem 1.3]{markman-torelli}. 
In particular, the Kummers 
$K_X(n\!-\!1)$ and $K_{\hat{X}}(n\!-\!1)$,
of a generic complex torus $X$ and its dual $\hat{X}$, are not
bimeromorphic. 
Namikawa proved this counter example in the case of 
Kummer fourfold (where $n=3$) \cite{namikawa}.

\hide{
The analogue of Theorem \ref{thm-Mon-2}, for hyperk\"{a}hler varieties 
deformations equivalent to 
the Hilbert scheme $S^{[n]}$ of length $n$ subschemes on a $K3$
surface $S$, was proven in \cite{markman-constraints}.
It leads to similar counter examples to the generic Torelli question. 
One significant difference exists; $Mon^2(S^{[n]})$ does contain reflections
and is equal to the whole group $\W_{S^{[n]}}$ associated to 
the lattice $H^2(S^{[n]},\Integers)$ by equation
(\ref{eq-W}). 

Let us stress, that there are two natural versions of the 
Generic Torelli Question. The above counter examples provide a 
negative answer to the stronger version, while the weaker version is still 
open. We introduce next these two version of Generic Torelli.
A {\em marking} for a hyperk\"{a}hler variety $Y$ is an isometry
$\eta : H^2(Y,\Integers)\rightarrow \Lambda$ with a fixed lattice $\Lambda$.
There is a (non-Housdorff) moduli space ${\frak M}_\Lambda$, of 
isomorphism classes of marked hyperk\"{a}hler varieties, fixing the lattice
$\Lambda$. The period domain $\Omega_\Lambda$ is an open analytic subset of 
the quadric in $\PP(\Lambda\otimes_\Integers\ComplexNumbers)$.
The marking $\eta$ conjugates $Mon^2(Y)$ to the
subgroup $Mon^2$ in $O(\Lambda)$, which fixes the connected component 
${\frak M}^0_\Lambda$ in ${\frak M}_\Lambda$ through a marked pair 
$(Y,\eta)$. Consider the commutative diagram
\[
\begin{array}{ccc}
{\frak M}^0_\Lambda & \LongRightArrowOf{p} & \Omega_\Lambda
\\
\downarrow & & \downarrow
\\
{\frak M}^0_\Lambda/Mon^2 & \LongRightArrowOf{\bar{p}} & 
\Omega_\Lambda/O(\Lambda),
\end{array}
\]
where $p$ is the period map of the weight $2$ Hodge structure. 
The quotient ${\frak M}^0_\Lambda/Mon^2$ is the moduli ``space'' 
parametrizing isomorphism classes of hyperk\"{a}hler varieties
deformation equivalent to $Y$. The quotient $\Omega_\Lambda/O(\Lambda)$
is the  moduli ``space'' of Hodge isometry classes of 
weight $2$ Hodge structures. 

The {\em Weak Generic Torelli Question} asks if the period map $p$
is generically bijective. The {\em Strong Generic Torelli Question} 
asks that question for the induced map $\bar{p}$ on the
level of equivalence classes. Now, the center $\{\pm 1\}$ of 
$O(\Lambda)$ acts trivially on $\Omega_\Lambda$ and $Mon^2$ 
is known to inject into $O(\Lambda)/\{\pm 1\}$. 
Hence, the degrees of the maps $p$ and $\bar{p}$ are related by
\[
\deg(\bar{p}) \ \ = \ \ \deg(p)\cdot [Mon^2:O(\Lambda)/\{\pm 1\}].
\]
In particular, the index $[Mon^2:O(\Lambda)/\{\pm 1\}]$
is a lower bound for the generic degree of the map $\bar{p}$.

The weak generic Torelli question suggests, that the 
birational class of $Y$, as well as the data encoded
by the cohomology groups $H^i(Y,\Integers)$, for $i>2$, are 
captured by the $Mon^2(Y)$-orbit of the (generic) period of a marked 
hyperk\"{a}hler variety. The $Mon^2(Y)$-orbits are smaller, in general,
than the $O(\Lambda)$-orbits, and hence encode more information 
than the Hodge isometry class of $H^2(Y,\Integers)$. Following is 
a more intrinsic formulation of the above statement, 
for two (unmarked) 
deformation equivalent hyperk\"{a}hler varieties $Y_1$ and $Y_2$. 
Denote by $O^2(Y_1,Y_2)$ the set of isometries from
$H^2(Y_1,\Integers)$ onto $H^2(Y_2,\Integers)$. Let
$O^2_{Hodge}(Y_1,Y_2)$ be the subset of isometries in $O^2(Y_1,Y_2)$
mapping $H^{2,0}(Y_1)$ to $H^{2,0}(Y_2)$. 
Let $O^2_{Mon}(Y_1,Y_2)$ be the subset of isometries induced by 
the monodromy along some path in the base of some family 
of hyperk\"{a}hler varieties, having $Y_1$ and $Y_2$ 
as the fibers over the end points of the path. 
$O^2_{Mon}(Y_1,Y_2)$ is a torsor of the groups $Mon^2(Y_i)$.
If the intersection $O^2_{Hodge}(Y_1,Y_2)\cap O^2_{Mon}(Y_1,Y_2)$ is
empty, then $Y_1$ and $Y_2$ are {\em not} bimeromorphic.
This follows from the fact, that birational hyperk\"{a}hler varieties are 
inseparable points in their moduli \cite{huybrects-basic-results}. 
The weak (generic) Torelli question is equivalent to the question:
{\em Are $Y_1$ and $Y_2$  bimeromorphic whenever 
$O^2_{Hodge}(Y_1,Y_2)\cap O^2_{Mon}(Y_1,Y_2)$ is
non-empty (and the period of $Y_i$ is generic)?} 
Theorem \ref{thm-Mon-2} calculates $O^2_{Mon}(Y_1,Y_2)$, 
whenever at least one element of $O^2_{Mon}(Y_1,Y_2)$ is known. 
Consider, for example, two complex tori $X_i$, $i=1,2$, and let 
$Y_i:=K_{X_i}(m)$. Then 
any orientation-preserving homomorphism from
$H^1(X_1,\Integers)$ to $H^1(X_2,\Integers)$ induces an explicit 
element of $O^2_{Mon}(Y_1,Y_2)$.

The image $Mon^3(Y)$ in $\Aut H^3(Y,\Integers)$, 
of the monodromy group, is descibed intrinsically
below (section ???). We describe here $Mon^3(K_X(m))$ for a generalized
Kummer, $m\geq 2$. The
cohomology group $H^3(K_X(m),\Integers)$ is isomorphic to 
$H^{odd}(X,\Integers)$ and 
is naturally a half-Spin representation of $\Spin(V)$. 
The $G(V,m+1)$-action on $H^3(K_X(m),\Integers)$, 
given in (\ref{eq-mon}), 
factors through the stabilizer $G(S^+)^{even}_{s_{m+1}}$ and the latter is 
represented in $H^3(K_X(m),\Integers)$ via
the restriction of the half-spin representation.
$Mon^3(K_X(m))$
is the image of $G(S^+)^{even}_{s_{m+1}}$ via the representation $\Tildemon$. 
(???) what happen with the sign change in the case of sign reversing
elements?

A central feature of the geometry of generalized Kummer varieties is 
{\em triality} for the integral subgroup $\Spin(V)$ of $\Spin(8)$. Triality identifies 
the period domains for the Hodge structures $H^2(K_X(m),\Integers)$
and $H^3(K_X(m),\Integers)$ of ranks $7$ and $8$ respectively
(see Sections \ref{sec-two-isomorphic-period-domains} and \ref{sec-third-intermediate-jacobians}). 
In Section \ref{sec-comparison-with-Verbitsky-Spin-7-representation} we use 
triality to relate Verbitsky's $\Spin(7)$-representation on the 
cohomology of $K_X(m)$ as the restriction of the $\Spin(8)$ 
representation on the cohomology of the collection of moduli spaces 
of sheaves on $X$ with a primitive Mukai vector. The $\Spin(8)$
representation is related to the group of auto-equivalences 
of the derived category of $X$, studied by Mukai, Polishchuk, and Orlov
\cite{mukai-spin,polishchuk-analogue-of-weil-representation,
orlov-abelian-varieties,golyshev-luntz-orlov,verbitsky-mirror-symmetry}. They show, more generally, that
the group of auto-equivalences 
of the derived category of an abelian variety $X$ of dimension $n$ acts on it cohomology $H^*(X,\Integers)$ via the subgroup of
$\Spin(4n)$ preserving the integral Hodge structure.
}

%
\subsection{A monodromy representation for more general moduli spaces}
Let $w\in H^{even}(X,\Integers)$ be a primitive Hodge class. Denote by $w_i$ the graded summand in $H^{2i}(X,\Integers)$.
Set $(w,w):=\int_X 2w_0w_2-w_1^2$. 
Assume that $w_0\geq 0$ and $(w,w)\leq -6$. 
Choose a $w$-generic ample line bundle $H$ on $X$ (see Section \ref{subsection-mukai-lattice} for the definition).
Then the moduli space $\M_H(w)$ of $H$-stable sheaves with Chern character $w$ is a projective, non-singular, connected, and holomorphic symplectic variety of dimension $2-(w,w)$, by results of Mukai and Yoshioka \cite{yoshioka-abelian-surface}. Yoshioka proved that the Albanese variety of $\M_H(w)$ is isomorphic to $X\times \Pic^0(X)$ and each fiber of the Albanese map 
$\alb:\M_H(w)\rightarrow X\times \Pic^0(X)$ is deformation equivalent to a generalized Kummer variety \cite{yoshioka-abelian-surface}. 
The analogue of Theorem \ref{thm-introduction-monodromy-representation-mu} for $\M_H(w)$ is proved in Corollary  \ref{cor-monodromy-representation-of-spin} below, where a group homomorphism
\[
\mon:G(S^+)^{even}_{w} \rightarrow Mon(\M_H(w))
\]
is constructed.
Proposition \ref{prop-introduction-overline-mon}
is stated and proved in Proposition \ref{prop-overline-mon} for any fiber $K_a(w)$ of the Albanese map 
$\alb:\M_H(w)\rightarrow X\times \Pic^0(X)$, replacing $\Gamma_X$ by the subgroup $\Gamma_w$ of $\Aut(K_a(w))$ 
acting trivially on $H^i(K_a(w),\Integers)$, for $i=2,3$.
The constructed homomorphism
\[
\overline{\mon}:G(S^+)^{even}_{w} \rightarrow Mon(K_a(w))/\Gamma_w
\]
is an isomorphism, by Theorem \ref{thm-Mon-2}.

%
\subsection{The Hodge conjecture for a generic abelian fourfold of Weil type of discriminant $1$}

A  $2n$-dimensional abelian variety $A$ is of {\em Weil type} if there exists an embedding
$\eta:K\hookrightarrow \End(A)\otimes_\Integers\RationalNumbers$  of an imaginary quadratic number field 
$K:=\RationalNumbers[\sqrt{-d}]$, where $d$ is a positive integer, such that the eigenspaces with eigenvalues $\sqrt{-d}$ and $-\sqrt{-d}$ for the action of  $\eta(\sqrt{-d})$ on $H^{1,0}(A)$ are both $n$-dimensional. 
A {\em polarized $2n$-dimensional abelian variety of Weil type} is a triple $(A,K,h),$ with $(A,K)$ as above and $h\in H^{1,1}(A,\Integers)$ an ample class, such that $\eta(\sqrt{-d})^*h=dh$. Any abelian variety of Weil type $(A,K)$ admits such an ample class $h$
\cite[Lemma 5.2(1)]{van-Geemen}. 

Polarized $2n$-dimensional abelian varieties of Weil type come in $n^2$-dimensional families \cite[Sec. 5.3]{weil,van-Geemen}.
The top exterior power $\wedge^{2n}_KH^1(A,\RationalNumbers)$, of $H^1(A,\RationalNumbers)$ as a $K$-vector space,
is naturally embedded as a subspace of $H^{n,n}(A,\RationalNumbers)$, which together with $h^n$ spans a $3$-dimensional subspace
over $\RationalNumbers$.
The generic abelian variety of Weil type $A$ has a cyclic Picard group but a $3$-dimensional $H^{n,n}(A,\RationalNumbers)$
\cite[Theorem 6.12]{weil,van-Geemen}. 
If $A$ is an abelian variety of dimension $4$, which is not isogeneous to a product of abelian varieties, and such that the cup product homomorphism $\Sym^2H^{1,1}(A,\RationalNumbers)\rightarrow H^{2,2}(A,\RationalNumbers)$ is not surjective, then $A$ is of Weil type, by \cite{moonen-zarhin}. This reduced the proof of the Hodge conjecture for abelian fourfolds to those of Weil type, by \cite[Theorem 4.11]{ramon-mari}.

Let $Nm:K^*\rightarrow \RationalNumbers^*$ be the norm homomorphism, sending $a+b\sqrt{-d}$ to $a^2+db^2$.
Associated to the isogeny class of a polarized abelian variety of Weil type $(A,K,h)$ is a discrete invariant in $\RationalNumbers^*/Nm(K^*)$, called its {\em discriminant} \cite[Lemma 5.2(3)]{van-Geemen}. Following is the second main result of this paper, which is proven in Section \ref{sec-hyperholomorphic-sheaves}.

\begin{thm}
\label{thm-intro-hodge-classes-of-weil-type-are-algebraic}
(Theorem \ref{thm-hodge-classes-of-weil-type-are-algebraic})
Let $(A,K,h)$ be a polarized abelian fourfold of Weil type of discriminant $1$. The $3$-dimensional subspace
of $H^{2,2}(A,\RationalNumbers)$, spanned by $h^2$ and $\wedge^{4}_KH^1(A,\RationalNumbers)$, consists of algebraic classes. In particular, the Hodge conjecture holds for the generic such $(A,K,h)$.
\end{thm}

The case $K=\RationalNumbers[\sqrt{-3}]$ was previously proven by Schoen  for arbitrary discriminant in \cite{schoen}
and for $K=\RationalNumbers[\sqrt{-1}]$ and  discriminant $1$ in \cite{schoen1}.

%
\subsection{The modular sheaf over a universal deformation of $\M_H(w)\times \M_H(w)$}
The proof of Theorem \ref{thm-intro-hodge-classes-of-weil-type-are-algebraic}
involves the construction of a coherent sheaf over every $4$-dimensional compact complex torus in a $5$-dimensional family of such tori, which contains a representative of each isogeny class of abelian $4$-folds of Weil type of discriminant $1$.
In Section \ref{sec-period-domain-for-trwo-families} we describe this $5$-dimensional family.
In Section \ref{sec-the-modular-sheaf} we describe the coherent sheaf over $X\times \Pic^0(X)$.
In Section \ref{sec-deforming-the-modular-sheaf} we briefly explain why this coherent sheaf deforms to one over each member of this $5$-dimensional family of complex tori.
In Section \ref{sec-outline-of-proof-of-algebraicity-of-Hodge-Weil-classes} we outline the proof of Theorem \ref{thm-intro-hodge-classes-of-weil-type-are-algebraic} about the algebraicity of the Hodge-Weil classes.
%
\subsubsection{A period domain for two families}
\label{sec-period-domain-for-trwo-families}
Let $X$ be an abelian surface and set $V:=H^1(X,\Integers)\oplus H^1(X,\Integers)^*$, 
$S^+:=H^{even}(X,\Integers)$ and $S^-:=H^{odd}(X,\Integers)$. We endow $S^+$ with the unimodular symmetric bilinear pairing (\ref{eq-Mukai-pairing})
of signature $(4,4)$ 
(minus the Mukai pairing). 
Let $w\in S^+$ be the Chern character of the ideal sheaf of a length $n+1$ subscheme, $n\geq 2$.
Let $w^\perp$ be the sublattice orthogonal to $w$ in $S^+$. Then $w^\perp$ is naturally isometric to $H^2(K_X(n),\Integers)$ (with minus the Beauville-Bogomolov-Fujiki pairing) and
the period domain of $2n$-dimensional  irreducible holomorphic symplectic manifolds deformation equivalent to the generalized Kummer variety $K_X(n)$ is
\[
\Omega_{w^\perp}:=
\{\ell\in \PP(w^\perp\otimes_\Integers\ComplexNumbers) \ : \ (\ell,\ell)=0, \ (\ell,\bar{\ell})<0\},
\]
by \cite{beauville-varieties-with-zero-c-1}.
Choose a basis $\{e_1,e_2,e_3,e_4\}$ of $H^1(X,\Integers)$, compatible with the orientation of $X$, and let $\{e_1^*,e_2^*,e_3^*,e_4^*\}$ be the dual basis. 
Consider the following classes in $\wedge^4V$.
\begin{eqnarray*}
\alpha&:=&{\textstyle \sum}_{i=1}^4e_i\wedge e_i^*,
\\
\beta&:=&e_1\wedge e_2\wedge e_3\wedge e_4,
\\
\gamma&:=&e_1^*\wedge e_2^*\wedge e_3^*\wedge e_4^*,
\\
c_w&:=&-(n+1)^2\alpha^2+4(n+1)^3\beta+4(n+1)\gamma.
\end{eqnarray*}

We refer to $c_w$ as the {\em Cayley class} due to Part (\ref{prop-item-introduction-equation-for-Cayley-class}) of  the following.

\begin{prop}
\label{prop-period-domain-for-two-families}
\begin{enumerate}
\item
\label{prop-item-introduction-equation-for-Cayley-class}
(Prop. \ref{prop-equation-for-Cayley-class})
The class $c_w$ spans the  $1$-dimensional subspace of $\wedge^4_\RationalNumbers V$ invariant under 
$\Spin(V)_w$. 
\item
(Lemmas \ref{lemma-J-ell-as-an-element-of-spin-V}, \ref{lemma-sub-hodge-structures})
$\Omega_{w^\perp}$ is also the period domain of integral weight $1$ Hodge structures $(V,J)$, where $J:V_\RealNumbers\rightarrow V_\RealNumbers$ is a complex structure satisfying:
\begin{enumerate}
\item $V^{1,0}$ and $V^{0,1}$ are maximal isotropic with respect to the symmetric bilinear pairing (\ref{eq-pairing-on-V-introduction}) on $V_\ComplexNumbers$.
\item
Each of the two lifts of 
$J\in SO(V_\RealNumbers)$  to  $\Spin(V_\RealNumbers)$ maps to an involution of $S^+_\RealNumbers$, determined by $J$ up to sign, one of which eigenspaces is a negative definite $2$-dimensional subspace of $w^\perp_\RealNumbers$.
\item
The class $c_w$ is of Hodge type $(2,2)$.
\end{enumerate}
\end{enumerate}
\end{prop}

Note that $\Omega_{w^\perp}$ is an open analytic subset of the quadric of isotropic lines in $w^\perp_\CC\subset S^+_\CC$. The correspondence between periods $\ell\in\Omega_{w^\perp}$ and maximal isotropic subspaces $V^{1,0}$ in $V_\CC$ is the restriction to 
$\Omega_{w^\perp}$ of the well known isomorphism between the quadric of isotropic lines in the half-spin representation $S^+_\CC$ of $\Spin(V_\CC)$ and one of the connected components of the maximal isotropic Grassmanian of $V_\CC$ (see \cite[III.1.6]{chevalley}).

Proposition \ref{prop-period-domain-for-two-families} gives rise to the universal torus
\begin{equation}
\label{eq-universal-torus}
\T\rightarrow \Omega_{w^\perp}
\end{equation}
over the $5$-dimensional period domain $\Omega_{w^\perp}$ and $c_w$ determines an integral class of Hodge type $(2,2)$
in the cohomology $H^4(T_\ell,\Integers)$ of the fiber $T_\ell$ over each $\ell\in \Omega_{w^\perp}$.

\begin{prop}
\label{prop-introduction-complex-multiplication}
(Proposition \ref{prop-Theta-h-is-a-Kahler-form} and Corollary \ref{cor-weil-type})
Let $h\in S^+$ satisfy $(h,w)=0$ and $(h,h)<0$. Set $d:=(w,w)(h,h)/4$. The restriction of the universal torus $\T$ to the four dimensional subspace $\Omega_{\{w,h\}^\perp}\subset \Omega_{w^\perp}$, consisting of periods $\ell$ orthogonal to $h$, is a complete family of polarized $4$-dimensional abelian $4$-fold of Weil type of discriminant $1$ with complex multiplication by $\RationalNumbers[\sqrt{-d}]$.
\end{prop}

All possible imaginary quadratic number fields arise, since the lattice $S^+$ is unimodular. 
The complex multiplication by $\sqrt{-d}$ is best explained by an integral version of triality for $\Spin(8)$. The groups $V$ and $S^-$ are the two half-spin representations of $\Spin(S^+)$. The elements $w, h$ of $S^+$ are elements of the 
Clifford algebra 
\[
C(S^+):=\oplus_{k=0}^\infty (S^+)^{\otimes k}/\langle w_1\otimes w_2+w_2\otimes w_1-(w_1,w_2) \ : \ w_1, w_2\in S^+\rangle. 
\]
The spin representation $V\oplus S^-$ is a $C(S^+)$-module, a fact which corresponds to an algebras
isomorphism  $m:C(S^+)\rightarrow \End(V\oplus S^-)$. Each of $h$ and $w$ maps each of the two direct summands of $V\oplus S^-$ to the other direct summand. Hence, the product $w\cdot h$ maps $V$ to itself. We have $w\cdot h+h\cdot w=(h,w)=0$,
by the defining relation of the Clifford algebra, and so $(m_w\circ m_h)^2=-m_w^2\circ m_h^2=-\frac{(w,w)(h,h)}{4}id_{V\oplus S^-}=(-d)id_{V\oplus S^-}$.

Proposition \ref{prop-introduction-complex-multiplication} was first proved by O'Grady in the following set-up.
The complex torus $T_\ell$, $\ell\in \Omega_{w^\perp}$, is isogenous to the third intermediate Jacobian of every marked $2n$-dimensional 
irreducible holomorphic symplectic manifold $Y$ in $\fM_{w^\perp}^0$ with period $\ell$, by Lemma
\ref{lemma-intermediate-jacobians}. 
O'Grady observed that every ample class $h$ on $Y$, with Beauville-Bogomolov-Fujiki degree $(h,h)=2k$, induces complex multiplication on the intermediate Jacobian with imaginary quadratic field $\RationalNumbers[\sqrt{-d}]$, where 
$d:=(n+1)k$  \cite{ogrady}. 

\hide{
In Section \ref{sec-two-isomorphic-period-domains} we show that 
$\Omega_{w^\perp}$ is also the period domain of integral weight $1$ Hodge structures $(V,J)$, where $J:V_\RealNumbers\rightarrow V_\RealNumbers$ is a complex structure, such that $V^{1,0}$ and $V^{0,1}$ are maximal isotropic with respect to the symmetric bilinear pairing (\ref{eq-pairing-on-V-introduction}) on $V_\ComplexNumbers$, and such that the class 
\[
c_w:=-(n+1)^2\alpha^2+4(n+1)^3\beta+4(n+1)\gamma 
\]
is of Hodge type $(2,2)$, where $\alpha:=\sum_{i=1}^4e_i\wedge e_i^*$, $\beta:=e_1\wedge e_2\wedge e_3\wedge e_4$, and 
$\gamma:=e_1^*\wedge e_2^*\wedge e_3^*\wedge e_4^*.$ 
The above class spans the  $1$-dimensional subspace of $\wedge^4_\RationalNumbers V$ invariant under 
$\Spin(V)_w$ (Prop. \ref{prop-equation-for-Cayley-class}). 
}
%
\subsubsection{The modular sheaf}
\label{sec-the-modular-sheaf}
Following is a second description of the class $c_w$. Let $\M(w)$ be the moduli space of rank $1$ torsion free sheaves on $X$ with Chern character $w$. Then $\M(w)$ is $2n+4$ dimensional. 
Let $\E$ be a universal sheaf over $X\times \M(w)$. Let $\pi_{ij}$ be the projection from 
$\M(w)\times X\times \M(w)$ onto the product of the $i$-th and $j$-th factors. 
Let
\begin{equation}
\label{eq-modular-sheaf}
E:=\SheafExt^1_{\pi_{13}}(\pi_{12}^*\E,\pi_{23}^*\E)
\end{equation}
be the relative extension sheaf over $\M(w)\times \M(w)$. Then $E$ is a reflexive torsion free sheaf of rank $2n+2$, which is locally free
away from the diagonal \cite[Rem. 4.6]{markman-hodge}. 
Given a sheaf $F\in \M(w)$, let $E_F$ be the restriction of $E$ to $\{F\}\times \M(w)$. 
Set 
$A:=X\times\Pic^0(X)$. $A$ acts faithfully on $\M(w)$, the first factor via push-forward by translation automorphisms of $X$ and the second factor by tensorization. For a generic sheaf $F'$ in $\M(w)$, the resulting morphism onto the $A$-orbit of $F'$
\[
\iota_{F'}:A\rightarrow \M(w)
\]
is an embedding. $\Spin(V)_w$ acts on $H^*(\M(w),\Integers)$ via a monodromy action 
(Theorem \ref{thm-introduction-monodromy-representation-mu}) and on $H^*(A,\Integers):=\wedge^*V$ via the natural action on the exterior algebra of the fundamental representation $V$. 

\begin{thm} 
\label{thm-introduction-kappa-class-is-non-zero-and-spin-7-invariant}
\begin{enumerate}
\item 
\label{thm-item-kappa-class-is-non-zero-and-spin-7-invariant}
(Theorem \ref{thm-kappa-class-is-non-zero-and-spin-7-invariant})
The class $c_2(\SheafEnd(E))\in H^4(\M(w)\times\M(w),\Integers)$
is $\Spin(V)_w$-invariant with respect to the diagonal monodromy representation of Theorem \ref{thm-introduction-monodromy-representation-mu}. Similarly, 
the class $c_2\left(\SheafEnd(E_F)\right)\in H^4(\M(w),\Integers)$ is $\Spin(V)_w$-invariant. 
\item
(Cor. \ref{cor-q-w-is-Spin-w-equivariant})
The homomorphism $\iota_{F'}^*:H^*(\M(w),\Integers)\rightarrow H^*(A,\Integers)$ is
$\Spin(V)_w$-equivariant. Hence, the class $\iota_{F'}^*c_2(\SheafEnd(E_F))$
is $\Spin(V)_w$-invariant. 
\item
\label{thm-item-equation-for-Cayley-class}
(Proposition \ref{prop-equation-for-Cayley-class})
The equality $c_w=\iota_{F'}^*c_2(\SheafEnd(E_F))$ holds.
\end{enumerate}
\end{thm}

The $\Spin(V)_w$-invariance of $c_2(\SheafEnd(E))$ follows from the automorphic property of the Chern character of the universal sheaf in Theorem
\ref{thm-introduction-monodromy-representation-mu}(\ref{thm-item-universal-sheaf-is-automorphic}).
The $\Spin(V)_w$-invariance of $c_w$ follows from Theorem \ref{thm-introduction-kappa-class-is-non-zero-and-spin-7-invariant}(\ref{thm-item-equation-for-Cayley-class}).

%
\subsubsection{Deforming the modular sheaf}
\label{sec-deforming-the-modular-sheaf}
We would like to deform the coherent sheaf $\iota_{F'}^*\SheafEnd(E_F)$ on $A=X\times\Pic^0(X)$ to a coherent sheaf over every fiber $T_\ell$, $\ell\in \Omega_{w^\perp}$, of the universal torus (\ref{eq-universal-torus}). This would prove the algebraicity of the $\Spin(V)_w$-invariant Hodge class $c_w$ whenever $T_\ell$ is algebraic. For that purpose we deform the moduli space $\M(w)$ and the sheaf $E_F$. We describe first the deformation of $\M(w)$.

There exists a moduli space $\fM_{w^\perp}$ of marked irreducible holomorphic symplectic manifolds and a surjective and generically injective period map $Per:\fM_{w^\perp}^0\rightarrow \Omega_{w^\perp}$ from each connected component $\fM_{w^\perp}^0$ of 
moduli \cite{huybrechts-torelli,verbitsky-torelli}. Choose a connected component $\fM_{w^\perp}^0$ containing a marked
generalized Kummer. 
There exists a universal family $p:\Y\rightarrow \fM_{w^\perp}^0$
of irreducible holomorphic symplectic manifolds of generalized Kummer deformation type \cite[Theorem 1.1]{markman-universal-family}. 
Pulling back to $\fM_{w^\perp}^0$ the universal torus $\T\rightarrow \Omega_{w^\perp}$ 
constructed in (\ref{eq-universal-torus}) 
via the period map we get a universal fiber product
$Per^*\T\times_{\fM_{w^\perp}^0}\Y$. The latter admits a diagonal action by a trivial group scheme $\underline{\Gamma_w}$ over $\fM_{w^\perp}^0$,
whose quotient is a universal deformation 
\begin{equation}
\label{eq-intro-universal-deformation-of-a-moduli-space-of-sheaves}
\Pi:\M\rightarrow \fM_{w^\perp}^0
\end{equation}
of the moduli space $\M(w)$. The family $\Pi$  is constructed in Equation (\ref{eq-universal-deformation-of-a-moduli-space-of-sheaves}).


\hide{
We use Verbitsky's theory  of hyperholomorphic sheaves \cite{kaledin-verbitski-book} in order to deform the pair
$(\M(w),E_F)$ to a pair $(\M_\ell,E_\ell)$, for every period
$\ell\in \Omega_{w^\perp}$. 
The deformation argument is analogous to the one used in \cite{markman-hodge} to prove a similar result for irreducible holomorphic symplectic manifolds deformation equivalent to moduli spaces of stable sheaves on a $K3$ surface.
}

The above discussion is worked out for more general smooth and compact moduli space $\M(w)$ over $X$, for more general primitive classes $w\in S^+$. We choose $\M(w)$, so that the universal sheaf $\E$ is twisted, and the sheaf $E_F$ is {\em maximally twisted},
i.e., its Brauer class in the analytic Brauer group $H^2_{an}(\M(w),\StructureSheaf{\M(w)}^*)$ has order equal to the rank of $E_F$.
It follows that the sheaf $E_F$ does not have any proper non-trivial subsheaves and 
is thus slope-stable with respect to every K\"{a}hler class on $\M(w)$.
The second Chern class of $\SheafEnd(E_F)$ is $\Spin(V)_w$-invariant, hence it remains of Hodge type $(2,2)$ over the whole of $\fM_{w^\perp}^0$, by Lemma \ref{lemma-invariance-under-diagonal-monodromy-action}. 
A theorem of Verbitsky states that if $G$ is a slope-stable reflexive sheaf and $c_2(\SheafEnd(G))$ remains of Hodge type over 
the whole of $\fM_{w^\perp}^0$, then $G$ deforms to a twisted sheaf over each fiber of the universal family $\Pi$ (see Theorem \ref{thm-deformability} below). Such sheaves $G$ are said to be {\em hyper-holomorphic}.
$E_F$ has these properties and we obtain the following result, which is the main cycle-theoretic construction of the paper.

\begin{thm}
\label{thm-deformability-introduction}
(Theorem \ref{thm-deformability})
The sheaf $E_F$ deforms with $\M(w)$ to a reflexive sheaf, locally free on the complement of a point, over every fiber 
of the universal family $\Pi$ given in (\ref{eq-intro-universal-deformation-of-a-moduli-space-of-sheaves}). Similarly, 
the sheaf $E$ given in (\ref{eq-modular-sheaf}) deforms with $\M(w)\times \M(w)$ to a reflexive sheaf, locally free away from the diagonal, over the cartesian square of  every fiber 
of the universal family $\Pi$. 
\end{thm}

Verbitsky's theorem applies to slope-stable reflexive sheaves over hyperk\"{a}hler manifolds. It applies in our setup, since every
hyperk\"{a}hler structure on a marked irreducible holomorphic symplectic manifold $Y$ with period $\ell$ determines a
hyperk\"{a}hler structure on the complex torus $T_\ell$, $\ell\in \Omega_{w^\perp}$, and both correspond to the same twistor line in $\Omega_{w^\perp}$, by Prop. \ref{prop-Theta-h-is-a-Kahler-form}. Hence, the restriction of the universal family $\Pi:\M\rightarrow \fM_{w^\perp}^0$ to twistor lines in $\fM_{w^\perp}^0$ consists of twistor deformations of the fiber of $\Pi$ as well.

%
\subsubsection{Outline of the proof of Theorem \ref{thm-intro-hodge-classes-of-weil-type-are-algebraic}}
\label{sec-outline-of-proof-of-algebraicity-of-Hodge-Weil-classes}
It remains to prove that the three dimensional subspace of Hodge-Weil classes in $H^{2,2}(T_\ell,\RationalNumbers)$ consists of algebraic classes, when $\ell$ is a period in $\Omega_{\{w,h\}^\perp}$ as in Proposition \ref{prop-introduction-complex-multiplication}. 
The embedding $\iota_{F'}:A\rightarrow \{F\}\times \M(w)$ deforms to an embedding
$\iota:T_\ell\rightarrow \{F_\ell\}\times \M_\ell$ associated to a choice of a pair of points $(F_\ell,F'_\ell)$ in the cartesian square
$\M_\ell\times\M_\ell$ of the fiber over $\ell$ of the universal family $\Pi$ given in (\ref{eq-intro-universal-deformation-of-a-moduli-space-of-sheaves}). Hence, 
the one-dimensional subspace of $\Spin(V)_w$ invariant classes in $H^{2,2}(T_\ell,\RationalNumbers)$ is spanned by the algebraic class $c_w$, by Theorems \ref{thm-introduction-kappa-class-is-non-zero-and-spin-7-invariant}(\ref{thm-item-equation-for-Cayley-class}) and 
\ref{thm-deformability-introduction}. Now $\RationalNumbers[\sqrt{-d}]$ acts on $T_\ell$ via rational correspondences, which are algebraic, and we show that the $\RationalNumbers[\sqrt{-d}]$-translates of $c_w$
and the square $\Theta_h^2$ of the polarization of $T_\ell$ span the three dimensional space of Hodge-Weil classes 
(Theorem \ref{thm-hodge-classes-of-weil-type-are-algebraic}).

%
\subsection{Surjectivity of the Abel-Jacobi map}
Let $\Pi:\M\rightarrow \fM^0_{\omega^\perp}$ be the universal deformation of the moduli space $\M_H(w)$ of sheaves on an abelian surface $X$ given in (\ref{eq-intro-universal-deformation-of-a-moduli-space-of-sheaves}). Assume that the dimension of $\M_H(w)$ is $\geq 8$.
Given $b\in \fM^0_{\omega^\perp}$, let $E_b$ be the
deformation of the modular sheaf  (\ref{eq-modular-sheaf}) constructed in Theorem \ref{thm-deformability-introduction} over the cartesian square $\M_b\times \M_b$ of the fiber $\M_b$ of $\Pi$. Let $e_b:Y_b\rightarrow \M_b$ be the inclusion of a fiber of the albanese map $alb:\M_b\rightarrow \Alb(\M_b)$. 
Let $J^2(Y_b):= H^3(Y_b,\ComplexNumbers)/\left[H^{2,1}(Y_b)+H^3(Y_b,\Integers)\right]$ be the intermediate Jacobian. Assume that $Y_b$ is projective. 
Then 
$J^2(Y_b)$ is the co-domain for the Abel-Jacobi map associated to any family of complex co-dimension $2$ algebraic cycles on $Y_b$ homologous to $0$. 

Given $F\in \M_b$, let $E_F$ be the restriction of $E_b$ to $\{F\}\times \M_b$. 
Fix a point $F_0\in \M_b$ and consider the map
\[
AJ_b: \M_b \rightarrow J^2(Y_b)
\]
sending $F\in \M_b$ to the Abel-Jacobi image of an algebraic cycle representing the Chow class
\begin{equation}
\label{eq-AJ-t}
e_b^*\left[c_2(E_F^\vee\stackrel{L}{\otimes}E_F)-c_2(E_{F_0}^\vee\stackrel{L}{\otimes}E_{F_0})\right],
\end{equation}
where $E_F^\vee$ is $R\SheafHom(E_F,\StructureSheaf{\M_b})$ and the tensor product is taken in the derived category.
The morphism $AJ_b$ factors through a morphism 
\begin{equation}
\label{eq-overline-AJ}
\overline{AJ}_b:\Alb(\M_b)\rightarrow J^2(Y_b),
\end{equation}
since the fibers of $alb$ are simply connected.

\begin{thm}
\label{thm-generalized-Hodge}
The morphism $AJ_b$ is surjective, for every $b\in \fM^0_{\omega^\perp}$ for which $Y_b$ is projective.
\end{thm}

Claire Voisin suggested to the author to prove 
the above theorem as a corollary of Theorem \ref{thm-deformability-introduction}. The proof of Theorem \ref{thm-generalized-Hodge} is provided in Section \ref{sec-generalized-Hodge-conjecture}.

\subsection{Organization of the paper}
\label{sec-organization}
In Section \ref{subsection-mukai-lattice} we recall Yoshioka's result that the fibers of the albanese map of a moduli space of stable sheaves on an abelian surface are irreducible holomorphic symplectic manifolds of generalized Kummer deformation type. We recall also the relationship between the second cohomology of the fibers and the Mukai lattice. 

In Section \ref{sec-spin-8-and-triality} we recall the integral versions of the Clifford algebra and Clifford groups associated to the cohomology of an abelian surface $X$. We then recall the integral version of triality for $\Spin(V)$, where $V$ is the rank $8$ lattice $H^1(X,\Integers)\oplus H^1(X,\Integers)^*$.

In Section \ref{sec-generators-for-the-stabilizer} we identify a set of generators for the stabilizer $\Spin(V)_{s_n}$ of the Chern character $s_n$ of the ideal sheaf of a length $n$ subscheme of an abelian surface.

In Section \ref{sec-equivariance-of-the-universal-sheaf} we construct a class $\gamma_g$ in the cohomology
of a product of two moduli spaces $\M_{H_1}(w_1)\times \M_{H_2}(w_2)$ of $H_i$-stable sheaves with Chern character $w_i$ over an abelian surface $X_i$ associated to a parity preserving isomorphism $g:H^*(X_1,\Integers)\rightarrow H^*(X_2,\Integers)$ satisfying $g(w_1)=w_2$. When $g$ is a parallel-transport operator,\footnote{\label{footnote-def-parallel-transport}
Let $X_1$, $X_2$ be compact K\"{a}hler manifolds. An isomorphism $g:H^*(X_1,\Integers)\rightarrow H^*(X_2,\Integers)$ is a {\em parallel-transport operator}, if there exists a family $\pi:\X\rightarrow B$ of compact K\"{a}hler manifolds, points $b_1, b_2\in B$, isomorphisms $\psi_i:X_i\rightarrow \X_{b_i}$ with the fibers over $b_i$, and a continuous path $\gamma$ from $b_1$ to $b_2$ in $B$, such that $\psi_{2,*}\circ g \circ \psi_1^*$ is induced by
parallel-transport in the local system $R\pi_*\Integers$ along $\gamma$.
}
or when $g$ is induced by an equivalence of the derived categories of $X_1$ and $X_2$, then $g\otimes \gamma_g:H^*(X_1\times \M_{H_1}(w_1))\rightarrow H^*(X_2\times \M_{H_2}(w_2))$ maps the class of a universal sheaf to a class of a universal sheaf and $\gamma_g$ is a parallel transport operator. We show that the assignment $g\mapsto \gamma_g$ is multiplicative and we extend the construction to include the contravariant functor of dualization of sheaves.

In section \ref{sec-derived-categories} we lift certain generators for the stabilizer $\Spin(V)_{s_n}$ found in Section \ref{sec-generators-for-the-stabilizer} to auto-equivalence of the derived category of an abelian surface $X$. 

In Section \ref{sec-monodromy-via-Fourier-Mukai} we use results of Yoshioka to show that the auto-equivalences found in Section
\ref{sec-derived-categories} map sheaves in $\M(s_n)$ to sheaves in $\M(s_n)$. This enables us to prove Theorem \ref{thm-introduction-monodromy-representation-mu} about the monodromy representation $\mon : G(S^+)_{s_n}^{even}\rightarrow \Aut(H^*(\M(s_n),\Integers))$.

In Section \ref{sec-four-groupoids} we extend the latter monodromy representation to an action of a groupoid $\G$ (a category all of which morphisms are isomorphisms). Objects $(X,w,H)$ of $\G$ consist of an abelian surface $X$, a primitive Chern character $w$, and a $w$-generic polarization $H$. Morphisms in $\Hom_\G((X_1,w_1,H_1),(X_2,w_2,H_2))$ are parallel-transport operators 
$\gamma: H^*(\M_{H_1}(w_1),\Integers)\rightarrow H^*(\M_{H_2}(w_2),\Integers)$. A result of Yoshioka 
implies the existence of such morphisms, whenever the dimensions of the two moduli spaces are equal. Yoshioka's result 
enables us to construct the monodromy representation 
$\mon : G(S^+)_{w}^{even}\rightarrow Mon(\M(w))$ 
for all primitive Chern characters $w$. We also establish 
in Section \ref{sec-four-groupoids} the $\Spin(V)_w$-equivariance of the homomorphism
$\iota_{F'}^*:H^*(\M_H(w),\Integers)\rightarrow H^*(X\times \hat{X},\Integers)$ of Theorem \ref{thm-introduction-kappa-class-is-non-zero-and-spin-7-invariant}. 

In Section \ref{subsec-translation-invariant-subring} we prove Theorem \ref{thm-Mon-2} about the monodromy group of generalized Kummer varieties. In section \ref{sec-comparison-with-Verbitsky-Spin-7-representation} we identify the Lie algebra of the Zariski closure of the monodromy group as a subalgebra of the Looijenga-Lunts-Verbitsky Lie algebra.  We then verify that monodromy invariant classes of an irreducible holomorphic symplectic manifold of generalized Kummer deformation type are Hodge classes. 

In Section \ref{sec-cayley-class}  we prove the $\Spin(V)_w$-invariance of $c_2(\SheafEnd(E))$ for the modular sheaf $E$ over $\M(w)\times\M(w)$ in Theorem \ref{thm-introduction-kappa-class-is-non-zero-and-spin-7-invariant}(\ref{thm-item-kappa-class-is-non-zero-and-spin-7-invariant}). We then prove that the Cayley class $c_w$ is the pullback of $c_2(\SheafEnd(E))$ 
(Theorem \ref{thm-introduction-kappa-class-is-non-zero-and-spin-7-invariant}(\ref{thm-item-equation-for-Cayley-class})).

In Section \ref{sec-two-isomorphic-period-domains} we construct the universal torus $\T$, given in Equation (\ref{eq-universal-torus}),  over the period domain $\Omega_{w^\perp}$ of 
irreducible holomorphic symplectic manifolds of generalized Kummer deformation type. 
In Section \ref{sec-polarization-map} we prove Proposition \ref{prop-introduction-complex-multiplication};
we construct the polarization $\Theta_h$ and the complex multiplication for the complex tori with periods in the $4$-dimensional subloci $\Omega_{\{w,h\}^\perp}$ in the $5$-dimensional period domain $\Omega_{w^\perp}$.
In Section \ref{sec-diagonal-twistor-lines} we construct a hyperk\"{a}hler structure on the complex torus $T_\ell$ associated with a K\"{a}hler class on an irreducible holomorphic symplectic manifold  with period $\ell$ (Proposition \ref{prop-Theta-h-is-a-Kahler-form}). In Section \ref{subsection-abelian-fourfolds-of-Weil-type} we prove that the subloci $\Omega_{\{w,h\}^\perp}$ parametrize abelian fourfolds of Weil-type of discriminant $1$. In Section \ref{subsection-a-universal-deformation-of-a-moduli-space-of-sheaves} we construct the universal deformation $\Pi:\M\rightarrow \fM^0_{w^\perp}$ of the moduli space of sheaves over the moduli space of marked irreducible holomorphic symplectic manifolds of generalized Kummer deformation type.
In Section \ref{sec-third-intermediate-jacobians} we prove that the torus $T_\ell$ is isogenous to the $3$-rd intermediate jacobian of the irreducible holomorphic symplectic manifold of generalized Kummer deformation type with period $\ell$.

In Section \ref{sec-hyperholomorphic-sheaves} we prove Theorem \ref{thm-intro-hodge-classes-of-weil-type-are-algebraic}
about the algebraicity of the Hodge-Weil classes on abelian fourfolds of Weil-type of discriminant $1$.

In Section \ref{sec-generalized-Hodge-conjecture} we prove Theorem \ref{thm-generalized-Hodge}
verifying the generalized Hodge conjecture for codimension $2$ algebraic cycles homologous to $0$ on every projective irreducible holomorphic symplectic manifold of generalized Kummer deformation type. 
\section{Notation}
\label{sec-notation}
Let $X$ be a complex projective abelian surface, $\hat{X}:=\Pic^0(X)$ its dual surface, 
and $X^{[n]}$ the Hilbert scheme of length $n$ zero-dimensional subschemes 
of $X$.
Let $X^{(n)}$ be the $n$-th symmetric product of $X$ and 
$
hc \ : \  X^{[n]} \rightarrow X^{(n)}
$
the Hilbert-Chow morphism. Let
$\sigma  :  X^{(n)} \rightarrow X$
be the morphism sending a cycle to its sum and let
$
\pi : X^{[n]} \rightarrow X
$
be the composition $\sigma\circ hc$. The  $(2n-2)$-dimensional generalized Kummer
$K_X(n\!-\!1)$ is the fiber of $\pi$ over $0\in X$. 
$K_X(m)$ is a simply connected projective holomorphically symplectic
variety. The algebra of holomorphic
differential forms on $K_X(m)$ is generated by the holomorphic  symplectic
form, which is unique up to a constant scalar 
\cite{beauville-varieties-with-zero-c-1}. Consequently, 
$K_X(m)$ admits a hyperk\"{a}hler structure. The first three Betti numbers $b_i(K_X(m))$
of $K_X(m)$, $m>1$, are $b_1(K_X(m))=0$, $b_2(K_X(m))=7$, and $b_3(K_X(m))=8$, by \cite[Theorem 2.4.11]{gottsche}.
\hide{
Let $\Gamma_X$ be the subgroup of $X$ of points of order $n$ and
$\nu: X\rightarrow X/\Gamma_X\cong X$ the quotient map.
We have a Cartesian diagram
\[
\begin{array}{ccc}
X\times K_X(n\!-\!1) & \LongRightArrowOf{\tilde{\nu}} & X^{[n]}
\\
\downarrow \ \hspace{3ex} & & \hspace{1ex} \ \downarrow \ \pi
\\
X  \ \hspace{3ex} & \LongRightArrowOf{\nu} & X,
\end{array}
\]
where the left vertical morphism is the projection.
}

Let $H$ be an ample line bundle on an abelian surface $X$. 
$\M_H(v)$ denotes the moduli space of Gieseker-Simpson $H$-stable coherent sheaves on $X$ with Chern character $v$.
We always assume that $v$ is primitive and $H$ is $v$-generic (see Section \ref{subsection-mukai-lattice}), so that $\M_H(v)$ is smooth and projective.
We denote by $K_a(v)$ the fiber of the Albanese morphism from $\M_H(v)$ to its Albanese variety over the point $a$. 
Then $K_0(s_{n})=K_X(n\!-\!1)$, when $s_n$ is the Chern character of the ideal sheaf of a length $n$ subscheme.  

When a citation of the form \cite[Theorem 6.12]{weil,van-Geemen} is used, then the statement number Theorem 6.12 is associated to the last entry \cite{van-Geemen} in the list.

A glossary of notation is included in Section \ref{sec-glossary-of-notation}.


%
\section{The Mukai lattice and the second cohomology of a moduli space of sheaves}
\label{subsection-mukai-lattice}
Let $X$ be an abelian surface. Set $S^+:=\oplus_{i=0}^2H^{2i}(X,\Integers)$. 
The {\em Mukai pairing} on $S^+$ is given by
\begin{equation}
\label{eq-Mukai-pairing-with-sign-as-in-Mukai}
\langle x,y\rangle := \int_X(x_1y_1-x_0y_2-x_2y_0),
\end{equation}
where $x=(x_0,x_1,x_2)$, $y=(y_0,y_1,y_2)$, and
$x_i$, $y_i$ are the graded summands in $H^{2i}(X,\Integers)$. We refer to $S^+$ as the {\em Mukai lattice} of $X$
and elements of $S^+$ will be called {\em Mukai vectors}.
Following Mukai we endow $S^+_\ComplexNumbers:=S^+\otimes_\Integers\ComplexNumbers$ with a weight $2$ Hodge structure
by setting $(S^+)^{2,0}:=H^{2,0}(X),$ $(S^+)^{0,2}:=H^{0,2}(X),$ and 
$(S^+)^{1,1}:=H^{0}(X)\oplus H^{1,1}(X)+H^4(X).$

Let $v=(r,c_1,\chi)\in S^+$ be a primitive Mukai vector, 
with $c_1\in H^2(X,\Integers)$
of Hodge type $(1,1)$. 
There is a system of hyperplanes in the ample cone of $X$, called $v$-walls,
that is countable but locally finite \cite[Ch. 4C]{huybrechts-lehn-book}.
An ample class is called {\em $v$-generic}, if it does not
belong to any $v$-wall. Choose a $v$-generic ample class $H$. 
Assume that the moduli space $\M_H(v)$, 
of Gieseker-Simpson $H$-stable sheaves on $X$ with Chern character 
$v$, is non-empty. Then $\M_H(v)$ is smooth, connected, projective, and holomorphic symplectic of dimension
$\langle v,v\rangle+2$ \cite{mukai-symplectic-structure}. 
The Albanese variety of $\M_H(v)$ is isomorphic to $X\times \hat{X}$, whenever the dimension of $\M_H(v)$ is at least $4$, 
and each fiber of the Albanese map $alb:\M_H(v)\rightarrow X\times\hat{X}$ is simply connected and deformation equivalent to 
the generalized Kummer variety of the same dimension \cite[Theorem 0.2]{yoshioka-abelian-surface}.

Denote by $K_a(v)$ the fiber of $alb$ over a point $a\in X\times \hat{X}$.
The second cohomology $H^2(K_a(v),\Integers)$ is endowed with a 
natural symmetric bilinear form, the Beauville-Bogomolov-Fujiki form 
\cite{beauville-varieties-with-zero-c-1}. 
Let $\E$ be a quasi-universal family over $X\times \M_H(v)$ of similitude $\sigma$, so that $\E$ restricts to $X\times\{t\}$, $t\in \M_H(v)$,
as the direct sum $E_t^{\oplus \sigma}$ of $\sigma$ copies of a sheaf $E_t$ representing the isomorphism class $t$.
Let $v^\perp$ be the sublattice of $S^+$ orthogonal to $v$ with respect to the Mukai pairing. Denote by $\pi_X$ and $\pi_\M$ the
projections from $X\times \M_H(v)$ onto the corresponding factor. Given  a Mukai vector $y=(y_0,y_1,y_2)$ set
$y^\vee:=(y_0,-y_1,y_2)$.
Let
\[
\theta':v^\perp\rightarrow H^2(\M_H(v),\Integers)
\]
be given by 
\[
\theta'(y)=
-\frac{1}{\sigma}\left[\pi_{\M,*}\left(ch(\E)\pi_X^*y^\vee\right)\right]_1,
\]
where the subscript $1$ denotes the graded summand in $H^2(\M_H(v),\Integers)$.
We denote the composition of $\theta'$ with restriction to the Albanese fiber by
\begin{equation}
\label{eq-mukai-isomorphism}
\theta : v^\perp \ \ \rightarrow \ \ H^2(K_a(v),\Integers).
\end{equation}

\begin{thm}
\label{thm-yoshioka}
\cite[Theorem 0.2]{yoshioka-abelian-surface}. If the dimension of $\M_H(v)$ is $\geq 8$, then
the homomorphism $\theta$ is a Hodge isometry with respect to the Mukai pairing and the Hodge structure on $v^\perp$ and the Beauville-Bogomolov-Fujiki pairing on $H^2(K_a(v),\Integers)$.
\end{thm}

\section{$\Spin(8)$ and triality}
\label{sec-spin-8-and-triality}
Let $X$ be an abelian surface. Set $S:=H^*(X,\Integers)$, $S^+:= H^{even}(X,\Integers)$, $S^-:=H^{odd}(X,\Integers)$, and
$V:= H^1(X,\Integers)\oplus H^1(\hat{X},\Integers)$.
In Section \ref{sec-Clifford-groups}
we review the basic facts about Clifford algebras and groups associated to the abelian surface $X$ via the rank $8$ lattice 
$V$. In Section \ref{sec-bilinear-operations}
we review the bilinear operations $V\otimes S^+ \rightarrow S^-$, 
$V\otimes S^- \rightarrow S^+$, and $S^+\otimes S^-\rightarrow V$ 
induced by the Clifford product. In Section \ref{sec-triality} we review triality for the arithmetic subgroup 
$\Spin(V)$
of $\Spin(8)$.

%
\subsection{The Clifford groups}
\label{sec-Clifford-groups}
 Endow $V$ with the symmetric bilinear form
\begin{equation}
\label{eq-pairing-on-V}
(a,b)_V \ \ := \ \ b_2(a_1)+a_2(b_1), 
\end{equation}
where  $a=(a_1,a_2)$, $b=(b_1,b_2)$, $a_1,b_1\in H^1(X,\Integers)$, 
and $a_2,b_2\in H^1(\hat{X},\Integers)$, and we use
the natural identification of $H^1(\hat{X},\Integers)$ with 
$H^1(X,\Integers)^*$. Then $V$ is an even unimodular lattice,
the orthogonal direct sum of four copies of the hyperbolic plane, 
and $Q(a):=\frac{1}{2}(a,a)_V$ is an integral quadratic form. 
Let $C(V)$ be the Clifford algebra, i.e., the quotient of the 
tensor algebra of $V$ by the relation 
\begin{equation}
\label{eq-defining-relation-of-Clifford-algebra}
v\cdot w + w\cdot v \ \ = \ \ (v,w)_V.
\end{equation}
As a general reference on Clifford algebras and the theory of spinors we recommend \cite{chevalley}.
The free abelian group $H^*(X,\Integers)$ is isomorphic to $\Wedge{*} H^1(X,\Integers)$ and 
is a $C(V)$-module and $C(V)$ maps isomorphically onto
$\End[H^*(X,\Integers)]$, by \cite[Prop. 3.2.1(e)]{golyshev-luntz-orlov}.
The $C(V)$-module structure of $H^*(X,\Integers)$ is seen as follows.
$V$ embeds in $C(V)$. $V$ embeds also in $\End[H^*(X,\Integers)]$ by sending $w\in H^1(X,\Integers)$ to the 
wedge product 
\begin{equation}
\label{eq-left-wedge-by-w}
L_w \ \ \ := \ \ \ w\wedge (\bullet) 
\end{equation}
and 
$\theta \in H^1(\hat{X},\Integers)\cong H^1(X,\Integers)^*$ 
to the corresponding derivation $D_\theta$ of
$H^*(X,\Integers)$ sending $H^i(X,\Integers)$ to $H^{i-1}(X,\Integers)$. The resulting homomorphism $m:V\rightarrow \End[H^*(X,\Integers)]$ satisfies the analogue of the relation (\ref{eq-defining-relation-of-Clifford-algebra}), i.e., $m(v)\circ m(w)+m(w)\circ m(v)=(v,w)id_{H^*(X,\Integers)}$. Hence $m$ extends to an algebra homomorphism $m:C(V)\rightarrow \End[H^*(X,\Integers)]$, by the universal property of $C(V)$.

We let 
\begin{equation}
\label{eq-tau}
\tau: C(V) \ \ \ \rightarrow \ \ \ C(V) 
\end{equation}
be the {\em main anti-automorphism},
sending $v_1\cdot \cdots \cdot v_r$ to $v_r\cdot \cdots \cdot v_1$. Let
\begin{equation}
\label{eq-main-involution}
\alpha \ : \  C(V) \ \ \ \rightarrow \ \ \ C(V) 
\end{equation}
the {\em main involution},
acting as multiplication by $-1$ on $C^{odd}(V)$ and as the identity on 
$C^{even}(V)$. The {\em conjugation} $x\mapsto x^*$ is the composition
of $\tau$ and $\alpha$. 
Set\footnote{In reference \cite{chevalley} these groups are called {\em Clifford groups} of the orthogonal group of $V$ with respect to the quadratic form $Q$. The group $G(V)$ is denoted in \cite{chevalley} by $\Gamma$, the group $G_0(V)$ by $\Gamma_0$, 
the group $G(V)^{even}$ by $\Gamma^+$, and the group $\Spin(V)$ by $\Gamma^+_0$.
The main anti-automorphism $\tau$ is denoted in \cite[II.3.5]{chevalley} by $\alpha$.}
\begin{eqnarray}
\nonumber
\Spin(V) & := & 
\{x \in C(V)^{even} \ : \ x\cdot x^*=1 \ \mbox{and} \  
x\cdot V \cdot x^*\subset V\},
\\
\nonumber
\Pin(V) & := & 
\{x \in C(V) \ : \ x\cdot x^*=1 \ \mbox{and} \  
x\cdot V \cdot x^*\subset V\},
\\
\nonumber
G_0(V) & := & 
\{x \in C(V) \ : \ x\cdot \tau(x)=1 \ \mbox{and} \  
x\cdot V \cdot \tau(x)\subset V\},
\\
\label{eq-Spin-Pin-and-G}
G(V) & := & 
\{x \in C(V)^\times 
\ :  \ x\cdot V \cdot x^{-1}\subset V\}
\\
G(V)^{even} & := & G(V)\cap C(V)^{even}.
\end{eqnarray}

The condition $x\cdot V \cdot x^{-1}\subset V$ in the definition of $G(V)$, 
combined with the fact that $V$ has even rank, imply that $x$ is 
either even or odd and $G(V)$ is contained in the union 
$C(V)^{even}\cup C(V)^{odd}$ \cite[II.3.2]{chevalley}.
Note that $x\in C(V)$ is invertible, if and only if $x\cdot x^*$ is, i.e., 
if and only if $x\cdot x^*=\pm 1$. Thus, $\Pin(V)$ is an index $2$
subgroup of $G(V)$. We get the extension 
\[
1 \rightarrow \Spin(V) \rightarrow G(V)^{even} \rightarrow 
\Integers/2\Integers \rightarrow 0.
\]
Given a vector $v\in V$, then $v^*=-v$ and $v\cdot v^*=-v\cdot v=-Q(v)$.
Given vectors $v_1$, $v_{-1}$ in $V$, with $Q(v_1)=1$ and $Q(v_{-1})=-1$,
then $v_{-1}$ belongs to $\Pin(V)$, $v_1$ belongs to $G(V)$,
and $(v_1\cdot v_{-1})$ belongs to $G(V)^{even}$. 

The standard representation $V$ is defined\footnote{Chevalley denotes the standard representation by $\chi  :  G(V) \rightarrow O(V)$ 
\cite[Sec. 2.3]{chevalley}. The convention in \cite[Sec. 2.7]{deligne-spinors} is different and there 
$\rho(x)(v)=\alpha(x)\cdot v \cdot x^{-1}$. The two representations agree on $G(V)^{even}$. 
}
by the homomorphism
\begin{eqnarray}
\label{eq-standard-representation-of-G-V}
\rho \ : \ G(V) & \longrightarrow & O(V)
\\
\nonumber
\rho(x)(v) & = & x\cdot v \cdot x^{-1}.
\end{eqnarray}
If $Q(v)=\pm 1$, then $-\rho(v)$ is reflection in $v$
\begin{equation}
\label{eq-Pin-acts-by-reflections}
-\rho(v)(w) \ \ = \ \ w \ - \ \frac{2(v,w)_V}{(v,v)_V}\cdot v, 
\ \ \ \forall w \in V.
\end{equation}

Let
\begin{equation}
\label{eq-m-from-C-V}
m:C(V)\rightarrow \End(S)
\end{equation} 
be the Clifford algebra representation and denote by
\begin{equation}
\label{eq-Cl}
m:G(V)\rightarrow GL(S)
\end{equation}
 as well its restriction to the subset $G(V)\subset C(V)$.

The character group $\Hom(O(V),\ComplexNumbers)$ of $O(V)$ is isomorphic
to $\Integers/2\Integers\times \Integers/2\Integers$ and is generated
by the determinant character $\det$ and the {\em orientation character}
\begin{equation}
\label{eq-orientation-character}
ort \ : \ O(V) \ \ \longrightarrow \ \ \{\pm 1\}.
\end{equation}
The latter is defined as follows. The positive cone in 
$V\otimes_\Integers\RealNumbers$ is homotopic to the $3$-sphere. 
The character $ort$ represents the action of isometries on the third 
singular cohomology of the positive cone \cite[Sec. 4]{markman-torelli}. 
Denote by $O_+(V)$ the kernel of $ort$. 

\begin{new-lemma}
\label{lemma-spin-surjects}
The homomorphism $\rho$ is surjective and it maps $\Pin(V)$ onto $O_+(V)$ and
$\Spin(V)$ onto $SO_+(V)$. 
\end{new-lemma}
\begin{proof}
The lattice $V$ is the orthogonal direct sum of four copies of the even unimodular rank $2$ hyperbolic lattice $U$. Hence, 
$O(V)$ is generated by the reflections $-\rho(v)$, given in (\ref{eq-Pin-acts-by-reflections}), in elements $v\in V$ with $(v,v)_V=\pm 2$, by \cite[4.3]{wall}. The element $-id_U\in O(U)$ is a product of two reflections. Hence, $-id_V$ is a product of $8$ reflections and so $O(V)$ is generated by the set $\{\rho(v) \ : \ v\in V, (v,v)_V=\pm 2\}$.
\end{proof}
The Lemma implies that 
the homomorphism $\rho$ pulls back the character 
$ort$ to the quotient homomorphism
\begin{eqnarray}
\label{eq-ort}
ort \ : \ G(V) & \longrightarrow & G(V)/\Pin(V) \ \ \cong \ \ 
\{\pm 1\},
\\
\nonumber
x & \mapsto & x\cdot x^* \in  \{\pm 1\}.
\end{eqnarray}
The determinant character of $O(V)$ is pulled back via $\rho$ to the {\em parity character}
\[
p \ : \ G(V) \ \ \ \rightarrow \ \ \ G(V)/G(V)^{even} \ \cong \ \{\pm 1\}.
\]
The product of the parity and orientation characters is 
the {\em norm character} denoted\footnote{In Chevalley's book  the norm character is denoted by $\lambda$
and in Deligne's notes by $N$ \cite{chevalley,deligne-spinors}.} 
\begin{eqnarray}
\label{eq-norm}
N \ : \ G(V) & \rightarrow & \{\pm 1\}
\\
\nonumber
x & \mapsto & x\cdot \tau(x). 
\end{eqnarray}
The subgroup $G_0(V)$ of $G(V)$ is the kernel of $N$.


$H^*(X,\Integers)$ is a subalgebra of $C(V)$, 
invariant under both $\tau$ and $\alpha$, and $\tau$ acts on 
$H^i(X,\Integers)$ by $(-1)^{i(i-1)/2}$. 
There is a natural symmetric pairing\footnote{This pairing is denoted by $\beta(s,t)$ in \cite[Sec. 3.2]{chevalley}. 
Another natural symmetric pairing $\tilde{\beta}(s,t):=\int_X s^*\cup t$ is considered in \cite[III.2.5]{chevalley}.
The isomorphism $S\rightarrow S^*$ induced by $\tilde{\beta}$ is the one induced by the Fourier-Mukai functor in Equation 
(\ref{eq-cohomological-fourier-mukai-homomorphism}) below. We chose to use the former pairing $\beta$, since it is the pairing used in the Triality Chapter IV of \cite{chevalley}.
} 
on $H^*(X,\Integers)$  defined by 
\begin{equation}
\label{eq-Mukai-pairing}
(s,t)_S \ \ := \ \ \int_X\tau(s)\cup t.
\end{equation}
We set $Q_S(s):=\frac{1}{2}(s,s)_S$, for $s\in S:=H^*(X,\Integers)$.
Given $x\in G(V)$, we have the equality 
\[
(x\cdot s,x\cdot t)_S \ \ \ = \ \ \ N(x)(s,t)_S,
\]
where $N$ is the norm character (\ref{eq-norm}). In particular,
an element $v\in V$ with $Q(v)=\pm 1$, satisfies
\begin{equation}
\label{eq-minus-1-vectors-do-not-act-via-isometries-on-spin-representation}
(v\cdot s,v\cdot t)_S \ \ \ = \ \ \ Q(v)(s,t)_S.
\end{equation}
The latter equation holds for every $v\in V$ \cite[IV.2.3]{chevalley}.

Let $(\bullet,\bullet)_{S^+}$ and $(\bullet,\bullet)_{S^-}$ 
be the restrictions of the pairing (\ref{eq-Mukai-pairing})  
to $S^+$ and $S^{-}$. Both are even symmetric 
unimodular bilinear pairings, which are $\Spin(V)$-invariant. 
Note that $-(\bullet,\bullet)_{S^+}$ is the pairing given in (\ref{eq-Mukai-pairing-with-sign-as-in-Mukai}), which is  known as the
{\em Mukai pairing} \cite{mukai-hodge,yoshioka-abelian-surface}.
The pairing on $V$ is $G(V)$ invariant. 
The pairings on $S^+$ and $S^-$ are {\em not} invariant under $G(V)^{even}$.
Rather, the pairing spans a rank $1$ module in
$\Sym^2(S^+)$ corresponding to the character (\ref{eq-ort}).
Let $\widetilde{O}(S^+)$ be the subgroup of 
$GL(S^+)$ preserving the bilinear pairing 
$(\bullet,\bullet)_{S^{+}}$ only up to sign. 
Define $\widetilde{O}(S^-)$ similarly. 
Let 
$\tilde{\alpha}\in GL(S)$
be the element acting as the identity on $S^+$ and by $-1$ on $S^-$.
It follows from the Triality Principle recalled below that $\tilde{\alpha}$ corresponds to the action of $j(-id)$, where $j$ 
is an order $3$ outer automorphism of $\Spin(V)$, and so $\tilde{\alpha}$ is the image of a unique element of $\Spin(V)$ in $GL(S)$ (Theorem \ref{thm-triality-principle}). Denote by 
\begin{equation}
\label{eq-central-element-tilde-alpha}
\tilde{\alpha}\in \Spin(V)
\end{equation}
the corresponding central element. 
Set $S\widetilde{O}(S^+):=\widetilde{O}(S^+)\cap SL(S^+)$ and $S\widetilde{O}(S^-):=\widetilde{O}(S^-)\cap SL(S^-)$.
$G(V)^{even}$ maps to $S\widetilde{O}(S^+)$ with kernel generated by 
$\tilde{\alpha}$, to $S\widetilde{O}(S^-)$ with kernel generated by 
$-\tilde{\alpha}$, and onto $SO(V)$ with kernel generated by $-1$. 
The image of $G(V)^{even}$ has index $2$ in
$S\widetilde{O}(S^+)$. 

%
\subsection{Bilinear operations via the Clifford product}
\label{sec-bilinear-operations}

The Clifford product defines the homomorphisms
$V\otimes S^+ \rightarrow S^-$, 
$V\otimes S^- \rightarrow S^+$, and
\begin{equation}
\label{eq-multiplication-of-S-plus-and-S-minus}
S^+\otimes S^- \ \ \longrightarrow \ \ V. 
\end{equation}
The latter is the composition of $V^*\rightarrow V$, given by 
the bilinear pairing $(\bullet,\bullet)_V$, and the homomorphism
$S^+\otimes S^- \rightarrow V^*$,
sending $s\otimes t$ to $((\bullet)\cdot s,t)_{S^{-}}\in V^*$.

\begin{example}
\label{example-Clifford-multiplication-by-s-n}
{\rm 
Let us calculate the homomorphisms
\begin{equation}
\label{eq-Clifford-multiplication-by-s-n}
m_{s_n} \ : \ V \ \ \rightarrow \ \ S^-
\end{equation}
corresponding to Clifford multiplication $(\bullet)\cdot s_n$ 
with the Chern character 
$s_n=(1,0,-n)\in S^+$ of the ideal sheaf of a length $n$ subscheme. 
Let $v=(w,\theta)\in V$, where $w\in H^1(X,\Integers)$ and
$\theta\in H^1(X,\Integers)^*$. Then
\[
m_{s_n}(v) \ := \ v\cdot (1,0,-n) \ = \ w\wedge (1,0,-n) + D_\theta(1,0,-n) \ = \
w - n D_\theta([pt]) = w + n PD^{-1}(\theta).
\]
In the last equality we used the isomorphism 
\begin{equation}
\label{eq-Poincare-Duality}
PD :  H^i(X,\Integers) \rightarrow H^{4-i}(X,\Integers)^*
\end{equation}
given by $\beta \mapsto \int_{X}\beta\wedge (\bullet)$. 
The equality
\begin{equation}
\label{eq-PD-of-D-theta-of-the-class-of-a-point}
[PD(D_\theta([pt]))](x):=\int_X  D_\theta([pt])\wedge x \ \ = \ \ -\theta(x), \ \ 
\forall x \in H^1(X,\Integers),
\end{equation}
follows from
\[
0 = D_\theta([pt]\wedge x) = D_\theta([pt])\wedge x + \theta(x)[pt].
\]
Note that the homomorphism (\ref{eq-Clifford-multiplication-by-s-n}) 
is equivariant with respect to the action of the subgroup
$G(V)^{even}_{s_n}$ of $G(V)^{even}$ stabilizing $s_n$. The integral 
representations $V$ and $S^-$ restrict to $G(V)^{even}_{s_n}$ as isogenous 
representations, which are not isomorphic. 
The co-kernel of $m_{s_n}$ is $H^3(X,\Integers/n\Integers)$ 
and is naturally identified with $H_1(X,\Integers/n\Integers)$
as well as with the group $\Gamma_X$ of points of order $n$ on $X$. 

The homomorphism (\ref{eq-multiplication-of-S-plus-and-S-minus}) 
restricts to $s_n\otimes S^-$ as a homomorphism
\begin{equation}
\label{eq-Clifford-multiplication-by-s-n-in-C-S-plus}
m_{s_n}^\dagger \ : \ S^- \ \ \rightarrow \ \ V,
\end{equation}
which is the adjoint of (\ref{eq-Clifford-multiplication-by-s-n})
with respect to the pairings $(\bullet,\bullet)_V$ and 
$(\bullet,\bullet)_{S^-}$. Explicitly, given $(w,\beta)\in S^-$,
with $w\in H^1(X,\Integers)$ and $\beta\in H^3(X,\Integers)$, then
\begin{equation}
\label{eq-s-n-dagger-from-S-minus-to-V}
m_{s_n}^\dagger(w,\beta) \ \ \ = \ \ \ (-nw,-PD(\beta)).
\end{equation}
In particular, $m_{s_n}^\dagger\circ m_{s_n}:V\rightarrow V$ and 
$m_{s_n}\circ m_{s_n}^\dagger:S^-\rightarrow S^-$ are both
multiplication by $Q_{S^+}(s_n)=-n$. 
}
\end{example}

\begin{rem}
\label{rem-Z-w}
Let $m_{s_n}:V\rightarrow S^-$ be the homomorphism given in (\ref{eq-Clifford-multiplication-by-s-n}).
Given a primitive element $w\in S^+$ with $(w,w)_{S^+}=-2n$, there exists an element $g\in \Spin(V)$ satisfying 
$g(s_n)=w$. Given $x\in V$, and regarding $x$ and $g$ as elements of $C(V)$, we have the equalities
\[
g(m_{s_n}(x))=g\cdot x \cdot s_n=(g\cdot x \cdot g^{-1})\cdot (g \cdot s_n)=m_w(\rho(g)(x)).
\]
Let $m_{s_n}^\dagger :S^-\rightarrow V$ be the adjoint of $m_{s_n}$ and define $m_w^\dagger$ similarly. 
We get the commutative diagram
\[
\xymatrix{
V\ar[r]^{m_{s_n}}\ar[d]_{\rho(g)} & S^-\ar[d]^{g} \ar[r]^{m_{s_n}^\dagger} &V\ar[d]_{\rho(g)}
\\
V\ar[r]_{m_w} & S^-\ar[r]_{m_w^\dagger} &V.
}
\]
In particular, the co-kernels of $m_{s_n}$ and $m_w$ are isomorphic and both
$m_w^\dagger\circ m_w$ and $m_w\circ m_w^\dagger$ are multiplication by $-n$. 

Let $m_w^{-1}:S^-_\RationalNumbers\rightarrow V_\RationalNumbers$ be the inverse of $m_w:V_\RationalNumbers\rightarrow S^-_\RationalNumbers$.
The quotient $m_w^{-1}(S^-)/V$ is then a subgroup $\Gamma_w$ of the compact torus $V_\RealNumbers/V$ isomorphic to
$(\Integers/n\Integers)^4$ canonically associated to $w$. The subgroup $\Gamma_w$ is the kernel of the homomorphism
$V_\RealNumbers/V\rightarrow S^-_\RealNumbers/S^-$ induces by $m_w$.
\end{rem}

\begin{example}
\label{example-an-element-of-the-even-clifford-group-of-S-plus}
{\rm
{\em An element of $G(S^+)^{even}_{s_n}$, which does not belong to 
$\Spin(V)_{s_n}$:}
Consider the two vectors $s_1=(1,0,-1)$ and $s_{}=(1,0,1)$ in $S^+$.
Then, $Q_S(s_1)=-1$, $Q_S(s_{})=1$, and $(s_1,s_{})_{S^+}=0$.
The composition $m_{s_{}}^\dagger\circ m_{s_1} : V\rightarrow V$  is given by
\[
m_{s_{}}^\dagger\circ m_{s_1}(w,\theta) \ \ \ = \ \ \ 
m_{s_{}}^\dagger(w,PD^{-1}(\theta)) \ \ \ = \ \ \ (w,-\theta). 
\]
Consequently, 
$m_{s_1}^\dagger\circ m_{s_{}}=(m_{s_{}}^\dagger\circ m_{s_1})^\dagger=
-(m_{s_{}}^\dagger\circ m_{s_1})^{-1}=-(m_{s_{}}^\dagger\circ m_{s_1})$.
This agrees with the relation 
$s_1\cdot s_{} + s_{}\cdot s_1=(s_1,s_{})_{S^+}=0$ in
the Clifford algebra $C(S^{+})$ (see also Equation (\ref{eq-composition-of-m-y-1-and-m-y-2}) below for a more conceptual interpretation). 
Similarly, $m_{s_1}\circ m_{s_{}}^\dagger: S^-\rightarrow S^{-}$ is
given by
\[
m_{s_1}\circ m_{s_{}}^\dagger(w,\beta) \ \ \ = \ \ \ (w,-\beta).
\]
The reflections $R_{s_1}(r,H,t)=(t,H,r)$ and 
$R_{s_{}}(r,H,t)=(-t,H,-r)$ of $S^+$ commute, and 
$R_{s_1}\circ R_{s_{}}(r,H,t)=(-r,H,-t)$. 
Let $\tilde{m}_{s_1\cdot s_{}}$ be the element of $GL(V)\times GL(S^+)\times GL(S^-)$
acting on $V$ via $m_{s_1}^\dagger\circ m_{s_{}}$, on $S^-$ via $m_{s_1}\circ m_{s_{}}^\dagger$, and on $S^+$ via $R_{s_1}\circ R_{s_{}}$.
Let $G(S^+)^{even}$ be the subgroup of $GL(V)\times GL(S^+)\times GL(S^-)$ generated by 
$Spin(V)$ and $\tilde{m}_{s_1\cdot s_{}}$. The notation $G(S^+)^{even}$ is motivated by triality (see Equation (\ref{eq-G-S-plus-even}) below).
Note that $R_{s_1}\circ R_{s_{}}(s_n)=-s_n$. Hence, 
both of the products $-\tilde{m}_{s_1\cdot s_{}}$ and $-\tilde{\alpha}\tilde{m}_{s_1\cdot s_{}}$, of
$\tilde{m}_{s_1\cdot s_{}}$ with
$-1$ or $-\tilde{\alpha}$ in $\Spin(V)$, 
belong to the stabilizer $G(S^+)^{even}_{s_n}$ of $s_n\in S^+$.
Set
\begin{equation}
\label{eq-tau-is-in-G-S-plus-even}
\tilde{\tau} \ \ \ := \ \ \ -\tilde{\alpha}\cdot \tilde{m}_{s_1\cdot s_{}} \ \ \in \ \  G(S^+)^{even}_{s_n}.
\end{equation}
Note that $\tilde{\tau}$ acts on $S^+$ and $S^-$ via the restriction of 
the main anti-involution $\tau$ of $C(V)$ to its subalgebra $H^*(X,\Integers)$. 
However, the action of $\tilde{\tau}$ on $V$ is not the identity, and so it does not agree with the restriction of the action of $\tau$ to the embedding of $V$ in $C(V)$.
}
\end{example}

\begin{new-lemma}
\label{lemma-stabilizer-in-Spin-V-maps-to-GL-4}
The image of the homomorphism
\[
G(V)^{even}_{s_n} \ \ \ \LongRightArrowOf{\rho} \ \ \ SO(V) \ \ \ 
\longrightarrow \ \ \ SO(V/nV)
\]
is equal to the subgroup of $SO(V/nV)$ leaving invariant the summand $H^1(X,\Integers/n\Integers)^*$
in the direct sum decomposition
\[
V/nV \ \ := \ \ H^1(X,\Integers/n\Integers) \ \oplus \ 
H^1(X,\Integers/n\Integers)^*.
\]
Consequently, we get the homomorphism
\begin{equation}
\label{eq-homomorphism-from-stabilizer-in-Spin-V-to-GL-4}
\bar{\rho} \ : \ G(V)^{even}_{s_n} \ \ \ \longrightarrow \ \ \ 
GL[H^1(X,\Integers/n\Integers)].
\end{equation}
\end{new-lemma}

\begin{proof}
Reduce modulo $n$ the homomorphisms $m_{s_n}$ 
given in (\ref{eq-Clifford-multiplication-by-s-n}). The kernel in $V/nV$ is 
$H^1(X,\Integers/n\Integers)^*$ and the kernel is $G(V)^{even}_{s_n}$-invariant. 
\end{proof}

\hide{
\begin{rem}
{\rm
The homomorphism (\ref{eq-homomorphism-from-stabilizer-in-Spin-V-to-GL-4})
induces further a homomorphism 
\[
SO(S^+)_{s_n} \ \ \longrightarrow \ \ PGL[H^1(X,\Integers/n\Integers)]
\ \ \cong \ \ 
O\left([\Wedge{2}H^1(X,\Integers)]/n\right).
\]
The latter homomorphism can be constructed directly as follows. 
The orthogonal complement $s_n^\perp$, of $s_n$ in $S^+$, 
is the direct sum
\[
s_n^\perp \ \ = \ \ \Wedge{2}H^1(X,\Integers) \oplus {\rm span}\{(1,0,n)\}
\]
The pairing $(\bullet,\bullet)_{S^+}$ restricts to the orthogonal
complement $s_n^\perp$ and reduces to a pairing 
on $s_n^\perp/n s_n^\perp$ with values in $\Integers/n^2\Integers$. 
Now $((1,0,n),(1,0,n))_{S^+}=2n$. 
The image of ${\rm span}\{(1,0,n)\}$ in $s_n^\perp/n s_n^\perp$
is $SO(S^+)_{s^n}$ invariant,
since it is the null space of the pairing, provided we further 
reduce its values modulo $n$. Hence, $SO(S^+)_{s^n}$ acts on
$s_n^\perp/[{\rm span}\{(1,0,n)\}+n s_n^\perp]$, which is
$[\Wedge{2}H^1(X,\Integers)]/n$.
}
\end{rem}
}

\subsection{Triality}
\label{sec-triality}
Set $A_X \ \ := \ \ V\oplus S^-\oplus S^+$. We endow $A_X$ with the bilinear pairing induced by that of each of the summands, so that $A_X$ is an orthogonal direct sum of the three lattices.
Define the commutative, but non-associative, algebra structure on $A_X$,
using the Clifford multiplication and the homomorphism 
(\ref{eq-multiplication-of-S-plus-and-S-minus}). 
The product of two elements in the same summand vanishes 
\cite[IV.2.2]{chevalley}. Any automorphism $\sigma$ 
of the algebra $A_X$, which leaves
$V$ and the sum $S^-\oplus S^+$ invariant, belongs to the image in $GL(A_X)$ of the kernel 
$G_0(V)\subset G(V)$ of the norm character (\ref{eq-norm}) via 
the homomorphism 
\begin{equation}
\label{eq-representation-of-G-V-on-A-X}
\tilde{\mu}:G(V)\rightarrow O(V)\times GL(S),
\end{equation}
given by $\tilde{\mu}(g)=(\rho(g),m(g))$,
where $\rho$ is given in (\ref{eq-standard-representation-of-G-V}) and $m$ in (\ref{eq-Cl})
(see \cite[IV.2.4]{chevalley}). Restriction of $\tilde{\mu}$ to $G_0(V)$ yields the faithful representation
\begin{equation}
\label{eq-representation-of-ker-N-on-A-X}
\tilde{\mu}: G_0(V)\rightarrow \Aut(A_X).
\end{equation}
If, in addition, each of the three summands is $\sigma$-invariant,
then $\sigma$ belongs to the image of $\Spin(V)$. This leads to a symmetric definition of
$\Spin(V)$ in terms of the algebra $A_X$.
Given an element $a$ of $A_X$, denote by $m_a$ the linear homomorphism of $A_X$ acting via multiplication by $a$.
Note, in particular, that the composition $m_{s_1}\circ m_{s_2}$ 
of two multiplications by elements $s_i$ of $S^+$
\[
(m_{s_1}\circ m_{s_2})(a)
\ \ := \ \ s_1\cdot (s_2\cdot a), \ \ a \in A_X,
\]
acts on the subspace $V\oplus S^-$ via an element 
$m_{s_1\cdot s_2}$
of $\Spin(V)$, 
provided $(s_1,s_1)_{S^+}=(s_2,s_2)_{S^+}=2$ 
or $(s_1,s_1)_{S^+}=(s_2,s_2)_{S^+}=-2$. 
The sublattice $S^+$ belongs to the kernel of $m_{s_1}\circ m_{s_2}$, but
$\tilde{\mu}(m_{s_1\cdot s_2})$ leaves $S^+$ invariant and acts on $S^+$ as an isometry.

Following is the {\em Principle of Triality}, adapted from \cite[Theorem IV.3.1 and Sec. 4.5]{chevalley}.

\begin{thm} 
\label{thm-triality-principle}
There exists an automorphism $J$ of order $3$ of the algebra $A_X$, preserving its bilinear pairing, with the following properties.
\begin{enumerate}
\item
\label{thm-item-cyclic-permutation-of-representations}
$J(V)=S^+,$ $J(S^+)=S^-$, and $J(S^-)=V$. 
\item
\label{thm-item-outer-automorphism-j}
$J^{-1}\tilde{\mu}(\Spin(V))J=\tilde{\mu}(\Spin(V))$, where $\tilde{\mu}$ is the representation (\ref{eq-representation-of-ker-N-on-A-X}).
Consequently, there exists a unique outer automorphism $j$ of $\Spin(V)$ of order $3$, satisfying
\[
\tilde{\mu}(j(g))=J^{-1}\tilde{\mu}(g)J.
\]
\end{enumerate}
\end{thm}

\begin{proof}
The proof consists of checking that the explicit construction of the automorphism $J$ in \cite[Theorem IV.3.1]{chevalley}
and of $j$ in \cite[Sec. 4.5]{chevalley}, which is stated for vector spaces, carries through for our lattices. We outline the construction.
Let $u_1$ be an element of $S^+$, satisfying $(u_1,u_1)_{S^+}=2$. We could choose for example $u_1=(1,0,1)$.
Let $t$ be the automorphism of the lattice $A_X$ mapping $V$ to $S^-$ 
via the Clifford action of elements of $V\subset C(V)$ on $u_1\in S^+\subset S$,
sending $S^-$ to $V$ via the product by $u_1$ in $A_X$, so via (\ref{eq-multiplication-of-S-plus-and-S-minus}),
and acting on $S^+$ by $-R_{u_1}$. 
Choose an element $x_1$ of $V$ satisfying $Q(x_1)=1$, so that $(x_1,x_1)_V=2$.
Regarding $x_1$ as an element of $C(V)$, via the embedding $V\subset C(V)$, we get that $x_1$ belongs to the subgroup $G_0(V)$ of $G(V)$.
Set 
$J:=\tilde{\mu}(x_1)t$, the composition of the automorphisms $\tilde{\mu}(x_1)$ and $t$ of the lattice $A_X$. 
$J$ is an isometry and an algebra automorphism of $A_X\otimes_\Integers\RationalNumbers$,
by \cite[Theorem IV.3.1]{chevalley} (see also the paragraph preceding \cite[IV.2.4]{chevalley}), and so of $A_X$ as well.
The properties in Part (\ref{thm-item-cyclic-permutation-of-representations}) of the Theorem follow by construction.
The identity $J^3=id$ holds for the vector space $A_X\otimes_\Integers\RationalNumbers$,
by the proof of \cite[Theorem IV.3.1]{chevalley}, and so must hold also for the $J$-invariant lattice $A_X$.
The invariance of $\Spin(V)$ in Part
(\ref{thm-item-outer-automorphism-j}) of the Theorem follows from that of $\Spin(V_\RationalNumbers)$,
by the proof in \cite[Sec. 4.5]{chevalley}, and the existence of $j$ follows, since the representation $\tilde{\mu}$ is faithful.
\end{proof}

As a corollary of the Triality Principle we get the following identification of $V\oplus S^-$ with the Clifford module 
$\wedge^* S^+$ of the integral Clifford algebra $C(S^+)$ associated to the decomposition 
$S^+=J(H^1(X,\Integers))\oplus J(H^1(X,\Integers)^*)$ as a direct sum of maximal isotropic sublattices.
Let $\tilde{J}:C(V)\rightarrow C(S^+)$ be the isomorphism extending the isomorphism $\restricted{J}{V}:V\rightarrow S^+$ induced by 
the isometry $J$. Let $Ad_J:\End(S^+\oplus S^-)\rightarrow \End(S^-\oplus V)$ be the isomorphism sending $f$ to $J\circ f\circ J^{-1}$.
Let $m:C(V)\rightarrow \End(S^+\oplus S^-)$ be the algebra homomorphism given in (\ref{eq-m-from-C-V}).
Given an element $x\in C(V)$, let $m_x\in \End(S^+\oplus S^-)$ be its image under $m$.

\begin{cor}
\label{cor-V-plus-S-minus-is-the-Clifford-module}
There exists a unique injective algebra homomorphism  
\[
m:C(S^+)\rightarrow \End(S^-\oplus V),
\] 
which restricts to the embedding of $S^+$ in $C(S^+)$ as the multiplication in $A_X$ by
elements of $S^+$. 
The following diagram is commutative
\[
\xymatrix{
V \ar[r]^{\subset} \ar[d]^{J} & C(V) \ar[r]^-m\ar[d]^{\tilde{J}} & \End(S^+\oplus S^-)\ar[d]^{Ad_J}
\\
S^+\ar[r]_{\subset} & C(S^+) \ar[r]_-m & \End(S^-\oplus V).
}
\]
In particular, given elements $x\in V$ and $y_1,y_2\in S^+$, the following equalities hold in $\End(S^-\oplus V)$.
\begin{eqnarray}
\label{eq-m-J-of-x}
m_{J(x)}&=&J\circ m_x\circ J^{-1},
\\
\label{eq-composition-of-m-y-1-and-m-y-2}
m_{y_1}\circ m_{y_2}+m_{y_2}\circ m_{y_1}&=&(y_1,y_2)_{S^+}\cdot id_{S^-\oplus V}.
\end{eqnarray}
\end{cor}

\begin{proof}
The left square is commutative, by definition of $\tilde{J}$.
The right lower arrow $m$ is defined by the requirement that the right square is commutative. 
It is an algebra homomorphism, since the upper arrow $m$ is, and it restricts to $S^+$ as multiplication in $A_X$, due to $J$ being an algebra automorphism of $A_X$.
Equality (\ref{eq-m-J-of-x}) follows from the commutativity of the diagram. We get
\[
m_{y_1}\circ m_{y_2}\stackrel{(\ref{eq-m-J-of-x})}{=}(J\circ m_{J^{-1}(y_1)}\circ J^{-1})\circ (J\circ m_{J^{-1}(y_2)}\circ J^{-1})=
J\circ m_{J^{-1}(y_1)}\circ m_{J^{-1}(y_2)}\circ J^{-1}.
\]
Equality (\ref{eq-composition-of-m-y-1-and-m-y-2}) follows:
\begin{eqnarray*}
m_{y_1}\circ m_{y_2}+m_{y_2}\circ m_{y_1}&=&
J\circ \left[m_{J^{-1}(y_1)}\circ m_{J^{-1}(y_2)}+m_{J^{-1}(y_2)}\circ m_{J^{-1}(y_1)}\right]\circ J^{-1}
\\
&=&
J\circ \left[(J^{-1}(y_1),J^{-1}(y_2))_{V}\cdot id_S\right]\circ J^{-1}
=(y_1,y_2)_{S^+}\cdot id_{S^-\oplus V},
\end{eqnarray*}
where the second Equality follow from the $C(V)$-module structure of $S$ 
and the defining relation (\ref{eq-defining-relation-of-Clifford-algebra}) of $C(V)$ and the last is due to $J$ being an isometry.
\end{proof}

Set $G(S^{+}):=\tilde{J}G(V)\tilde{J}^{-1}\subset C(S^+)$ and define its subgroup 
 $G(S^{+})^{even}$ similarly,
\begin{eqnarray}
\label{eq-G-S-plus-even}
G(S^{+})^{even} & := & \tilde{J} G(V)^{even} \tilde{J}^{-1},
\\
\nonumber
\Spin(S^+) & := & \tilde{J} [\Spin(V)] \tilde{J}^{-1}.
\end{eqnarray}
Let $\tilde{m}:G(S^+)\rightarrow GL(A_X)$ be the unique homomorphism making the following diagram commutative.
\begin{equation}
\label{eq-tilde-m}
\xymatrix{
 G(V) \ar[r]^-{\tilde{\mu}} \ar[d]_{Ad_{\tilde{J}}}& O(V)\times GL(S^+\oplus S^-) \ar[d]^{Ad_J}
 \\
G(S^+)\ar[r]_-{\tilde{m}} & O(S^+)\times GL(S^-\oplus V)
}
\end{equation}
Note the equality 
\begin{equation}
\label{eq-Spin-V-equals-Spin-S-plus}
\tilde{m}(\Spin(S^+))=\tilde{\mu}(\Spin(V))
\end{equation}
as subgroups of $\Aut(A_X)$, by Theorem \ref{thm-triality-principle}(\ref{thm-item-outer-automorphism-j}). 
However, $\tilde{m}(-1)$ and $\tilde{\mu}(-1)$ are distinct elements of the center. The element $\tilde{\mu}(-1)$ acts as the identity on $V$ and as $-id_S$ on $S=S^+ \oplus S^-$.
The element $\tilde{m}(-1)$ acts as the identity on $S^+$ and via scalar multiplication by $-1$ on $V\oplus S^-$.

Let
\[
\rho:G(S^+)\rightarrow O(S^+)
\]
be the composition of $\tilde{m}$ with the restriction homomorphism to the $G(S^+)$-invariant direct summand $S^+$
of $A_X$. The representation $\rho$ of $G(S^+)$ is analogous to that of $G(V)$ given in (\ref{eq-standard-representation-of-G-V}).
Explicitly,  $G(S^{+})^{even}$ is 
generated by  $\tilde{\mu}(\Spin(V))$ and one additional automorphism 
$\tilde{m}_{s_1\cdot s_{-1}},$ 
acting 
on $V\oplus S^{-}$ via 
$m_{s_1}\circ m_{s_{-1}}$, which is the composition of two multiplications
\[
(m_{s_1}\circ m_{s_{-1}})(a) \ \ := \ \ s_1\cdot (s_{-1}\cdot a), \ \ a \in A_X,
\]
where $s_i$, $i=1,-1$, are two elements of $S^+$ and $(s_i,s_i)_{S^+}=2i$. 
One can take, for instance, $s_1=(1,0,1)$ and $s_{-1}=(1,0,-1)$ as in
Example \ref{example-an-element-of-the-even-clifford-group-of-S-plus}. 
The action of the element $\tilde{m}_{s_1\cdot s_{-1}}$ on $S^+$ is 
the composition $R_{s_1}\circ R_{s_{-1}}$ of the two reflections.
Then $G(S^{+})^{even}$ preserves $(\bullet,\bullet)_{S^+}$ and maps into $SO(S^+)$, but its image 
in $SL(V)$ is contained in $S\widetilde{O}(V)$. In particular, $G(S^{+})^{even}$ is not contained in 
the image of $G(V)$ in $GL(A_X)$ via $\rho\times m$.

\hide{
Motivated by triality, we let $\Spin(S^+)$ 
be the image of $\Spin(V)$  in $\Aut(A_X)$. 
Let $\Pin(S^+)\subset GL(A_X)$ be the subgroup $J Pin(V) J^{-1}$. It is generated by $\Spin(V)$
and one additional element $\tilde{m}_{s_{-1}}$, acting on $V\oplus S^-$ via multiplication by 
$s_{-1}$ and on $S^+$ via the reflection $R_{s_{-1}}$ (in analogy to (\ref{eq-Pin-acts-by-reflections})).
Let $G(S^+)$ be the subgroup $J G(V) J^{-1}$ of $GL(A_X)$ generated by $\Pin(S^+)$ and   $G(S^{+})^{even}$.
Given an element $s_1\in S^+$, satisfying $(s_1,s_1)_{S^+}=2$, we get the element
\begin{equation}
\label{eq-tilde-m-s-1}
\tilde{m}_{s_1}
\end{equation}
of $G(S^+)$ satisfying $\tilde{m}_{s_1}\tilde{m}_{s_1\cdot s_{-1}}=\tilde{m}_{s_{-1}}$.
The sublattice $S^+$ of $A_X$ is $G(S^+)$ invariant, yielding the homomorphism 
\[
\rho \ : \ G(S^+) \ \ \longrightarrow \ \ O(S^+)
\]
analogous to (\ref{eq-standard-representation-of-G-V}).
}
Note that if $s\in G(S^+)^{even}$ does not belong to $\Spin(S^+)$, then $\rho(s)$
acts on $S^+$ as an isometry, $V$ and $S^-$ are $s$-invariant, but the action of $s$ on $V\oplus S^-$ reverses the sign of the 
pairing, so its restriction belongs to $\widetilde{O}(V\oplus S^-)$ but not to $O(V\oplus S^-)$, in analogy to 
(\ref{eq-minus-1-vectors-do-not-act-via-isometries-on-spin-representation}).

Consider the composite homomorphism 
\[
G(S^+)^{even}_{s_n} \ \ \ \longrightarrow \ \ \ 
S\widetilde{O}(V) \ \ \ \longrightarrow \ \ \
S\widetilde{O}(V/nV).
\]
Its composition with $Ad_{\tilde{J}}:\Spin(V)\rightarrow \Spin(S^+)$ agrees with the restriction of the homomorphism $\bar{\rho}$,
given in (\ref{eq-homomorphism-from-stabilizer-in-Spin-V-to-GL-4}), to
$\Spin(V)$.
The image of $G(S^+)^{even}_{s_n}$ in $S\widetilde{O}(V/nV)$ 
leaves the subgroup $H^1(X,\Integers/n\Integers)^*$ invariant. 
The quotient of $(V/nV)$ by $H^1(X,\Integers/n\Integers)^*$ is naturally isomorphic to $H^1(X,\Integers/n\Integers)$.
We obtain a homomorphism
\begin{equation}
\label{eq-homomorphism-from-stabilizer-in-G-S-plus-even-to-GL}
G(S^+)^{even}_{s_n} \ \ \ \longrightarrow \ \ \ 
GL[H^1(X,\Integers/n\Integers)].
\end{equation}

\section{Generators for the stabilizer $\Spin(S^+)_{s_n}$}
\label{sec-generators-for-the-stabilizer}

Given a lattice $\Lambda$, 
let $\Reflection(\Lambda)$ be the group generated by 
reflections in $+2$ and $-2$ vectors in $\Lambda$. Set
$S\Reflection(\Lambda):=\Reflection(\Lambda)\cap SO(\Lambda)$ and 
$S\Reflection_+(\Lambda):=\Reflection(\Lambda)\cap SO_+(\Lambda)$.
The former is generated by elements, which are products of an even
number of reflections.
$S\Reflection_+(\Lambda)$ is generated by elements, 
which are either products of an even
number of reflections in $+2$ vectors, or products of an even
number of reflections in $-2$ vectors.
Let $U$ be the rank $2$ even unimodular hyperbolic lattice.
The lattice $S^+$ is isometric to $U^{\oplus 4}$. 
The orthogonal group is generated by reflections, 
$O(U^{\oplus k})=\Reflection(U^{\oplus k})$, for 
$k\geq 3$ (see \cite[4.3]{wall}). Consequently, we get the following.
\begin{cor}
\label{cor-SO-plus-of-direct-sums-of-hyperbolic-plane}
$SO_+(U^{\oplus k})=S\Reflection_+(U^{\oplus k})$, for 
$k\geq 3$.
\end{cor}

Set $e_1:=(1,0,0)\in S^+$ and $e_2:=(0,0,1)\in S^+$.
Let $\Spin(V)_{e_1,e_2}$ be the subgroup of $\Spin(V)$ stabilizing both $e_1$ and $e_2$.
\begin{new-lemma}
\label{lemma-stabilizer-of-H0-and-H4-is-SL-4}
$\Spin(V)_{e_1,e_2}$ leaves invariant each of the subspaces $H^1(X,\Integers)$ and $H^1(\hat{X},\Integers)$ of $V$, 
its action on $H^1(X,\Integers)$ factors through an isomorphism 
\[
f:\Spin(V)_{e_1,e_2}\rightarrow SL(H^1(X,\Integers)),
\]
and its action on $H^1(\hat{X},\Integers)$ factors through an isomorphism with $SL(H^1(\hat{X},\Integers))$.
\end{new-lemma}

\begin{proof}
Consider $V$ as a subspace of $C(V)$ as in Section \ref{sec-spin-8-and-triality}.
$H^1(X,\Integers)$  is the intersection of $V$ with the left ideal of $C(V)$ annihilating $e_2$ and
$H^1(\hat{X},\Integers)$  is the intersection of $V$ with the left ideal annihilating $e_1$.
In the language of \cite[Sec. 3.1]{chevalley} the elements $e_i$, $i=1,2$, are {\em pure spinors} and each of the two maximal isotropic sublattices of $V$ is the one associated to the pure spinor. If $g$ belongs to $\Spin(V)_{e_1,e_2}$ and $v$ to $H^1(X,\Integers)$, then
\[
\rho(g)(v)\cdot e_i=(g\cdot v\cdot g^{-1})\cdot e_i=(g\cdot v\cdot g^{-1})\cdot (g\cdot e_i)=
g\cdot (v\cdot e_i),
\]
where $\cdot$ denotes multiplication in $C(V)$ as well as the module action of $C(V)$ on $S$, and the second equality follows since $g$ stabilizes $e_i$.
Hence, $\rho(g)(v)\cdot e_i=0$, if and only if $v\cdot e_i=0$.
Thus, $\Spin(V)_{e_1,e_2}$ leaves  each of $H^1(X,\Integers)$ and $H^1(\hat{X},\Integers)$ invariant. 
Furthermore, $\tilde{\mu}\left(\Spin(V)_{e_1,e_2}\right)=\tilde{m}\left(\Spin(S^+)_{e_1,e_2}\right)$, by (\ref{eq-Spin-V-equals-Spin-S-plus}), and the latter is $\Spin(H^2(X,\Integers))$, since
$H^2(X,\Integers)$ is the subspace of $S^+$ orthogonal to $\{e_1,e_2\}$. 
$\Spin(H^2(X,\RationalNumbers))$ acts on each of its half-spin representations $H^1(X,\RationalNumbers)$ and $H^1(\hat{X},\RationalNumbers)$
via an injective homomorphism into $SL(H^1(X,\RationalNumbers))$ and $SL(H^1(\hat{X},\RationalNumbers))$, by
\cite[III.7.2]{chevalley}. Hence, the homomorphism $f$ is well defined and injective.

Conversely, $SL(H^1(X,\Integers))$ acts via isometries on $\Wedge{2}H^1(X,\Integers)$ and we get the commutative diagram
\[
\xymatrix{
\Spin(V)_{e_1,e_2}\ar[rr]^{f} \ar[dr] & & SL(H^1(X,\Integers)) \ar[dl]
\\
& SO_+(\Wedge{2}H^1(X,\Integers))
}
\]
$SO_+(\Wedge{2}H^1(X,\Integers))$ is generated by products of two reflections in classes $c\in H^2(X,\Integers)$ with $(c,c)_{S^+}=2$
and products of two reflections in classes $c\in H^2(X,\Integers)$ with $(c,c)_{S^+}=-2$, by Corollary
\ref{cor-SO-plus-of-direct-sums-of-hyperbolic-plane}. 
The left slanted arrow is surjective, since these generators are in its image. The kernel of the right slanted arrow is $-id$,
which is in the image of $f$, since $\tilde{m}(-1)$, given in (\ref{eq-tilde-m}), acts as the identity on $S^+$ and maps to $-id_V$.
Hence, $f$ is surjective as well.
The verification of the statement for $H^1(\hat{X},\Integers)$ is identical.
\end{proof}

Let $s_n:=(1,0,-n)$, $n\geq 3$. 

\begin{new-lemma}
\label{lemma-generators-for-stabilizer-in-SO}
\begin{enumerate}
\item
\label{lemma-item-stabilizer-is-generated-by-reflections}
The stabilizer $SO(S^+)_{s_n}$ is equal to 
the subgroup $S\Reflection(s_n^\perp)$ of the reflection group 
of the orthogonal complement $s_n^\perp\subset S^+$ of $s_n$. 
\item
\label{lemma-item-reflections-by-line-bundles-generate-stabilizer}
$\Reflection(s_n^\perp)$ is generated by 
$O(H^2(X,\Integers))$ and reflections in vectors of the form 
$(1,A,n)$, with $A\in H^2(X,\Integers)$ a primitive class satisfying
$\int_X A\cup A=2n-2$. 
\item
\label{lemma-item-even-number-of-reflections-by-line-bundles-generate}
$S\Reflection(s_n^\perp)$
is generated by $SO(H^2(X,\Integers))$ and products $R_{t_1} R_{t_2}$ of
reflections, where $t_i=(1,A_i,n)\in s_n^\perp$ are vectors of the above
form. 
\end{enumerate}
\end{new-lemma}


\begin{proof}
\ref{lemma-item-stabilizer-is-generated-by-reflections},
\ref{lemma-item-reflections-by-line-bundles-generate-stabilizer})
It suffices to prove that $O(S^+)_{s_n}$ is generated by 
$O(H^2(X,\Integers))=\Reflection(H^2(X,\Integers))$ and
$+2$ vectors of the form $(1,A,n)$, with $A$ primitive. 
The argument proving Lemma 7.4 in \cite{markman-monodromy-I} applies  
to show that the stabilizer $O(S^+)_{s_n}$ is generated by 
$O(H^2(X,\Integers))$ and 
reflections in $+2$ elements $(r,A,rn)\in s_n^\perp$, 
with $A\in H^2(X,\Integers)$ primitive. 
If $r=0$, then $R_{(r,A,rn)}$ belongs to 
$\Reflection(H^2(X,\Integers))$. 
If $r\neq 0$, then the reflection $R_{(r,A,rn)}$ is a composition of 
reflections
$R_{t_i}$, with $t_i=(1,A_i,n)$ a $+2$ vector in $s_n^\perp$, 
and $A_i$ primitive, by the argument proving Lemma 7.7 in 
\cite{markman-monodromy-I}. 

\ref{lemma-item-even-number-of-reflections-by-line-bundles-generate})
Let $t=(0,A,0)$ and $u=(1,B,n)$, with $\int_X A\cup A=\pm 2$ and 
$\int_X B\cup B=2n-2$. 
The equality $R_uR_t=R_tR_tR_uR_t=R_tR_{R_t(u)}$ implies that every element of 
$S\Reflection(s_n^\perp)$ can be written as a product
\begin{equation}
\label{eq-normalized-product-of-reflections}
R_{t_1}\cdots R_{t_k}R_{u_1}\cdots R_{u_\ell},
\end{equation}
with $k+\ell$ even, $t_i$ are $\pm 2$ vectors in $H^2(X,\Integers)$,
and $u_j=(1,B_j,n)$ are $+2$ vectors in $s_n^\perp$, with $B_j$ primitive.

It remains to prove that we can choose $\ell$ to be even as well. 
The proof is similar to that of Lemma 7.7 in \cite{markman-monodromy-I}. 
Assume $k$ and $\ell$ are odd. Choose $A\in H^2(X,\Integers)$
satisfying $(A,A)_{S^+}=2$ and $(A,B_1)_{S^+}=-1$. Then $A+B_1=R_A(B_1)$
is a primitive class in $H^2(X,\Integers)$. 
Set $t_{k+1}:=(0,A,0)$.
Then $(t_{k+1},u_1)=-1$ and $v:=R_{u_1}(t_{k+1})=t_{k+1}+u_1=(1,A+B_1,n)$.
Thus, the subgroup generated by the three reflections 
$\{R_{t_{k+1}},R_{u_1},R_{v}\}$ is the permutation group $\Sym_3$, and
\[
R_{u_1} \ \ = \ \ R_{t_{k+1}}R_{u_1}R_v.
\]
Substitute the right hand side for $R_{u_1}$ in 
(\ref{eq-normalized-product-of-reflections}) 
to replace $\ell$ by $\ell+1$.
\end{proof}

\begin{new-lemma}
\label{lemma-generators-for-stabilizer-in-G-V}
The stabilizer $\Spin(S^+)_{s_n}$ is generated by
$\Spin(S^+)_{e_1,e_2}\cong SL[H^1(X,\Integers)]$ 
and products $t_1t_2\in \Spin(S^+)_{s_n}$, where each $t_i=(1,A_i,n)$
is a $+2$ vector in $s_n^\perp$ and $A_i$ is a primitive class in
$H^2(X,\Integers)$.
\end{new-lemma}

\begin{proof}
Let $SO_+(S^+)_{s_n}$ be the kernel of the restriction of the orientation character 
(\ref{eq-orientation-character}) to $SO(S^+)_{s_n}$.
The homomorphism $\Spin(S^+)_{s_n}\rightarrow SO_+(S^+)_{s_n}$ is surjective, by Lemma 
\ref{lemma-spin-surjects}, and its kernel 
is equal to the kernel\footnote{It is the order $2$ subgroup generated by $s\cdot s$, where
$s\in S^+$ is a class with $Q_{S^+}(s)=-1$.
}
of $\Spin(S^+)\rightarrow SO_+(S^+)$ and is 
contained in $\Spin(S^+)_{e_1,e_2}$. 
The homomorphism $f:\Spin(S^+)_{e_1,e_2}\rightarrow SO_+(H^2(X,\Integers))$ is surjective,
by Lemma \ref{lemma-stabilizer-of-H0-and-H4-is-SL-4}.
Hence, it suffices to prove that $SO_+(S^+)_{s_{n}}$ is
generated by $SO_+[H^2(X,\Integers)]$ and the products $R_{t_1} R_{t_2}$ of
the reflections in the vectors $t_i$. 
Let $g$ be an element of $SO_+(S^+)_{s_n}$. Then $g=R_{a_1}\cdots R_{a_k}R_{t_1}\cdots R_{t_\ell}$
with $a_i\in H^2(X,\Integers)$, $(a_1,a_2)_{S+}=\pm 2$, $t_i$ of the above form, and $\ell$ is even, by
Lemma 
\ref{lemma-generators-for-stabilizer-in-SO} (\ref{lemma-item-even-number-of-reflections-by-line-bundles-generate}). 
Let $k=k_++k_-$, where $k_-$ is the number of $a_i$ with $(a_1,a_1)_{S^+}=-2$.  
Then $k_-$ is even, since $g$ belongs to $SO_+(S^+)_{s_n}$ and $(t_i,t_i)=2$. Hence, $k_+$ is even as well 
and $R_{a_1}\cdots R_{a_k}$ belongs to $SO_+[H^2(X,\Integers)]$.
\end{proof}

%
\section{Equivariance of the Chern character of a universal sheaf}
\label{sec-equivariance-of-the-universal-sheaf}
Fix a non-negative integer $n$.
Let $\C_n$ be the category whose objects are triples $(X,H,w)$, where $X$ is an abelian surface, $w=(r,\beta,s)\in H^{even}(X,\Integers)$ is a primitive class with $r\geq 0$, $\beta\in H^{1,1}(X,\Integers)$, $(w,w)_{S^+}=-2n$, $H$ is a $w$-generic polarization,
and such that there exists a universal sheaf $\E$ over $X\times\M_H(w)$. 
A morphism in $\Hom_{\C_n}((X_1,w_1,H_1),(X_2,w_2,H_2))$ consists of a pair $(g,\epsilon)$, where $\epsilon$
is in $\Integers/2\Integers$ and $g:H^*(X_1,\Integers)\rightarrow H^*(X_2,\Integers)$ is a group homomorphism mapping
$H^{even}(X_1,\Integers)$ to $H^{even}(X_2,\Integers)$ and
$H^{odd}(X_1,\Integers)$ to $H^{odd}(X_2,\Integers)$ and satisfying the following condition.
There exists a ring isomorphism 
\[
\gamma:H^*(\M_{H_1}(w_1),\Integers)\rightarrow H^*(\M_{H_2}(w_2),\Integers)
\] 
such that for a choice of universal sheaves $\E_i$ over $X_i\times \M_{H_i}(w_i)$ and
some class $c\in H^2(\M_{H_2}(w_2),\Integers)$
we have 
\begin{eqnarray}
\label{condition-to-be-a-morphism-in-C-n}
(g\otimes \gamma)(ch(\E_1)
)=ch(\E_2)
exp(c), & \mbox{if} & \epsilon=0,
\\
\nonumber
((\tau g\tau)\otimes\gamma)(ch(\E_1)
)=ch(\E_2^\vee)
exp(c), & \mbox{if} & \epsilon=1,
\end{eqnarray}
where $\tau:H^*(X_j,\Integers)\rightarrow H^*(X_j,\Integers)$ acts on $H^i(X_j,\Integers)$ by $(-1)^{i(i-1)/2}$.
In the displayed equation above 
$\E_2^\vee$ is the derived dual, and the product is by the pull backs of 
$exp(c)$ via the projection to $\M_{H_2}(w_2)$. 

Composition of morphisms $(h,\epsilon)\in \Hom_{\C_n}((X_1,w_1,H_1),(X_2,w_2,H_2))$ and
$(g,\delta)\in \Hom_{\C_n}((X_2,w_2,H_2),(X_3,w_3,H_3))$ is given by
\begin{equation}
\label{eq-composition-inC-n}
(h,\epsilon)\circ (g,\delta)=(\tau^{\delta}h\tau^{\delta}g,\epsilon+\delta).
\end{equation}

We prove in this section that  $\C_n$ is indeed a category with the above composition rule, namely, the existence of a ring isomorphism 
$\gamma:H^*(\M_{H_1}(w_1),\Integers)\rightarrow H^*(\M_{H_3}(w_3),\Integers)$ and a class 
$c\in H^2(\M_{H_3}(w_3),\Integers)$ needed for the right hand side of (\ref{eq-composition-inC-n}) to satisfy one of the two
conditions displayed in (\ref{condition-to-be-a-morphism-in-C-n}) (Corollary \ref{cor-composition} and Lemmas \ref{lemma-composition-of-gamma-g-1-and-gamma-h-1} and \ref{lemma-composition-of-gamma-g-0-and-gamma-h-1}).
Furthermore, $\gamma$ 
in (\ref{condition-to-be-a-morphism-in-C-n}) is uniquely determined by 
$(g,\epsilon)$ via an explicit formula (\ref{eq-gamma-delta}) (Lemma \ref{lemma-recovering-f}). Everything follows formally from the expression of the class of the diagonal 
in $H^*(\M_H(w)\times \M_H(w),\Integers)$ in terms of a universal sheaf. 
We do not discuss existence of morphisms in this section. In Section \ref{sec-four-groupoids}
$\Hom_{\C_n}((X_1,w_1,H_1),(X_2,w_2,H_2))$ is shown to be non-empty (Theorem \ref{thm-Hom-G3-non-empty}) and 
in Corollary \ref{cor-monodromy-representation-of-spin} an injective homomorphism\footnote{
The homomorphism sends $f\in G(S^+_X)^{even}_w$ to $(\tau^{ort(f)}\tilde{m}_f,ort(f))$, where $ort$ is the character (\ref{eq-ort-S+})
and $\tilde{m}:G(S^+_X)^{even}\rightarrow SO(S^+_X)\times S\widetilde{O}(S^-_X)$ is given in Equation (\ref{eq-tilde-m}).
}
$
G(S^+_X)^{even}_w\rightarrow \Aut_{\C_n}((X,w,H))
$
is constructed for $n\geq 3$.
In the current section \ref{sec-equivariance-of-the-universal-sheaf} we treat the case where $X$ is a $K3$ surface as well.

Given a smooth projective variety $M$, we denote by 
\begin{eqnarray*}
\ell \ : \ \oplus_{i}H^{2i}(M,\RationalNumbers) & \longrightarrow &
\oplus_{i}H^{2i}(M,\RationalNumbers)
\\
(r+a_1+a_2+\cdots ) & \mapsto & 1+a_1 + (\frac{1}{2}a_1^2-a_2) + \cdots 
\end{eqnarray*}
the universal polynomial map, which takes the exponential Chern character
of a complex of sheaves to its total Chern class. 
We let 
\[
D_M: H^{even}(M,\Integers)\rightarrow H^{even}(M,\Integers)
\]
be the dualization automorphism acting by $(-1)^i$ on $H^{2i}(M,\Integers)$.

Let $X_1$ and $X_2$ be two abelian or $K3$ surfaces and let
$w_i\in H^{even}(X_i,\Integers)$ be two Mukai vectors. 
Assume that the moduli space
$\M(w_i)$ of Gieseker-Simpson stable sheaves on $X_i$ (with respect to a choice of
polarizations, when the rank is different from $1$) is compact, for $i=1,2$, and that
$\dim(\M(w_1))=\dim(\M(w_2))$. Set $m:=\dim(\M(w_i))$, $i=1,2$.
Denote by $\pi_i$ the projection from 
$X_1 \times \M(w_1)\times X_2 \times \M(w_2)$ onto the $i$-th factor.
Set 
\[
D:=D_{X_1 \times \M(w_1)\times X_2 \times \M(w_2)}.
\]
Given a class $\delta$ in $H^{even}(X_1\times X_2,\Integers)$,
classes $\alpha_i$ in $H^{even}(X_i\times \M(w_i),\RationalNumbers)$, 
and an element 
$\epsilon\in \Integers/2\Integers$, we define a class 
$\gamma_{\delta,\epsilon}(\alpha_1,\alpha_2)$ in 
$H^{2m}(\M(w_1)\times \M(w_2),\RationalNumbers)$ by
\begin{equation}
\label{eq-gamma-delta}
\gamma_{\delta,\epsilon}(\alpha_1,\alpha_2) \ = \ 
c_m\left(
\left\{
\ell\left(
\pi_{24_*}\left[
D^{1-\epsilon}[\pi_{12}^*(\alpha_1)\cdot\pi_{13}^*(\delta)]
\cdot\pi_{34}^*(\alpha_2)\cdot
\pi_1^*(\sqrt{td_{X_1}})\cdot\pi_3^*(\sqrt{td_{X_2}})
\right]
\right)
\right\}^{-1}\right).
\end{equation}
(The Todd classes $td_{X_i}$ are equal to $1$ for an abelian surface).
If $\E_i$ is a complex of sheaves on $X_i\times \M(w_i)$, 
we denote 
$\gamma_{\delta,\epsilon}(ch(\E_1)\cdot\sqrt{td_{X_1}},
ch(\E_2)\cdot\sqrt{td_{X_2}})$ 
also by $\gamma_{\delta,\epsilon}(\E_1,\E_2)$. 

Let $\M_H(w)$ be a smooth and projective 
$m$-dimensional moduli space of $H$-stable sheaves on
a $K3$ or abelian surface $X$. Denote by $p_i$ the projection from $\M_H(w)\times X \times \M_H(w)$ onto the $i$-th factor and by $p_{ij}$ the projection onto the product of the $i$-th and $j$-th factors.

\begin{thm} 
\label{thm-diagonal}
Let $\E_1$, $\E_2$ be any two universal families of sheaves over $X\times \M_H(w)$ 
\begin{enumerate}
\item
\label{thm-item-class-of-diagonal}
\cite{markman-diagonal}
The class of the diagonal, in the Chow ring of 
$\M_H(w)\times \M_H(w)$, is identified by 
\[
c_m\left[- \ 
p_{13_!}\left(
p_{12}^*(\E_1)^\vee\stackrel{L}{\otimes}p_{23}^*(\E_2)
\right)
\right],
\]
where both the dual
$(\E_1)^\vee$ and the tensor product are taken in the derived category.
\item
\cite[Theorem 1]{markman-integer}
The integral cohomology $H^*(\M_H(w),\Integers)$ is torsion free.
\end{enumerate}
\end{thm}

\begin{rem}
\label{rem-class-of-the-diagonal-in-terms-of-twisted-universal-sheaves}
A version of Theorem \ref{thm-diagonal} holds for a projective 
moduli space $\M_H(w)$ of $H$-stable sheaves on
a $K3$ or abelian surface $X$, even if the universal sheaves $\E_1$ and $\E_2$ are twisted with respect to 
the pullback to $X\times \M_H(w)$ of a Brauer class $\theta\in H^2_{an}(\M_H(w),\StructureSheaf{\M_H(w)}^*)$ (a \v{C}ech cohomology class for the analytic topology). In that case a universal class $e$ in the topological $K$-ring 
$K(X\times \M_H(w))$
of $X\times \M_H(w)$ was constructed in 
\cite[Sec. 3.1]{markman-integer}, unique up to tensorization by the class of a topological complex line bunde. 
The statement of Theorem \ref{thm-diagonal} then holds replacing $\E_i$ by any such universal class $e_i$ in $K(X\times \M_H(w))$, by \cite[Prop. 24]{markman-integer}.  
\end{rem}

We denote by $\Delta_X$ the class of the diagonal in $X\times X$.
Given a homomorphism  
$g:H^*(X_1,\Integers)\rightarrow H^*(X_2,\Integers)$, preserving 
the parity of the cohomological degree, 
we get the class $(1\times g)(\Delta_{X_1})$ in 
$H^{even}(X_1\times X_2,\Integers)$ inducing $g$.  
Set
\begin{equation}
\label{eq-gamma-g}
\gamma_{g,\epsilon}(\alpha,\beta) \ := \ 
\gamma_{(1\times g)(\Delta_{X_1}),\epsilon}(\alpha,\beta).
\end{equation}

\noindent
When the parameter $\epsilon$ is omitted, it is understood to be zero.
When $X_1=X_2$ and $g=id$, Grothendieck-Riemann-Roch yields the equality
\begin{equation}
\label{eq-gamma-id-translated-by-GRR}
\gamma_{id}(\E_1,\E_2) \ = \ c_m\left\{-p_{13_!}\left(p_{12}^*(\E_1)^\vee
\stackrel{L}{\otimes}
p_{23}^*(\E_2)\right)\right\}.
\end{equation} 
In Section \ref{sec-monodromy-via-Fourier-Mukai} we will see
examples where 
$\gamma_g(\E_1,\E_2)$ is the class of the graph of an isomorphism, when $g$ is induced by a stability-preserving 
auto-equivalence of the derived category of the surface.

Identifying $H^*(\M(w_1),\Integers)$ with its dual, 
via Poincare-Duality $x\mapsto \int_{\M(w_1)}x\cup (\bullet)$, 
we will view the class $\gamma_{\delta,\epsilon}(\E_1,\E_2)$ 
as a homomorphism (preserving the grading) from 
$H^*(\M(w_1),\Integers)$ to $H^*(\M(w_2),\Integers)$.
We identify $H^*(X,\Integers)$ with its dual via Poincare-Duality as well
and regard classes in $H^*(X_1\times X_2,\Integers)$ as
homomorphisms from $H^*(X_1,\Integers)$ to $H^*(X_2,\Integers)$. 

Let $d_X\in \Aut[H^*(X,\ComplexNumbers)]$ and 
$d_{\M(w_2)}\in \Aut[H^*(\M(w_2),\ComplexNumbers)]$ 
be graded ring automorphisms preserving the intersection pairings and satisfying
\begin{equation}
\label{eq-factorization-of-D}
D_{X\times \M(w_2)} \ \ \ = \ \ \ d_X\otimes d_{\M(w_2)}
\end{equation}
as automorphisms of $H^{even}(X\times \M(w_2),\ComplexNumbers)$.
Clearly, $d_X$ 
determines $d_{\M(w_2)}$ uniquely, 
and each is determined by the above equation, 
up to a constant factor on each graded summand of the cohomology groups.
When $X$ is a $K3$ surface, the odd cohomology groups of 
both $X$ and $\M(w_2)$ vanish, and  we get a natural factorization 
by setting $d_X=D_X$ and $d_{\M(w_2)}=D_{\M(w_2)}$.
When $X$ is an abelian surface, we can let $d_X$ and $d_{\M(w_2)}$ act on 
the $i-th$ cohomology via multiplication by $(\sqrt{-1})^i$. 
Note that $d_X$ has order $4$. Nevertheless, conjugation by $d_X$
has order $2$ and the corresponding inner automorphism 
of $\Aut [H^*(X,\Integers)]:=
GL[H^{even}(X,\Integers)]\times GL[H^{odd}(X,\Integers)]$
is independent of the choice of 
$d_X$ in the factorization (\ref{eq-factorization-of-D}). 

Let $\tau\in \Aut [H^*(X,\Integers)]$ be the element acting by 
$(-1)^{i(i-1)/2}$ on $H^i(X,\Integers)$. Then $d_X$ commutes with $\tau$
and the corresponding inner automorphisms of $\Aut [H^*(X,\Integers)]$
are equal.\footnote{In Section \ref{sec-spin-8-and-triality} the automorphism
$\tau$ is extended to the main anti-automorphism
(\ref{eq-tau}) of the Clifford algebra $C(V)$. $\tau$ 
is an element of  $G(S^+)^{even}$ (see Example 
\ref{example-an-element-of-the-even-clifford-group-of-S-plus}). 
}
Note that $d_X$ is an isometry with respect to the pairing
(\ref{eq-Mukai-pairing}) on $H^*(X,\ComplexNumbers)$ 
(since $d_X$ commutes with $\tau$). 
Similarly, $d_{\M(w_2)}$ preserves the Poincare duality pairing. 

\medskip
Let $g:H^*(X_1,\Integers)\rightarrow H^*(X_2,\Integers)$ 
be a linear homomorphism, preserving 
the parity of the cohomological degree.
Note that 
\begin{equation}
\label{eq-conjugation-by-d-X-i}
d_{X_2}gd_{X_1}^{-1}=\tau g\tau
\end{equation}
and is hence an integral homomorphism.
Indeed 
$d_{X}\tau$ acts on $H^{even}(X)$ as the identity and on 
$H^{odd}(X)$ via multiplication by a scalar. Hence,
$
d_{X_2}\tau f=fd_{X_1}\tau,
$
for every linear homomorphism $f:H^*(X_1,\Integers)\rightarrow H^*(X_2,\Integers)$ preserving the parity of the grading. 
Applying the latter equality with $f=\tau g\tau$ we get:
$
d_{X_2}gd_{X_1}^{-1}=d_{X_2}\tau(\tau g\tau)\tau d_{X_1}^{-1}=\tau g\tau.
$
Given $\epsilon\in \Integers/2\Integers$, we set $d_{X_i}^\epsilon:=
\left\{\begin{array}{ccc}
d_{X_i}&\mbox{if} & \epsilon=1
\\
id & \mbox{if} & \epsilon=0.
\end{array}\right.
$
We use this notation only in conjugation, were the equality
$d_{X_2}^\epsilon g (d_{X_1}^\epsilon)^{-1}=\tau^\epsilon g\tau^\epsilon$ makes it unambiguous, since $\tau$ is an involution.

Let $\gamma:H^*(\M(w_1),\Integers)\rightarrow H^*(\M(w_2),\Integers)$
be an isomorphism of graded rings. 
Assume that universal sheaves $\E_i$
exist over $X_i\times \M(w_i)$, for $i=1,2$. 

\begin{defi}
\label{def-g-gamma-maps-universal-classes-to-such}
{\rm
\begin{enumerate}
\item
\label{def-item-maps-universal-class-to-such}
Assume that $g(w_1)=w_2$. 
We say that $g\otimes\gamma$ 
{\em maps a universal class  of $\M(w_1)$ to a universal class of $\M(w_2)$},
if
\[
(g\otimes\gamma)(ch(\E_1)\sqrt{td_{X_1}}) \ \ \ = \ \ \
[ch(\E_2)\sqrt{td_{X_2}}]\pi_{\M(w_2)}^*\exp(c_g),
\]
where the class $c_g\in H^2(\M(w_2),\Integers)$ is characterized by 
\begin{equation}
\label{eq-c-g}
\rank(w_2)\pi_{\M(w_2)}^*c_g=c_1(\E_2)-[(g\otimes\gamma)(ch(\E_1)\sqrt{td_{X_1}})]_1.
\end{equation}
\item
Assume that $g(w_1)=(w_2)^\vee$. 
We say that $g\otimes\gamma$ 
{\em maps a universal class to the dual of a universal class},
if
\[
(g\otimes\gamma)(ch(\E_1)\sqrt{td_{X_1}}) \ \ \ = \ \ \
[ch(\E_2^\vee)\sqrt{td_{X_2}}]\pi_{\M(w_2)}^*\exp(c_g),
\]
where $\rank(w_2)\pi_{\M(w_2)}^*c_g=-c_1(\E_2)-[(g\otimes\gamma)(ch(\E_1)\sqrt{td_{X_1}})]_1.$
\item
\label{def-item-g-gamma-1-maps-dual-of-univ-class-to-univ-class}
We say that $\gamma_{g,\epsilon}(\E_1,\E_2)$ 
{\em maps a universal class  
of $\M(w_1)$ to 
a universal class (or the dual of a universal class) of $\M(w_2)$},
if $d_{X_2}^{\epsilon}g(d_{X_1}^{\epsilon})^{-1}\otimes
\gamma_{g,\epsilon}(\E_1,\E_2)$ does. 
\end{enumerate}
}
\end{defi}

\hide{
\begin{rem}
\label{rem-maps-universal-class-to-universal-class-iff}
Let $\gamma:H^*(\M(w_1),\Integers)\rightarrow H^*(\M(w_2),\Integers)$
be an isomorphism of graded rings. 
Note that $g\otimes \gamma$ maps a universal class to a universal class, if and only if
$d_{X_2}gd_{X_1}^{-1}\otimes \gamma$ maps the dual of a universal class to the dual of a universal class, 
since $d_{\M(w_2)}\gamma d^{-1}_{\M(w_1)}=\gamma$, and so 
\[
(d_{X_2}\otimes d_{\M(w_2)})\circ(g\otimes \gamma)\circ (d^{-1}_{X_1}\otimes d^{-1}_{\M(w_2)})=
d_{X_2}gd_{X_1}^{-1}\otimes \gamma.
\]
Similarly, $g\otimes \gamma$ maps a universal class to the dual of a universal class, if and only if
$d_{X_2}gd_{X_1}^{-1}\otimes \gamma$ maps the dual of a universal class to a universal class.
\end{rem}

}

The following is a characterization of the class
$\gamma_g(\E_{1},\E_{2})$. 
Let $(X_1,\LB_1)$ and $(X_2,\LB_2)$ be polarized K3 or abelian surfaces, 
$\M_{\LB_1}(w_1)$ and $\M_{\LB_2}(w_2)$ compact moduli spaces of 
stable sheaves, and $\E_i$ a universal sheaf over 
$X_i\times \M_{\LB_i}(w_i)$. 

\begin{new-lemma}
\label{lemma-recovering-f}
\cite[Lemma 5.2]{markman-monodromy-I}
Suppose that $f: H^*(\M_{\LB_1}(w_1),\Integers) \rightarrow 
H^*(\M_{\LB_2}(w_2),\Integers)$ 
is a {\em ring} isomorphism, 
$g : H^*(X_1,\Integers) \rightarrow H^*(X_2,\Integers)$ 
a linear homomorphism, preserving 
the parity of the cohomological degree, 
and  $f\otimes g$ 
maps a universal class of 
$\M_{\LB_1}(w_1)$ to a  universal class of 
$\M_{\LB_2}(w_2)$.
Then
$
[f] \ = \ \gamma_g(\E_1,\E_2).
$
In particular, given $g$, a ring isomorphism $f$, satisfying
the condition above, is {\em unique} (if it exists).
\end{new-lemma}

\begin{cor}
\label{cor-composition}
Assume that $\gamma_{g}(\E_1,\E_2)$
maps a universal class of $\M(w_1)$ to a universal class of $\M(w_2)$,
$\gamma_{h}(\E_2,\E_3)$
maps a universal class of $\M(w_2)$ to a universal class of $\M(w_3)$,
and both $\gamma_{g}(\E_1,\E_2)$ and $\gamma_{h}(\E_2,\E_3)$ are 
ring isomorphisms.
Then $\gamma_{hg}(\E_1,\E_3)=\gamma_{h}(\E_2,\E_3)\circ\gamma_{g}(\E_1,\E_2)$.
\end{cor}

\begin{new-lemma}
\label{lemma-on-gamma-and-epsilon}
Let $g:H^*(X_1,\Integers)\rightarrow H^*(X_2,\Integers)$
be a linear homomorphism.
\begin{enumerate}
\item
\label{lemma-item-gamma-g-1-in-terms-of-gamma-g-0}
${\displaystyle
\gamma_{g,1}(\E_1,\E_2) \ \ \ = \ \ \ 
(d_{\M(w_1)}^{-1}\otimes 1)\gamma_{d^{-1}_{X_2}g,0}(\E_1,\E_2)  \ \ \ = \ \ \ 
(d_{\M(w_1)}\otimes 1)\gamma_{d_{X_2}g,0}(\E_1,\E_2).
}$
\item
\label{lemma-item-gamma-g-1-as-a-homomorphism}
When regarded as homomorphisms, then
\begin{equation}
\label{eq-gamma-g-1-as-a-homomorphism}
\gamma_{g,1}(\E_1,\E_2) \ \ \ = \ \ \ 
\gamma_{d^{-1}_{X_2}g,0}(\E_1,\E_2)\circ d_{\M(w_1)} 
\ \ \ = \ \ \ 
\gamma_{d_{X_2}g,0}(\E_1,\E_2)\circ d^{-1}_{\M(w_1)}. 
\end{equation}
Consequently, we have the following equalities:
\begin{eqnarray}
\label{eq-g-tensor-gamma-g-1-composied-with-dualization}
(d^{-1}_{X_2}gd_{X_1}\otimes 
\gamma_{g,1}(\E_1,\E_2))\circ (d_{X_1}^{-1}\otimes d_{\M(w_1)}^{-1})
& = &
d_{X_2}^{-1}g\otimes \gamma_{d_{X_2}^{-1}g,0}(\E_1,\E_2), 
\\
\nonumber
(d_{X_2}gd_{X_1}^{-1}\otimes 
\gamma_{g,1}(\E_1,\E_2))\circ (d_{X_1}\otimes d_{\M(w_1)})
& = &
d_{X_2}g\otimes \gamma_{d_{X_2}g,0}(\E_1,\E_2). 
\end{eqnarray}
\item
\label{lemma-item-if-gamma-g-1-maps-then-gamma-g-0-maps}
If $\gamma_{g,1}(\E_1,\E_2)$ maps a universal class of 
$\M(w_1)$ to the {\em dual} of a universal class of $\M(w_2)$, then 
$d_{X_2}g \otimes \gamma_{d_{X_2}g,0}(\E_1,\E_2)$ 
maps a universal class of $\M(w_1)$
to a universal class of $\M(w_2)$. 
\end{enumerate}
\end{new-lemma}

\begin{proof}
\ref{lemma-item-gamma-g-1-in-terms-of-gamma-g-0})
\hide{
If 
$f:H^1(X_1,\Integers)\rightarrow H^*(X_2,\Integers)$ 
is a ring isomorphism, then
$(f\otimes f)(\Delta_{X_1})=\Delta_{X_2}$, so 
$(f\otimes 1)(\Delta_{X_1})=(f\otimes 1)(f^{-1}\otimes f^{-1})(\Delta_{X_2})=
(1\otimes f^{-1})(\Delta_{X_2})$. If $g$ preserves the pairing 
(\ref{eq-Mukai-pairing}) then 
\[
\int_{X_1}\tau(\alpha)\cup\beta = 
\int_{X_2}\tau g(\alpha)\cup g(\beta) =
\int_{X_2}(\tau g \tau^{-1})\tau(\alpha)\cup g(\beta).
\]
Consequently, $(\tau g \tau^{-1}\otimes g)(\Delta_{X_1}) = \Delta_{X_2}$
and
\[
(1\otimes g)(\Delta_{X_1}) \ \ \ = \ \ \ 
(\tau g^{-1}\tau^{-1}\otimes 1)(\Delta_{X_2}).
\] 
Since $d_{X_1}$ commutes with $\tau$ and is a ring isomorphism, we get
\begin{equation}
\label{eq-g-d-X}
(d_{X_1}\otimes 1)(1\otimes g)(\Delta_{X_1}) = 
(\tau d_{X_1}g^{-1}\tau^{-1}\otimes 1)(\Delta_{X_2}) =
(1\otimes gd^{-1}_{X_1})(\Delta_{X_1}).
\end{equation}
}
The class $\gamma_{g,1}(\alpha_1,\alpha_2)$ involves the term
${\displaystyle \pi_{13_*}\left[
\pi_{12}^*(\alpha_1)\pi_{24}^*[(1\otimes g)(\Delta_{X_1})]\pi_{34}^*(\alpha_2)
\right]}$.
The corresponding class in the definition of 
$\gamma_{h,0}(\alpha_1,\alpha_2)$ is:
\[
\begin{array}{lc}
\pi_{13_*}\left[
D\left\{
\pi_{12}^*(\alpha_1)\pi_{24}^*[(1\otimes h)(\Delta_{X_1})]
\right\}
\pi_{34}^*(\alpha_2)
\right] &  =
\\
\pi_{13_*}\left[
\pi_{12}^*(D\alpha_1)\pi_{24}^*[
(d_{X_1}\otimes d_{X_2})(1\otimes h)(\Delta_{X_1})
]
\pi_{34}^*(\alpha_2)
\right]. 
\end{array}
\]
Integrating first along the $X_1$ factor and using the fact that $d_{X_1}$
is a ring automorphism preserving the intersection pairing we get:
\[
\pi_{13_*}\left[
\pi_{12}^*((d_{\M(w_1)}\otimes 1)\alpha_1)\pi_{24}^*[
(1\otimes d_{X_2}h)(\Delta_{X_1})
]
\pi_{34}^*(\alpha_2)
\right].
\]
Now we can ``pull $(d_{\M(w_1)}\otimes 1)$ out'' as it commutes with 
the Gysin map $\pi_{13_*}$ (by the projection formula). Being
a ring automorphism, $(d_{\M(w_1)}\otimes 1)$ commutes
with the map $\ell$, the inversion, and projection on the 
class of degree $2m$. Setting $h=d_{X_2}^{-1}g$ we get the first identity.
The second identity follows using the same argument, the fact that $\gamma_{g,1}(\E_1,\E_2)$ and $(1\otimes h)(\Delta_X)$ are even cohomology classes, and the identities $D^2=id$ on the even cohomology of the products
$X_i\times\M(w_i)$ and $X_1\times X_2$, so
$d_{\M(w_1)}\otimes d_{X_1}$ and $d_{\M(w_1)}^{-1}\otimes d_{X_1}^{-1}$ restrict to the same automorphism of the even cohomology and so do
$d_{X_1}\otimes d_{X_2}$ and $d_{X_1}^{-1}\otimes d_{X_2}^{-1}$.

\ref{lemma-item-gamma-g-1-as-a-homomorphism})
The equalities in (\ref{eq-gamma-g-1-as-a-homomorphism}) follow from Part (\ref{lemma-item-gamma-g-1-in-terms-of-gamma-g-0}).
Each equation in (\ref{eq-g-tensor-gamma-g-1-composied-with-dualization}) follows from the corresponding equation 
in (\ref{eq-gamma-g-1-as-a-homomorphism}).

\ref{lemma-item-if-gamma-g-1-maps-then-gamma-g-0-maps}) 
The equality $\gamma_{g,1}(\E_1,\E_2)\circ d_{\M(w_1)}=d_{\M(w_2)} \circ \gamma_{g,1}(\E_1,\E_2)$ holds,
since $\gamma_{g,1}(\E_1,\E_2)$ is assumed to be a graded ring isomorphism. 
The second equality below follows.
\begin{eqnarray*}
d_{X_2}g\otimes \gamma_{d_{X_2}g,0}(\E_1,\E_2)
&\stackrel{(\ref{eq-g-tensor-gamma-g-1-composied-with-dualization})}{=} &
d_{X_2}g\otimes (
\gamma_{g,1}(\E_1,\E_2)
\circ d_{\M(w_1)}
)
=
d_{X_2}g\otimes 
(
d_{\M(w_2)} \circ 
\gamma_{g,1}(\E_1,\E_2)
)
\\
&=& (d_{X_2}\otimes d_{\M(w_2)})\circ (g\otimes \gamma_{g,1}(\E_1,\E_2)).
\end{eqnarray*}
The latter is assumed to map a universal class to a universal class. Hence, so does 
$d_{X_2}g\otimes \gamma_{d_{X_2}g,0}(\E_1,\E_2)$.
\end{proof}

\begin{new-lemma}
\label{lemma-composition-of-gamma-g-1-and-gamma-h-1}
Assume that $\gamma_{g,1}(\E_1,\E_2)$
maps a universal class of $\M(w_1)$ to the dual of 
a universal class of $\M(w_2)$,
$\gamma_{h,1}(\E_2,\E_3)$
maps a universal class of $\M(w_2)$ to the dual of 
a universal class of $\M(w_3)$, and both $\gamma_{g,1}(\E_1,\E_2)$
and $\gamma_{h,1}(\E_2,\E_3)$ are ring isomorphisms.
Then $\gamma_{\tau h\tau g,0}(\E_1,\E_3)=\gamma_{h,1}(\E_2,\E_3) \circ 
\gamma_{g,1}(\E_1,\E_2)$ and it is a ring isomorphism that
maps a universal class of $\M(w_1)$ to a universal class of $\M(w_3)$. 
\end{new-lemma}

\begin{proof}
The equality
\[
\begin{array}{l}
[d_{X_3}h d_{X_2}^{-1}\otimes \gamma_{h,1}(\E_2,\E_3)] \circ
D^{-1}_{X_2\times \M(w_2)}\circ 
[d_{X_2}gd_{X_1}^{-1}\otimes \gamma_{g,1}(\E_1,\E_2)] \circ
D_{X_1\times \M(w_1)} \ \ = \ \ 
\\
(d_{X_3}h d_{X_2}^{-1}g)\otimes 
[\gamma_{h,1}(\E_2,\E_3) \circ 
\gamma_{g,1}(\E_1,\E_2)] 
\end{array}
\]
follows from the definition of $d_{X_i}$ and $D_{\M(w_i)}$, $i=1,2,3$.
The left hand side maps a universal class of $\M(w_1)$ to a universal class of $\M(w_3)$, by Lemma
\ref{lemma-on-gamma-and-epsilon} (\ref{lemma-item-if-gamma-g-1-maps-then-gamma-g-0-maps}). 
The right hand side is 
the tensor product of the integral homomorphism $d_{X_3}h d_{X_2}^{-1}g$ with a ring isomorphism. 
Lemma \ref{lemma-recovering-f} implies that the latter ring isomorphism must be 
$\gamma_{d_{X_3}hd_{X_2}^{-1}g,0}(\E_1,\E_3)$. Consequently, 
\[
(d_{X_3}hd_{X_2}^{-1}g)\otimes [\gamma_{d_{X_3}hd_{X_2}^{-1}g,0}(\E_1,\E_3)]
\]
maps a universal class of $\M(w_1)$ to a universal class of $\M(w_3)$. 
Finally, the equality $\gamma_{d_{X_3}hd_{X_2}^{-1}g,0}(\E_1,\E_3)=
\gamma_{\tau h\tau g,0}(\E_1,\E_3)$ follows from the 
equality $d_{X_3}hd_{X_2}^{-1}=\tau h\tau$ (see Equation (\ref{eq-conjugation-by-d-X-i})). 
\end{proof}


\begin{new-lemma}
\label{lemma-composition-of-gamma-g-0-and-gamma-h-1}
Assume that $\gamma_{g,0}(\E_1,\E_2)$
maps a universal class of $\M(w_1)$ to a universal class of $\M(w_2)$,
$\gamma_{h,1}(\E_2,\E_3)$
maps a universal class of $\M(w_2)$ to the dual of 
a universal class of $\M(w_3)$, 
$\gamma_{f,0}(\E_3,\E_4)$ maps a universal class of $\M(w_3)$ to a universal class of $\M(w_4)$
and  $\gamma_{g,0}(\E_1,\E_2)$,
$\gamma_{h,1}(\E_2,\E_3)$, and $\gamma_{f,0}(\E_3,\E_4)$ are ring isomorphisms.
Then $\gamma_{hg,1}(\E_1,\E_3)=\gamma_{h,1}(\E_2,\E_3) \circ 
\gamma_{g,0}(\E_1,\E_2)$ and it is a ring isomorphism that
maps a universal class of $\M(w_1)$ to the dual of a universal class of $\M(w_3)$. 
Similarly, $\gamma_{f,0}(\E_3,\E_4)\gamma_{h,1}(\E_2,\E_3)=\gamma_{\tau f\tau h,1}(\E_2,\E_4)$ and it is a ring isomorphism that
maps a universal class of $\M(w_2)$ to the dual of a universal class of $\M(w_4)$. 
\end{new-lemma}

\begin{proof}
The proof is similar to that of Lemma \ref{lemma-composition-of-gamma-g-1-and-gamma-h-1}. We check only the latter equality.
$\gamma_{h,1}(\E_2,\E_3)=\gamma_{d_{X_3}h,0}(\E_2,\E_3)\circ d_{\M(w_2)}^{-1},$
by Lemma \ref{lemma-on-gamma-and-epsilon}(\ref{lemma-item-gamma-g-1-as-a-homomorphism}). 
Hence,
\[
\gamma_{f,0}(\E_3,\E_4)\gamma_{h,1}(\E_2,\E_3)=
\gamma_{fd_{X_3}h,0}(\E_2,\E_4)\circ d_{\M(w_2)}^{-1},
\]
by Corollary \ref{cor-composition}.
The right hand side is equal to $\gamma_{d_{X_4}^{-1} f d_{X_3} h,1}(\E_2,\E_4)$, by 
Lemma \ref{lemma-on-gamma-and-epsilon}(\ref{lemma-item-gamma-g-1-as-a-homomorphism}). 
The latter is equal to $\gamma_{\tau f\tau h,1}(\E_2,\E_4)$, by Equation (\ref{eq-conjugation-by-d-X-i}).
\end{proof}

Let $g$ and $\gamma$ be as in Definition \ref{def-g-gamma-maps-universal-classes-to-such}(\ref{def-item-maps-universal-class-to-such}).

\begin{new-lemma}
\label{lemma-dual-of-a-universal-class-to-dual-of-a-universal-class}
If $g\otimes\gamma$ maps a universal class of $\M(w_1)$ to a universal class of $\M(w_2)$, then $(\tau g\tau)\otimes\gamma$ maps the dual of a universal class of $\M(w_1)$ to the dual of a universal class of $\M(w_2)$.
\end{new-lemma}

\begin{proof} Set $D_i:=D_{X_i\times\M(w_i)}$, $i=1,2$. We have
\begin{eqnarray*}
(\tau g\tau\otimes \gamma)(ch(\E_1)^\vee\sqrt{td_{X_1}}) &=&
(D_2(g\otimes\gamma)D_1)\left(D_1(ch(\E_1)\sqrt{td_{X_1}})\right)=
\\
(D_2(g\otimes\gamma))\left(ch(\E_1)\sqrt{td_{X_1}}\right)&=&D_2\left(ch(\E_2)\sqrt{td_{X_2}}\exp(c_g)\right)=
ch(\E_2)^\vee\sqrt{td_{X_2}}\exp(-c_g).
\end{eqnarray*}
\end{proof}
\hide{
\section{The automorphism group of a generic manifold of generalized Kummer type}
\label{sec-generic-automorphisms}

Let $\widetilde{\Gamma}_X$ be the subgroup of $\Aut(X)$ generated by 
the group $\Gamma_X$ of translations by points of order $n$ and by
multiplication by $-1$. Then $\widetilde{\Gamma}_X$ is the semi-direct 
product and hence an extension 
\[
0 \rightarrow \Gamma_X \rightarrow \widetilde{\Gamma}_X
\rightarrow \Integers/2\Integers \rightarrow 0.
\]
$\widetilde{\Gamma}_X$ acts on $X^{[n]}$ and the fiber $K_X(n\!-\!1)$
of $\pi$ over $0\in X$ is $\widetilde{\Gamma}_X$ invariant. 
$\widetilde{\Gamma}_X$ acts trivially on $H^2(K_X(n\!-\!1),\Integers)$. 

\begin{thm}
\label{thm-generic-automorphism} (??? Generalize for monodromy operators acting trivially on $H^2$ ???)
Let $u$ be an automorphism of $K_X(n\!-\!1)$ acting trivially on 
$H^2(K_X(n\!-\!1),\Integers)$. Then $u$ belongs to the image of 
$\widetilde{\Gamma}_X$ in $\Aut(K_X(n\!-\!1))$.
\end{thm}

\begin{proof}
???
\end{proof}
}

\section{Equivalences of derived categories}
\label{sec-derived-categories}

Let $X$ be an abelian surface and let $V$, $S^+$, and $S^-$, be the regular and half spin integral representations of $\Spin(V)$ recalled in
Section \ref{sec-Clifford-groups}.
Let $\Aut(D^b(X))$ be the group of auto-equivalences of the bounded derived category of coherent sheaves on  $X$.
Mukai, Polishchuk, and Orlov, constructed a homomorphism $\Aut(D^b(X))\rightarrow \Spin(V)$, whose image is equal to the 
subgroup preserving the Hodge structure of $V$ (see \cite[Theorem 3.5]{mukai-spin},
\cite[Prop. 4.3.7]{golyshev-luntz-orlov}, and \cite[Prop. 9.48]{huybrechts-book}). 
Their result holds for abelian varieties of arbitrary dimension. For abelian surfaces 
we get a homomorphism $\Aut(D^b(X))\rightarrow \Spin(S^+)$ using the equality (\ref{eq-Spin-V-equals-Spin-S-plus}). 
In Corollary \ref{cor-reflections-in-two-line-bundles} 
below we exhibit an explicit lift to $\Aut(D^b(X))$ of products $m_s m_t\in \Spin(S^+)$,
for certain pairs of elements $s,t\in S^+$ each of self-intersection $2$. These lifts will be shown in Section 
\ref{sec-monodromy-via-Fourier-Mukai} to induce isomorphisms of certain moduli spaces of stable sheaves.
\hide{
{\bf A dictionary:}
\[
\begin{array}{lll}
\underline{\mbox{Derived Categories}}
&
\underline{\mbox{Spinors/Cohomology}}
\\
D^b(X\times X) & C(V)\cong \End[H^*(X,\Integers)] \ \ \  
\mbox{Clifford algebra}
\\
D^b(X) & S^+\oplus S^- \cong H^*(X,\Integers) \ \ \  \mbox{Clifford module}
\\
D^b(X\times \hat{X}) & V \ \ \  \mbox{the standard representation}
\\
Auteq D^b(X) & \Spin(V)
\\
Auteq D^b(X)\ni\Phi_\E \mapsto \E \in D^b(X\times X) &
\Spin(V)\subset C(V)
\\
\Phi_{S_X}:D^b(X\times \hat{X})\rightarrow D^b(X\times X) &
V\subset C(V) 
\\
\E\mapsto \Phi^{-1}_{S_X}\circ Ad_\E\circ \Phi_{S_X}&
C(V)\supset \Spin(V) \ni x \mapsto x\cdot(\bullet)\cdot x^{-1}\in  SO(V).
\\
\hspace{0ex} [1] \ \ \mbox{The shift} & -1 \in \Spin(V)
\\
\Phi_\P:D^b(X)\rightarrow D^b(\hat{X}), 
\ \P \ \mbox{the Poincare l.b.} &
(-1)^{i(i+1)/2}PD : H^i(X,\Integers)\rightarrow H^{4-i}(\hat{X},\Integers)
\\
\Psi_\P:D^b(\hat{X})\rightarrow D^b(X), 
\ \P \ \mbox{the Poincare l.b.} &
(-1)^{i(i+1)/2}PD : H^i(\hat{X},\Integers)\rightarrow H^{4-i}(X,\Integers)
\end{array}
\]

Provide a dictionary, describing how autoequivalences of $D^b(X)$, 
such as tensorization by a line-bundle $L$ with $c_1(L)\neq 0$,
acts on $V$, $S^+$, and $S^{-}$. 
}

%
\subsection{Tensorization by a line bundle}
Let $F$ be a line bundle on $X$.
Denote by $\phi_F\in \Spin(V)$ the element corresponding to the auto-equivalence of $D^b(X)$ of tensorization by $F$.
Explicitly, $\phi_F$ is the element of $\Spin(V)$, which maps to the element of $GL(S)=GL(H^*(X,\Integers))$ acting by multiplication by the Chern character of $F$. That this element of $GL(S)$ is the image of a unique element of $\Spin(V)$ is proven directly in 
\cite[III.1.7]{chevalley}.

\begin{new-lemma}
\label{lemma-tensorization-by-line-bundle-F}
$\phi_F$ acts on $A_X$ as follows:
\begin{enumerate}
\item
On $S^+$: Given $(r,H,s)\in S^+$,
\[
\phi_F(r,H,s) \ \ = \ \ (r,H+rc_1(F),s+r\frac{c_1(F)^2}{2}+H\wedge c_1(F)).
\]
\item
\label{lemma-item-action-on-S-minus}
On $S^-$: Given $(w,w')\in H^1(X,\Integers)\oplus H^3(X,\Integers)$,
\[
\phi_F(w,w') \ \ = \ \ (w,w'+c_1(F)\wedge w).
\]
\item
\label{lemma-item-action-on-V-by-tensorization-by-line-bundle}
On $V$: Given $(w,\theta)\in H^1(X,\Integers)\oplus H^1(X,\Integers)^*$,
\[
\rho(\phi_F)(w,\theta) \ \ = \ \ 
(w-D_\theta(c_1(F)),\theta).
\]
\end{enumerate}
\end{new-lemma}

\begin{proof}
The action on $H^*(X,\Integers)$ is multiplication by
the Chern character $ch(F):=1+c_1(F)+\frac{c_1(F)^2}{2}$ of $F$.
For the action on $V$, embed $V$ in $C(V)\subset \End[H^*(X,\Integers)]$ 
sending $(w,\theta)$ to $L_w+D_\theta$ as in equation
(\ref{eq-left-wedge-by-w}). Conjugation yields:
\begin{eqnarray*}
\phi_F\circ (L_w)\circ (\phi_F)^{-1} & = & L_w,
\\
\phi_F\circ (D_\theta)\circ (\phi_F)^{-1}  & = &
D_\theta-L_{D_\theta(c_1(F))},
\end{eqnarray*}
where the last equality is verified as follows. Set $f:=c_1(F)$. 
For $a\in S$ we get
\[
ch(F)\left(D_\theta (ch(F^{-1})a)\right)=D_\theta(a)+ch(F)[D_\theta(ch(F^{-1}))]a,
\]
and $ch(F)D_\theta(ch(F^{-1}))=[1+f+f^2/2][-D_\theta(f)+fD_\theta(f)]=-D_\theta(f).$
(Compare Part (\ref{lemma-item-action-on-V-by-tensorization-by-line-bundle}) with the last displayed formula on page 74 in the proof of \cite[III.1.7]{chevalley}).
\end{proof}

%
\subsection{Fourier-Mukai transform with kernel the Poincare line bundle}
Let $\pi_X$ and $\pi_{\hat{X}}$ be the projections from $X\times \hat{X}$ onto the corresponding factor.
Let $\P$ be the normalized Poincare line bundle over $X\times \hat{X}$. $\P$ restricts to $X\times \{t\}$, $t\in \hat{X}$, as a line bundle in the isomorphism class  $t$ and to $\{0\}\times \hat{X}$ as the trivial line bundle.
The Fourier-Mukai functor 
$\Phi_\P:D^b(X)\rightarrow D^b(\hat{X})$ with kernel $\P$ is given by
$R\pi_{\hat{X},*}(L\pi_X^*(\bullet)\otimes\P)$.
Let
$\iota:H^*(X,\Integers)^*\rightarrow H^*(\hat{X},\Integers)$ be the natural isomorphism identifying 
$H^j(X,\Integers)^*$ and $H^j(\hat{X},\Integers)$. 
Explicitly, $\iota^{-1}:H^1(\hat{X},\Integers)\rightarrow H^1(X,\Integers)^*$ is the dual of the composition of the isomorphisms
\[
H^1(X,\Integers)\rightarrow H_1(\Pic^0(X),\Integers)\RightArrowOf{=}H_1(\hat{X},\Integers)\rightarrow H^1(\hat{X},\Integers)^*,
\]
where the left one is induced by the identification $\Pic^0(X)=H^1(X,\StructureSheaf{X})/H^1(X,\Integers)$ via the exponential sequence.
For $k>1$, $\iota$ is given by the composition
$[\wedge^k H^1(X,\Integers)]^*\cong \wedge^k[H^1(X,\Integers)^*]\rightarrow \wedge^kH^1(\hat{X},\Integers)$, where the left isomorphism is the natural one and the right is the $k$-th exterior power of $\iota:H^1(X,\Integers)^*\rightarrow H^1(\hat{X},\Integers)$.
On the level of cohomology $\Phi_\P$ induces 
$\phi_\P:=\sum_{i=0}^4\phi^i_\P$, where
\begin{equation}
\label{eq-cohomological-fourier-mukai-homomorphism}
\phi^i_\P \ \ := \ \ (-1)^{i(i+1)/2}\iota\circ PD_X \ \ : \ \ H^i(X,\Integers)
\rightarrow H^{4-i}(\hat{X},\Integers),
\end{equation}
and $PD_X$ is given in Equation (\ref{eq-Poincare-Duality}), by \cite[Prop. 1.17]{mukai-fourier-functor-and-its-applications}.
Let $\Psi_\P:D^b(\hat{X})\rightarrow \nolinebreak D^b(X)$ be the integral functor with kernel $\P^\vee\otimes \pi_{\hat{X}}^*\omega_{\hat{X}}$, where 
$\P^\vee$ is the line bundle dual to $\P$. Then $\Psi_\P[2]$ is the left adjoint of $\Phi_\P$, by \cite[Prop. 5.9]{huybrechts-book}. We have the natural isomorphism 
\begin{equation}
\label{eq-Psi-P-inverse-of-Phi-P}
\Psi_\P[2]\circ\Phi_\P\cong id_{D^b(X)},
\end{equation}
since $\Phi_\P$ is an equivalence \cite[Theorem 2.2]{mukai-duality}. 

\hide{
%
\subsection{Mukai's $SL(2,\Integers)$ action}
Let $X$ be an abelian surface and $L$ a line bundle on $X$ with $\ell:=c_1(L)$ satisfying $(\ell,\ell)_{S}=-2$.
Then one of $L$ or $L^{-1}$ is a principal polarization. Set $s_1:=(1,0,1)$, $s_{-1}:=(1,0,-1)\in S^+.$
Note that $\{s_1,s_{-1},\ell\}$ is an orthogonal subset of $S^+$, $(s_1,s_1)_S=2$, and $(s_{-1},s_{-1})_S=-2.$
We get the element $\phi_L\in \Spin(V)$ and the element $m_\ell m_{s_{-1}}\in \Spin(S^+)$. 
Let $\lambda:H^1(X,\Integers)^*\rightarrow H^1(X,\Integers)$ be given by
$\lambda(\theta):=D_\theta(\ell)$. 
Set $U:=\left(
\begin{array}{cc}
1&-1
\\
0&1
\end{array}
\right)$
and $J:=\left(
\begin{array}{cc}
0&1
\\
-1&0
\end{array}
\right)$. Then $U$ and $J$ generate $SL(2,\Integers)$, the relation $(UJ^{-1})^3=id$ holds, and $SL(2,\Integers)$ is the quotient of the free group generated by $U$ and $J$ modulo that relation.
Consider the homomorphism $\eta:SL(2,\Integers)\rightarrow SO(V)$ given by
\[
\eta\left(
\begin{array}{cc}
a&b
\\
c&d
\end{array}
\right)
:=
\left(
\begin{array}{cc}
a \cdot id_{H^1(X)}&b \lambda
\\
c\lambda^{-1}&d \cdot id_{H^1(X)^*}
\end{array}
\right)
\]
The element $\tilde{m}_{s_{-1}}\tilde{m}_\ell=(\tilde{m}_\ell\tilde{m}_{s_{-1}})^{-1}$ corresponds to 
the auto-equivalence $\varphi_L^*\circ \Phi_\P$ of $D^b(X)$, where $\Phi_\P:D^b(X)\rightarrow D^b(\hat{X})$ is the Fourier-Mukai functor with kernel $\P$ and $\varphi_L^*$ is the pull-back via the isomorphism $\varphi_L:X\rightarrow \hat{X}$, given by
$x\mapsto \tau_x^*L\otimes L^*$, where $\tau_x:X\rightarrow X$ is translation by $x$.
The following is a cohomological version of a theorem of Mukai (see \cite[Prop. 9.30]{huybrechts-book}).

\begin{new-lemma}
There exists an injective homomorphism 
$\tilde{\eta}:SL(2,\Integers)\rightarrow \Aut(A_X)$ lifting $\eta$ and determined by
$\tilde{\eta}(U)=\tilde{\mu}(\phi_L)$ and $\tilde{\eta}(J)=\tilde{m}_\ell\tilde{m}_{s_{-1}}$.
In particular, $\tilde{\mu}(\phi_L)$ acts on $V$ via $\eta(U)$
and $\tilde{m}_\ell\tilde{m}_{s_{-1}}$ via $\eta(J)$.
\end{new-lemma}

\begin{proof}
The equality $\tilde{\mu}(\phi_L)(w,\theta)=(w-\lambda(\theta),\theta)$, for $(w,\theta)\in V$, $w\in H^1(X,\Integers)$ and $\theta\in H^1(X,\Integers)^*$, follows from 
Lemma \ref{lemma-tensorization-by-line-bundle-F}(\ref{lemma-item-action-on-V-by-tensorization-by-line-bundle}).

$\tilde{m}_\ell\tilde{m}_{s_{-1}}$ acts on $S^+$ as the composition $R_\ell R_{s_{-1}}$ 
of the reflections, so it acts by $-1$ on $\mbox{span}\{\ell,s_{-1}\}$ and as the identity on its orthogonal complement. 
Set $e_1:=(1,0,0)=(s_1+s_{-1})/2$ and 
$e_2:=(0,0,1)=(s_1-s_{-1})/2$. 
Then $H^1(X,\Integers)=\ker(m_{e_2}:V\rightarrow S^-)$ and $H^1(X,\Integers)^*=\ker(m_{e_1}:V\rightarrow S^-)$. Now
$\tilde{m}_\ell\tilde{m}_{s_{-1}}$ interchanges $e_1$ and $e_2$. Thus, $(m_\ell m_{s_{-1}})m_{e_1}=m_{e_2}(m_\ell m_{s_{-1}})$.
Hence, $\tilde{m}_\ell\tilde{m}_{s_{-1}}$ maps $H^1(X,\Integers)$
to $H^1(X,\Integers)^*$ and vice versa. 

$m_\ell:V\rightarrow S^-$ restricts to $H^1(X,\Integers)^*$ as $\lambda$. Furthermore, $m_\ell m_\ell=-1\in \Spin(S^+).$
So $m_\ell:S^-\rightarrow V$ restricts to $H^1(X,\Integers)$ as $-\lambda^{-1}$. So, given $w\in H^1(X,\Integers)\subset V$,
\[
(\tilde{m}_\ell \tilde{m}_{s_{-1}})(w)=m_\ell(L_w(s_{-1}))=m_\ell(w)=-\lambda^{-1}(w).
\]
The equality $(\tilde{m}_\ell\tilde{m}_{s_{-1}})^2=\tilde{m}(-1)$ implies that $\tilde{m}_\ell \tilde{m}_{s_{-1}}$
restricts to $H^1(X,\Integers)^*$ as $\lambda$. Hence, $(\tilde{m}_\ell \tilde{m}_{s_{-1}})(w,\theta)=(\lambda(\theta),-\lambda^{-1}(w)).$
We conclude that the restrictions of $\tilde{\mu}(\phi_L)$ and $\tilde{m}_\ell\tilde{m}_{s_-1}$ to $V$ act as stated.

The saturation of the sublattice of $S^+$ spanned by $\beta:=\{s_1,s_{-1},\ell\}$ is invariant under 
$\tilde{\mu}(\phi_L)$ and $\tilde{m}_\ell\tilde{m}_{s_-1}$ and the $\beta$-matrices of their restrictions are
\[
[\tilde{m}_\ell\tilde{m}_{s_-1}]_\beta=
\left(
\begin{array}{ccc}
1 & 0 & 0
\\ 
0 & -1 & 0
\\
0 & 0 & -1
\end{array}
\right), \ \ 
\mbox{and} \ \
[\tilde{\mu}(\phi_L)]_\beta=
\left(
\begin{array}{ccc}
3/2 & 1/2 & 1
\\ 
-1/2 & 1/2 & -1
\\
1 & 1 & 1
\end{array}
\right).
\]
Hence, $(\tilde{\mu}(\phi_L)(\tilde{m}_\ell\tilde{m}_{s_-1})^{-1})^3$ acts as the identity on both $V$ and 
$\mbox{span}\{s_1,s_{-1},\ell\}$. In particular, it is not equal to the image $\tilde{\mu}(-1)$ of the non trivial element  of the kernel of
$\rho:\Spin(V)\rightarrow SO_+(V)$, as the latter acts by $-1$ on $S^+$. Hence, $(\tilde{\mu}(\phi_L)(\tilde{m}_\ell\tilde{m}_{s_-1})^{-1})^3=1$ and 
$\eta$ factors through a homomorphism 
$\tilde{\eta}:SL(2,\Integers)\rightarrow \Aut(A_X)$ given by 
$\tilde{\eta}(U)=\tilde{\mu}(\phi_L)$ and $\tilde{\eta}(J)=\tilde{m}_\ell\tilde{m}_{s_-1}$.
\end{proof}

If, furthermore, $X$ is simple, so that $\End(X)\cong\Integers$, then $\tilde{\mu}(\phi_L)$, $\tilde{m}_\ell\tilde{m}_{s_-1}$,
and $\tilde{\mu}(-1)$, generate the subgroup of $\tilde{\mu}(\Spin(V))$ acting via Hodge isometries on $V$.

We return to a general abelian surface $X$.
Let $PD_X:H^i(X,\Integers)\rightarrow H^{4-i}(X,\Integers)^*$ and 
$PD_{\hat{X}}:H^i(\hat{X},\Integers)\rightarrow H^{4-i}(\hat{X},\Integers)^*$
be the isomorphisms given in Equation (\ref{eq-Poincare-Duality}). 

}

The following Lemma deals with some delicate sign issues.

\begin{new-lemma}
\label{lemma-two-Poincare-dualities}
The following equalities hold for classes $\theta$ in $H^j(X,\Integers)$ and $\omega$ in $H^{4-j}(X,\Integers)$.
\begin{enumerate}
\item
\label{lemma-item-PD-compatible-with-topological-pairing}
$\int_X\theta\wedge\omega = \int_{\hat{X}}\iota(PD_X(\theta))\wedge \iota(PD_X(\omega)).$
\item
\label{lemma-item-comparison-of-two-Poincare-dualities}
$PD_{\hat{X}}^{-1}((\iota^*)^{-1}(\theta))=(-1)^j \iota(PD_X(\theta)).
$
\item
\label{lemma-item-fourier-mukai-induces-isometry}
The isomorphism $\phi_\P$ induces an isometry from 
$S_X:=H^*(X,\Integers)$ to $S_{\hat{X}}:=H^*(\hat{X},\Integers)$ with respect to the pairings given in (\ref{eq-Mukai-pairing}).
\end{enumerate}
\end{new-lemma}

\begin{proof}
Part \ref{lemma-item-PD-compatible-with-topological-pairing}) 
Let $\{e_1,e_2,e_3,e_4\}$ be a basis of $H^1(X,\Integers)$, compatible with the orientation, so satisfying $\int_Xe_1\wedge e_2\wedge e_3\wedge \nolinebreak e_4=1$,  and let $\{f_1,f_2,f_3,f_4\}$ be a dual basis of $H^1(\hat{X},\Integers)$, so that
$(\iota^{-1}(f_i))(e_j)=\delta_{i,j}$. The element $PD_X(e_1)\in H^3(X,\Integers)^*$ sends $e_2\wedge e_3\wedge e_4$ to $1$
and its kernel consists of $\{e_1\wedge e_i\wedge e_j \ : \ 1<i<j\}$, so $\iota(PD_X(e_1))=f_2\wedge f_3\wedge f_4$. Similarly, given a permutation $\sigma$ of $\{1,2,3,4\}$, 
\begin{eqnarray*}
\iota(PD_X(e_{\sigma(1)}))&=&sgn(\sigma)f_{\sigma(2)}\wedge f_{\sigma(3)}\wedge f_{\sigma(4)},
\\
\iota(PD_X(e_{\sigma(1)}\wedge e_{\sigma(2)}))&=&sgn(\sigma) f_{\sigma(3)}\wedge f_{\sigma(4)},
\\
\iota(PD_X(e_{\sigma(1)}\wedge e_{\sigma(2)}\wedge e_{\sigma(3)}))&=&sgn(\sigma)  f_{\sigma(4)}.
\end{eqnarray*}
The sign of the cyclic shift of $4$ elements is $-1$, and so 
\begin{eqnarray*}
\iota(PD_X(e_{\sigma(2)}\wedge e_{\sigma(3)}\wedge e_{\sigma(4)}))&=& -sgn(\sigma)  f_{\sigma(1)},
\\
\iota(PD_X(e_{\sigma(3)}\wedge e_{\sigma(4)}))&=& sgn(\sigma)  f_{\sigma(1)}\wedge  f_{\sigma(2)}.
\end{eqnarray*}
We get
$\int_{\hat{X}}\iota(PD_X(e_{\sigma(1)}))\wedge
\iota(PD_X(e_{\sigma(2)}\wedge e_{\sigma(3)}\wedge e_{\sigma(4)}))
=
-\int_{\hat{X}}f_{\sigma(2)}\wedge f_{\sigma(3)}\wedge f_{\sigma(4)}\wedge f_{\sigma(1)}=
\int_{\hat{X}} f_{\sigma(1)}\wedge f_{\sigma(2)}\wedge f_{\sigma(3)}\wedge f_{\sigma(4)}=sgn(\sigma)=
\int_X e_{\sigma(1)}\wedge e_{\sigma(2)}\wedge e_{\sigma(3)}\wedge e_{\sigma(4)}.
$ Similarly,
\begin{eqnarray*}
\int_{\hat{X}}\iota(PD_X(e_{\sigma(1)}\wedge e_{\sigma(2)}))\wedge
\iota(PD_X(e_{\sigma(3)}\wedge e_{\sigma(4)}))=
\int_{\hat{X}}f_{\sigma(3)}\wedge f_{\sigma(4)}\wedge f_{\sigma(1)}\wedge f_{\sigma(2)}
\\
=
\int_{\hat{X}} f_{\sigma(1)}\wedge f_{\sigma(2)}\wedge f_{\sigma(3)}\wedge f_{\sigma(4)}=
\int_X e_{\sigma(1)}\wedge e_{\sigma(2)}\wedge e_{\sigma(3)}\wedge e_{\sigma(4)}.
\end{eqnarray*}


Part \ref{lemma-item-fourier-mukai-induces-isometry} follows from Part \ref{lemma-item-PD-compatible-with-topological-pairing} and
Equation (\ref{eq-cohomological-fourier-mukai-homomorphism}).
Indeed,
\begin{eqnarray*}
\int_X\theta\wedge\omega&\stackrel{(\ref{eq-Mukai-pairing})}{=}&(-1)^{j(j-1)/2}(\theta,\omega)_{S_X},
\\
\int_{\hat{X}}\iota(PD_X(\theta))\wedge\iota(PD_X(\omega))&\stackrel{(\ref{eq-cohomological-fourier-mukai-homomorphism})}{=}&
(-1)^{j(j+1)/2}(-1)^{(4-j)(4-j+1)/2}\int_{\hat{X}}\phi_\P(\theta)\wedge\phi_\P(\omega)
\\
&=&
(-1)^{j(j-1)/2}(-1)^{(4-j)(4-j-1)/2}\int_{\hat{X}}\phi_\P(\theta)\wedge\phi_\P(\omega)
\\
&\stackrel{(\ref{eq-Mukai-pairing})}{=}&
(-1)^{\frac{j(j-1)}{2}}(-1)^{\frac{(4-j)(4-j-1)}{2}}(-1)^{\frac{(4-j)(4-j-1)}{2}}(\phi_\P(\theta),\phi_\P(\omega))_{S_{\hat{X}}}
\\
&=& (-1)^{\frac{j(j-1)}{2}}(\phi_\P(\theta),\phi_\P(\omega))_{S_{\hat{X}}}.
\end{eqnarray*}
Hence, $(\theta,\omega)_{S_X}=(\phi_\P(\theta),\phi_\P(\omega))_{S_{\hat{X}}}$, 
by Part \ref{lemma-item-PD-compatible-with-topological-pairing}.

Part \ref{lemma-item-comparison-of-two-Poincare-dualities}) Let $\gamma$ be a class in $H^j(\hat{X},\Integers)$.
We have the equalities
\begin{equation}
\label{eq-iota-inverse-gamma-of-theta}
(\iota^{-1}(\gamma))(\theta)=((\iota^{-1})^*(\theta))(\gamma)=((\iota^*)^{-1}(\theta))(\gamma)=
\int_{\hat{X}}PD_{\hat{X}}^{-1}((\iota^*)^{-1}(\theta))\wedge\gamma,
\end{equation}
where the last follows from 
the definition of $PD_{\hat{X}}$.
Now $\gamma=\iota(PD_X(\delta))$, for some $\delta\in H^{4-j}(X,\Integers)$.
Part \ref{lemma-item-PD-compatible-with-topological-pairing} yields the second equality below.
\begin{eqnarray*}
\int_{\hat{X}}\iota(PD_X(\theta))\wedge \gamma&=&\int_{\hat{X}}\iota(PD_X(\theta))\wedge \iota(PD_X(\delta))
\\
&=&
\int_X\theta\wedge\delta=(-1)^{j(4-j)}\int_X\delta\wedge\theta=(-1)^j\int_X\delta\wedge\theta
\\
&=&
(-1)^j (PD_X(\delta)(\theta))=(-1)^j (\iota^{-1}(\gamma))(\theta)
\\
&\stackrel{(\ref{eq-iota-inverse-gamma-of-theta})}{=}&
(-1)^j\int_{\hat{X}}PD_{\hat{X}}^{-1}((\iota^*)^{-1}(\theta))\wedge\gamma,
\end{eqnarray*}
for all $\gamma$ in $H^j(\hat{X},\Integers)$.
The equality in Part \ref{lemma-item-comparison-of-two-Poincare-dualities} follows.
\end{proof}

\begin{new-lemma}
\label{lemma-conjugation-of-derivative-by-phi-P}
The equality
$
\iota\left(PD_X(\beta\wedge w)\right)=D_{(\iota^*)^{-1}(w)}(\iota(PD_X(\beta)))
$
holds, for all  $w \in H^1(X,\Integers)$ and all $\beta\in H^2(X,\Integers)$.
\end{new-lemma}

\begin{proof}
Let $\gamma$ be a class in $H^1(X,\Integers)$.
Set $\tilde{w}:=(\iota^*)^{-1}(w)$ and $\tilde{\gamma}:=(\iota^*)^{-1}(\gamma)$.
Then $(\iota^*)^{-1}(w\wedge\gamma)=\tilde{w}\wedge\tilde{\gamma}$. We have
\begin{eqnarray*}
D_{\tilde{\gamma}}(D_{\tilde{w}}(\iota(PD_X(\beta))))&=&
(\tilde{w}\wedge\tilde{\gamma})(\iota(PD_X(\beta)))=
\int_{\hat{X}} PD^{-1}_{\hat{X}}(\tilde{w}\wedge\tilde{\gamma})\wedge \iota(PD_X(\beta))
\\
&\stackrel{{\rm \ Lemma \ \ref{lemma-two-Poincare-dualities}(\ref{lemma-item-comparison-of-two-Poincare-dualities})}}{=}&
\int_{\hat{X}}\iota(PD_X(w\wedge\gamma))\wedge \iota(PD_X(\beta))
\\
&\stackrel{{\rm Lemma \ \ref{lemma-two-Poincare-dualities}(\ref{lemma-item-PD-compatible-with-topological-pairing})}}{=}&
\int_Xw\wedge\gamma\wedge\beta=\int_X\beta\wedge w\wedge\gamma=(PD_X(\beta\wedge w))(\gamma)
\\
&=& D_{\tilde{\gamma}}(\iota(PD_X(\beta\wedge w))).
\end{eqnarray*}
\end{proof}

Let
$
\varphi_\P:V_X\rightarrow V_{\hat{X}}
$
be the isomorphism given by 
\begin{equation}
\label{eq-varphi-P}
\varphi_\P(w,\theta)=-(\iota(\theta),(\iota^*)^{-1}(w)),
\end{equation}
for all $(w,\theta)\in V_X$, $w\in H^1(X,\Integers)$ and $\theta\in H^1(X,\Integers)^*$.

\begin{new-lemma}
\label{lemma-phi-P-is-Spin-Spin-equivariant}
The following diagram is commutative.
\[
\xymatrix{
V_X\ar[r]^-{\subset} \ar[d]_{\varphi_\P} & C(V_X) \ar[r] & \End(S_X)\ar[d]^{Ad_{\phi_\P}} &\ni &f \ar[d]
\\
V_{\hat{X}} \ar[r]_-{\subset} & C(V_{\hat{X}}) \ar[r] & \End(S_{\hat{X}}) & \ni &\phi_\P f\phi_\P^{-1}.
}
\]
\end{new-lemma}

\begin{proof}
For $w\in H^1(X,\Integers)$ we need to show that $D_{\varphi_\P(w)}=Ad_{\phi_\P}(L_w)$. 
Evaluating both sides on $\gamma\in H^*(\hat{X},\Integers)$ the equality becomes
\begin{equation}
\label{eq-Ad-phi-P-of-L-w}
(\phi_\P\circ L_w\circ\phi_{\P}^{-1})(\gamma)=-D_{(\iota^*)^{-1}(w)}(\gamma).
\end{equation}
The above equality holds for all $\gamma\in H^2(\hat{X},\Integers)$, by Lemma \ref{lemma-conjugation-of-derivative-by-phi-P}
applied with $\beta:=\phi_\P^{-1}(\gamma)$. Both sides of Equation (\ref{eq-Ad-phi-P-of-L-w}) 
vanish for $\gamma\in H^0(\hat{X},\Integers)$.
For $\gamma:=[pt_{\hat{X}}]\in H^4(\hat{X},\Integers)$, the left hand side of (\ref{eq-Ad-phi-P-of-L-w}) is $\phi_\P(w)$.
The right hand side is
\[
-D_{(\iota^*)^{-1}(w)}[pt_{\hat{X}}]\stackrel{(\ref{eq-PD-of-D-theta-of-the-class-of-a-point})}{=}PD_{\hat{X}}^{-1}((\iota^*)^{-1}(w))
\stackrel{{\rm Lemma \ \ref{lemma-two-Poincare-dualities}(\ref{lemma-item-comparison-of-two-Poincare-dualities})}}{=}
-\iota(PD_X(w))=\phi_\P(w).
\]
Equation (\ref{eq-Ad-phi-P-of-L-w}) thus holds for all $\gamma\in H^{even}(\hat{X},\Integers)$. 
We have
$\phi_\P(L_{e_{\sigma(1)}}(\phi_\P^{-1}(f_{\sigma(1)})))=
\phi_\P(L_{e_{\sigma(1)}}(-sgn(\sigma)e_{\sigma(2)}\wedge e_{\sigma(3)}\wedge e_{\sigma(4)}))=
\phi_\P(-[pt_X])=-1\in H^0(\hat{X},\Integers)$. Similarly, 
$-D_{(\iota^*)^{-1}(e_{\sigma(1)})}(f_{\sigma(1)})=-1$. Both sides of (\ref{eq-Ad-phi-P-of-L-w}) vanish for 
$\gamma=f_i$ and $w=e_j$ if $i\neq j$. The case $\gamma\in H^3(\hat{X},\Integers)$ checks as well.

We denote by $\iota_{\hat{X}}:H^j(\hat{X},\Integers)^*\rightarrow H^j(\hat{\hat{X}},\Integers)$ the homomorphism  analogous to $\iota$. 
Let $\hat{\P}$ be the Poincare line bundle over $\hat{X}\times \hat{\hat{X}}$. 
The composition $\phi_{\hat{\P}}\circ\phi_\P:S_X\rightarrow S_{\hat{\hat{X}}}$ is equal to $\iota_{\hat{X}}\circ (\iota^*)^{-1}$.
Indeed,
\[
\phi_{\hat{\P}}\circ\phi_\P
=(-1)^{j(j+1)/2}(-1)^{(4-j)(4-j+1)/2}\iota_{\hat{X}}PD_{\hat{X}}\iota PD_X=
(-1)^j\iota_{\hat{X}}PD_{\hat{X}}\iota PD_X=
\iota_{\hat{X}}\circ (\iota^*)^{-1},
\]
where the last equality follows from Lemma \ref{lemma-two-Poincare-dualities}(\ref{lemma-item-comparison-of-two-Poincare-dualities}). 
Hence, given $\theta\in H^1(\hat{X},\Integers)^*$, we have
\[
Ad_{\phi_{\hat{\P}}}(-D_\theta)\stackrel{(\ref{eq-Ad-phi-P-of-L-w})}{=}Ad_{\phi_{\hat{\P}}}\circ Ad_{\phi_\P}(L_{\iota^*\theta})=
Ad_{[\iota_{\hat{X}}(\iota^*)^{-1}]}(L_{\iota^*\theta}).
\]
Now, $\iota_{\hat{X}}\circ (\iota^*)^{-1}$ is a cohomology ring isomorphism. Hence the right hand side above is equal to
$L_{(\iota_{\hat{X}}\circ (\iota^*)^{-1})(\iota^*\theta)}=L_{\iota_{\hat{X}}(\theta)}=-L_{\varphi_{\hat{\P}}(\theta)}.$ 
We conclude the equality
$Ad_{\phi_{\hat{\P}}}(D_\theta)=L_{\varphi_{\hat{\P}}(\theta)}$, for the dual of every abelian surface. Hence,
the equality $Ad_{\phi_{\P}}(D_\theta)=L_{\varphi_{\P}(\theta)}$ holds, for every $\theta\in H^1(X,\Integers)^*$.
\hide{
$\varphi_\P(\theta)$ is defined by the identity 
\[
(\phi_\P D_\theta \phi_\P^{-1})(\lambda)=\varphi_\P(\theta)\wedge \lambda,
\]
for all $\lambda\in H^3(\hat{X},\Integers)$. The latter is equivalent to 
$\phi_\P D_\theta(\gamma)=\varphi_\P(\theta)\wedge \phi_\P(\gamma)$, for all $\gamma\in H^1(X,\Integers)$.
The equality becomes
\begin{equation}
\label{eq-Ad-phi-of-theta}
D_\theta(\gamma)=\int_{\hat{X}}\varphi_\P(\theta)\wedge \phi_\P(\gamma).
\end{equation}
Let $\theta'\in H^3(X,\Integers)$ be the class satisfying $\varphi_\P(\theta)=\iota(PD_X(\theta'))$. 
The right hand side of the latter displayed equality is equal to $\int_X\theta'\wedge(-\gamma)$, 
by Lemma \ref{lemma-two-Poincare-dualities} (\ref{lemma-item-PD-compatible-with-topological-pairing}) and 
the identity $\phi_\P(\gamma)=-\iota(PD_X(\gamma))$. Now
\[
\int_X\theta'\wedge(-\gamma)=(-PD_X(\theta'))(\gamma)=(-\iota^{-1}(\varphi_\P(\theta)))(\gamma).
\]
The left hand side of (\ref{eq-Ad-phi-of-theta}) is just $\theta(\gamma)$. Hence, 
$\varphi_\P(\theta)=-\iota(\theta)$ as claimed.

$\varphi_\P(w)\in H^1(\hat{X},\Integers)^*$ is defined by the identity 
\[
(\varphi_\P(w))(\gamma)=\phi_\P(w\wedge \phi_\P^{-1}(\gamma)),
\]
for all $\gamma\in H^1(\hat{X},\Integers)$. Equivalently,
\[
\int_Xw\wedge \lambda=(\varphi_\P(w))(\iota(PD_X(\lambda))),
\]
for all $\lambda\in H^3(X,\Integers)$.
Set $\beta:=PD_{\hat{X}}^{-1}(\varphi_\P(w))$ and $\tilde{\beta}:=PD_X^{-1}(\iota^{-1}(\beta))$. 
The right hand side displayed above is equal to $\int_X\tilde{\beta}\wedge \lambda,$
by Lemma \ref{lemma-two-Poincare-dualities} (\ref{lemma-item-PD-compatible-with-topological-pairing}).
Hence, $\tilde{\beta}=w$ and so $\varphi_\P(w)=PD_{\hat{X}}(\iota(PD_X(w)))=-(\iota^*)^{-1}(w)$ as claimed, where the latter equality follows by Lemma \ref{lemma-two-Poincare-dualities}(\ref{lemma-item-comparison-of-two-Poincare-dualities}).
}
\end{proof}

\begin{rem}
The isomorphism $\iota_{\hat{X}}\circ(\iota^*)^{-1}:H^1(X,\Integers)\rightarrow H^1(\hat{\hat{X}},\Integers)$ corresponds to
the standard isomorphism $st:X\rightarrow \hat{\hat{X}}$.
The isomorphism $\varphi_\P$ corresponds to an isomorphism $\tilde{\varphi}_\P:X\times \hat{X}\rightarrow \hat{X}\times \hat{\hat{X}}$,
which pulls back the line bundle $\hat{\P}$ to $\P^{-1}$. 
The isomorphism $-\tilde{\varphi}_\P$ pulls back the line bundle $\hat{\P}$ to $\P^{-1}$ as well and it restricts to the second factor $\hat{X}$ as the identity onto the first factor $\hat{X}$ of the image. Similarly, $-\tilde{\varphi}_\P$ restricts to the first factor $X$ as the standard isomorphism $st$ onto the second factor $\hat{\hat{X}}$ of the image.
Mukai and Orlov composed $\tilde{\varphi}_\P$ with the isomorphism 
$id_{\hat{X}}\times -(st)^{-1}:\hat{X}\times \hat{\hat{X}}\rightarrow \hat{X}\times X$ obtaining 
an isomorphism $\tilde{\phi}_\P:X\times \hat{X}\rightarrow \hat{X}\times X$ which pulls back 
$\hat{\P}$ to $\P$, and whose square is minus the identity, see 
 \cite[Remark 9.12 and Example 9.38 (v) page 213]{huybrechts-book} or \cite[Remark 2.2]{orlov-abelian-varieties}. 
\end{rem}

The natural isomorphism 
$\varphi_\P:V_X\rightarrow V_{\hat{X}}$ and $\phi_\P:S_X\rightarrow S_{\hat{X}}$ combine to yield 
the linear isomorphism $\phi_\P:A_X\rightarrow A_{\hat{X}}$, which is an isometry. 
The isometry $\varphi_\P$ has a unique extension to an algebra isomorphism $\tilde{\varphi}_\P:C(V_X)\rightarrow C(V_{\hat{X}})$,
by the definition of the Clifford algebras, and $\tilde{\varphi}_\P(\Spin(V_X))=\Spin(V_{\hat{X}})$, by definition of the spin groups.
Lemma \ref{lemma-phi-P-is-Spin-Spin-equivariant} implies the equality
\[
\tilde{\mu}(\tilde{\varphi}_\P(g))=Ad_{\phi_\P}(\tilde{\mu}(g)),
\]
for all $g\in \Spin(V_X)$, where $\tilde{\mu}:\Spin(V_X)\rightarrow \Aut(A_X)$ is the restriction of the homomorphism given in (\ref{eq-representation-of-ker-N-on-A-X})
and we denote by $\tilde{\mu}:\Spin(V_{\hat{X}})\rightarrow \Aut(A_{\hat{X}})$ its analogue.
\hide{
Note that the free abelian group $A_{\hat{X}}$ is dual to $A_X$.
The isomorphism $\phi_\P$ restricts to each of the direct summands $V_X$ and $S_X^+$ of $A_X$ as the isomorphism with its dual induced by the
$\Spin(V_X)$-invariant bilinear pairings (\ref{eq-pairing-on-V}) and (\ref{eq-Mukai-pairing}). Similarly, $-\phi_\P$ restricts to $S_X^-$ as the isomorphism with its dual induced by the
$\Spin(V_X)$-invariant bilinear pairing (\ref{eq-Mukai-pairing}). Hence, if we view $A_{\hat{X}}\cong A_X^*$ as the dual $\Spin(V_X)$-representation,
then $\phi_\P$ is $\Spin(V_X)$-equivariant.
}

%
\subsection{Lifting products of two $(+2)$-reflections  to auto-equivalences}
Consider the natural homomorphism 
\begin{eqnarray}
\label{eq-homomorphism-from-Pic-X-to-Pic-X-hat}
\Pic(X) & \rightarrow & \Pic(\hat{X}),
\\
\nonumber
F &\mapsto & \hat{F}:=\det\left(\Phi_\P(F)\right)^{-1}. 
\end{eqnarray}
Then $c_1(\hat{F})=\iota(PD(c_1(F)))$, by (\ref{eq-cohomological-fourier-mukai-homomorphism}).

Let $F$ be a line bundle on $X$, $\phi_F\in \Spin(V_X)$ the isometry of
$H^*(X,\Integers)$ induced by tensorization with $F$.
Set $s:= (1,0,1)\in S^+_X$ and
$\hat{s}:= (1,0,1)\in S^+_{\hat{X}}$. 
Then $(s,s)_{S^+_X}=2$. 
Denote by $R_s$ the reflection (\ref{eq-Pin-acts-by-reflections}) in $s$.

\begin{new-lemma}
\label{lemma-auto-equivalence-acts-as-two-reflections}
The autoequivalence 
$\Phi_{\P}^{-1}\circ (\otimes \hat{F}) \circ \Phi_\P\circ (\otimes F^{-1})$
of $D^b(X)$ maps to the element $\tilde{m}_s\cdot \tilde{m}_{\phi_F(s)}$ of $\Aut(A_X)$. 
In other words, 
\begin{equation}
\label{eq-product-of-two-reflections-via-FM}
\tilde{\mu}\left(\phi_{\P}^{-1}\circ \phi_{\hat{F}} \circ \phi_{\P}\circ \phi_F^{-1} \right)
\ \ = \ \  
\tilde{m}_s\cdot \tilde{m}_{\phi_{F}(s)},
\end{equation}
where $\tilde{\mu}$ is the homomorphism  in equation (\ref{eq-representation-of-G-V-on-A-X})
and $\tilde{m}$ is defined in (\ref{eq-tilde-m}).
The displayed element acts on $S^+_X$ via the composition $R_s\circ R_{\phi_{F}(s)}$ of the two reflections.
\end{new-lemma}

\begin{proof}
Consider the composition $\eta:=\tilde{m}_{\hat{s}}\circ \phi_\P:A_X\rightarrow A_{\hat{X}}$.
Then $\eta$ maps $H^{even}(X,\Integers)$ to $H^{even}(\hat{X},\Integers)$  and 
preserves the grading and $\eta:S^+_X\rightarrow S^+_{\hat{X}}$
is given by
\[
\eta(r,H,t) \ \ = \ \ -(r,\iota(PD(H)),t).
\]

Set $\tilde{\phi}_F:=\tilde{\mu}(\phi_F)\in \Aut(A_X)$ and define $\tilde{\phi}_{\hat{F}}$ similarly.
We claim that conjugation yields the equality
\begin{equation}
\label{eq-conjugation-of-phi-F-by-eta}
\eta\circ\tilde{\phi}_F\circ \eta^{-1} \ \ \ = \ \ \ \tilde{\phi}_{\hat{F}}
\end{equation}
in $\tilde{\mu}(\Spin(V_{\hat{X}}))$.
It suffices to verify that both sides act the same way on $S^+_{\hat{X}}$ and on $V_{\hat{X}}$, since $\Spin(V_{\hat{X}})$
acts faithfully on the direct sum $S^+_{\hat{X}}\oplus V_{\hat{X}}$. 
Both sides of equalition (\ref{eq-conjugation-of-phi-F-by-eta})
map to the same element of $SO(S^+_{\hat{X}})$, by the above computation of $\eta$.
Let $\theta$ be a class in $H^1(X,\Integers)$, 
$\omega$ a class in $H^3(X,\Integers)$, set $\tilde{\theta}:=(\iota^*)^{-1}(\theta)\in H^1(\hat{X},\Integers)^*$ and
$\hat{\omega}:=\iota(PD_X(\omega))\in H^1(\hat{X},\Integers).$ 
Consider the element $(\hat{\omega},\tilde{\theta})\in V_{\hat{X}}$. 
We have the following equalities
\begin{equation}
\label{eq-action-of-phi-F-on-an-element-of-V-hat-X}
\tilde{\phi}_{\hat{F}}(\hat{\omega},\tilde{\theta})=(\hat{\omega}-D_{\tilde{\theta}}(c_1(\hat{F})),\tilde{\theta})=
(\hat{\omega}-D_{\tilde{\theta}}(\iota(PD_X(c_1(F)))),\tilde{\theta}),
\end{equation}
where the first follows from Lemma \ref{lemma-tensorization-by-line-bundle-F}(\ref{lemma-item-action-on-V-by-tensorization-by-line-bundle})
and the second from the equality $c_1(\hat{F})=\iota(PD_X(c_1(F)))$ observed above.
Let us evaluate $\eta\tilde{\phi}_{F}\eta^{-1}(\hat{\omega},\tilde{\theta}).$
Note that $\tilde{m}_{\hat{s}}^{-1}=\tilde{m}_{\hat{s}}$.
Example \ref{example-Clifford-multiplication-by-s-n} with $n=-1$ yields the equality
\[
\tilde{m}_{\hat{s}}^{-1}(\hat{\omega},\tilde{\theta})=(\hat{\omega},-PD_{\hat{X}}^{-1}(\tilde{\theta}))
\]
in $S^-_{\hat{X}}$.
Now $PD_{\hat{X}}^{-1}(\tilde{\theta})=\phi_\P(\theta)$, 
by Lemma \ref{lemma-two-Poincare-dualities}(\ref{lemma-item-comparison-of-two-Poincare-dualities}), 
and $\hat{\omega}=\phi_\P(\omega)$. Hence,
\begin{equation}
\label{eq-phi-P-and-PD-hat-X}
\phi_\P^{-1}(\hat{\omega},-PD_{\hat{X}}^{-1}(\tilde{\theta}))=(-\theta,\omega)
\end{equation}
in $S^-_X$.
The equality
$\tilde{\phi}_F(-\theta,\omega)=(-\theta,\omega-c_1(F)\wedge\theta)$ holds, 
by Lemma \ref{lemma-tensorization-by-line-bundle-F}(\ref{lemma-item-action-on-S-minus}).
We have the following two equalities, the first by (\ref{eq-phi-P-and-PD-hat-X}):
\begin{eqnarray*}
\phi_\P(-\theta,\omega-c_1(F)\wedge\theta)&=&(\hat{\omega}-\iota(PD_X(c_1(F)\wedge\theta)),-PD_{\hat{X}}^{-1}(\tilde{\theta})),
\\
\tilde{m}_{\hat{s}}(\hat{\omega}-\iota(PD_X(c_1(F)\wedge\theta)),-PD_{\hat{X}}^{-1}(\tilde{\theta}))&=&
(\hat{\omega}-\iota(PD_X(c_1(F)\wedge\theta)),\tilde{\theta}).
\end{eqnarray*}
The right hand side above is equal to the right hand side of Equation (\ref{eq-action-of-phi-F-on-an-element-of-V-hat-X}), by
Lemma \ref{lemma-conjugation-of-derivative-by-phi-P}. Hence, both sides of (\ref{eq-conjugation-of-phi-F-by-eta})
map to the same isometry of $V_{\hat{X}}$ as well and Equation (\ref{eq-conjugation-of-phi-F-by-eta}) is verified.

Substitute the equality (\ref{eq-conjugation-of-phi-F-by-eta}) into 
(\ref{eq-product-of-two-reflections-via-FM}) to get the following equalities in $\Aut(A_{\hat{X}})$.\\
$
\phi_{\P}^{-1}\circ \left(\eta\circ\tilde{\phi}_F\circ \eta^{-1}\right) 
\circ \phi_{\P}\circ \tilde{\phi}_F^{-1} \ \ = \ \  
(\phi_\P^{-1}\circ \tilde{m}_{\hat{s}}\circ \phi_\P)\circ
\tilde{\phi}_F\circ
(\phi_\P^{-1}\circ (\tilde{m}_{\hat{s}})^{-1}\circ \phi_\P)\circ
\tilde{\phi}_F^{-1} 
\\
\stackrel{(\ref{eq-conjugation-of-m-s-by-phi-P})}{=}
\tilde{m}_s\circ \tilde{\phi}_F\circ (\tilde{m}_s)^{-1}\tilde{\phi}_F^{-1} 
\ \ \stackrel{(\ref{eq-conjugation-of-m-s-by-phi-F})}{=} \ \ 
\tilde{m}_s\cdot \tilde{m}_{\phi_{F}(s)}.
$
\end{proof}

\begin{new-lemma}
\label{lemma-conjugation-of-m-s-by-phi-P}
The following equalities hold in $GL(A_X)$.
\begin{eqnarray}
\label{eq-conjugation-of-m-s-by-phi-P}
\phi_\P^{-1}\tilde{m}_{\hat{s}}\phi_\P&=&\tilde{m}_s,
\\
\label{eq-conjugation-of-m-s-by-phi-F}
\tilde{\phi}_F\tilde{m}_s\tilde{\phi}_F^{-1}&=&\tilde{m}_{\phi_F(s)}.
\end{eqnarray}
\end{new-lemma}

\begin{proof}
Equation (\ref{eq-conjugation-of-m-s-by-phi-P}): 
$\phi_\P$ restricts to an isometry from $S_X^+$ to $S_{\hat{X}}^+$, by Lemma \ref{lemma-two-Poincare-dualities}.
The right hand side acts on $S^+$ by $-R_s$, where $R_s$ is the reflection in $s$. The left hand side acts by  
$-R_{\phi_{\P}^{-1}(\hat{s})}=-R_s$  as well.

The element $\tilde{m}_s$ maps $S^-_X$ to $V_X$ and $V_X$ to $S^-_X$ and
$\tilde{m}_s^2=1\in \Aut(A_X)$. Hence, it suffices to check that both sides restrict to the same homomorphism from $S^-_X$ to $V_X$.
Let $(\alpha,\beta)\in S^-_X$, $\alpha\in H^1(X,\Integers)$ and $\beta\in H^3(X,\Integers)$. 
This checks as follows:
\begin{eqnarray*}
\phi_\P(\tilde{m}_s(\alpha,\beta))\stackrel{(\ref{eq-s-n-dagger-from-S-minus-to-V})}{=}\varphi_\P(\alpha,-PD_X(\beta))
\stackrel{{\rm Lemma \ \ref{lemma-phi-P-is-Spin-Spin-equivariant}}}{=}
(\iota(PD_X(\beta)),-(\iota^*)^{-1}(\alpha)).
\end{eqnarray*}

\begin{eqnarray*}
\tilde{m}_{\hat{s}}(\phi_\P(\alpha,\beta))=\tilde{m}_{\hat{s}}(\iota PD_X(\beta),-\iota PD_X(\alpha))
\stackrel{(\ref{eq-s-n-dagger-from-S-minus-to-V})}{=}(\iota PD_X(\beta),-PD_{\hat{X}}(-\iota PD_X(\alpha)))
\\
\stackrel{{\rm Lemma \ \ref{lemma-two-Poincare-dualities}(\ref{lemma-item-comparison-of-two-Poincare-dualities})}}{=}
(\iota(PD_X(\beta)),-(\iota^*)^{-1}(\alpha)).
\end{eqnarray*}

Equation (\ref{eq-conjugation-of-m-s-by-phi-F}): 
The restriction of both sides to $S^+_X$ are equal, since $\tilde{\phi}_F$ restricts to $S^+_X$ as an isometry.
Both sides map $S^-_X$ to $V_X$ and $V_X$ to $S^-_X$.
The square of both sides is the identity, so it suffices to check that both sides restrict to the same homomorphism from $S^-_X$ to $V_X$. This checks as follows. Let $(\alpha,\beta)\in S^-_X$.
\begin{eqnarray*}
\tilde{\phi}_F(\tilde{m}_s(\alpha,\beta))&\stackrel{(\ref{eq-s-n-dagger-from-S-minus-to-V})}{=}&\tilde{\phi}_F(\alpha,-PD_X(\beta))
\stackrel{{\rm Lemma \ \ref{lemma-tensorization-by-line-bundle-F}}}{=}
(\alpha+D_{PD_X(\beta)}(c_1(F)),-PD_X(\beta)),
\\
\tilde{m}_{\phi_F(s)}(\tilde{\phi}_F(\tilde{m}_s(\alpha,\beta)))&=&\tilde{m}_{\phi_F(s)}(\alpha+D_{PD_X(\beta)}(c_1(F)),-PD_X(\beta))
\\
&=&
(\alpha+D_{PD_X(\beta)}(c_1(F)))\wedge\phi_F(s)-
D_{PD_X(\beta)}(\phi_F(s)),
\\
\tilde{\phi}_F(\alpha,\beta)
&\stackrel{{\rm Lemma \ \ref{lemma-tensorization-by-line-bundle-F}}}{=}&
(\alpha,\beta+c_1(F)\wedge\alpha).
\end{eqnarray*}
Substituting $\phi_F(s)=1+c_1(F)+c_1(F)^2/2+[pt_X]$ we see that the right hand sides of the two lines above are equal, using also the equality $D_{PD_X(\beta)}([pt])=-\beta$, which follows from (\ref{eq-PD-of-D-theta-of-the-class-of-a-point}).
\end{proof}

Assume now that $F_1$ and $F_2$ are two line bundle on $X$.
Let $\Phi_{F_i}:D^b(X)\rightarrow D^b(X)$ be tensorization by $F_i$.

\begin{cor}
\label{cor-reflections-in-two-line-bundles}
The auto-equivalence $\left(\Phi_{F_2^{-1}}\circ [1]\circ\Psi_\P\circ\Phi_{\hat{F}_2}\right)\circ
\left(\Phi_{\hat{F}_1^{-1}}\circ [1] \circ \Phi_\P\circ \Phi_{F_1}\right)$ of $D^b(X)$ is mapped to the element 
$\tilde{m}_{\phi^{-1}_{F_2}(s)}\circ \tilde{m}_{\phi^{-1}_{F_1}(s)}$ of $\Aut(A_X)$,
which acts on $S^+_X$ as the composition
$R_{\phi^{-1}_{F_2}(s)}\circ R_{\phi^{-1}_{F_1}(s)}$ 
of two reflections in the $+2$ vectors
$\phi^{-1}_{F_i}(s)$, $i=1,2$. 
\begin{equation}
\label{eq-reflections-in-two-line-bundles}
\left(\Phi_{F_2^{-1}}\circ [1]\circ\Psi_\P\circ\Phi_{\hat{F}_2}\right)\circ
\left(\Phi_{\hat{F}_1^{-1}}\circ [1] \circ \Phi_\P\circ \Phi_{F_1}\right) 
\ \ \ \mapsto \ \ \ 
\tilde{m}_{\phi^{-1}_{F_2}(s)}\circ \tilde{m}_{\phi^{-1}_{F_1}(s)}.
\end{equation}
\end{cor}

\begin{proof}
The left hand side translates via Equation 
(\ref{eq-Psi-P-inverse-of-Phi-P}) to the left hand side below.
\begin{eqnarray*}
\phi_{F_2^{-1}}(\phi_\P^{-1}\phi_{\hat{F}_2\otimes\hat{F}_1^{-1}}\phi_\P\phi_{F_1\otimes F_2^{-1}})\phi_{F_2}
&\stackrel{{\rm Lemma \ \ref{lemma-auto-equivalence-acts-as-two-reflections}}}{=}&
\phi_{F_2^{-1}}\left(\tilde{m}_s\circ \tilde{m}_{\phi_{F_1^{-1}\otimes F_2}(s)}\right)\phi_{F_2}
\\
&\stackrel{(\ref{eq-conjugation-of-m-s-by-phi-F})}{=}&
\tilde{m}_{\phi^{-1}_{F_2}(s)}\circ \tilde{m}_{\phi^{-1}_{F_1}(s)}.
\end{eqnarray*}

\end{proof}

\section{Monodromy of moduli spaces  via Fourier-Mukai functors}
\label{sec-monodromy-via-Fourier-Mukai}
 
Let $\Phi:D^b(X_1)\rightarrow D^b(X_2)$ be an equivalence of derived categories of abelian surfaces, which maps $H_1$-stable sheaves with a primitive Chern character $w_1$ to $H_2$-stable sheaves with Chern character $w_2$. In Section \ref{subsection-stability-preserving-FM-transformations} we note that $\Phi$ induces an isometry $\phi:H^*(X_1,\Integers)\rightarrow H^*(X_2,\Integers)$ and an isomorphism $f:\M_{H_1}(w_1)\rightarrow \M_{H_2}(w_2)$, such that 
$\phi\otimes f_*:H^*(X_1\times \M_{H_1}(w_1),\Integers)\rightarrow H^*(X_2\times \M_{H_2}(w_2),\Integers)$
maps a universal class to a universal class.
In Section \ref{sec-the-monodromy-representation-of-a-rank-1-moduli-space} we prove 
Theorem \ref{thm-introduction-monodromy-representation-mu} constructing the homomorphism
$\mon:G(S^+)^{even}_{s_n}\rightarrow Mon(\M(s_n))$.
In Section \ref{sec-half-spin} we construct monodromy equivariant homomorphisms from the half-spin representations $S^+_X$ and $S^-_X$ to canonical quotients of $H^d(\M(s_n),\Integers)$.

%
\subsection{Stability preserving Fourier-Mukai transformations}
\label{subsection-stability-preserving-FM-transformations}
Let $X_i$, $i=1,2$, be abelian surfaces, $v_i\in S^+_{X_i}$ a primitive Mukai vector, and $H_i$ a $v_i$-generic polarization on $X_i$.
Assume that the integral transform $\Phi_F:D^b(X_1)\rightarrow D^b(X_2)$ with kernel an object $F$ in $D^b(X_1\times X_2)$
is an equivalence, and $\Phi_F(E)$ is an $H_2$-stable sheaf with Mukai vector $v_2$, for every $H_1$-stable coherent sheaf $E$ on $X_1$  with Mukai vector $v_1$. 
Denote by 
$\Phi_F\boxtimes id_{\M(v_1)}:D^b(X_1\times\M_{H_1}(v_1))\rightarrow D^b(X_2\times\M_{H_1}(v_1))$ the integral transform with kernel
the object in $D^b(X_1\times \M_{H_1}(v_1)\times X_2\times \M_{H_1}(v_1))$, which is the derived tensor product of the pull backs of the object $F$ and the structure sheaf of the diagonal in the cartesian square of $\M_{H_1}(v_1)$.
Then $\Phi_F\boxtimes id_{\M(v_1)}$ is an equivalence \cite[Assertion 1.7]{orlov-abelian-varieties}. 
Let $\E_{v_1}$ be a (possibly twisted) universal sheaf over $X_1\times \M_{H_1}(v_1).$
The object $\Phi_F\boxtimes id_{\M(v_1)}(\E_{v_1})$ is represented by a flat (possibly twisted by the pullback of a Brauer class from $\M_{H_1}(v_1)$) family of $H_2$-stable coherent sheaves with Mukai vector $v_2$
on $X_2$, by \cite[Theorem 1.6]{mukai-fourier-functor-and-its-applications}. 
Let $f:\M_{H_1}(v_1)\rightarrow \M_{H_2}(v_2)$ be the classifying morphism associated to this family. Then $f$ is easily seen to be an open immersion, which must be surjective, by compactness of $\M_{H_1}(v_1)$ and irreducibility of $\M_{H_2}(v_2)$ (Theorem \ref{thm-yoshioka}). Hence, $f$ is an isomorphism.
Let $\phi_F:H^*(X_1,\Integers)\rightarrow H^*(X_2,\Integers)$ be the parity preserving isomorphism induced by $\Phi_F$ \cite[Cor. 9.43]{huybrechts-book}. Then $\phi_F(v_1)=v_2$. Let $e_{v_i}\in K(X_i\times\M_{H_i}(v_i))$ be a universal class (see Remark \ref{rem-class-of-the-diagonal-in-terms-of-twisted-universal-sheaves}).

\begin{new-lemma}
\label{lemma-isomorphism-of-cohomology-rings-of-moduli-spaces-is-equal-to-gamma}
$\phi_F\otimes f_*:H^*(X_1\times\M_{H_1}(v_1),\RationalNumbers)\rightarrow H^*(X_2\times\M_{H_2}(v_2),\RationalNumbers)$
maps a universal class to a universal class, in the sense of Definition
\ref{def-g-gamma-maps-universal-classes-to-such}. In particular, $f_*=\gamma_{\phi_F}(e_{v_1},e_{v_2}).$
\end{new-lemma}

\begin{proof}
The statement was proven in detail in \cite[Lemma 5.6]{markman-monodromy-I} in the case of moduli spaces of sheaves on $K3$ surfaces. The proof for abelian surfaces is identical. We briefly outline the argument in the case of fine moduli spaces admitting untwisted universal sheaves $\E_{v_i}$, $i=1,2$. In that case $(id_{X_1}\times f)^*\E_{v_2}$ represents the object
$\Phi_F\boxtimes id_{\M(v_1)}(\E_{v_1})$, possibly after replacing $\E_{v_2}$ by its tensor product with the pullback of a line bundle on
$\M_{H_2}(v_2)$. Now $ch\left[\Phi_F\boxtimes id_{\M(v_1)}(\E_{v_1})\right]=(\phi_F\otimes id)(ch(\E_{v_1}))$.
Hence, $(\phi_F\otimes f_*)$ maps $ch(\E_{v_1})$ to $ch(\E_{v_2})$.
The equality $f_*=\gamma_{\phi_F}(\E_{v_1},\E_{v_2})$ now follows from Lemma \ref{lemma-recovering-f}.
\end{proof}

%
\subsection{The monodromy representation of $G(S^+)^{even}_{s_n}$}
\label{sec-the-monodromy-representation-of-a-rank-1-moduli-space}
Let $X$ be an abelian surface. Let $H$ be an ample line bundle on $X$ with 
$\chi(H)=n\geq 2$ (and genus $n+1$ and degree $2n$). 
Given a length $n+1$ subscheme $Z\subset X$,
we get the equality
\[
ch(I_Z\otimes H) \ \ \ = \ \ \ (1,H,-1).
\]
In particular, $z_{n+1}:=ch(I_Z\otimes H)$ is orthogonal to $s:=(1,0,1)$,
$(z_{n+1},s)_{S^+}=0$. 
Let $\hat{X}$ be the dual surface and $\hat{H}$ the line bundle associated to 
$H$ via (\ref{eq-homomorphism-from-Pic-X-to-Pic-X-hat}). 
Then $\hat{H}$ is ample as well.

\begin{prop}
\label{prop-yoshioka-F-P-induces-an-isomorphism}
(\cite{yoshioka-abelian-surface} Proposition 3.5)
Let $E$ be a $\mu$-stable sheaf with $ch(E)=(r,H,a)$, $a<0$.
Then the sheaf cohomology $\H^i(\Phi_\P(E))$ vanishes, for $i\neq 1$, and 
$\H^1(\Phi_\P(E))$ is a $\mu$-stable sheaf with Chern character 
$(-a,\hat{H},-r)$. 
In particular, if $H$ generates $\Pic(X)$, then the composition
$[1]\circ \Phi_\P$ induces an isomorphism of moduli spaces
\[
\M_H(r,H,a) \ \ \ \LongIsomRightArrow \ \ \ \M_{\hat{H}}(-a,\hat{H},-r).
\]
\end{prop}

We conclude that the composition 
$\Phi_{\hat{H}^{-1}}\circ [1] \circ \Phi_\P\circ \Phi_H$
induces an isomorphism of the moduli spaces
\begin{equation}
\label{eq-isomorphism-between-moduli-spaces-of-ideals-of-colength-n-plus-1}
X^{[n+1]}\times \hat{X}  \  \cong \ \M_H(1,0,-n-1)
\ \ \ \LongIsomRightArrow \ \ \ 
\M_{\hat{H}}(1,0,-n-1)  \  \cong \ \hat{X} ^{[n+1]}\times X
\end{equation}
as well as of the Albanese fibers (the generalized Kummer varieties).
An isomorphism of cohomology rings induced via pushforward by an isomorphism is a parallel-transport-operator, by definition (see Footnote \ref{footnote-def-parallel-transport} in Sec. \ref{sec-organization}).
Let $\E_X$ and $\E_{\hat{X}}$ be the universal ideal sheaves.
Setting 
$g:=-\phi_{\hat{H}}^{-1} \circ \phi_\P\circ \phi_H$, we get that 
(\ref{eq-isomorphism-between-moduli-spaces-of-ideals-of-colength-n-plus-1})
induces the parallel transport operator $\gamma_g(\E_X,\E_{\hat{X}})$ and the 
isomorphism $g\otimes \gamma_g(\E_X,\E_{\hat{X}})$ 
maps a universal class of $\M_H(1,0,\!-n\!-\!1)$
to a universal class of $\M_{\hat{H}}(1,0,\!-n\!-\!1)$, by Lemma
\ref{lemma-isomorphism-of-cohomology-rings-of-moduli-spaces-is-equal-to-gamma}.
The latter property is stable under deformation of the pair $(X,H)$, 
dropping the condition that $H$ is ample. 
Consequently, we get the following corollary.

\begin{cor}
\label{cor-composition-of-two-reflections-by-minus-two-vectors-is-a-monodromy-operator}
Let $F_1$ and $F_2$ be two line bundles with $c_1(F_i)$ primitive and $\chi(F_i)=n$. 
The left hand side of equality
(\ref{eq-reflections-in-two-line-bundles}),
and hence its right hand side 
\[
g':=\tilde{m}_{(1,F_2^{-1},n+1)}\circ\tilde{m}_{(1,F_1^{-1},n+1)},
\] 
satisfy the property that
$\gamma_{g'}(\E_X,\E_X)$ is a monodromy operator which
maps a universal class of $\M_H(1,0,-n-1)$
to another universal class of $\M_H(1,0,-n-1)$. 
\end{cor}
Note that
$(1,F_i^{-1},n+1)=\phi_{F_i}^{-1}(s)$, $i=1,2$, 
are two vectors in $S^+_X$ orthogonal to the Chern character $s_{n+1}$ and of square $2$ with respect to $(\bullet,\bullet)_{S^+_X}$, so that $g'$ belongs to $\Spin(S^+_X)_{s_{n+1}}$. 

Let the ample line bundle $H$ on $X$ have Euler characteristic $n+2$, 
$n\geq 2$. Given a length $n+1$ subscheme $Z\subset X$,
we get 
\[
z_{n+1} \ \ := \ \ ch(I_Z\otimes H) \ \ = \ \ (1,H,1).
\]
In particular, $z_{n+1}$ is orthogonal to the $-2$ vector 
$s_1:=(1,0,-1)$. Let $D_{\hat{X}} : D^b(\hat{X})\rightarrow D^b(\hat{X})^{op}$
be the functor taking $E\in D^b(\hat{X})$ to 
$E^\vee:=R\SheafHom(E,\StructureSheaf{\hat{X}})$. Set
\[
\G_\P \ \ := \ \ D_{\hat{X}}\circ \Phi_\P[2].
\]
Given $E\in D^b(X)$, set
\[
\G^i_\P(E) \ \ := \ \ \H^i(\G_\P(E)),
\]
the $i$-th sheaf cohomology.
If $E$ is a coherent sheaf, then 
$\G_\P^i(E)$ is isomorphic to the relative extension sheaf 
$\SheafExt^i_{\pi_{\hat{X}}}(\P\otimes \pi_X^*(E),\StructureSheaf{X\times\hat{X}})$, by Grothendieck-Verdier duality and the triviality of the relative canonical line bundle $\omega_{\pi_{\hat{X}}}$.

\begin{prop}
\label{prop-yoshioka-G-P-induces-an-isomorphism}
(\cite{yoshioka-abelian-surface} Proposition 3.2)
Let $E$ be a $\mu$-stable sheaf with $ch(E)=(r,H,a)$, $a>0$.
Then the sheaf cohomology $\G^i_\P(E)$ vanishes, for $i\neq 2$, and 
$\G^2_\P$ is a $\mu$-stable sheaf with Chern character 
$(a,c_1(\hat{H}),r)$. 
In particular, if $H$ generates $\Pic(X)$, then the composition
$\G_\P[2]$ induces an isomorphism of moduli spaces
\[
\M_H(r,H,a) \ \ \ \LongIsomRightArrow \ \ \ \M_{\hat{H}}(a,\hat{H},r).
\]
\end{prop}

We conclude that the composition 
$\Phi_{\hat{H}^{-1}}\circ \G_\P[2]\circ \Phi_H$ induces another isomorphism 
between the moduli spaces in equation
(\ref{eq-isomorphism-between-moduli-spaces-of-ideals-of-colength-n-plus-1}).

Denote by 
\begin{equation}
\label{eq-tau-X}
\tau_X \ : \ A_X \ \ \ \longrightarrow \ \ \ A_X
\end{equation}
the element 
of 
$G(S^+)^{even}$  in equation (\ref{eq-tau-is-in-G-S-plus-even}).
Then $\tau_X$ acts on $S^+$ via $D_X:=-R_s\circ R_{s_1}$, where 
$s=(1,0,1)$ and $s_1=(1,0,-1)$. 
\hide{
We set 
\[
g_\P \ \ \ := \ \ \ \tau_{\hat{X}}\circ \phi_\P.
\]
The operator $\tau_X$ is introduced in order to factor
the dualization involution $D_{X\times \M}\in\Aut[H^{even}(X\times \M)]$
on the product of $X$ and a moduli space of sheaves on $X$. 
We set $\tau_\M$ to be the unique solution of the equation 
\[
1\otimes \tau_\M \ \ = \ \ (\tau_X\otimes 1)\circ D_{X\times \M}.
\]
Explicitly, $\tau_\M$ acts on $H^i(\M,\Integers)$ via 
$(-1)^{i(i+1)/2}$. 
We get the factorization $D_{X\times \M}=\tau_X\otimes\tau_\M$.

\begin{rem}
{\rm
(??? $\rightarrow$ not used)
Using the operator $\tau_{\M(w)}$ one gets a symmetric unimodular bilinear 
pairing on $H^*(\M(w),\Integers)$ defined as in equation 
(\ref{eq-Mukai-pairing}).
We consider the factorization of $D_{X\times \M}$ in order to 
represent in $A_X$ the monodromy group $Mon(\M)$ of the moduli space, 
rather than an index $2$ subgroup of $Mon(\M)$. 
While the monodromy action on $H^*(\M,\Integers)$ is natural, 
its representation $A_X$ will depend on the choice of the central
element $\alpha_{S^+}$ (??? define ???) in the above factorization. The value of 
the orientation character (\ref{eq-ort}) of an element $g$ of 
$G(S^+)^{even}_{s_n}\subset GL(A_X)$ will determine the value of 
the parameter $\epsilon$ in the equation
(\ref{eq-gamma-delta}) for the element $\gamma_{g,\epsilon}$
of $Mon(\M)$. 
}
\end{rem}
}

\begin{cor}
\label{cor-an-orientation-reversing-monodromy}
\begin{enumerate}
\item
\label{prop-item-orientation-reversing-isomorphism}
The composition 
\[
\begin{array}{rl}
\gamma_{\phi_{\hat{H}}^{-1}}(\E_{(1,\hat{H},1)},\E_{\hat{X}})
\circ \gamma_{\phi_\P,1}(\E_{(1,H,1)},\E_{(1,\hat{H},1)}) \circ
\gamma_{\phi_H}(\E_X,\E_{(1,H,1)}) : & 
H^*(\M_H(1,0,-1-n),\Integers)  
\\
 \rightarrow & H^*(\M_{\hat{H}}(1,0,-1-n),\Integers)
\end{array}
\]
is equal to $\gamma_{\tau_{\hat{X}}\phi^{-1}_{\hat{H}}\tau_{\hat{X}}\phi_\P\phi_H,1}(\E_X,\E_{\hat{X}})$, is
induced by an isomorphism of the moduli spaces, and maps a universal class to the dual of a universal class.
\hide{
\item
\label{prop-item-orientation-reversing-isomorphism-maps-univ-to-univ}
The homomorphism in part (\ref{prop-item-orientation-reversing-isomorphism}) is equal to
$\gamma_{\left(
\phi_{\hat{H}}^{-1}d_{\hat{X}}\phi_\P\phi_H,0
\right)}
\left(
\E_X,\E_{\hat{X}}
\right)
\circ d_{\M(w)}^{-1}$, where $w=(1,0,-1-n)$.
The class 
$\gamma_{\left(
\phi_{\hat{H}}^{-1}d_{\hat{X}}\phi_\P\phi_H,0
\right)}
\left(
\E_X,\E_{\hat{X}}
\right)
$
induces a ring isomorphism and 
\[
\left(
\phi_{\hat{H}}^{-1}d_{\hat{X}}\phi_\P\phi_H
\right)\otimes
\gamma_{\left(
\phi_{\hat{H}}^{-1}d_{\hat{X}}\phi_\P\phi_H,0
\right)}
\left(
\E_X,\E_{\hat{X}}
\right)
\]
maps a universal class to a universal class.
}
\item
\label{cor-item-an-orientation-reversing-monodromy-op}
Let $H$ and $F$ be line bundles on $X$ with 
$\chi(F)=n$, $\chi(H)=n+2$. Assume that the classes $c_1(F)$ and $c_1(H)$ in 
$H^2(X,\Integers)$ are primitive. 
Set 
\[
\psi:=
-\tau_X\phi_F^{-1}\phi_\P^{-1}\phi_{\hat{F}}\phi_{\hat{H}}^{-1}
\tau_{\hat{X}}\phi_\P\phi_H.
\]
Then the automorphism
$\gamma_{\psi,1}(\E_X,\E_X)$ of $H^*(\M(1,0,-1-n),\Integers)$ 
is induced by a monodromy operator.
\end{enumerate}
\end{cor}

\begin{proof}
\ref{prop-item-orientation-reversing-isomorphism})
Let $f:\M_H(1,0,-1-n)   \rightarrow \M_{\hat{H}}(1,0,-1-n)$ 
be the isomorphism in Proposition 
\ref{prop-yoshioka-G-P-induces-an-isomorphism}.
Let $\eta:\M_H(1,H,1)   \rightarrow \M_{\hat{H}}(1,\hat{H},1)$ be the associated isomorphism. 
By construction, we have an isomorphism, in the derived category of $\hat{X}\times M_H(1,H,1)$,
\[
(1_{\hat{X}}\times\eta)^*\left(\E_{(1,\hat{H},1)}^\vee\right) \ \ \ 
\cong \ \ \ \Phi_{\P\boxtimes\StructureSheaf{\Delta_{\M(1,H,1)}}}(\E_{(1,H,1)})
\]
for suitably chosen universal sheaves.
Consequently, we have the equality
\[
ch(\E^\vee_{(1,\hat{H},1)})=(\phi_\P\otimes \eta_*)(ch(\E_{(1,H,1)})),
\]
and so $\phi_\P\otimes \eta_*$ maps a universal class of $\M_H(1,H,1)$ to the dual of a universal class of $\M_{\hat{H}}(1,\hat{H},1)$.
Hence, $(d_{\hat{X}}^{-1}\phi_\P)\otimes (d_{\M(1,\hat{H},1)}^{-1}\eta_*)$ maps 
a universal class of $\M_H(1,H,1)$ to a universal class of $\M_{\hat{H}}(1,\hat{H},1)$. 
Lemma \ref{lemma-recovering-f} implies the equality
\[
d_{\M(1,\hat{H},1)}^{-1}\eta_*=\gamma_{d^{-1}_{\hat{X}}\phi_\P,0}(\E_{(1,H,1)},\E_{(1,\hat{H},1)}).
\]
We get
\[
\eta_*=d_{\M(1,\hat{H},1)}\gamma_{d^{-1}_{\hat{X}}\phi_\P,0}(\E_{(1,H,1)},\E_{(1,\hat{H},1)})=
\gamma_{d^{-1}_{\hat{X}}\phi_\P,0}(\E_{(1,H,1)},\E_{(1,\hat{H},1)})d_{\M(1,H,1)}.
\]
The right hand side is equal to $\gamma_{\phi_\P,1}(\E_{(1,H,1)},\E_{(1,\hat{H},1)})$,
by Equation (\ref{eq-gamma-g-1-as-a-homomorphism}).
Hence, the graph of $\eta$ 
is Poincare-Dual to the class 
$\gamma_{\phi_\P,1}(\E_{(1,H,1)},\E_{(1,\hat{H},1)})$.
The composition in the statement of Part (\ref{prop-item-orientation-reversing-isomorphism}) is equal to 
$\gamma_{\tau_{\hat{X}}\phi^{-1}_{\hat{H}}\tau_{\hat{X}}\phi_\P\phi_H,1}(\E_X,\E_X)$,
by Lemma \ref{lemma-composition-of-gamma-g-0-and-gamma-h-1}.

\hide{
\ref{prop-item-orientation-reversing-isomorphism-maps-univ-to-univ})
We have seen that $\gamma_{\phi_\P,1}(\E_{(1,H,1)},\E_{(1,\hat{H},1)})$ maps a universal class to the dual of a universal class  in the proof of part (\ref{prop-item-orientation-reversing-isomorphism}).
Thus, the homomorphism
\[
d_{\hat{X}}\phi_\P\otimes 
\gamma_{d_{\hat{X}}\phi_\P,0}(\E_{(1,H,1)},\E_{(1,\hat{H},1)})
\]
maps a universal class to a universal class, by Lemma 
\ref{lemma-on-gamma-and-epsilon}
part \ref{lemma-item-if-gamma-g-1-maps-then-gamma-g-0-maps}.
The conclusion now follows from the 
composition property in Corollary \ref{cor-composition}.
}
\ref{cor-item-an-orientation-reversing-monodromy-op}) 
The homomorphism 
$\gamma_{-\phi_F^{-1}\phi_\P^{-1}\phi_{\hat{F}},0}(\E_{\hat{X}},\E_X)$
is a deformation of the isomorphism in Proposition
\ref{prop-yoshioka-F-P-induces-an-isomorphism} and is hence a parallel transport operator. 
The equality
\[
\gamma_{\psi,1}(\E_X,\E_X)=
\gamma_{-\phi_F^{-1}\phi_\P^{-1}\phi_{\hat{F}},0}(\E_{\hat{X}},\E_X)
\gamma_{\tau_{\hat{X}}\phi^{-1}_{\hat{H}}\tau_{\hat{X}}\phi_\P\phi_H,1}(\E_X,\E_X)
\]
follows from Lemma \ref{lemma-composition-of-gamma-g-0-and-gamma-h-1}. The latter is a composition of a parallel transport operator and an isomorphism, and is hence a parallel transport operator (so a monodromy operator).
\hide{
The isomorphism
\[
\gamma_{\phi_{\hat{H}}^{-1}d_{\hat{X}}\phi_\P}
(\E_{(1,H,1)},\E_{\hat{X}})\circ
d^{-1}_{\M(1,H,1)}\circ
\gamma_{\phi_H}(\E_X,\E_{(1,H,1)})
\]
is equal to the one in part 
(\ref{prop-item-orientation-reversing-isomorphism}),
by Lemma \ref{lemma-on-gamma-and-epsilon} (\ref{lemma-item-gamma-g-1-as-a-homomorphism}). 
Set $g:=-\phi_F^{-1}\phi_\P^{-1}
\phi_{\hat{F}}\phi_{\hat{H}}^{-1}d_{\hat{X}}\phi_\P\phi_H$. 
The composition of the two isomorphisms above is equal to
\begin{equation}
\label{eq-auxiliary-composite-isomorphism}
\gamma_{d_X(d_X^{-1}g)}(\E_X,\E_X)\circ d^{-1}_{\M(1,0,-1-n)}.
\end{equation}
The equality (\ref{eq-conjugation-by-d-X-i}) of conjugations by $d_X$ and $\tau_X$ yields
$d_X^{-1}g=\psi$. The isomorphism 
(\ref{eq-auxiliary-composite-isomorphism}) is equal to 
$\gamma_{\psi,1}(\E_X,\E_X)$, by 
Lemma \ref{lemma-on-gamma-and-epsilon} (\ref{lemma-item-gamma-g-1-as-a-homomorphism}).
}
\end{proof}

Let 
\begin{equation}
\label{eq-ort-S+}
ort:G(S^+)^{even}\rightarrow \Integers/2\Integers
\end{equation}
be the pull back of the character of $O(S^+)$ 
corresponding to the orientation of the positive cone in $S^+_\RealNumbers$ and analogous to the character 
(\ref{eq-orientation-character}). 
The kernel of $ort$ is $\Spin(S^+)$.
Let $\Integers/2\Integers\ltimes \Spin(S^+)$ be the semi-direct product
with multiplication given by $(\epsilon_1,g_1)(\epsilon_2,g_2):=(\epsilon_1+\epsilon_2,\tau_X^{\epsilon_2}g_1\tau_X^{\epsilon_2}g_2)$.
We have the isomorphism 
\[
\lfloor:G(S^+)^{even}\rightarrow \Integers/2\Integers\ltimes \Spin(S^+),
\]
given by $\lfloor(g)=(ort(g),\tau_X^{ort(g)}g)$.
Set $s_n:=(1,0,-n)$, $n\geq 3$. Note that $\tau_X$ belongs to $G(S^+)^{even}_{s_n}$, 
by Example \ref{example-an-element-of-the-even-clifford-group-of-S-plus}. 
Let $\gamma:\Integers/2\Integers\rtimes \Spin(S^+)\rightarrow 
\End H^*(\M(s_n),\Integers)$ send $(\epsilon,g)$ to the correspondence homomorphism 
induced by the class $\gamma_{g,\epsilon}(\E_{s_n},\E_{s_n})$. Set 
\begin{equation}
\label{eq-mon-is-gamma-circ-lfloor}
\mon:=\gamma\circ \lfloor : G(S^+)^{even}\rightarrow \End H^*(\M(s_n),\Integers).
\end{equation}
Let $D_{X\times \M(s_n)}$ be the automorphism of $H^{even}(X\times \M(s_n),\Integers)$ acting 
on $H^{2i}(X\times \M(s_n),\Integers)$ by $(-1)^i$.
Let $\tau_\M$ to be the unique solution of the equation 
\begin{equation}
\label{eq-tau-M}
1\otimes \tau_\M \ \ = \ \ (\tau_X\otimes 1)\circ D_{X\times \M(s_n)}.
\end{equation}
Explicitly, $\tau_\M$ acts on $H^j(\M(s_n),\Integers)$ via 
$(-1)^{j(j+1)/2}$. 
We get the factorization $D_{X\times \M}=\tau_X\otimes\tau_\M$. 
We get the group homomorphism 
\begin{eqnarray*}
G(S^+)^{even}_{s_n}&\rightarrow &G(S^+)^{even}_{s_n}\times \Aut H^*(\M(s_n),\Integers)
\\
g & \mapsto & (\tau_X^{ort(g)}g\tau_X^{ort(g)},\tau_\M^{ort(g)}\mon(g)),
\end{eqnarray*}
whose image consists of pairs mapping a universal class to a universal class, 
by the following theorem. Set $\tau:=\tau_X$.

\begin{thm}
\label{thm-monodromy-representation-mu}
\begin{enumerate}
\item
\label{thm-item-group-homomorphism}
The class
\[
\gamma_{\tau^{ort(g)}g,ort(g)}(\E_{s_n},\E_{s_n}) \ \ \
\in \ \ \ H^{4n+4}(\M(s_n)\times \M(s_n),\Integers)
\]
induces a graded ring automorphism, for every $g\in G(S^+)_{s_n}^{even}$.
The resulting map
\begin{equation}
\label{eq-homomorphism-gamma-from-G-S-plus-even-s-n-to-Mon}
\mon \ : \ G(S^+)_{s_n}^{even} \ \ \ \longrightarrow \ \ \
\Aut H^*(\M(s_n),\Integers)
\end{equation}
is a group homomorphism.
\item
\label{thm-item-gamma-tau-g}
If $ort(g)=0$, then $g\otimes \gamma_{g,0}(\E_{s_n},\E_{s_n})$ 
maps a universal class to a universal class.
\item
\label{thm-item-gamma-tau-g-1}
If $ort(g)=1$, then $g\tau\otimes \gamma_{\tau g,1}(\E_{s_n},\E_{s_n})$ 
maps a universal class to the dual of a universal class.
\item
\label{thm-item-included-in-the-monodromy}
The image of the homomorphism $\mon$ in 
(\ref{eq-homomorphism-gamma-from-G-S-plus-even-s-n-to-Mon})
is contained in the monodromy group $Mon(\M(s_n))$.
\end{enumerate}
\end{thm}

\begin{proof}
\ref{thm-item-group-homomorphism}) The fact that the map $\mon$ is a group homomorphism would follow, once the rest of the statements in 
parts (\ref{thm-item-group-homomorphism}), (\ref{thm-item-gamma-tau-g}) and 
(\ref{thm-item-gamma-tau-g-1}) of the Theorem are proven, by 
Corollary \ref{cor-composition}, and 
Lemmas \ref{lemma-composition-of-gamma-g-1-and-gamma-h-1} and \ref{lemma-composition-of-gamma-g-0-and-gamma-h-1}.

We first prove part (\ref{thm-item-group-homomorphism}) for the subgroup  of $\Spin(S^+)_{s_n}$ stabilizing both $(1,0,0)$ and $(0,0,1)$.
This subgroup is the image of $SL(H^1(X,\Integers))$ via the inverse of the isomorphism $f$ in 
Lemma \ref{lemma-stabilizer-of-H0-and-H4-is-SL-4}. A {\em marked compact complex torus of dimension $2$} is a pair 
$(X',\eta)$ consisting of a compact complex torus $X'$ of dimension $2$ and an isomorphism $\eta:H^1(X',\Integers)\rightarrow H^1(X,\Integers)$,
where $X$ is our fixed abelian surface. Two pairs $(X_1,\eta_1)$ and $(X_2,\eta_2)$ are {\em isomorphic}, if there exists an isomorphism $f:X_1\rightarrow X_2$, such that $\eta_2=\eta_1f^*$.
Let $\fM$ be the moduli space of isomorphism classes of marked compact complex $2$-dimensional tori.
Let $\fM^0$ be the connected component containing $(X,id)$. The group $GL(H^1(X,\Integers))$ acts on $\fM$, via
$g(X',\eta)=(X',g\eta)$, and the subgroup $SL(H^1(X,\Integers))$ leaves $\fM^0$ invariant. We have a universal torus $\pi:\X\rightarrow \fM^0$, a relative Douady space $\X^{[n]}\rightarrow \fM^0$ of length $n$ zero-dimensional subschemes of fibers of $\pi$,
a relative dual torus $\hat{\pi}:\hat{\X}\rightarrow \fM^0$, and so a relative moduli space $\M_\X(s_n):=\hat{\X}\times_{\fM^0}\X^{[n]}\rightarrow \fM^0$ of rank $1$ torsion free sheaves on fibers of $\pi$ with Chern character $s_n$. Furthermore, we have a universal 
sheaf $\E$ over $\X\times_{\fM^0}\M_\X(s_n)$. 

Let $\M_X(s_n)$ be the fiber of $\M_\X(s_n)$ over $(X,id)$.
The group $SL(H^1(X,\Integers))$ acts via monodromy operators on $H^*(X,\Integers)$ and on $H^*(\M_X(s_n),\Integers)$, by
\cite[Lemma 6.6]{markman-monodromy-I}. The Chern character $ch(\E)$ maps to a global flat section of the local system 
$R\Pi_*\RationalNumbers$, where $\Pi:\X\times_{\fM^0}\M_\X(s_n)\rightarrow \fM^0$ is the natural morphism.
Hence, $ch(\E)$ restricts to an $SL(H^1(X,\Integers))$-invariant class in $H^*(X\times \M_X(s_n))$, under the diagonal monodromy action. The statement of Part (\ref{thm-item-group-homomorphism}) follows for the image of $SL(H^1(X,\Integers))$ in 
$\Spin(S^+)_{s_n}$, by Lemma \ref{lemma-recovering-f}. The statement of Part (\ref{thm-item-group-homomorphism}) follows
for the elements  of $\Spin(S^+)_{s_n}$ which are the compositions $\tilde{m}_{t_1}\tilde{m}_{t_2}$, where $t_i=(1,A_i,n)\in s_n^\perp$,  satisfying $(t_i,t_i)_{S^+}=2$, and $A_i$ is a primitive class in $H^2(X,\Integers)$, by Corollary \ref{cor-composition-of-two-reflections-by-minus-two-vectors-is-a-monodromy-operator}. The statement of Part (\ref{thm-item-group-homomorphism}) follows for the whole of
$\Spin(S^+)_{s_n}$, since the latter is generated by the image of $SL(H^1(X,\Integers))$  and compositions $\tilde{m}_{t_1}\tilde{m}_{t_2}$ as above, by Lemma \ref{lemma-generators-for-stabilizer-in-G-V}. 

Part (\ref{thm-item-gamma-tau-g}) 
is evident for elements of $SL(H^1(X,\Integers))$ and it was verified in Corollary \ref{cor-composition-of-two-reflections-by-minus-two-vectors-is-a-monodromy-operator} for the above mentioned compositions $\tilde{m}_{t_1}\tilde{m}_{t_2}$, so it follows for all elements of $\Spin(S^+)_{s_n}$. 



In the verification of Part (\ref{thm-item-gamma-tau-g-1}) we use the evident identity 
$d_X(\tau g)d_X^{-1}=\tau(\tau g)\tau^{-1}=
g\tau$ and Definition \ref{def-g-gamma-maps-universal-classes-to-such} part 
\ref{def-item-g-gamma-1-maps-dual-of-univ-class-to-univ-class}. Then Corollary \ref{cor-an-orientation-reversing-monodromy} verifies Part (\ref{thm-item-gamma-tau-g-1})
for $g=\tau\psi$, which belongs to $G(S^+)_{s_n}^{even}$ but not to $\Spin(S^+)_{s_n}$.
The result of Part (\ref{thm-item-group-homomorphism}) thus extends to  
$G(S^+)_{s_n}^{even}$. 

\ref{thm-item-included-in-the-monodromy})
The statement was verified for elements of $SL(H^1(X,\Integers))$ above, and is established in 
Corollaries \ref{cor-composition-of-two-reflections-by-minus-two-vectors-is-a-monodromy-operator} and 
\ref{cor-an-orientation-reversing-monodromy}, hence it is verified for a set of generators of $G(S^+)_{s_n}^{even}$. 
\end{proof}

%
\subsection{The half-spin representations as subquotients of the cohomology ring}
\label{sec-half-spin}
We identify next $H^2(K_X(n-1),\Integers)$ as a $G(S_X^+)^{even}_{s_n}$-module by determining the equivariance property of the
isomorphism $\theta$ in (\ref{eq-mukai-isomorphism}). We do this more generally for all degrees as follows.
Let $\tilde{I}_d$ be the ideal in $H^*(\M(s_n),\Integers)$ generated by $H^i(\M(s_n),\Integers)$, $1\leq i \leq d$, denote by $\tilde{I}^j_d$ its graded summand of degree $j$, and set $I:=\oplus_{d\geq 2}\tilde{I}_{d-1}^d$.
Set $Q(\M(s_n)):=H^*(\M(s_n),\Integers)/I$. 
In particular, $Q^1(\M(s_n))=H^1(\M(s_n),\Integers)$, and $Q^d(\M(s_n))=H^d(\M(s_n),\Integers)/I^d$, where $I^d:=\tilde{I}_{d-1}^d$.
For example, 
\begin{eqnarray*}
I^2&=&H^1(\M(s_n),\Integers)\cup H^1(\M(s_n),\Integers),
\\
I^3&=&H^1(\M(s_n),\Integers)\cup H^2(\M(s_n),\Integers),
\\
I^4&=&H^1(\M(s_n),\Integers)\cup H^3(\M(s_n),\Integers)+H^2(\M(s_n),\Integers)\cup H^2(\M(s_n),\Integers).
\end{eqnarray*}
We have the homomorphism $\tilde{\theta}:S_X\rightarrow H^*(\M(s_n),\RationalNumbers)$, given by 
\begin{equation}
\label{eq-tilde-theta-homomorphism}
\tilde{\theta}(\lambda):=\pi_{\M,*}\left[\pi_X^*(\tau_X(\lambda))\cup ch(\E)\right].
\end{equation}
Set $S^1_X:=S_X^-$, $S^2_X:=S_X^+\cap s_n^\perp$, and for $j>2$ let $S^j_X$ be $S^+_X$, if $j$ is even, and $S^-_X$ if $j$ is odd. Let 
\begin{equation}
\label{eq-tilde-theta-j-rational}
\tilde{\theta}_j: S^j_X\rightarrow Q^j(\M(s_n))\otimes_\Integers\RationalNumbers
\end{equation}
be the composition of $\tilde{\theta}$ with the inclusion $S^j_X\subset S_X$ and the projection $H^*(\M(s_n),\RationalNumbers)\rightarrow Q^j(\M(s_n))\otimes_\Integers\RationalNumbers.$ 
Note that $\tilde{\theta}_j$ is independent of the choice of the universal sheaf, by our definition of $S^j_X$ and
$Q^j(\M(s_n)).$
Note also that $Q^2(\M(s_n))$ is isomorphic to the second integral cohomology $H^2(K_a(s_n),\Integers)$ of the generalized Kummer, $\tilde{\theta}_2$ above is injective and has integral values, and 
the integral isomorphism $\theta$ in Equation (\ref{eq-mukai-isomorphism}) factors through 
$\tilde{\theta}_2$, since $H^1(K_a(s_n),\Integers)$ vanishes. 

The action of any graded ring automorphism of $H^*(\M(s_n),\Integers)$ descends to an action on $Q^j(\M(s_n))$,
for all $j\geq 1$. 
Similarly, the action of $\tau_\M$, given in (\ref{eq-tau-M}), descends to one on $Q^j(\M(s_n))$,
for all $j\geq 1$. 

\begin{cor}
\label{cor-monodromy-equivariance-of-theta}
The image of $\tilde{\theta}_j$, $j\geq 1$, is invariant under the monodromy action of $G(S^+_X)^{even}_{s_n}$ via the homomorphism $\mon$ given in Equation (\ref{eq-homomorphism-gamma-from-G-S-plus-even-s-n-to-Mon}) and the image spans $Q^j(\M(s_n))\otimes_\Integers\RationalNumbers$. Furthermore, for all $\lambda\in S^j_X$ and all $g\in G(S^+_X)^{even}_{s_n}$ we have
\begin{equation}
\label{eq-monodromy-equivariance-of-theta}
\mon_g\left(\tilde{\theta}_j(\lambda)\right)=\tau_\M^{ort(g)}\left[
\tilde{\theta}_j\left(\tau_X^{ort(g)}g\tau_X^{ort(g)}(\lambda)\right)
\right].
\end{equation}
In particular, $\tilde{\theta}_j$ is $\mon$-equivariant with respect to the subgroup $\Spin(S^+_X)_{s_n}$ of $G(S^+_X)^{even}_{s_n}$.
\end{cor}

\begin{proof}
The image of $\tilde{\theta}_j$ spans $Q^j(\M(s_n))\otimes_\Integers\RationalNumbers$, since the K\"{u}nneth factors of $ch(\E)$
generate $H^*(\M(s_n),\RationalNumbers)$, by \cite[Cor. 2]{markman-diagonal}.
Let $\{x_1, \dots, x_{16}\}$ be a basis of $S_X$, with each class either even or odd. Let 
$ch(\E):=\sum_i x_i\otimes e_i$ be its K\"{u}nneth decomposition.
Then $\tilde{\theta}(\lambda)=\sum_i(\lambda,x_i)_{S_X}e_i$.
If $g$ belongs to $\Spin(V)_{s_n}$, so that $ort(g)=0$, then
\begin{eqnarray*}
\mon_g(\tilde{\theta}(g^{-1}(\lambda)))&=&\mon_g(\sum_i(g^{-1}(\lambda),x_i)_{S_X}e_i)=
\sum_i(\lambda,g(x_i))_{S_X}\mon_g(e_i)
\\
&=&\pi_{\M,*}\left[\pi_X^*(\tau_X(\lambda))\cup (g\otimes\mon_g)(ch(\E))\right].
\end{eqnarray*}
Consequently, we get
\begin{equation}
\label{eq-equivariance-of-tilde-theta}
\mon_g(\tilde{\theta}(g^{-1}(\lambda)))=\tilde{\theta}(\lambda)\exp(c_g),
\end{equation}
where the class $c_g$ is given in Equation (\ref{eq-c-g}), 
since $(g\otimes\mon_g)(ch(\E))$ is a universal class, by Theorem \ref{thm-monodromy-representation-mu}.
The projection of the right hand side to $Q(\M(s_n))\otimes_\Integers\RationalNumbers$ is equal to that of $\tilde{\theta}(\lambda)$.
The identity (\ref{eq-monodromy-equivariance-of-theta}) follows. 

Assume $ort(g)=1$. Then
\begin{eqnarray}
\label{eq-mu-g-when-ort-g-is-1} \hspace{2Ex} \
\mon_g(\tilde{\theta}(\tau_X g^{-1}\tau_X(\lambda)))&=&
\sum_i(\lambda,\tau_X g\tau_X(x_i))_{S_X}\mon_g(e_i)
\\
\nonumber
&=& \pi_{\M,*}\left[
\pi_X^*(\tau_X(\lambda))\cup
\left\{((\tau_X g\tau_X)\otimes \mon_g)(ch(\E))
\right\}
\right].
\end{eqnarray}
Now, $((g\tau_X)\otimes \mon_g)(ch(\E))$ is the dual of a universal class, by Theorem \ref{thm-monodromy-representation-mu}. So
\[
((\tau_X g\tau_X)\otimes \mon_g)(ch(\E))
\]
projects to the image of $(1\otimes \tau_{\M})(ch(\E))$ in $H^*(X,\Integers)\otimes Q(\M(s_n))$. We conclude that
the right hand side of (\ref{eq-mu-g-when-ort-g-is-1}) and $\tau_\M(\tilde{\theta}(\lambda))$ project to the same class in $Q(\M(s_n))$. Hence, so does the left hand side of  (\ref{eq-mu-g-when-ort-g-is-1}).
We have thus verified the equation obtained from (\ref{eq-monodromy-equivariance-of-theta})
by substituting $(\tau_Xg\tau_X)(\lambda)$ for $\lambda$. 
We conclude that 
the identity (\ref{eq-monodromy-equivariance-of-theta})  follows in this case as well.
\end{proof}

We construct next an integral analogue of the  homomorphism (\ref{eq-tilde-theta-j-rational}) into $Q^j(\M(s_n))$.
Given a topological space $M$, 
let $K(M):=K^0(M)\oplus K^1(M)$  be the topological $K$-ring of $M$. 
The Chern character induces a ring isomorphism 
$ch:K(M)\otimes_\Integers\RationalNumbers\rightarrow H^*(M,\RationalNumbers)$ sending $K^0(M)$ into $H^{even}(M,\RationalNumbers)$ and $K^1(M)$ into
$H^{odd}(M,\RationalNumbers)$ \cite[V.3.26]{karoubi}. 
The Chern classes $c_{j/2}(x)\in H^{j}(M,\Integers)$ are defined for an odd integer $j\geq 1$ in \cite[Def. 19]{markman-integer}.
They satisfy the equality
\begin{equation}
\label{eq-chern-class-for-a-half-integer}
ch_{k-\frac{1}{2}}(x)=\frac{(-1)^{k-1}}{(k-1)!}c_{k-\frac{1}{2}}(x),
\end{equation}
for $x\in K^1(\M(s_n))$ and a positive integer $k$, by \cite[Lemma 22(2)]{markman-integer}.

For the abelian surface $X$ the Chern character is integral and we get an 
isomorphism $ch:K(X)\rightarrow H^*(X,\Integers)$. Integrality follows for $K^0(X)$, since the intersection pairing on $H^2(X,\Integers)$ is even, and for $K^1(X)$ since the coefficient of the right hand side of Equation (\ref{eq-chern-class-for-a-half-integer}) is $\pm1$ for $k=1,2$. 
Surjectivity of $ch:K(X)\rightarrow H^*(X,\Integers)$ follows, since $H^*(X,\Integers)$ is generated by $H^1(X,\Integers)$,
the latter is spanned by pullback of classes via maps to a circle $S^1$, and surjectivity of $ch:K(S^1)\rightarrow H^*(S^1,\Integers)$ is clear.
The K\"{u}nneth theorem in $K$-theory yields an isomorphism
\[
\left[K^0(X)\otimes K^0(\M(s_n))\right]\oplus \left[K^1(X)\otimes K^1(\M(s_n))\right]\rightarrow
K^0(X\times\M(s_n)),
\]
by \cite[Cor. 2.7.15]{atiyah}. Choose a basis $\{x_1, \dots, x_{16}\}$ of $K(X)$, which is a union of a basis of $K^0(X)$ and a basis of
$K^1(X)$.
Let $[\E]\in K(X\times\M(s_n))$ be the class of a universal sheaf. We get the K\"{u}nneth decomposition
\[
[\E]=\sum_{i=1}^{16} x_i\otimes e_i,
\]
with $e_i$ either in $K^0(\M(s_n))$ or $K^1(\M(s_n))$.

\begin{thm}
\label{thm-chern-classes-of-kunneth-factors-are-integral-generators}
\cite[Theorem 1]{markman-integer}
The Chern classes of $\{e_i\ : \ 1\leq i \leq 16\}$ generate the integral cohomology ring $H^*(\M(s_n),\Integers)$.
\end{thm}

Let $\pi_X^!:K(X)\rightarrow K(X\times \M(s_n))$ be the pull back homomorphism and denote by 
$\pi_{\M,!}:K(X\times \M(s_n))\rightarrow K(\M(s_n))$ the Gysin homomorphism. The involution $\tau_X$ acts on $H^i(X,\Integers)$ via $(-1)^{i(i-1)/2}$, as in 
(\ref{eq-Mukai-pairing}), and we 
denote  the involution 
$ch^{-1}\circ\tau_X\circ ch:K(X)\rightarrow K(X)$ by $\tau_X$ as well. 
We get the homomorphism
$
e:K(X)\rightarrow K(\M(s_n)),
$
given by
\[
e(\lambda):=\pi_{\M,!}\left(\pi_X^!(\tau_X(\lambda))\cup [\E]\right).
\]
The Chern classes of $e(x_i)$, $1\leq i\leq 16$, generate $H^*(\M(s_n),\Integers)$, 
by Theorem \ref{thm-chern-classes-of-kunneth-factors-are-integral-generators}.

Let $\bar{c}_j:K(\M(s_n))\rightarrow Q^{2j}(\M(s_n))$ be the composition of the Chern class map 
$c_j:K(\M(s_n))\rightarrow H^{2j}(\M(s_n),\Integers)$ with the natural projection $H^{2j}(\M(s_n),\Integers)\rightarrow Q^{2j}(\M(s_n))$.
Then $\bar{c}_j$ is a group homomorphism. This is proven in \cite[Prop. 2.6]{markman-constraints}, for an integer $j\geq 0$ and
$K^0(\M(s_n))$, and for a half integer $j$ and $K^1(\M(s_n))$ it follows from Equation (\ref{eq-chern-class-for-a-half-integer}) and the linearity of the Chern character homomorphism, 
hence of the integral homomorphism $(j-\frac{1}{2})!ch_j$. 

Let 
\begin{equation}
\label{eq-double-tilde-theta-j}
\bar{\theta}_j:S_X^j\rightarrow Q^j(\M(s_n)),
\end{equation} 
$j\geq 1$, 
be the composition of 
\begin{equation}
\label{eq-tilde-theta-homomorphism-integral-version}
S_X:=H^*(X,\Integers)\RightArrowOf{ch^{-1}} K(X) \RightArrowOf{e} K(\M(s_n))\RightArrowOf{\bar{c}_j} Q^{j}(\M(s_n))
\end{equation}
 with the inclusion
$S_X^j\rightarrow S_X$. 

\begin{new-lemma}
\label{lemma-tilde-theta-j-is-an-isomorphism}
If the ranks of $S^j_X$ and $Q^j(\M(s_n))$ are equal, then $\bar{\theta}_j$ is an isomorphism and $Q^j(\M(s_n))$
is torsion free. This is the case for $1\leq j\leq 3$.
\end{new-lemma}
\begin{proof}
The first statement is clear, since
the homomorphism $\bar{\theta}_j$ is surjective, for all $j\geq 0$, 
by Theorem \ref{thm-chern-classes-of-kunneth-factors-are-integral-generators}.
$Q^1(\M(s_n))=H^1(\M(s_n),\Integers)$ and its rank is equal to the first Betti number $8$ of the Albanese $X\times\hat{X}$ of
$\M(s_n)$. 
The rank of $Q^2(\M(s_n))$ is equal to that of $H^2(K_X(n-1),\Integers)$, which is $7$, by Theorem \ref{thm-yoshioka}, since $H^1(K_X(n-1),\Integers)$ vanishes. $Q^3(\M(s_n))$ is isomorphic to $Q^3(X^{[n]})$, 
since $\M(s_n)$ is isomorphic to $\hat{X}\times X^{[n]}$ and $H^*(\hat{X},\Integers)$ is generated by $H^1(\hat{X},\Integers)$. 
Now $b_3(X^{[n]})=40$, by G\"{o}ttsche's formula \cite[Cor. 2.3.13]{gottsche}, and
\[
8=\rank(S^-)\geq \rank(Q^3(X^{[n]}))\geq 40-b_3(X)-\rank(Q^2(X^{[n]}))b_1(X)=40-4-7\times 4=8.
\]
Hence, the ranks of $S^-$ and $Q^3(\M(s_n))$ are equal.
\end{proof}

\begin{new-lemma}
\label{lemma-two-theta-j-are-proportional}
The composition of $\bar{\theta}_{2j}$ with the natural homomorphism 
$Q^{2j}(\M(s_n))\rightarrow Q^{2j}(\M(s_n))\otimes_\Integers\RationalNumbers$ is equal to
$(-1)^{j-1}(j-1)!\tilde{\theta}_{2j}$, if $j$ is an integer, and to 
\[
(-1)^{j-\frac{1}{2}}\left(j-\frac{1}{2}\right)!\tilde{\theta}_{2j},
\] 
if $j$ is  half an odd integer. 
Consequently, the integral version of Equation (\ref{eq-monodromy-equivariance-of-theta}), 
with $\tilde{\theta}_k$ replaced by $\bar{\theta}_k$, holds whenever $Q^k(\M(s_n))$ is torsion free, and in particular for 
$k\in\{1,2,3\}$.
\end{new-lemma}

\begin{proof}
We have the commutative diagram
\[
\xymatrix{
K(X)\ar[r]^{\tau_X}\ar[d]_{ch}&K(X)\ar[r]^{\pi_X^!}\ar[d]_{ch}&K(X\times\M)\ar[r]^{\cup[\E]}\ar[d]_{ch}&
K(X\times\M)\ar[r]^{\pi_{\M,!}}\ar[d]_{ch}&K(\M)\ar[d]_{ch}
\\
H^*(X,\Integers) \ar[r]_{\tau_X}&H^*(X,\Integers)\ar[r]_{\pi_X^*}&H^*(X\times\M,\RationalNumbers)\ar[r]_{\cup ch(\E)}& 
H^*(X\times\M,\RationalNumbers) \ar[r]_{\pi_{M,*}}&H^*(\M,\RationalNumbers).
}
\]
The first (left) square commutes, by definition of the top $\tau_X$, the second and third by well known properties of the Chern character, and the fourth by the topological version of Grothendieck-Riemann-Roch and the triviality of the Todd class of $X$. Finally, 
let $\bar{ch}_j:K(\M)\rightarrow Q^{2j}(\M)\otimes_\Integers\RationalNumbers$ be the composition of $ch_j$ with the quotient homomorphism 
$H^{2j}(\M,\RationalNumbers)\rightarrow Q^{2j}(\M)\otimes_\Integers\RationalNumbers$. Then 
$\bar{c}_j=(-1)^{j-\epsilon}(j-\epsilon)!\bar{ch}_j$, where $\epsilon=1$ if $j$ is an integer, and $\epsilon=\frac{1}{2}$ if $j$ is half an odd integer (see for example
 \cite[Lemma 22]{markman-integer}).
\end{proof}

%
\section{Four groupoids}
\label{sec-four-groupoids}
We extend in Corollary \ref{cor-monodromy-representation-of-spin} the  representation of $G(S^+)^{even}_{s_n}$ in the monodromy group of the moduli space
$\M(s_n)$ of rank $1$ sheaves, given in Theorem \ref{thm-monodromy-representation-mu}, to more general moduli spaces. This is achieved by extending the monodromy group symmetry of a single moduli space to a symmetry of the collection of all
smooth and compact moduli spaces $\M_H(w)$ of stable sheaves on  abelian surfaces with respect to a groupoid, whose morphisms are parallel transport operators. In Corollary \ref{cor-q-w-is-Spin-w-equivariant} 
we construct a $\Spin(S^+)_w$-equivariant homomorphism from the cohomology
$H^*(\M_H(w),\Integers)$ of a moduli space of sheaves on an abelian surface $X$ to $H^*(X\times\hat{X},\Integers)$.

A groupoid is a category whose morphisms are all isomorphisms.
Let $\G_1$ be the groupoid whose objects are abelian varieties and whose morphisms 
$\Hom_{\G_1}(X,Y)$ are objects $\E$ in $D^b(X\times Y)$, such that the integral transform 
$\Phi_\E:D^b(X)\rightarrow D^b(Y)$ with kernel $\E$ is an equivalence of triangulated categories. 
Composition of morphisms corresponds to convolution of the objects. 
Given abelian varieties $X$ and $Y$ of the same dimension,
let $U(X\times \hat{X},Y\times\hat{Y})$ be set of isomorphisms $f:X\times \hat{X}\rightarrow Y\times\hat{Y}$,
such that $f^*:H^1(Y\times\hat{Y},\Integers)\rightarrow H^1(X\times\hat{X},\Integers)$ induces an isometry with respect to the 
symmetric bilinear pairing (\ref{eq-pairing-on-V}).\footnote{
See \cite[Def. 9.46]{huybrechts-book} for a matrix form characterization of elements of $U(X\times \hat{X},Y\times\hat{Y})$.
}
Let $\G_2$ be the groupoid whose objects are abelian varieties and whose morphisms 
$\Hom_{\G_2}(X,Y)$ are isomorphisms $f:X\times \hat{X}\rightarrow Y\times \hat{Y}$ in the subset 
$U(X\times \hat{X},Y\times\hat{Y})$. 
Composition of morphisms is the usual composition of isomorphisms. The following combines results of Orlov and Polishchuk.

\begin{thm} \cite[Prop. 9.39, Excercise 9.41, and Prop. 9.48]{huybrechts-book}
\label{thm-isomorphism-associated-to-an-equivalence}
There exists an explicit full functor $f:\G_1\rightarrow \G_2$ sending the object $X$ to itself and associating 
to an equivalence 
$\Phi_\E:D^b(X)\rightarrow D^b(Y)$ with kernel $\E\in D^b(X\times Y)$ 
an isomorphism $f_\E:X\times \hat{X}\rightarrow Y\times\hat{Y}$ in $U(X\times \hat{X},Y\times\hat{Y})$.
\end{thm}

Let $\G_3$ be the groupoid whose objects are 
triples $(X,w,H)$, where $X$ is an abelian surface, $w\in S^+_X$ is a primitive Mukai vector, and $H$ is a $w$-generic polarization,
such that the moduli space $\M_H(w)$ has dimension $\geq 4$. 
Morphisms in $\Hom_{\G_3}[(X_1,w_1,H_1),(X_2,w_2,H_2)]$ are pairs $(g,\gamma)$, where
$g:H^*(X_1,\Integers)\rightarrow  H^*(X_2,\Integers)$ is an isometry, with respect to the pairings (\ref{eq-Mukai-pairing}),
preserving the parity of the grading and satisfying $g(w_1)=w_2$, and $\gamma$ is a graded
ring isomorphism
\[
\gamma:H^*(\M_{H_1}(w_1),\Integers)\rightarrow H^*(\M_{H_2}(w_2),\Integers).
\]
Each pair $(g,\gamma)$ is assumed to be the composition 
$(g_k,\gamma_k)\circ (g_{k-1},\gamma_{k-1}) \circ \cdots \circ (g_1,\gamma_1)$, where $(g_i,\gamma_i)$ is of one of three types.
Type 1: $g_i:H^*(X_i,\Integers)\rightarrow H^*(X_{i+1},\Integers)$ is induced by  
 an equivalence 
$\Phi_{\E_i}:D^b(X_i)\rightarrow D^b(X_{i+1})$ of triangulated categories which maps $H_i$-stable sheaves with Mukai vector $w_i$ to
$H_{i+1}$-stable sheaves with Mukai vector $w_{i+1}$, 
and
$\gamma_i$ is induced by an isomorphism $\tilde{\gamma}_i:\M_{H_i}(w_i)\rightarrow \M_{H_{i+1}}(w_{i+1})$
of moduli spaces, which is in turn induced by $\Phi_{\E_i}$ (see Section \ref{subsection-stability-preserving-FM-transformations}).
Type 2:  $g_i:H^*(X_i,\Integers)\rightarrow H^*(X_{i+1},\Integers)$ and $\gamma_i:H^*(\M_{H_i}(w_i),\Integers)\rightarrow H^*(\M_{H_{i+1}}(w_{i+1}),\Integers)$ are parallel transport operators associated to a continuous path from a point $b_0$ to a point $b_1$ in the complex analytic base $B$ of a smooth 
family $\Pi:\M\rightarrow B$ of moduli spaces of stable sheaves corresponding
to a family $\pi:\X\rightarrow B$ of abelian surfaces, and a section $w$ of $R^{even}\pi_*\Integers$ of Hodge type,
as well as a section $h$ of $R^2\pi_*\Integers$ of Hodge type (not necessarily continuous), such that $h(b)$ is a
$w(b)$-generic polarization on the fiber $X_b$ of $\pi$ and 
the fiber $\M_b$ of $\Pi$  is a smooth and projective moduli space $\M_{h(b)}(w(b))$ of $h(b)$-stable sheaves over $X_b$, for all $b\in B$. An isomorphism is chosen between $X_i$ and $X_{b_0}$, mapping $w_1$ and $H_1$ to  $w(b_0)$ and $h(b_0)$.
An isomorphism is chosen between $X_{i+1}$ and $X_{b_1}$ with the analogous properties. The chosen isomorphisms yield  isomorphisms between $\M_{b_0}$ and $\M_{H_i}(w_i)$ and between $\M_{b_1}$ and $\M_{H_{i+1}}(w_{i+1})$.
Type \nolinebreak{3}: Analogous to type 2 but the Mukai vectors are all $(1,0,-n)$, the data $h$ is dropped, 
the family $\pi:\X\rightarrow B$ is a smooth and proper family of two dimensional complex tori, and 
each fiber $\M_b$ 
is the product  $X_b^{[n]}\times \hat{X}_b$, where the first factor is the Douady space of length $n$ subschemes.

\begin{rem}
\label{rem-morphism-of-groupoid-maps-universal-class-to-same}
Note that in each of the three types  of morphisms $(g_i,\gamma_i)$ above $g_i\otimes\gamma_i$ maps a universal class to a universal class, in the sense of
Definition \ref{def-g-gamma-maps-universal-classes-to-such}. This is obvious for types 2 and 3, and for type 1 it follows from Lemma \ref{lemma-isomorphism-of-cohomology-rings-of-moduli-spaces-is-equal-to-gamma}.
Consequently, the same holds for their composition $(g,\gamma)$, by Corollary \ref{cor-composition}.
In particular, a morphism $(g,\gamma)$ is determined already by $g$, since 
$\gamma=\gamma_g(\E_{w_1},\E_{w_2})$, by Lemma \ref{lemma-recovering-f}.  
Whenever non-empty,
$\Hom_{\G_3}[(X,s_n,H),(X,w,H)]$ is a right $\Aut_{\G_3}(X,s_n,H)$-torsor and a left $\Aut_{\G_3}(X,w,H)$-torsor.
$\Aut_{\G_3}(X,s_n,H)$ contains $\Spin(S^+_X)_{s_n}$, for $s_n:=(1,0,-n)$ and $n\geq 3$, 
by Theorem \ref{thm-monodromy-representation-mu}.
\end{rem}
The main result of \cite{yoshioka-abelian-surface} may be stated as follows.

\begin{thm}
\label{thm-Hom-G3-non-empty}
The set $\Hom_{\G_3}[(X_1,w_1,H_1),(X_2,w_2,H_2)]$ is non-empty, for every two objects $(X_i,w_i,H_i)$, $i=1,2$, of $\G_3$ with
$(w_1,w_1)_{S^+_{X_1}}=(w_2,w_2)_{S^+_{X_2}}$.
\end{thm}

Let $(X,w,H)$ be an object of $\G_3$, with $(w,w)_{S^+_X}=-2n$, $n\geq 3$, and denote by $e_w\in K(X\times \M_H(w))$ the class of a possibly twisted universal sheaf.
\begin{cor}
\label{cor-monodromy-representation-of-spin}
The class $\gamma_{\tau_X^{ort(g)}g,ort(g)}(e_w,e_w)\in H^{4n+4}(\M_H(w)\times\M_H(w),\Integers)$ induces a monodromy operator
$
\mon(g)\in \Aut(H^*(\M_H(w),\Integers)),
$
for every $g\in G(S^+_X)^{even}_w$. 
The resulting map 
\[
\mon:G(S^+_X)^{even}_w\rightarrow \Aut(H^*(\M_H(w),\Integers))
\]
is a group homomorphism.
The analogues of statements (\ref{thm-item-gamma-tau-g}) and (\ref{thm-item-gamma-tau-g-1}) of 
Theorem \ref{thm-monodromy-representation-mu} hold as well.
\end{cor}

\begin{proof}
Choose a morphism $(g,\gamma)\in \Hom_{\G_3}[(X_1,s_n,H_1),(X,w,H)]$, where $n=-(w,w)_{S^+_X}/2$. Such a morphism exists by Theorem \ref{thm-Hom-G3-non-empty}.
Let $e_{s_n}$ be the class of a universal sheaf over $X_1\times \M_{H_1}(s_n)$. Then $\gamma=\gamma_g(e_{s_n},e_w)$, by Remark
\ref{rem-morphism-of-groupoid-maps-universal-class-to-same}. Now $g:S_{X_1}\rightarrow S_X$ 
is the composition of parallel transport operators and isomorphisms induced by equivalences of derived categories of abelian surfaces
and $g(s_n)=w$. Thus, $g$ conjugates $G(S^+_{X_1})^{even}_{s_n}$ to $G(S^+_X)^{even}_w$. 
Given $f\in G(S^+_{X_1})^{even}_{s_n}$, let $\mon(f)\in Mon(\M_{H_1}(s_n))$ be the monodromy operator of 
Theorem \ref{thm-monodromy-representation-mu}. Then $\mon(f)=\gamma_{\tau_{X_1}^{ort(f)}f,ort(f)}(e_{s_n},e_{s_n})$. 
The conjugate
$\gamma\circ \mon(f)\circ \gamma^{-1}$ is equal to 
$\gamma_{\tau_X^{ort(f)}gfg^{-1},ort(f)}(e_w,e_w)$, by Lemma \ref{lemma-composition-of-gamma-g-0-and-gamma-h-1}. 
The latter is just $\mon(gfg^{-1})$, since $ort(gfg^{-1})=ort(f)$.
Hence, the map $\mon$ of the current Corollary is the conjugate via $\gamma$ and $g$ of 
the homomorphism $\mon$
of Theorem \ref{thm-monodromy-representation-mu}.
\[
\mon(h)=\gamma\circ\mon(g^{-1}hg)\circ\gamma^{-1},
\]
for every $h\in G(S^+_X)^{even}_w$. It is thus a group homomorphism into the monodromy group. If $ort(h)=0$, then 
$h\otimes\mon(h)$ maps a universal class to a universal class, since $g^{-1}hg\otimes\mon(g^{-1}hg)$ and $g\otimes\gamma$ do.
The case $ort(h)=1$ is similar.
\end{proof}

Let $\G_4$ be the groupoid whose objects are two dimensional
compact complex tori $X$, and let
$\Hom_{\G_4}(X,Y)$ consist of ring isomorphisms 
$f:H^*(X\times \hat{X},\Integers)\rightarrow H^*(Y\times \hat{Y},\Integers)$, 
each of which is the composition $f_k\circ f_{k-1} \circ \cdots \circ f_1$
of a sequence of isomorphisms 
$f_i:H^*(X_i\times \hat{X}_i,\Integers)\rightarrow H^*(X_{i+1}\times \hat{X}_{i+1},\Integers)$ of one of two types.
Type 1: $X_i$ and $X_{i+1}$ are projective and $f_i$ is induced by an isomorphism
$\tilde{f}:X_i\times \hat{X}_i\rightarrow X_{i+1}\times \hat{X}_{i+1}$ in
$U(X_i\times \hat{X}_i, X_{i+1}\times \hat{X}_{i+1})$.
Type 2: $f_i$ is the parallel transport operator associated to a continuous path from a point $b_0$ to a point $b_1$ in the base $B$ of a 
smooth and proper family $\pi:\X\rightarrow B$ of two dimensional compact complex tori. Isomorphisms are chosen between 
$\X_{b_0}$ is and $X_i$ and between $\X_{b_1}$ and $X_{i+1}$. 

We define next a functor
$F:\G_3\rightarrow \G_4$ as follows. $F$ sends the object $(X,w,H)$ to $X$. $F$ sends a morphism 
$(g,\gamma):(X_1,w_1,H_1)\rightarrow (X_2,w_2,H_2)$ of type 1,
corresponding to a Fourier-Mukai transformation $\Phi_\E:D^b(X_1)\rightarrow D^b(X_2)$ with kernel $\E\in D^b(X_1\times X_2)$,
to the isomorphism $f_{\E,*}:H^*(X_1\times\hat{X}_1,\Integers)\rightarrow H^*(X_2\times\hat{X}_2,\Integers)$
induced by the isomorphism $f_\E:X_1\times\hat{X}_1\rightarrow X_2\times\hat{X}_2$ of Theorem
\ref{thm-isomorphism-associated-to-an-equivalence}. Morphisms of types 2 and 3 in $\G_3$ are associated to continuous paths in the bases of families of two dimensional complex tori $X_b$
and $F$ sends these to the associated parallel transport operators of the fourfolds $X_b\times \hat{X}_b$.

Let $\Alg$ be the category of commutative algebras with a unit. Let
$\Psi:\G_3\rightarrow \Alg$ be the functor which sends an object $(X,w,H)$ to $H^*(\M_H(w),\Integers)$.
$\Psi$ sends a morphism $(g,\gamma)$ in $\G_3$ to $\gamma$.
Let $\Sigma:\G_4\rightarrow \Alg$ be the functor, which sends $X$ to $H^*(X\times\hat{X},\Integers)$. 
$\Sigma$ sends a morphism in $\G_4$ to itself. We get a second functor $\Sigma\circ F$ from $\G_3$ to $\Alg.$

Given an object $(X,w,H)$ of $\G_3$ and a generic $H$-stable coherent sheaf $F$ on $X$ of Mukai vector $w$, we get the
embedding 
\begin{equation}
\label{eq-iota-F}
\iota_F:X\times \hat{X}\rightarrow \M_H(w)
\end{equation} 
given by  $\iota_F(x,L)=\tau_{x,*}(F)\otimes L$, where $L\in \hat{X}$, $x\in X$, 
and $\tau_x:X\rightarrow X$ 
sends $x'$ to $x+x'$. We postpone the proof that $\iota_F$ is an embedding to Lemma \ref{lemma-Gamma-v}(\ref{lemma-item-acts-faithfully}). The homomorphism 
$\iota_F^*:H^*(\M_H(w),\Integers)\rightarrow H^*(X\times\hat{X},\Integers)$ is independent of the choice of such a generic $F$,
as the data $\iota_F^*$ is discrete and depends continuously on $F$. We thus denote $\iota_F^*$ also by 
\begin{equation}
\label{eq-q-w}
q_w:H^*(\M_H(w),\Integers)\rightarrow H^*(X\times\hat{X},\Integers).
\end{equation}

\begin{prop}
\label{prop-q-is-a-natural-transformation}
The assignment $(X,w,H)\mapsto q_w$ defines a natural transformation $q$ from $\Psi$ to $\Sigma\circ F$.
\end{prop}

\begin{proof}
Given an isomorphism
$\gamma:H^*(\M_{H_1}(w_1),\Integers)\rightarrow H^*(\M_{H_2}(w_2),\Integers)$
corresponding to a morphism $(g,\gamma)$ in $\Hom_{\G_3}[(X_1,w_1,H_1),(X_2,w_2,H_2)]$
and $H_i$ stable sheaves $F_i$ on $X_i$ with Mukai vectors $w_i$, $i=1,2$, we need to prove that the following diagram is commutative.
\[
\xymatrix{
H^*(\M_{H_1}(w_1),\Integers) \ar[r]^\gamma \ar[d]_{\iota_{F_1}^*}& H^*(\M_{H_2}(w_2),\Integers) \ar[d]^{\iota_{F_2}^*}
\\
H^*(X_1\times \hat{X}_1,\Integers) \ar[r]_{\Sigma(F(g,\gamma))}& H^*(X_2\times \hat{X}_2,\Integers).
}
\]
The commutativity for morphisms $(g,\gamma)$ of type 2 and 3 is obvious.
Consider next the case where $(g,\gamma)$ is of type 1, associated to a stability preserving Fourier-Mukai transformation $\Phi_\E:D^b(X_1)\rightarrow D^b(X_2)$,
which induces an isomorphism 
\[
\tilde{\gamma}:\M_{H_1}(w_1)\rightarrow \M_{H_2}(w_2) 
\]
of the two moduli spaces,
and we choose $F_2$ to be a sheaf representing the object $\Phi_\E(F_1)$. The commutativity of the above diagram follows from 
the commutativity of the diagram
\[
\xymatrix{
X_1\times \hat{X}_1\ar[r]\ar[d]_{f_\E}&\Aut(\M_{H_1}(v_1))\ar[d]_{Ad_{\tilde{\gamma}}}&g \ar[d]
\\
X_2\times\hat{X}_2\ar[r]&\Aut(\M_{H_2}(v_2))& \tilde{\gamma}g\tilde{\gamma}^{-1}.
}
\]
The commutativity of the latter diagram follows from the analogous commutativity when we regard $X_i\times \hat{X}_i$ as a subgroup of the group of autoequivalences of the derived categories $D^b(X_i)$ of $X_i$ and regard $\M_{H_i}(v_i)$ as a subset of objects in $D^b(X_i)$, for $i=1,2$, see
\cite[Cor. 9.58]{huybrechts-book}. 
\end{proof}

The group $\Spin(S_X^+)_{s_n}$ acts on $H^*(\M_H(s_n),\Integers)$ via the monodromy representation 
$\mon$ in (\ref{eq-homomorphism-gamma-from-G-S-plus-even-s-n-to-Mon}). 
Corollary \ref{cor-monodromy-representation-of-spin} implies that $\Spin(S_X^+)_{w}$ similarly acts on $H^*(\M_H(w),\Integers)$. The group
$H^1(X\times\hat{X},\Integers)$ is the representation $V_X$ of $\Spin(S^+_X)$
and so $\Spin(S_X^+)_{w}$ acts on $H^*(X\times\hat{X},\Integers)\cong \wedge^*V_X.$

\begin{cor}
\label{cor-q-w-is-Spin-w-equivariant}
The homomorphism $q_w:H^*(\M_H(w),\Integers)\rightarrow H^*(X\times\hat{X},\Integers)$, given in (\ref{eq-q-w}), is 
$\Spin(S_X^+)_{w}$ equivariant.
\end{cor}

\begin{proof}
The proof of Theorem \ref{thm-monodromy-representation-mu} exhibits the image of $\Spin(S^+_X)_{s_n}$ via  $\mon$ as a subgroup of the automorphism group
$\Aut_{\G_3}(X,s_n,H)$. Conjugating by a morphism in $\Hom_{\G_3}[(X,s_n,H),(X,w,H)]$ we get that $\Spin(S_X^+)_{w}$ is a subgroup of the automorphism group
$\Aut_{\G_3}(X,w,H)$, for every object $(X,w,H)$ of $\G_3$, by Corollary \ref{cor-monodromy-representation-of-spin}.
Now $q_w$ is $\Aut_{\G_3}(X,w,H)$-equivariant, by Proposition \ref{prop-q-is-a-natural-transformation}.
\end{proof}

Let $\bar{\theta}_1:S^-_X\rightarrow H^1(\M(w),\Integers)$ be the isomorphism given in Equation (\ref{eq-double-tilde-theta-j}).
Let $m:S^+_X\rightarrow \Hom(S^-_X,V_X)$ be the homomorphism given in Corollary \ref{cor-V-plus-S-minus-is-the-Clifford-module}.
\begin{new-lemma}
\label{lemma-q-w-circ-theta-1-is-m-w}
The composition $q_w\circ \bar{\theta}_1:S_X^-\rightarrow H^1(X\times\hat{X},\Integers)=V_X$ is equal to either $m_w$ or $-m_w$.
\end{new-lemma}

\begin{proof}
Both $q_w$ and $\bar{\theta}_1$ are $\Spin(S_X^+)_{w}$ equivariant, and thus so is their composition. 
The homomorphism $q_w$ is equivariant, by Corollary \ref{cor-q-w-is-Spin-w-equivariant}.
Equivariance of $\bar{\theta}_1$ is proven in Lemma \ref{lemma-two-theta-j-are-proportional} when $w=s_n$ and the proof goes through in the general case, once we replace Theorem \ref{thm-monodromy-representation-mu} by Corollary \ref{cor-monodromy-representation-of-spin}.
Hence, $q_w\circ \bar{\theta}_1$ is a multiple $km_w$, for some integer $k$. It remains to prove that $|k|=1$.
It suffices to prove it for $w=s_n$, by Theorem \ref{thm-Hom-G3-non-empty}.
It suffices to prove that the cardinality of $\coker(q_{s_n})$ is equal to that of the group $\Gamma_{s_n}$ of $n$-torsion points of $X$, by Remark \ref{rem-Z-w} and the fact that $\tilde{\theta}_1$ is an isomorphism. The co-kernel of $q_{s_n}$ is equal to the co-kernel of the composition
\[
H^1(\Alb(\M(s_n)),\Integers)\rightarrow H^1(\M(s_n),\Integers)\RightArrowOf{\iota_F^*}H^1(X\times\hat{X},\Integers), 
\]
since the left arrow is an isomorphism and the right is $q_{s_n}$. Now, $\Alb(\M(s_n))=X\times\hat{X}$ and the above displayed 
the composition is the pullback by the homomorphism
$X\times \hat{X}\rightarrow X\times \hat{X}$ corresponding to multiplication by $n$ on the first factor and the identity on the second, whose kernel is $\Gamma_{s_n}$, see the proof of \cite[Theorem 7]{GS}.
\end{proof}

\section{The monodromy of a generalized Kummer}
\label{sec-monodromy-of-kummers}
We prove Theorem \ref{thm-Mon-2}, about $Mon^2(Y)$ for an irreducible holomorphic symplectic manifold $Y$ deformation equivalent to a generalized Kummer, in Section \ref{subsec-translation-invariant-subring}. In Section \ref{sec-comparison-with-Verbitsky-Spin-7-representation} we relate the Lie algebra of the Zariski closure of the monodromy integral $\Spin(7)$-representation we constructed
on the cohomology of $Y$ to an action of a Lie algebra constructed by Verbitsky. We use it to show that $\Spin(7)$-invariant classes are Hodge classes (Lemma \ref{lemma-invariance-under-both-translation-and-spin-actions}).
%
\subsection{The monodromy action on the translation-invariant subring}
\label{subsec-translation-invariant-subring}
Let $\M(v):=\M_H(v)$ be a smooth and compact moduli space of $H$-stable sheaves of primitive Mukai vector $v$ of dimension $m\geq 8$ over an abelian surface $X$. The Albanese variety $\Alb^0(\M(v))$ is the connected component of the identity in the larger group
$D(\M(v))$, the Deligne cohomology group of $\M(v)$, see \cite{EZ}. They fit in the exact sequence 
\[
0\rightarrow \Alb^0(\M(v))\rightarrow D(\M(v))\rightarrow H^{m,m}(\M(v),\Integers)\rightarrow 0.
\]
Denote by $\Alb^d(\M(v))$ the connected component of Deligne cohomology mapping to $d$ times the class Poincare dual to the class of a point in  
$H^{m,m}(\M(v),\Integers)$. Let 
\[
\alb:\M(v)\rightarrow \Alb^1(\M(v))
\]
be the Albanese morphism.
The abelian fourfold $A:=X\times \Pic^0(X)$ acts on $\M(v)$.
Given a point $F\in \M(v)$ the action yields the morphisms
$\iota_F:X\times \Pic^0(X)\rightarrow \M(v)$ and $alb\circ \iota_F:A\rightarrow \Alb^1(\M(v))$.
Define the morphism $\bar{q}:A\rightarrow \Alb^0(\M(v))$ by
\[
\bar{q}(g):=
(alb\circ \iota_F)(g)-alb(F),
\]
where the difference is defined, since $\Alb^1(\M(v))$ is an $\Alb^0(\M(v))$-torsor.
Then $\bar{q}$ is a group homomorphism, since every morphism of abelian varieties mapping the identity to the identity is
a group homomorphism. The morphism $\bar{q}$ is independent of the point $F$ of $\M(v)$, since it depends on $F$ continuously and varies in a discrete group. We have
\[
(alb\circ \iota_F)(g_1+g_2)=\bar{q}(g_1+g_2)+alb(F)=\bar{q}(g_1)+[\bar{q}(g_2)+alb(F)]=\bar{q}(g_1)+(alb\circ \iota_F)(g_2).
\]
Thus, the action fits in the commutative diagram
\begin{equation}
\label{eq-commutative-diagram-of-actions}
\xymatrix{
A\times \M(v) \ar[d]_{\bar{q}\times alb}\ar[r]^-\lambda & \M(v) \ar[d]_{alb}
\\
\Alb^0(\M(v))\times \Alb^1(\M(v)) \ar[r]_-{\bar{\lambda}} & \Alb^1(\M(v)),
}
\end{equation}
where $\lambda$ and $\bar{\lambda}$ are the action morphisms. 
The {\em anti-diagonal action} of $a\in A$ on $A\times \M(v)$ sends $(b,F)$ to $(b-a,\lambda(a,F))$.

Choose a point $a\in \Alb^1(\M(v))$ and denote by $K_a(v)$ the fiber of $\alb$ over $a$. Let $\iota_a:K_a(v)\rightarrow \M(v)$ be the inclusion.
The pullback homomorphism $\iota_a^*:H^i(\M(v),\Integers)\rightarrow H^i(K_a(v),\Integers)$ factors through a homomorphism
\begin{equation}
\label{eq-h-i}
h_i:Q^i(\M(v))\rightarrow H^i(K_a(v),\Integers),
\end{equation}
for $i=2,3$, since $H^1(K_a(v),\Integers)$ vanishes. Furthermore, $h_2$ is an isomorphism, by Theorem \ref{thm-yoshioka}, and $h_3$ is injective with finite co-kernel, since $Q^3(\M(v))$ is torsion free, by Lemma \ref{lemma-tilde-theta-j-is-an-isomorphism},  and pullback by the covering map
$q:A\times K_a(v)\rightarrow \M(v)$ induces an isomorphism 
$q^*:H^3(\M(v),\RationalNumbers)\rightarrow H^3(A\times K_a(v),\RationalNumbers),$
by G\"{o}ttsche's formula for the betti numbers of $K_a(v)$ \cite[Prop. 2.4.12]{gottsche}. 

\begin{new-lemma}
\label{lemma-Gamma-v}
\begin{enumerate}
\item
\label{lemma-item-acts-faithfully}
The morphism $\iota_F$ is an embedding, for generic $F$. In particular, 
the abelian fourfold $A:=X\times \Pic^0(X)$ acts faithfully on $\M(v)$. 
\item
\label{lemma-item-kernel-is-Gamma-v}
The composition
\[
S_X^-\IsomRightArrowOf{\tilde{\theta}_1} H^1(\M(v),\Integers) \IsomRightArrowOf{(alb^*)^{-1}}H^1(\Alb^0(\M(v)),\Integers)\RightArrowOf{\bar{q}^*}
H^1(A,\Integers)\cong V
\]
 is $m_v$ or $-m_v$. 
In particular,
The kernel of the homomorphism
$\bar{q}:A\rightarrow \Alb^0(\M(v))$ is the subgroup $\Gamma_v$ of Remark \ref{rem-Z-w}. 
Consequently, $\Gamma_v$ acts on each fiber $K_a(v)$ of 
the Albanese morphism. 
\item
\label{lemma-item-M-is-quotient-by-Gamma-v}
$\M(v)$ is isomorphic to the quotient of
$A\times K_a(v)$ by the anti-diagonal action of $\Gamma_v$. 
\item
\label{lemma-item-Gamma-v-embedds-in-Mon}
$\Gamma_v$ acts trivially on $H^i(K_a(v),\Integers)$, $i=2,3$, but embeds in $Mon(K_a(v))$.
The image of $\Gamma_v$ in $Mon(K_a(v))$ is characterized as the subgroup of $Mon(K_a(v))$ acting trivially on $H^i(K_a(v),\Integers)$, $i=2,3$. 
\end{enumerate}
\end{new-lemma}
\begin{proof}
Part  (\ref{lemma-item-kernel-is-Gamma-v}) was established in Lemma \ref{lemma-q-w-circ-theta-1-is-m-w}. 
Let $s_n$ be the Mukai vector $(1,0,-n)$ of the ideal sheaf of a length $n$ subscheme of $X$.
Parts 
(\ref{lemma-item-acts-faithfully}) and (\ref{lemma-item-M-is-quotient-by-Gamma-v})
are known when $v=s_n$, see for example the proof of  \cite[Theorem 7]{GS}.
These statements follow from the case of $s_n$, whenever there exists a Fourier-Mukai equivalence $\Phi:D^b(X)\rightarrow D^b(X')$ of the derived categories mapping the Mukai vector $v$ to $s_n$ and inducing
an isomorphism between the moduli spaces $\M(v)$ and $\M(s_n)$. 
The fact that the action of $X\times \Pic^0(X)$ on $\M(v)$ conjugates to that of  $X'\times \Pic^0(X')$ on $\M(s_n)$ 
follows from Orlov's characterization of $X\times \Pic^0(X)$ 
as the connected component of the identity of the subgroup of the group of autoequivalences of $D^b(X)$, which act trivially on the cohomology of $X$ \cite[Cor. 9.57]{huybrechts-book}.
The statements 
follow for a general moduli space $\M(v)$ as above, by Yoshioka's proof that $\M(v)$ is connected to $\M(s_n)$ via a sequence of stability preserving Fourier-Mukai transformations (inducing isomorphisms of moduli spaces) and deformations of the abelian surface
(Theorem \ref{thm-Hom-G3-non-empty}).

Part (\ref{lemma-item-Gamma-v-embedds-in-Mon}) The statement again reduces to the case $v=s_n$.
When $v=s_n$,  the abelian fourfold $X\times \Pic^0(X)$ naturally acts on $\M(v)$ and the subgroup $\Gamma_X$, 
of torsion points of $X$ of order $n$, is a subgroup of the first factor, which coincides with the subgroup $\Gamma_{s_n}$ of Remark \ref{rem-Z-w}. Write $a=(x_0,L)$, $x_0\in X$ and $L\in \Pic^0(X)$. Let $\nu_a$ be the automorphism of $K_a(s_n)$ induced by 
pullback of sheaves on $X$ by the automorphism $x\mapsto 2x_0-x$ of $X$, followed by tensorization by $L^2$.
$\Gamma_X$  acts on $K_a(s_n)$ by translation and is equal to the subgroup of its automorphism group which acts 
trivially on $H^i(K_a(s_n),\Integers)$, $i=2,3$, while the subgroup of the automorphism group of $K_a(s_n)$ which acts trivially on 
$H^2(K_a(s_n),\Integers)$ is generated by $\Gamma_X$ and $\nu_a$,
by \cite[Theorem 3 and Corollary 5]{BNS}. The automorphism $\nu_a$ acts on $H^3(K_a(s_n),\RationalNumbers)$ via multiplication by $-1$ and 
$\Gamma_X$ embeds in $Mon(K_a(s_n))$, by \cite[Theorem 1.3]{oguiso}. 
The subgroup of $Mon(K_a(s_n))$ acting trivially on $H^2(K_a(s_n),\Integers)$ is known to be induced by automorphisms, by
\cite[Theorem 2.1]{hassett-tschinkel-lagrangian-planes}, and  thus contains the image of $\Gamma_X$ as an index $2$ subgroup.
\end{proof}

\begin{prop}
\label{prop-overline-mon}
There exists a unique injective homomorphism\footnote{Once Theorem \ref{thm-Mon-2} is proven it would follow that this homomorphism is in fact an isomorphism}
\[
\overline{\mon}:G(S^+_X)^{even}_{v}\rightarrow Mon(K_a(v))/\Gamma_v
\]
such that both $h_2$ and $h_3$, given in (\ref{eq-h-i}), are $G(S^+_X)^{even}_{v}$ equivariant with respect to the homomorphisms $\mon$,
given in Corollary \ref{cor-monodromy-representation-of-spin}, and $\overline{\mon}$.
\end{prop}

\begin{proof}
Let $\pi:\M\rightarrow T$ be a smooth and proper family  of K\"{a}hler manifolds with fiber $\M(v)$ over a point $t_0$ of an analytic space $T$. We get the commutative diagram of
the relative Albanese variety $\RelAlb^1_\pi$ of degree $1$
\[
\xymatrix{
K_a(v) \ar[r]^{\iota_a} \ar[d] & \M(v) \ar[d]_{\alb}\ar[r]^{\subset} & \M \ar[d]_{\widetilde{\alb}} \ar@/^4pc/[dd]_{\pi}
\\
\{a\}\ar[r]^{\subset} \ar[rd]& \Alb^1(\M(v)) \ar[r]^{\subset}\ar[d] & \RelAlb^1_\pi\ar[d]_{p}
\\
& \{t_0\} \ar[r]_{\subset} & T.
}
\]
The morphism $p$ is a fibration with connected fibers. Hence, the homomorphism $p_*:\pi_1(\RelAlb^1_\pi,a)\rightarrow \pi_1(T,t_0)$ is surjective. Let $g\in Mon(\M(v))$ be a monodromy operator corresponding to a class $\gamma\in \pi_1(T,t_0)$. Choose a class
$\tilde{\gamma}$ in $\pi_1(\RelAlb^1_\pi,a)$, such that $p_*(\tilde{\gamma})=\gamma$. Let $\tilde{g}$ be the monodromy operator of
$K_a(v)$ corresponding to $\tilde{\gamma}$. Then the pullback homomorphism
$\iota_a^*:H^*(\M(v),\Integers)\rightarrow H^*(K_a(v),\Integers)$ is $(g,\tilde{g})$-equivariant,
\[
\iota_a^*(g(x))=\tilde{g}(\iota_a^*(x)),
\] 
for all $x\in H^*(\M(v),\Integers)$, since the evaluation homomorphism
\[
p^*R\pi_*\Integers\rightarrow R\widetilde{\alb}_*\Integers
\]
is a global section, hence monodromy invariant. It follows that the homomorphisms $h_2$ and $h_3$ are 
$(g,\tilde{g})$-equivariant as well. 
Now, the image of $\tilde{g}$ in $Mon(K_a(v))/\Gamma_v$ is determined by its action on 
$H^i(K_a(v),\RationalNumbers)$, for $i=2,3$. Indeed, if $\tilde{g}_1$, $\tilde{g}_2\in Mon(K_a(v))$ and 
$\tilde{g}_1\tilde{g}_2^{-1}$ acts trivially on $H^i(K_a(v),\RationalNumbers)$, for $i=2,3$, then $\tilde{g}_1\tilde{g}_2^{-1}$ belongs to $\Gamma_v$, 
since $\Gamma_v$ is equal to the subgroup of $Mon(K_a(v))$ acting trivially on both $H^2(K_a(v),\Integers)$ and $H^3(K_a(v),\Integers)$, by Lemma \ref{lemma-Gamma-v}.
The homomorphism 
\[
h_i:Q^i(\M(v))\otimes_\Integers\RationalNumbers\rightarrow H^i(K_a(v),\RationalNumbers)
\]
are isomorphisms for $i=2,3$, as noted in the paragraph below Equation (\ref{eq-h-i}). Hence, the image of $\tilde{g}$ in $Mon(K_a(v))/\Gamma_v$ is uniquely determined by $g$.
We get a canonical homomorphism
\[
Mon(\M(v))\rightarrow Mon(K_a(v))/\Gamma_v.
\]
Define $\overline{\mon}$ as the composition of the above homomorphism with the homomorphism $\mon$
given in Corollary \ref{cor-monodromy-representation-of-spin}. The homomorphism $\overline{\mon}$ is injective, since $h_i$ is injective and $G(S^+_X)^{even}_{v}$-equivariant, for $i=2,3$,  and $G(S^+_X)^{even}_{v}$ acts faithfully on the direct sum of $Q^i(\M(v))\otimes\RationalNumbers$, for $i=2,3$.
\end{proof}

Pulling back the extension 
\[
0\rightarrow \Gamma_v\rightarrow Mon(K_a(v))\rightarrow Mon(K_a(v))/\Gamma_v\rightarrow 0
\]
via $\overline{\mon}$ we get the extension
\begin{equation}
\label{eq-extension-denining-tildeGSplus}
0\rightarrow \Gamma_v\rightarrow \tildeGSplus{v}\rightarrow G(S^+_X)^{even}_{v} \rightarrow 0,
\end{equation}
where $\tildeGSplus{v}$ is a subgroup of $Mon(K_a(v))$, by the injectivity of $\overline{\mon}$.

\begin{proof}[Proof of Theorem \ref{thm-Mon-2}]
We prove the inclusion $\W^{\det\cdot \chi}\subset \overline{\mon}^2(G(S^+_X)^{even}_{s_n})$, as
the reverse inclusion was proven by Mongardi in \cite{mongardi}.
The restriction homomorphism from $H^2(\M(s_n),\Integers)$ to $H^2(K_X(n-1),\Integers)$ factors through the isomorphism
$h_2:Q^2(\M(s_n))\rightarrow H^2(K_X(n-1),\Integers)$, given in (\ref{eq-h-i}),
by Theorem \ref{thm-yoshioka}. 
Denote by $\W^{\det\cdot \chi}$ the corresponding subgroup of $O(s_n^\perp)$ as well.
Proposition \ref{prop-overline-mon} reduces the proof to checking that the isomorphism $\bar{\theta}_2:s_n^\perp\rightarrow Q^2(\M(s_n))$, given in (\ref{eq-double-tilde-theta-j}), 
conjugates the image of $\mon(G(S^+_X)^{even}_{s_n})$ in $GL[Q^2(\M(s_n),\Integers)]$ onto
$\W^{\det\cdot \chi}$. 
The subgroup $\Spin(S^+)_{s_n}$ of $G(S^+_X)^{even}_{s_n}$ maps onto 
$SO_+(S^+)_{s_n}$, by Corollary \ref{cor-SO-plus-of-direct-sums-of-hyperbolic-plane}.
The image of $\mon(\Spin(S^+)_{s_n})$ in $GL[Q^2(\M(s_n),\Integers)]$ is conjugated onto
$SO_+(S^+)_{s_n}$ 
via $\bar{\theta}_2$, by the $\Spin(S^+)_{s_n}$-equivariance of the latter established in Lemma \ref{lemma-two-theta-j-are-proportional}. 
$\W^{\det\cdot \chi}$ is generated by $SO_+(S^+)_{s_n}$ and the involution of $s_n^\perp$ sending $(1,0,n)$ to $-(1,0,n)$ and acting as the identity on $H^2(X,\Integers)$. The involution $\tau_X$, given in (\ref{eq-tau-X}), is an element of $G(S^+_X)^{even}_{s_n}$.
The former involution of $s_n^\perp$ is precisely 
$\bar{\theta}_2^{-1}\circ\mon_{\tau_X}\circ \bar{\theta}_2$, by the Equality
\[
\mon_{\tau_X}(\tilde{\theta}_2(\lambda))=-\bar{\theta}_2(\tau_X(\lambda)),
\]
for all $\lambda\in s_n^\perp$, which is a special case of
Corollary \ref{cor-monodromy-equivariance-of-theta} and Lemma \ref{lemma-two-theta-j-are-proportional}. 

It remains to prove 
the surjectivity of the homomorphism (\ref{eq-overline-mon}).
Let $\Aut_0(K_a(s_n))$ be the subgroup of $\Aut(K_a(s_n))$ consisting of automorphisms acting trivially on $H^2(K_a(s_n),\Integers)$.
The subgroup  of $Mon(K_a(s_n))$, acting trivially on $H^2(K_a(s_n),\Integers)$, is known to be 
the image of $\Aut_0(K_a(s_n))$, by
\cite[Theorem 2.1]{hassett-tschinkel-lagrangian-planes}. 
It suffices to prove that the image of $\overline{\mon}$ 
contains the image of $\Aut_0(K_a(s_n))$, by 
the surjectivity of $\overline{\mon}^2$.
Now $\Aut_0(K_a(s_n))$ is generated by  $\Gamma_{s_n}$ and the automorphism $\nu_a$ of order $2$ in the proof of
Lemma \ref{lemma-Gamma-v} (\ref{lemma-item-Gamma-v-embedds-in-Mon}), and action of the coset $\nu_a\Gamma_{s_n}$ on $H^3(K_a(s_n),\RationalNumbers)$ is equal to the action of $\overline{\mon}(-1)$, as both act as minus the identity. 
Furthermore, $\overline{\mon}(-1)$ acts trivially 
on $H^2(K_a(s_n),\Integers)$.
It follows that the subgroup $\{1,-1\}\subset G(S^+_X)^{even}_{s_n}$  is mapped via $\overline{\mon}$
onto the image of $\Aut_0(K_a(s_n))$ in $Mon(K_a(s_n))/\Gamma_{s_n}$.
\end{proof}

\hide{
Show that we can reconstruct $\M_X(v)$ from 
$K_X(v)$ as 
$[K_X(v)\times J^3(K_X(v))]/\Gamma_X$. The construction generalizes for a
deformation $Y$ of $K_X(v)$. 

Conclude, that $Mon(K_X(v))$ is a semi-direct-product of 
$H^1(X,\Integers/n\Integers)$ and $Mon(M_X(v))$,
where $n=$ \dots
}

%
\subsection{Comparison with Verbitsky's Lie algebra representation}
\label{sec-comparison-with-Verbitsky-Spin-7-representation}
\begin{new-lemma}
\label{lemma-a-normal-subgroup}
The homomorphism $\overline{\mon}$ of Proposition \ref{prop-overline-mon} embeds $\Spin(S^+_X)_v$ as a normal subgroup of $Mon(K_a(v))/\Gamma_v$.
\end{new-lemma}

\begin{proof}
The image of $\Spin(S^+_X)_v$ has index $2$ in $Mon(K_a(v))/\Gamma_v$, by Proposition \ref{prop-overline-mon} and the surjectivity statement of Theorem
\ref{thm-Mon-2} and is thus a normal subgroup. We provide next a proof independent of the surjectivity result, so independent of the result in \cite{mongardi}. 
We may assume that $v=(1,0,-1-n)$. The homomorphism $\overline{\mon}$ is injective, by Proposition \ref{prop-overline-mon}.
The group $O_+(S^+)_v$ naturally embeds in $O_+(v^\perp)$ and the image is a normal subgroup,
by \cite[Lemma 4.10]{markman-monodromy-I} (that Lemma is stated for the Mukai lattice of a $K3$ surface, but the same proof applies to the Mukai lattice of an abelian surface). Hence, $SO_+(S^+)_v$ embeds as a normal subgroup, the image being the intersection of two normal subgroups of $O_+(v^\perp)$. The restriction homomorphism
$r:Mon(K_a(v))\rightarrow Mon^2(K_a(v))$ factors through a homomorphism
$\bar{r}:Mon(K_a(v))/\Gamma_v\rightarrow Mon^2(K_a(v))$. $Mon^2(K_a(v))$ is naturally identified with a subgroup of $O_+(v^\perp)$
and the inverse image via $\bar{r}$ of $SO_+(S^+)_v$ is thus a normal subgroup of $Mon(K_a(v))/\Gamma_v$. It remains to show that this inverse image is $\overline{\mon}(\Spin(S^+_X)_v)$. Note that $\Spin(S^+_X)_v$ surjects onto $SO_+(S^+)_v$, by Lemma
\ref{lemma-spin-surjects}, and its kernel has order two generated by an element acting via scalar multiplication by $-1$ on $V$ and $S^-$.
The group $\Gamma_v$ has index two in the kernel of $r$, by the proof of Lemma \ref{lemma-Gamma-v}(\ref{lemma-item-Gamma-v-embedds-in-Mon}).
Hence, the kernel of $\bar{r}$ has order $2$ and is thus contained in $\overline{\mon}(\Spin(S^+_X)_v)$.
\end{proof}

Let $Y$ be an irreducible holomorphic symplectic manifold of complex dimension $2n$. 
We recall next Verbitsky's construction of a Lie algebra representation on the cohomology of $Y$.
Set $b_2:=\dim H^2(Y,\RealNumbers)$.
Let $h\in \End[H^*(Y,\RealNumbers)]$
be the endomorphism acting via scalar multiplication by $i-2n$ on $H^i(Y,\RealNumbers)$.
Given a class $a\in H^2(Y,\RealNumbers)$
let $e_a\in \End[H^*(Y,\RealNumbers)]$ be given by cup product with $a$. The class $a$ is called of {\em Lefschetz type}, if there exists an endomorphism $f_a\in \End[H^*(Y,\RealNumbers)]$, satisfying the $\LieAlg{sl}_2$ commutation relations
\[
[e_a,f_a]=h, \ [h,e_a]=2e_a, \ [h,f_a]=-2f_a.
\]
Such $f_a$ is unique, if it exists. The triple $\{e_a,h,f_a\}$ is called a {\em Lefschetz triple.}

Let $\LieAlg{g}(Y)$ be the Lie subalgebra of $\End[H^*(Y,\RealNumbers)]$ generated by all Lefschetz triples.
Denote by $\LieAlg{g}_k(Y)$ its graded summand of degree $k$. 
Let $Prim^k(Y)\subset H^k(Y,\RealNumbers)$ be the subspace
annihilated by $\LieAlg{g}_{-2}(Y)$ and set $Prim(Y):=\oplus_kPrim^k(Y)$.
Let $A_2\subset H^*(Y,\RealNumbers)$ be the subring generated by $H^2(Y,\RealNumbers)$.
The following theorem was proven by Verbitsky \cite{verbitsky-announcement} and in a detailed form by Looijenga and Lunts.

\begin{thm}
\label{thm-LL}
\begin{enumerate}
\item
\label{thm-item-verbitsky-lie-algebra}
\cite[Prop. 4.5]{looijenga-lunts}
$\LieAlg{g}(Y)$ is isomorphic to $\LieAlg{so}(4,b_2-2,\RealNumbers)$ and
$\LieAlg{g}_0(Y)\cong \LieAlg{so}(H^2(Y,\RealNumbers))\oplus\RealNumbers h.$
The homomorphism $e:H^2(Y,\RealNumbers)\rightarrow \LieAlg{g}_2(Y)$, sending $a$ to $e_a$, is an isomorphism.
$\LieAlg{g}_k(Y)$ vanishes if $k$ does not belong to $\{-2,0,2\}$. 
\item
\label{thm-item-Poincare-pairing}
\cite[Prop. 1.6]{looijenga-lunts} $\LieAlg{g}(Y)$ preserves, infinitesimally, the Poincare pairing on $H^*(Y,\RealNumbers)$.
\item
\label{thm-item-Prim-k-generate}
\cite[Prop. 1.6 and Cor. 2.3]{looijenga-lunts}
$H^*(Y,\RealNumbers)$ is the orthogonal direct sum, with respect to the Poincare pairing, of the $A_2$ submodules generated by $Prim^k(Y)$, $0\leq k \leq 2n$,
\[
H^*(Y,\RealNumbers)=\oplus_{k=0}^{2n} A_2 \cdot Prim^k(Y).
\]
\item
\label{thm-item-irreducible-g-Y-submodules}
\cite[Cor. 1.13]{looijenga-lunts}
Let $W$ be an irreducible $\LieAlg{g}_0(Y)$-submodule of $Prim^k(Y)$. Then the $A_2$-submodule generated by $W$ is an irreducible
$\LieAlg{g}(Y)$-submodule. Conversely, all irreducible $\LieAlg{g}(Y)$-submodules are of this type.
\item
\label{thm-item-hodge-endomorphism}
\cite[Theorem 7.1]{verbitsky-mirror-symmetry} The Hodge endomorphism of $H^*(Y,\ComplexNumbers)$, which acts on $H^{p,q}(Y)$ via scalar multiplication by $\sqrt{-1}(p-q)$, is an element of the semisimple summand $\LieAlg{so}(H^2(Y,\ComplexNumbers))$ of
$\LieAlg{g}_0(Y)\otimes_\RealNumbers\ComplexNumbers$.
\end{enumerate}
\end{thm}

Let $Y$ be an irreducible holomorphic symplectic manifold of generalized Kummer type.
Let $\Gamma$ be the subgroup of $\Aut(Y)$ acting trivially on $H^i(Y,\RationalNumbers)$, $i=2,3$. 
The $\Gamma$-action commutes with the $\LieAlg{g}(Y)$-action, since $\Gamma$ acts trivially on $H^2(Y,\RealNumbers)$ and $f_a$ is uniquely determined by $e_a$, for each Lefschetz triple $\{e_a,h,f_a\}.$
Hence, the $\Gamma$-invariant subring $H^*(Y,\RealNumbers)^\Gamma$ is a $\LieAlg{g}(Y)$-submodule of $H^*(Y,\RealNumbers)$.
Let $A_k\subset H^*(Y,\RealNumbers)^\Gamma$, $k\geq 0$, be the subalgebra generated by 
$\oplus_{i=0}^k H^i(Y,\RealNumbers)^\Gamma$. This definition of $A_2$ agrees with the one above, since $\Gamma$ acts trivially on $H^2(Y,\RealNumbers)$, by Lemma \ref{lemma-Gamma-v}(\ref{lemma-item-Gamma-v-embedds-in-Mon}).
Set $(A_k)^j:=A_k\cap H^j(Y,\RealNumbers)$.
Let $A'_k$ be the $A_2$-submodule of $H^*(Y,\RealNumbers)^\Gamma$ generated by $Prim(Y)\cap A_k$. 
Then $A'_k$ is the maximal $\LieAlg{g}(Y)$-submodule of $H^*(Y,\RealNumbers)^\Gamma$, which is contained in $A_k$,
by parts (\ref{thm-item-Prim-k-generate}) and (\ref{thm-item-irreducible-g-Y-submodules}) of Theorem \ref{thm-LL}.
The Poincare pairing restricts to a non-degenerate pairing on each irreducible $\LieAlg{g}(Y)$-submodule, by 
the second paragraph in the proof of \cite[Prop. 1.6]{looijenga-lunts}.
Hence, the Poincare pairing restricts to $A'_k$ as a non-degenerate pairing. 
Set $C_k:=H^k(Y,\RealNumbers)$, for $0\leq k\leq 3$. For $k\geq 4$, set
\[
C_k := (A'_{k-2})^\perp\cap H^k(Y,\RealNumbers)^\Gamma,
\]
where the orthogonal complement is taken with respect to the Poincare pairing. 

\begin{new-lemma}
\label{lemma-subspaces-C-k}
\begin{enumerate}
\item
\label{lemma-item-C-k-is-direct-summand}
$H^k(Y,\RealNumbers)^\Gamma$ admits a monodromy invariant and $\LieAlg{g}_0(Y)$-invariant decomposition
\begin{equation}
\label{eq-decomposition-of H-k-into-C-k}
H^k(Y,\RealNumbers)^\Gamma = (A_{k-2})^k\oplus C_k.
\end{equation}
\item
\label{lemma-item-C-k-generate}
The subspaces $C_k$, $k\geq 2$, generate $H^*(Y,\RealNumbers)^\Gamma$ as a ring.
\end{enumerate}
\end{new-lemma}

\begin{proof} (\ref{lemma-item-C-k-is-direct-summand})
$\Gamma$ is a normal subgroup of $Mon(Y)$ and so $H^*(Y,\RealNumbers)^\Gamma$ is $Mon(Y)$-invariant. Indeed, given
$\alpha\in H^*(Y,\RealNumbers)^\Gamma$, $g\in Mon(Y)$, and $\gamma\in \Gamma$, then $g^{-1}\gamma g$ belongs to $\Gamma$ and so 
$\gamma(g(\alpha))=g(g^{-1}\gamma g)(\alpha)=g(\alpha)$. Hence, $g(\alpha)$ belongs to $H^*(Y,\RealNumbers)^\Gamma$.
The $Mon(Y)$ action on $H^*(Y,\RealNumbers)^\Gamma$ factors through $Mon(Y)/\Gamma$. The proof of 
\cite[Cor. 4.6(1)]{markman-monodromy-I}
now applies to the $Mon(Y)/\Gamma$-action on $H^*(Y,\RealNumbers)^\Gamma$
(instead of the $Mon(X)$-action on $H^*(X,\RealNumbers)$ for $X$ of $K3^{[n]}$-deformation type).

(\ref{lemma-item-C-k-generate}) Follows easily by induction from part (\ref{lemma-item-C-k-is-direct-summand}).
\end{proof}

\begin{new-lemma}
\label{lemma-decomposition-of-C-2i}
\begin{enumerate}
\item
\label{lemma-item-decomposition-of-C-2i}
$C_{2i}$, $i\geq 2$, admits a $Mon(Y)/\Gamma$-invariant decomposition 
\[
C_{2i}=C_{2i}'\oplus C_{2i}''.
\]
$C_{2i}'$ either vanishes, or it is a one-dimensional representation of $Mon(Y)/\Gamma$.
$C_{2i}''$ either vanishes, or it is isomorphic to the tensor product of $H^2(Y,\RealNumbers)$ with a one-dimensional representation of 
$Mon(Y)/\Gamma$. 
\item
\label{lemma-item-decomposition-of-C-2i+1}
$C_{2i+1}$, $i\geq 1$, either vanishes, or it is an irreducible $8$-dimensional representation of $Mon(Y)/\Gamma$.
If $C_{2i+1}$ does not vanish and $Y=K_a(v)$, for $v=(1,0,\!-\!1\!-\!n)\in S^+_X$, then $C_{2i+1}$ is the spin representation for the monodromy representation of $\Spin(S^+_X)_v$ given in Proposition \ref{prop-overline-mon}.
\end{enumerate}
\end{new-lemma}

\begin{proof} 
As the decomposition (\ref{eq-decomposition-of H-k-into-C-k}) is $Mon(Y)$-invariant, we may prove the statements for $Y=K_a(v)$ and
$\Gamma=\Gamma_v$,
where $v=(1,0,\!-\!1\!-\!n)\in S^+_X$ for an abelian surface $X$.

The moduli space $\M(v)$ is the quotient of $K_a(v)\times A$ by the anti-diagonal action of $\Gamma_v$, by 
Lemma \ref{lemma-Gamma-v} (\ref{lemma-item-M-is-quotient-by-Gamma-v}). 
Denote the quotient morphism by 
$j:K_a(v)\times A\rightarrow \M(v)$.
Then the pull back homomorphism $j^*:H^*(\M(v),\RealNumbers)\rightarrow H^*(K_a(v)\times A,\RealNumbers)^{\Gamma_v}$ is surjective. Choose a point $\tilde{a}$ of $A$ over $a$ and let $\iota_{\tilde{a}}:K_a(v)\rightarrow K_a(V)\times A$ be the
natural inclusion onto $K_a(v)\times \{\tilde{a}\}$. Then $\iota_a=j\circ\iota_{\tilde{a}}$. 
The homomorphism $\iota_{\tilde{a}}^*:H^*(K_a(v)\times A,\RealNumbers)\rightarrow H^*(K_a(v),\RealNumbers)$ is surjective
and $\Gamma_v$-equivariant, since the action of $\Gamma_v$ on $H^*(A,\RealNumbers)$ is trivial.
Hence, $\iota_{\tilde{a}}^*:H^*(K_a(v)\times A,\RealNumbers)^{\Gamma_v}\rightarrow H^*(K_a(v),\RealNumbers)^{\Gamma_v}$
is surjective as well. It follows that 
\begin{equation}
\label{eq-pullback-by-iota-a-is-surjective}
\iota_a^*:H^*(\M(v),\RealNumbers)\rightarrow H^*(K_a(v),\RealNumbers)^{\Gamma_v}
\end{equation}
is the composition $\iota_{\tilde{a}}^*\circ j^*$ of two surjective homomorphisms, hence itself surjective.

Let $B_k$ be the projection to $H^k(\M(v),\RealNumbers)$ of 
the image of the homomorphism $\tilde{\theta}:S_X^k\otimes_\Integers\RealNumbers \rightarrow H^*(\M(v),\RealNumbers)$,
given in (\ref{eq-tilde-theta-homomorphism}). The subspaces $B_k$ generate the cohomology ring
$H^*(\M(v),\RealNumbers)$, by \cite[Cor. 2]{markman-diagonal}. Let $\overline{B}_k$ be the image of $B_k$ in
$H^*(K_a(v),\RealNumbers)^{\Gamma_v}$ via the restriction homomorphism $\iota_a^*$. 
The subspaces $\overline{B}_k$ generate the cohomology ring
$H^*(K_a(v),\RealNumbers)^{\Gamma_v}$, by the surjectivity of (\ref{eq-pullback-by-iota-a-is-surjective}). 
Hence, $B_k+(A_{k-2})^k=H^k(K_a(v),\RealNumbers)^{\Gamma_v}$.
Consequently, $\overline{B}_k$ surjects onto the direct summand $C_k$, for all $k$. 
We get a surjective $\Spin(S^+_X)_v$-equivariant homomorphism 
$S_X^k\otimes_\Integers\RealNumbers\rightarrow C_k$, with respect to the spin representation on $S_X$ and the monodromy representation on $C_k$. 
The rest of the proof of (\ref{lemma-item-decomposition-of-C-2i}) and (\ref{lemma-item-decomposition-of-C-2i+1})
is identical to that of \cite[Lemma 4.8]{markman-monodromy-I}, where the decomposition for even $k$ follows from the decomposition $S_X^+\otimes_\Integers\RealNumbers=\RealNumbers v\oplus (v^\perp\otimes_\Integers\RealNumbers)$.
\end{proof}

The quotient $Mon(Y)/\Gamma$ has a canonical normal subgroup $N$ obtained by conjugating the subgroup 
$\overline{\mon}(\Spin(S^+_X)_v)$ of Lemma \ref{lemma-a-normal-subgroup} via a parallel transport operator from
$H^*(K_a(v),\Integers)$ to $H^*(Y,\Integers)$.

\begin{new-lemma}
\label{lemma-compatibility-with-verbitsky}
The Lie algebra of the Zariski closure $N_\ComplexNumbers$ of $N$ in 
$GL[H^*(Y,\ComplexNumbers)^\Gamma]$ is equal to the semi-simple direct summand 
$\LieAlg{so}(H^2(Y,\ComplexNumbers))$
of the complexification 
$\LieAlg{g}_0(Y)\otimes_\RealNumbers\ComplexNumbers$ of Verbitsky's Lie algebra
$\LieAlg{g}_0(Y)$ introduced in Theorem \ref{thm-LL}(\ref{thm-item-verbitsky-lie-algebra}).
\end{new-lemma}

\begin{proof}
Verbitsky's Lie algebra $\LieAlg{g}_0(Y)$ is a monodromy invariant subalgebra of 
$\LieAlg{gl}[H^*(Y,\ComplexNumbers)^\Gamma]$, and so it suffices to prove the statement for $Y=K_a(v)$,
$v=(1,0,-1-n)$. 
Let
\[
\nu:\Spin(H^2(K_a(v),\ComplexNumbers))\rightarrow GL(H^*(K_a(v),\ComplexNumbers)^{\Gamma_v})
\]
be the integration of the infinitesimal action of the semisimple part of Verbitsy's Lie algebra $\LieAlg{g}_0(K_a(v))$ in Theorem 
\ref{thm-LL}(\ref{thm-item-verbitsky-lie-algebra}) to the group action of the corresponding simply connected group. 
Under the identification of $v^\perp\otimes_\Integers\ComplexNumbers$ with $H^2(K_a(v),\ComplexNumbers)$
we may view $\Spin(S^+_X)_v$ as a Zariski dense arithmetic subgroup of $\Spin(H^2(K_a(v),\ComplexNumbers))$.
We claim that the $\overline{\mon}$ representation of $\Spin(S^+_X)_v$, given in Proposition \ref{prop-overline-mon},
extends to a representation of $\Spin(H^2(K_a(v),\ComplexNumbers))$, which we again denote by $\overline{\mon}$.
The proof is identical to that of \cite[Lemma 4.11(3)]{markman-monodromy-I}, and uses the fact the the subspaces $C_k$ generate $H^*(K_a(v),\ComplexNumbers)^{\Gamma_v}$, by Lemma \ref{lemma-subspaces-C-k}(\ref{lemma-item-C-k-generate}), and each of the representations  $C_k\otimes_\RealNumbers\ComplexNumbers$ of $\Spin(S^+_X)_v$ is induced from a representation  of $\Spin(H^2(K_a(v),\ComplexNumbers))$, by Lemma \ref{lemma-decomposition-of-C-2i}.
(Contrast this with the congruence representation in Lemma \ref{lemma-stabilizer-in-Spin-V-maps-to-GL-4}).

We adapt next the proof of \cite[Lemma 4.13]{markman-monodromy-I} to our set-up.
The monodromy equivariance of Verbisky's representation $\nu$
yields the equality
\begin{equation}
\label{eq-monodromy-equivariance-of-nu}
\overline{\mon}(g)\nu(f)\overline{\mon}(g)^{-1}=\nu(gfg^{-1}),
\end{equation}
for all $f\in \Spin(H^2(K_a(v),\ComplexNumbers))$ and all $g\in \Spin(S^+_X)_v$.  The equality holds  also for all $g$ in
$\Spin(H^2(K_a(v),\ComplexNumbers))$, by the density of $\Spin(S^+_X)_v$. Let
\[
\eta:\Spin(H^2(K_a(v),\ComplexNumbers))\rightarrow GL\left(H^*(K_a(v),\ComplexNumbers)^{\Gamma_v}\right)
\]
be given by $\eta(g):=\nu(g)^{-1}\overline{\mon}(g).$ We have
\[
\eta(g)\nu(f)\eta(g)^{-1}=\nu(g)^{-1}\overline{\mon}(g)\nu(f)\overline{\mon}(g)^{-1}\nu(g)
\stackrel{(\ref{eq-monodromy-equivariance-of-nu})}{=}
\nu(g)^{-1}\nu(gfg^{-1})\nu(g)=\nu(f).
\]
Hence, $\nu(f)$ commutes with $\eta(g)$, for all $f,g\in\Spin(H^2(K_a(v),\ComplexNumbers))$.
We get
\[
\eta(fg)=\nu(fg)^{-1}\overline{\mon}(fg)=\nu(g)^{-1}\eta(f)\overline{\mon}(g)=\eta(f)\eta(g).
\]
Hence, $\eta$ is a representation of $\Spin(H^2(K_a(v),\ComplexNumbers))$.

Taking $g=f$ the commutation of $\nu(g)$ and $\eta(g)$ yields:
\[
\nu(g)^{-1}\overline{\mon}(g)\nu(g)=\eta(g)\nu(g)=\nu(g)\eta(g)=\overline{\mon}(g).
\]
Hence, $\nu(g)$ commutes with $\overline{\mon}(g)$, for all $g\in \Spin(H^2(K_a(v),\ComplexNumbers))$.
The subspaces $C_k$ are $\Spin(H^2(K_a(v),\ComplexNumbers))$-invariant with respect to both $\nu$ and $\overline{\mon}$,
by Lemma \ref{lemma-subspaces-C-k}. The irreducible monodromy subrepresentations $C'_{2i}$ and $C''_{2i}$ are non-isomorphic, if they do not vanish, by Lemma \ref{lemma-decomposition-of-C-2i}. Hence, each is also $\nu$-invariant, by the latter commutativity.

We claim that each of $C_{2i+1}$, $C_{2i}'$, and $C_{2i}''$ is an irreducible $\nu$ representation as well, if it does not vanish.
The statement is clear for the one-dimensional $C_{2i}'$.
Note that each non-trivial representation of $\Spin(H^2(K_a(v),\ComplexNumbers))$ of dimension 
$7$,  is necessarily irreducible. Similarly, each
 $8$-dimensional representation, which does not contain a trivial subrepresentation, is necessarily irreducible. 
The Hodge structure of each $\LieAlg{g}_0(K_a(v))$-submodule in $H^*(K_a(v),\ComplexNumbers)$ is determined by the 
$\LieAlg{g}_0(K_a(v))$-action, by Theorem \ref{thm-LL}(\ref{thm-item-irreducible-g-Y-submodules}).
Each $\nu$-subrepresentation of $H^{2p}(K_a(v),\ComplexNumbers)$, which is not of Hodge type $(p,p)$, is thus a non-trivial $\nu$ representation and each odd-degree subrepresentation is non-trivial. Irreducibility of $C_{2i+1}$ follows, if it does not vanish, as it is $8$-dimensional in that case. If non-zero, $C_{2i}''$ is $7$-dimensional  with Hodge numbers 
$(h^{i+1,i-1},h^{i,i},h^{i-1,i+1})=(1,5,1)$, since the surjective homomorphism $S^k_X\otimes_\Integers\RealNumbers\rightarrow C^k$
constructed in the proof of Lemma \ref{lemma-decomposition-of-C-2i} was a Hodge homomorphism. Hence, if non-zero,
$C''_{2i}$ is an irreducible $\nu$-representation. 

We have seen that $\nu(f)$ commutes with $\eta(g)$, for all $f,g\in\Spin(H^2(K_a(v),\ComplexNumbers))$.
Hence, $\eta$ must act on $C'_{2i}$ and $C_{2i}''$ via scalar multiplication, as they are irreducible subrepresentations 
of $\nu$, which appear with multiplicity one in the $\eta$ invariant $\nu$-representation $C_{2i}$. Similarly, $\eta$ acts on $C_k$, for odd $k$, via scalar multiplication. But $\Spin(H^2(K_a(v),\ComplexNumbers))$ does not have any non-trivial one-dimensional representations. Hence, each $C_k$  is a trivial $\eta$-representation, and hence, so is the subring $H^*(K_a(v),\ComplexNumbers)^{\Gamma_v}$ they generate (Lemma \ref{lemma-subspaces-C-k}(\ref{lemma-item-C-k-generate})). We conclude the equality $\nu=\overline{\mon}$ and so
\[
N_\ComplexNumbers:=\overline{\mon}\left(\Spin(H^2(K_a(v),\ComplexNumbers))\right)=\nu\left(\Spin(H^2(K_a(v),\ComplexNumbers))\right).
\]
\end{proof}

\begin{new-lemma}
\label{lemma-invariance-under-both-translation-and-spin-actions}
Every class in $H^{2p}(Y,\RationalNumbers)^\Gamma$, which is $N$-invariant, is of Hodge type $(p,p)$.
\end{new-lemma}

\begin{proof}
An $N$-invariant class $\alpha\in H^{2p}(Y,\RealNumbers)^\Gamma$ is annihilated by the Lie algebra of the identity component of the Zariski closure $N_\ComplexNumbers$ of $N$ in $GL[H^*(Y,\ComplexNumbers)^\Gamma]$. The latter Lie algebra is equal to 
the semisimple summand of the complexification $\LieAlg{g}_0(Y)\otimes_\RealNumbers\ComplexNumbers$, by Lemma
\ref{lemma-compatibility-with-verbitsky}. Hence, $\alpha$ is annihilated by the Hodge endomorphism, 
by Theorem \ref{thm-LL}(\ref{thm-item-hodge-endomorphism}). 
\end{proof}
%
\section{The Cayley class as a characteristic class}
\label{sec-cayley-class}
We prove Theorem \ref{thm-introduction-kappa-class-is-non-zero-and-spin-7-invariant} in this  section
exhibiting the $\Spin(V)_w$-invariant Cayley class $c_w\in H^4(X\times\hat{X},\Integers)$ as the second Chern class of the pull back
of a sheaf $\SheafEnd(E_F)$ on $\M_H(w)$ (Proposition \ref{prop-equation-for-Cayley-class}) such that
$c_2(\SheafEnd(E_F))$ is $\Spin(V)_w$-invariant (Theorem \ref{thm-kappa-class-is-non-zero-and-spin-7-invariant}).

Let $\M(w):=\M_H(w)$ be a smooth and compact moduli space of $H$-stable sheaves with a primitive Mukai vector $w$ over an abelian surface $X$. Assume\footnote{Theorem \ref{thm-monodromy-representation-mu} should hold, more generally, for moduli spaces of dimension $\geq 4$.  Once verified, Corollary \ref{cor-monodromy-representation-of-spin} would then follow for moduli spaces of dimension $\geq 4$ as well. The results of this section would then extend to moduli spaces of dimension $\geq 4$.} 
that the dimension $m$ of $\M(w)$ is greater than or equal $8$.
Set $\hat{X}:=\Pic^0(X)$.
Given a point of $\M(w)$ representing the isomorphism class of a sheaf $F$, let
$\iota_F:X\times \hat{X}\rightarrow \M(w)$ be the morphism 
given in Equation (\ref{eq-iota-F}).
Assume $F$ is generic, so that the morphism $\iota_F$ is an embedding, by Lemma \ref{lemma-Gamma-v}.
Let $\U$ be a universal sheaf over $X\times \M(w)$, possibly twisted by a Brauer class 
$\theta\in H^2_{an}(\M(w),\StructureSheaf{\M(w)}^*)$. Let $pr_i$ be the  projection from $X\times \M(w)$ onto the $i$-th factor, $i=1,2$,  and let
\[
\Phi_\U:D^b(X)\rightarrow D^b(\M(w),\theta)
\]
be the integral functor $Rpr_{2,*}(Lpr_1^*(\bullet)\stackrel{L}{\otimes} \U)$
with kernel $\U$. Let $E_F$ be the first sheaf cohomology
$\H^1(\Phi_\U(F^\vee))$, where $F^\vee:=R\SheafHom(F,\StructureSheaf{X})$ is the object derived dual to $F$. Then $E_F$ is the relative extension sheaf
$\SheafExt^1_{pr_2}(Lpr_1^*F,\U)$. Under the identification $V\cong H^1(X\times \hat{X},\Integers)$ we see that
$\wedge^4V\cong H^4(X\times \hat{X},\Integers)$ is a $\Spin(S^+_X)$ representation and so restricts to a $\Spin(S^+_X)_w$ representation.
The $\Spin(S^+_X)_w$ invariant subgroup of $\wedge^4V$ has rank $1$
\cite[Sec. 2.1]{munoz}. We will refer to either one of its integral generators as ``the'' {\em Cayley class} of $\Spin(S^+_X)_w$ 
(which belongs to $H^4(X\times \hat{X},\Integers)$ and depends on $w$).
The Cayley class is algebraic, by the following result.

\begin{thm}
\label{thm-kappa-class-is-non-zero-and-spin-7-invariant}
\begin{enumerate}
\item
\label{thm-item-c2-End-E-is-monodromy-invariant}
The sheaf $E$, given in (\ref{eq-modular-sheaf}), is a reflexive sheaf of rank $m-2$, which is locally free away from the diagonal in $\M(w)\times\M(w)$.
The class $c_2(\SheafEnd(E))\in H^4(\M(w)\times\M(w),\Integers)$ 
is $\Spin(S^+_X)_w$-invariant, with respect to the diagonal monodromy representation of Corollary
\ref{cor-monodromy-representation-of-spin}.
\item
\label{thm-item-E-F-is-reflexive}
$E_F$ is a reflexive sheaf of rank $m-2$, which is locally free over $\M(w)\setminus\{F\}$.
\item
\label{thm-item-c-2-is-monodromy-invariant}
The class $c_2(\SheafEnd(E_F))$ is $\Spin(S^+_X)_w$ invariant, with respect to the monodromy representation of Corollary
\ref{cor-monodromy-representation-of-spin}.
\item
\label{thm-Cayley-class-is-kappa-class}
The class $\iota_F^*(c_2(\SheafEnd(E_F)))$ in $H^4(X\times \hat{X},\Integers)$ is non-zero and
$\Spin(S^+_X)_w$ invariant.
\end{enumerate}
\end{thm}

The Theorem is proved at the end of this section.
We will need to treat first the case $w=s_n:=(1,0,-n)$. 
Let $\P$ be the Poincare line bundle over $X\times \hat{X}$, normalized so that its restriction to $\{0\}\times \hat{X}$ is trivial.
Denote by $[pt_X]\in H^4(X,\Integers)$ the class Poincare-dual to a point and define $[pt_{\hat{X}}]\in 
H^4(\hat{X},\Integers)$ similarly. Let $\pi_i$ be the projection from $X\times\hat{X}$ onto the $i$-th factor, $i=1,2$.

\begin{prop}
\label{prop-equation-for-Cayley-class}
When $w=s_n$, the following equality holds:
\begin{equation}
\label{eq-cayley-class}
\iota^*_F(c_2(\SheafEnd(E_F)))=-n^2c_1(\P)^2+4n^3\pi_1^*[pt_X]+4n\pi_2^*[pt_{\hat{X}}]
\end{equation}
and the above class is $\Spin(S^+_X)_{s_n}$ invariant.
\end{prop}

\begin{proof}
$\M(s_n)$ is isomorphic to $X^{[n]}\times \hat{X}$. Let $\pi_{ij}$ be the projection from
$X\times X^{[n]}\times \hat{X}$ onto the product of the $i$-th and $j$-th factors.
Let $\Z\subset X\times X^{[n]}$ be the universal subscheme and $I_\Z$ its ideal sheaf.
Then $\U:= \pi_{12}^*I_\Z\otimes\pi_{13}^*\P$ is a universal sheaf. Set $F:=I_Z$, where $Z$ is a reduced length $n$ subscheme supported on the set
$\{z_i\ : \ 1\leq i\leq n\}$, consisting of $n$ distinct points. 
Let $\tau_x:X\rightarrow X$ be the translation $\tau_x(x')=x+x'$ by $x\in X$.
The image of
$\iota_F:X\times \hat{X}\rightarrow \M(s_n)$ is the subset
\[
\{I_{\tau_x(Z)}\otimes L \ : \ x\in X, \ L\in \hat{X}\}.
\] 
If $\tau_{x_1}(Z)=\tau_{x_2}(Z)$ then $n(x_2-x_1)=0$ and translation by $x_2-x_1$ permutes the support of $Z$.
Assume that $n(z_i-z_j)\neq 0$, for some pair $i,j$, so that $\iota_F$ is an embedding. 
Let $\Delta_i\subset X\times X$ be the translate of the diagonal $\Delta$ by $(z_i,0)$, $1\leq i\leq n$.
Let $p_{ij}$ be the projection from $X\times X\times\hat{X}$ onto the product of the $i$-th and $j$-th factors.
The pullback $\iota_F^*\U$ of $\U$ to $X\times (X\times \hat{X})$
via $id_X\times\iota_F$ is thus $p_{12}^*I_{\cup_{i=1}^n \Delta_i}\otimes p_{13}^*\P.$
Furthermore, $\iota_F^*\U$ is isomorphic to the derived pullback $L\iota_F^*\U$, as $\U$ is flat over $\M(s_n)$. 
Let $\delta:X\times \hat{X}\rightarrow X\times X\times \hat{X}$ be the diagonal embedding
$(x,L)\mapsto (x,x,L)$. Then the class $[\iota_F^*\U]$ of $\iota_F^*\U$ in the topological $K$-group of 
$X\times X\times \hat{X}$ is equal to 
$[p_{13}^*\P]-n[\delta_*\P]$.
\[
[\iota_F^*\U]=[p_{13}^*\P]-n[\delta_*\P].
\]
The class $[F^\vee]$ of $F^\vee$ in the topological $K$-group of $X$ is equal to $[F]$, hence to $[\StructureSheaf{X}]-n[\ComplexNumbers_0]$, where $\ComplexNumbers_0$ is the sky-scraper sheaf at the origin.
Consider the cartesian diagram
\[
\xymatrix{
X\times X\times \hat{X} \ar[r]^-{p_{13}} \ar[d]_{p_{23}} & X\times \hat{X} \ar[r]^-{\pi_1} \ar[d]^{\pi_2} & X
\\
X\times \hat{X}\ar[r]_{\pi_2}& \hat{X}.
}
\]
We have an isomorphism of functors
$Rp_{23,*}\circ Lp_{13}^*\cong L\pi_2^*\circ R\pi_{2,*}$, by cohomology and base change. 
Hence, 
\[
\Phi_{p_{13}^*\P}(\bullet):=Rp_{23,*}\circ Lp_{13}^*(L\pi_1^*(\bullet)\otimes\P)\cong 
L\pi_2^*\circ R\pi_{2,*}(L\pi_1^*(\bullet)\otimes\P)\cong L\pi_2^*\circ \Phi_\P(\bullet).
\]
We get an isomorphism of integral functors
$\Phi_{p_{13}^*\P}\cong L\pi_2^*\circ \Phi_\P$.
Now, $\Phi_\P(\ComplexNumbers_0)\cong\StructureSheaf{\hat{X}}$ and 
$\Phi_\P(\StructureSheaf{X})\cong \ComplexNumbers_{\hat{0}}[-2]$, where $\hat{0}$ is the origin of $\hat{X}$. 
Hence,
$[\Phi_{p_{13}^*\P}(F^\vee)]=-n[\StructureSheaf{X\times\hat{X}}]+\pi_2^![\ComplexNumbers_{\hat{0}}]$.

The integral functor $\Phi_{\delta_*\P}:D^b(X)\rightarrow D^b(X\times \hat{X})$ is just the composition
\[
\xymatrix{
D^b(X)\ar[r]^-{L\pi_1^*} & D^b(X\times\hat{X}) \ar[r]^{\otimes\P} & D^b(X\times\hat{X})
}
\]
of derived pullback and tensorization by $\P$, since $p_{23}\circ\delta:X\times\hat{X}\rightarrow X\times\hat{X}$ is the identity morphism.
Hence, $\Phi_{\delta_*\P}(\ComplexNumbers_0)\cong L\pi_1^*\ComplexNumbers_0$,  
$\Phi_{\delta_*\P}(\StructureSheaf{X})\cong\P$, and
\[
[\Phi_{\delta_*\P}(F^\vee)]=
[\P]-n\pi_1^![\ComplexNumbers_0].
\]
We conclude that
\[
[\Phi_{\iota_F^*\U}(F^\vee)]=-n[\StructureSheaf{X\times\hat{X}}]+\pi_2^![\ComplexNumbers_{\hat{0}}]
-n[\P]+n^2\pi_1^![\ComplexNumbers_0].
\]
The first three terms of its Chern character are thus
\[
-ch[\Phi_{\iota_F^*\U}(F^\vee)]=
2n+nc_1(\P)+(n/2)c_1(\P)^2-n^2\pi_1^*[pt_X]-\pi_2^*[pt_{\hat{X}}]+\dots.
\]
Given a class $\alpha$ in the topological $K$ group, of non-zero rank $r$, set
\[
\kappa(\alpha):=ch(\alpha)\exp(-c_1(\alpha)/r).
\]
So the graded summand of degree $4$ of 
the class 
\[
\kappa(-[\Phi_{\iota_F^*\U}(F^\vee)]):=ch[-\Phi_{\iota_F^*\U}(F^\vee)] \left(1-(1/2)c_1(\P)+(1/8)c_1(\P)^2+\dots \right)
\]
is 
\[
\kappa_2(-[\Phi_{\iota_F^*\U}(F^\vee)])=(n/4)c_1(\P)^2-n^2\pi_1^*[pt_X]-\pi_2^*[pt_{\hat{X}}].
\]
If we set $b:=[-\Phi_{\iota_F^*\U}(F^\vee)]$, then
$c_2(b\otimes b^*)=-ch_2(b\otimes b^*)=-4n\kappa_2(b),$ which is equal to the right hand side of (\ref{eq-cayley-class}).
Finally, $\H^i(\Phi_\U(F^\vee))$ vanishes, if $i\not\in \{1,2\}$, and $\H^2(\Phi_\U(F^\vee))$ is supported on the point
of $\M(w)$ representing $F$, and so on a subvariety of co-dimension $\geq 4$. Hence,
$c_i(E_F)=c_i(\H^1(\Phi_\U(F^\vee))=c_i(-\Phi_\U(F^\vee))$, for $i=1,2$, and so
$
\iota_F^*c_2\left(\SheafEnd(E_F)\right)=\iota_F^*c_2\left([\Phi_\U(F^\vee)]\otimes [\Phi_\U(F^\vee)]^*\right)=
c_2\left(L\iota_F^*\left([\Phi_\U(F^\vee)]\otimes [\Phi_\U(F^\vee)]^*\right)\right)=
c_2\left(L\iota_F^*[\Phi_\U(F^\vee)]\otimes L\iota_F^*[\Phi_\U(F^\vee)]^*\right)=
c_2\left([\Phi_{\iota_F^*\U}(F^\vee)]\otimes [\Phi_{\iota_F^*\U}(F^\vee)]^*\right),
$
where the last equality follows from cohomology and base change and the isomorphism $L\iota_F^*\U\cong\iota_F^*\U$ observed above.
Equality (\ref{eq-cayley-class}) thus follows.

The invariance of the class (\ref{eq-cayley-class}) would follow once we show that 
\begin{equation}
\label{eq-monodromy-invariant-chern-class-on-moduli}
ch\left([\Phi_\U(F^\vee)]\otimes [\Phi_\U(F^\vee)]^*\right)
\end{equation}
is invariant with respect to the monodromy action of $\Spin(S^+_X)_{s_n}$ via $\mon$ in 
Equation (\ref{eq-homomorphism-gamma-from-G-S-plus-even-s-n-to-Mon}), since the homomorphism 
$\iota_F^*:H^4(\M(s_n),\Integers)\rightarrow H^4(X\times\hat{X},\Integers)$ is 
$\Spin(S^+_X)_{s_n}$ equivariant, by Corollary \ref{cor-q-w-is-Spin-w-equivariant}. 
The Chern character $ch(\Phi_\U(F^\vee))$ is equal to 
$\tilde{\theta}(ch(F))$, where $\tilde{\theta}$ is given in (\ref{eq-tilde-theta-homomorphism}).
Now, $\mon_g(\tilde{\theta}(g^{-1}(ch(F))))=\tilde{\theta}(ch(F))\exp(c_g)$, for $g\in \Spin(S^+_X)_{s_n}$,
by Equation (\ref{eq-equivariance-of-tilde-theta}). The invariance of the class (\ref{eq-monodromy-invariant-chern-class-on-moduli})
follows, since $ch(F)=s_n$, and so $g^{-1}(ch(F))=ch(F)$.
%
\end{proof}

\begin{proof}[Proof of Theorem \ref{thm-kappa-class-is-non-zero-and-spin-7-invariant}]
Part \ref{thm-item-c2-End-E-is-monodromy-invariant}) The reflexivity and rank statements are proven in \cite[Prop. 4.1 and Rem. 4.6]{markman-hodge}. We prove the $\Spin(S^+_X)_w$-invariance of $c_2(\SheafEnd(E))$.
Given $g\in \Spin(S^+_X)_w$, we get that
$(g\otimes\mon_g)(ch(\E))=ch(\E)\exp(c_g),$ by Corollary \ref{cor-monodromy-representation-of-spin}, where $c_g$ is given in Equation (\ref{eq-c-g}). Hence, $(\tau g\tau\otimes \mon_g)(ch(\E)^\vee)=ch(\E)^\vee\exp(-c_g)$, by Lemma \ref{lemma-dual-of-a-universal-class-to-dual-of-a-universal-class}. 
The chern character of the object $\F:=R\pi_{13,*}R\SheafHom\left(\pi_{12}^*\E,\pi_{23}^*\E\right)$ in
$D^b(\M(w)\times\M(w))$ is obtained by contracting the tensor product
$ch(\E)^\vee\otimes ch(\E)$ in $H^*(X\times\M(w)\times X\times \M(w))$ with the class in $H^*(X)\otimes H^*(X)$
corresponding to minus the Mukai pairing (\ref{eq-Mukai-pairing}). The latter class is invariant under $(\tau g\tau)\otimes g$. We get the equality 
\begin{eqnarray*}
(\mon_g\otimes\mon_g)(ch(\F))&=&
ch(\F)\pi_1^*\exp(-c_g)\pi_2^*\exp(c_g).
\end{eqnarray*}
The $\Spin(S^+_X)_w$-invariance of $ch(R\SheafHom(\F,\F))$ follows.
Now, $ch(\F)=ch(\StructureSheaf{\Delta_{\M(w)}})-ch(E)$, since the first sheaf cohomology of $\F$ is $E$ and the second is $\StructureSheaf{\Delta_{\M(w)}}$ and all other sheaf cohomologies vanish, by 
\cite[Prop. 4.1 and Rem. 4.6]{markman-hodge}. The statement follows, since the classes $ch_i(\StructureSheaf{\Delta_{\M(w)}})$ vanish for $i<m$.

Part (\ref{thm-item-E-F-is-reflexive}) of the Theorem is proven in \cite[Prop. 4.1 and Rem. 4.6]{markman-hodge}.
Part (\ref{thm-item-c-2-is-monodromy-invariant}) 
The proof of the $\Spin(S^+_X)_w$-invariance of $c_2(\SheafEnd(E_F))$ with respect to the monodromy representation of 
Corollary \ref{cor-monodromy-representation-of-spin} follows from Equation (\ref{eq-equivariance-of-tilde-theta}) by the same argument used in Proposition \ref{prop-equation-for-Cayley-class} to prove the invariance when $w=s_n$.
Part (\ref{thm-Cayley-class-is-kappa-class})
The $\Spin(S^+_X)_w$ equivariance of $\iota^*_F$ is established in Corollary \ref{cor-q-w-is-Spin-w-equivariant}.
The non-vanishing of the pulled back class $\iota^*_Fc_2(\SheafEnd(E_F))$ is checked for $w=s_n$ in Proposition
\ref{prop-equation-for-Cayley-class}. It follows for all $w$ by Remark \ref{rem-morphism-of-groupoid-maps-universal-class-to-same} and Theorem \ref{thm-Hom-G3-non-empty}.
\end{proof}

\section{Period domains}
\label{sec-period-maps}

In Section \ref{sec-two-isomorphic-period-domains} we construct the universal torus $\T$, given in Equation (\ref{eq-universal-torus}),  over the period domain $\Omega_{w^\perp}$ of 
irreducible holomorphic symplectic manifolds of generalized Kummer deformation type. 
In Section \ref{sec-polarization-map} we prove Proposition \ref{prop-introduction-complex-multiplication};
we construct the polarization $\Theta_h$ and the complex multiplication for the complex tori with periods in the $4$-dimensional subloci $\Omega_{\{w,h\}^\perp}$ in the $5$-dimensional period domain $\Omega_{w^\perp}$.
In Section \ref{sec-diagonal-twistor-lines} we construct a hyperk\"{a}hler structure on the complex torus $T_\ell$ associated with a K\"{a}hler class on an irreducible holomorphic symplectic manifold  with period $\ell$ (Proposition \ref{prop-Theta-h-is-a-Kahler-form}). In Section \ref{subsection-abelian-fourfolds-of-Weil-type} we prove that the subloci $\Omega_{\{w,h\}^\perp}$ parametrize abelian fourfolds of Weil-type of discriminant $1$. In Section \ref{subsection-a-universal-deformation-of-a-moduli-space-of-sheaves} we construct the universal deformation $\pi:\M\rightarrow \fM^0_{w^\perp}$ of the moduli space of sheaves over the moduli space of marked irreducible holomorphic symplectic manifolds of generalized Kummer deformation type.
In Section \ref{sec-third-intermediate-jacobians} we prove that the torus $T_\ell$ is isogenous to the $3$-rd intermediate jacobian of the irreducible holomorphic symplectic manifold of generalized Kummer deformation type with period $\ell$.

%
\subsection{Two isomorphic period domains}
\label{sec-two-isomorphic-period-domains}
Keep the notation of Section \ref{sec-Clifford-groups}. Set $S^+_\ComplexNumbers:=S^+\otimes_\Integers\ComplexNumbers$ and define
$S^-_\ComplexNumbers$ and $V_\ComplexNumbers$ similarly.
Let $\ell$ be an isotropic line in $S^+_\ComplexNumbers$.
Clifford multiplication $S^+_\ComplexNumbers\otimes V_\ComplexNumbers\rightarrow S^-_\ComplexNumbers$
restricts to $\ell\otimes V_\ComplexNumbers$ as a homomorphism of rank $4$, whose kernel is $\ell\otimes Z_\ell$, for a maximal isotropic subspace $Z_\ell$ of $V_\ComplexNumbers$,
by \cite[III.1.4 and IV.1.1]{chevalley}. 
The image of $\ell\otimes V_\ComplexNumbers$ is a maximal isotropic subspace of $S^-_\ComplexNumbers$. Conversely, $\ell$ is the kernel of the homomorphism
$S^+_\ComplexNumbers\rightarrow \Hom(Z_\ell,S^-_\ComplexNumbers)$, induced by Clifford multiplication. 
We get an isomorphism between the quadric in $\PP(S^+_\ComplexNumbers)$ of isotropic lines and 
a connected component $IG^+(4,V_\ComplexNumbers)$
of the grassmannian $IG(4,V_\ComplexNumbers)$ of maximal isotropic subspaces of $V_\ComplexNumbers$. Similarly,
interchanging the roles of $S^+_\ComplexNumbers$ and $S^-_\ComplexNumbers$ we get that the quadric of
isotropic lines in $S^-_\ComplexNumbers$ is isomorphic to the other connected component  $IG^-(4,V_\ComplexNumbers)$ of the grassmannian of maximal isotropic subspaces of $V_\ComplexNumbers$  \cite[III.1.6]{chevalley}. A maximal isotropic subspace of $V_\ComplexNumbers$ associated to an isotropic line of $S^+_\ComplexNumbers$ is called {\em even}. It is called {\em odd}, if it is associated to 
an isotropic line in $S^-_\ComplexNumbers$. 

Set 
\[
\Omega_{S^+} := \{\ell\in \PP(S^+_\ComplexNumbers) \ : \ (\ell,\ell)=0, (\ell,\bar{\ell})<0\},
\]
where the pairing is associated to the pairing $(\bullet,\bullet)_{S}$ given in (\ref{eq-Mukai-pairing}).
$\Omega_{S^+}$ is a connected open analytic subset of the quadric hyperplane of isotropic lines.
The discussion above yields a $\Spin(S^+_\RealNumbers)$-equivariant embedding
\begin{equation}
\label{eq-embeding-zeta-of-period-domains}
\zeta \ : \ \Omega_{S^+} \ \ \rightarrow \ \ IG^+(4,V_\ComplexNumbers)
\end{equation}
of $\Omega_{S^+}$ as an open subset of $IG^+(4,V_\ComplexNumbers)$ in the analytic topology. 
$\Spin(S^+_\RealNumbers)$ acts transitively 
on $\Omega_{S^+}$ (see \cite[Sec. 4.1]{huybrechts-period-domains}), and so the image of $\zeta$ is an open
$\Spin(S^+_\RealNumbers)$-orbit.
The two maximal isotropic subspaces $Z_\ell$ and 
$Z_{\bar{\ell}}$ of $V_\ComplexNumbers$ are transversal. Indeed, their intersection is even dimensional, by
\cite[III.1.10]{chevalley}, it is not $4$-dimensional, since $\ell\neq \bar{\ell}$, and it is not $2$-dimensional, by
\cite[III.1.12]{chevalley}, since the two dimensional subspace $\ell+\bar{\ell}$ is not isotropic.
$Z_{\bar{\ell}}$ is the complex conjugate of $Z_\ell$ as the map $\zeta$ is defined over $\RealNumbers$, since Clifford multiplication was defined over $\Integers$.
Let $J_\ell:V_\ComplexNumbers\rightarrow V_\ComplexNumbers$ be the endomorphism acting on $Z_\ell$
by $i$ and on $Z_{\bar{\ell}}$ by $-i$. Then $V_\RealNumbers$ is $J_\ell$-invariant and $J_\ell$ induces a complex structure on $V_\RealNumbers$. So the choice of $\ell\in \Omega_{S^+}$ endows $S^+$ with an integral weight $2$ Hodge structure,
such that $(S^+_\ComplexNumbers)^{2,0}=\ell$, 
and it endows $V$ with an integral weight $1$ Hodge structure, such that $V^{1,0}=Z_\ell$. 

Consider the real plane
\begin{equation}
\label{eq-P-ell}
P_\ell:=[\ell+\bar{\ell}]\cap S^+_\RealNumbers. 
\end{equation}
Let $\{e_1,e_2\}$ be an orthogonal basis of $P_\ell$, satisfying 
$
(e_1,e_1)_{S^+}=(e_2,e_2)_{S^+}=-2,
$
such that $\ell$ is spanned by the isotropic vector $e_1-ie_2$.
Let 
\[
m:C(S^+)\rightarrow \End(S^-\oplus V) 
\]
be the homomorphism of 
Corollary \ref{cor-V-plus-S-minus-is-the-Clifford-module} and denote by $m$ also its extension to the corresponding complex vector spaces. Recall the equality $\tilde{m}(\Spin(S^+))=\tilde{\mu}(\Spin(V))$, established in (\ref{eq-Spin-V-equals-Spin-S-plus}).
It identifies $\Spin(S^+)$ with $\Spin(V)$.

\begin{new-lemma}
\label{lemma-J-ell-as-an-element-of-spin-V}
The complex structure $J_\ell$ is the element of $SO(V_\RealNumbers)$, which is the image of the element 
$m_{e_1}\circ m_{e_2}$ of $\Spin(S^+_\RealNumbers)$.
\end{new-lemma}

\begin{proof}
$m_{e_1-ie_2}$ is a nilpotent element of square zero, and $Z_\ell=m_{e_1-ie_2}(S^-_\ComplexNumbers)\subset V_\ComplexNumbers$. Now, 
\[
(m_{e_1}\circ m_{e_2})\circ m_{e_1-ie_2}=im_{e_1-ie_2}.
\]
Hence, $(m_{e_1}\circ m_{e_2})$ acts on $Z_\ell$ via multiplication by $i$.
Similarly, $Z_{\bar{\ell}}=m_{e_1+ie_2}(S^-_\ComplexNumbers)$ 
and $(m_{e_1}\circ m_{e_2})$ acts on $Z_{\bar{\ell}}$ by $-i$. Hence, $(m_{e_1}\circ m_{e_2})$ is a lift of $J_\ell$ to an element of $\Spin(S^+_\RealNumbers)$. 
\end{proof}

Note that $m_{e_1}\circ m_{e_2}$ acts on $S^+_\RealNumbers$
with $P_\ell$  as the $-1$ eigenspace and $P_\ell^\perp$ as the $1$ eigenspace. 
Let $\Spin(V_\RealNumbers)_\ell$ be the subgroup of elements acting as the identity on the plane $P_\ell$
and let $\Spin(V_\RealNumbers)_{\ell^\perp}$ be the subgroup of elements acting as the identity on the orthogonal complement $P_\ell^\perp$.
Then $\Spin(V_\RealNumbers)_{\ell^\perp}$ is isomorphic to $U(1)$ and contains $m_{e_1}\circ m_{e_2}$.
It consists of elements of the form $a+bm_{e_1}\circ m_{e_2}$, $a,b\in\RealNumbers$, $a^2+b^2=1$.
Clearly, $J_\ell$ commutes with
the action of the subgroup $\Spin(V_\RealNumbers)_\ell\times \Spin(V_\RealNumbers)_{\ell^\perp}$ of $\Spin(V_\RealNumbers)$.
$\Spin(V_\RealNumbers)_\ell$ acts on $Z_\ell$ and  $Z_{\bar{\ell}}$ and the two representations are dual with respect to the bilinear pairing of $V_\ComplexNumbers$. Elements of $\Spin(V_\RealNumbers)_{\ell^\perp}$ act on $Z_\ell$ via scalar product by 
$e^{i\theta}$, and on $Z_{\bar{\ell}}$ by $e^{-i\theta}$, for some $\theta\in\RealNumbers$.

Given a class $w\in S^+$, with $(w,w)<0$, let
\begin{equation}
\label{eq-Omega-w-perp}
\Omega_{w^\perp}:= \{\ell\in \Omega_{S^+} \ : \ (\ell,w)=0\}
\end{equation}
be the corresponding hyperplane section of $\Omega_{S^+}$. The space $\Omega_{w^\perp}$ is connected as well.
$\Omega_{w^\perp}$ is the period domain of irreducible holomorphic symplectic manifolds deformation equivalent to generalized Kummers of dimension $2n$, if $(w,w)=-2n-2$ and $n\geq 2$, by Theorem \ref{thm-yoshioka}. Given $\ell\in \Omega_{w^\perp}$, the integral weight $1$ Hodge structure $(V,J_\ell)$
has the additional property that $\wedge^4V$ admits the integral $\Spin(S^+)_w$-invariant Cayley class, recalled in Section \ref{sec-cayley-class}, which is of Hodge type $(2,2)$, by the following Lemma.
Given a class $h\in w^\perp$, let $\Spin(S^+)_{w,h}$ be the subgroup of $\Spin(S^+)$ stabilizing both $w$ and $h$.

\begin{new-lemma}
\label{lemma-sub-hodge-structures}
\begin{enumerate}
\item
\label{lemma-item-spin-7-invariant}
Any $\Spin(S^+)_w$ invariant class in $\wedge^{2p}V_\ComplexNumbers$
is of Hodge-type $(p,p)$ with respect to $J_\ell$, for all $\ell\in \Omega_{w^\perp}$.
\item
\label{lemma-item-spin-6-invariant}
Any $\Spin(S^+)_{w,h}$ invariant class in $\wedge^{2p}V_\ComplexNumbers$
is of Hodge-type $(p,p)$ with respect to $J_\ell$, for all $\ell\in \Omega_{w^\perp}$, such that $(h,\ell)=0$.
\end{enumerate}
\end{new-lemma}

\begin{proof}
(\ref{lemma-item-spin-7-invariant})
The Zariski closure of the image of $\Spin(S^+)_w$ in $GL(\wedge^{2p}V_\ComplexNumbers)$ contains the image of 
$\Spin(S^+_\RealNumbers)_w$,
and hence also that of $\Spin(V_\RealNumbers)_{\ell^\perp}$, for all $\ell\in \Omega_{w^\perp}$. A class in $\wedge^{2p}V_\ComplexNumbers$ is of Hodge-type $(p,p)$ with respect to $J_\ell$, if and only if it is 
$\Spin(V_\RealNumbers)_{\ell^\perp}$-invariant.

(\ref{lemma-item-spin-6-invariant}) 
$\Spin(V_\RealNumbers)_{\ell^\perp}$ is contained in $\Spin(V_\RealNumbers)_{w,h}$. Hence, the Zariski closure of the image of
$\Spin(S^+)_{w,h}$ contains that of $\Spin(V_\RealNumbers)_{\ell^\perp}$.
\end{proof}

Given a negative definite three dimensional subspace $W$ of $w^\perp_\RealNumbers$ we get a subgroup $\Spin(S^+_\RealNumbers)_{W^\perp}$ of 
$\Spin(S^+_\RealNumbers)_w$, isomorphic\footnote{$\Spin(W)$ is isomorphic to $SU(2)$ and the even Clifford algebra $C(W)^{even}$ is the quaternion algebra $\HH$, by \cite[Ch. V, Example 1.5(3) and Cor. 2.10]{lam}.} 
to $SU(2)$, consisting of elements of $\Spin(S^+_\RealNumbers)$ acting as the identity on 
$W^\perp$.
It fits in the following cartesian diagram
\[
\xymatrix{
\Spin(S^+_\RealNumbers)_{W^\perp} \ar[r]^{\subset} \ar[d]_{2:1} & \Spin(S^+_\RealNumbers)_w\ar[d]^{2:1}
\\
SO(W)\ar[r]_{\subset} & SO(S^+_\RealNumbers)_w,
}
\]
where the bottom horizontal homomorphism extends an isometry of $W$ to an isometry of $S^+_\RealNumbers$ acting as the identity on $W^\perp$.
When $W$ is spanned by the real and imaginary parts of an element of $\ell$ and a K\"{a}hler class on a marked irreducible holomorphic symplectic manifold $Y$ with period $\ell$, then $W$ corresponds to a hyperk\"{a}hler structure on $Y$ \cite[1.17]{huybrects-basic-results}. We show in sections 
\ref{sec-polarization-map} and \ref{sec-diagonal-twistor-lines} that the subgroup $\Spin(S^+_\RealNumbers)_{W^\perp}$ associated to $W$ 
determines a hyperk\"{a}hler structure on the complex torus $T_\ell:=V_\ComplexNumbers/[Z_\ell+V]$ and $\Spin(S^+_\RealNumbers)_{W^\perp}$ acts on $V_\ComplexNumbers$ as the group of unit quaternions (Proposition \ref{prop-Theta-h-is-a-Kahler-form} below). 

\subsection{The polarization map $w^\perp\rightarrow \wedge^2V^*$} 
\label{sec-polarization-map}
The construction in this section is inspired by 
O'Grady's recent observation that the third intermediate Jacobians,  of projective irreducible holomorphic symplectic varieties of generalized Kummer type, are abelian fourfolds of Weil type \cite{ogrady}.
\label{sec-polarizations}
Let $w\in S^+$ be a primitive class satisfying $(w,w)_{S^+}=-2n$, where $n$ is a positive integer.
Let $w^\perp$ be the sublattice of $S^+$ orthogonal to $w$.
Given a class $h\in S^+$ we get the endomorphism $m_h$ of $S^-\oplus V$, 
given in Corollary \ref{cor-V-plus-S-minus-is-the-Clifford-module}.  It maps $V$ to $S^-$ and $S^-$ to $V$. 
The multiplication by $h$ in $A_X$ leaves the direct summand $S^-\oplus V$ invariant and restricts to $m_h$, by definition. 

Let 
\begin{equation}
\label{eq-Theta-prime}
\Theta':w^\perp\rightarrow \Hom(V,V)
\end{equation}
send $h$ to the restriction $\Theta'_h$ of $m_{w}\circ m_h$ to $V$.
Note that $m_{w}\circ m_h+m_h\circ m_{w}$ restricts to $V$ as $(w,h)_{S^+}\cdot id_V$, by Equation (\ref{eq-composition-of-m-y-1-and-m-y-2}). The latter scalar endomorphism vanishes due to the fact that $h$ is in $w^\perp$. 
Furthermore, $m_h\circ m_{w}$ restricts to $V$ as the adjoint of the restriction of $m_{w}\circ m_h$ with respect to the pairing $(\bullet,\bullet)_V$, by definition of the multiplication in $A_X$. 
Hence, $\Theta'_h$ is anti-self-dual with respect to $(\bullet,\bullet)_V$. 
The isomorphism $V\rightarrow V^*$, given by $x\mapsto (x,\bullet)_V$,
induces an isomorphism $a:\Hom(V,V)\rightarrow V^*\otimes V^*$, given by 
\[ 
a(f)(x,y):=(f(x),y)_V,
\]
for all $x,y\in V$.
An anti-self-dual homomorphism is sent by $a$ to $\wedge^2 V^*$. Hence, we get the composite homomorphism
\begin{equation}
\label{eq-Theta}
\Theta:=a\circ \Theta':w^\perp\rightarrow \wedge^2 V^*,
\end{equation}
sending $h$ to $\Theta_h$, where $\Theta_h(x,y):=(\Theta'_h(x),y)_V$.
More equivariantly, 
$\Theta$ is the composition of the embedding $w^\perp\rightarrow \wedge^2S^+$, sending $h$ to $w\wedge h$, with a $\Spin(V)$-modules isomorphism
$\wedge^2S^+\cong\wedge^2V^*$ (see the 5-th displayed formula on page 96 of \cite[Sec. II.4]{chevalley} for the latter).

Let $ort_{S^+}:G(S^+)^{even}\rightarrow \{\pm 1\}$ be the character (\ref{eq-ort-S+}), except that here it will be convenient to 
have it take values in the multiplicative group $\{\pm 1\}$ rather than in $\Integers/2\Integers$.
An element $g$ of $G(S^+)^{even}$ acts on $V$ as an isometry, if $ort_{S^+}(g)=1$, and it reverses the sign of $(\bullet,\bullet)_V$,
if $ort_{S^+}(g)=-1$. 

\hide{
Let $G(S^+)^{even}_{w^\perp}$ be the subgroup of $G(S^+)^{even}$, which leaves the rank one sublattice $\Integers w$ 
of $S^+$ invariant and acts on it via the character $ort_{S^+}$.
In other words, an element $g$ of $G(S^+)^{even}_{w^\perp}$ sends $w$ to $ort_{S^+}(g)w$. 
The subgroup $G(S^+)^{even}_{w^\perp}$ is generated by the stabilizer $\Spin(V)_{w}$ of $w$ and the element
$\tilde{m}_{s_1\cdot s_{-1}}$ used above in (\ref{eq-tilde-m-s-1-s-minus-1}) to describe $G(S^+)^{even}$.
Note that $G(S^+)^{even}_{w^\perp}$ is the image of the stabilizer $G(S^+)^{even}_{w}$ of $w$ via an automorphism of 
$G(S^+)^{even}$, as explained in Section \ref{sec-automorphism-of-G-S-plus}.
}

\begin{new-lemma}
\label{lemma-Theta-spans}
$\Theta$ spans a $G(S^+)^{even}_{w}$-invariant saturated rank $1$ subgroup in $\Hom(w^\perp,\wedge^2V^*)$ 
and $G(S^+)^{even}_{w}$ acts on it via the character $ort_{S^+}$. $\Theta$ spans also the $\Spin(V)_{w}$-invariant subgroup
of $\Hom(w^\perp,\wedge^2V^*)$.
\end{new-lemma}

\begin{proof}
The homomorphism $\Theta'':S^+\rightarrow \Hom(V,S^-)$, $h\mapsto m_h$,
 is $\Spin(V)$-equivariant, since $\Spin(V)$ acts by algebra automorphisms of $A_X$.
The homomorphism $m_{w}:S^-\rightarrow V$ is $\Spin(V)_{w}$-equivariant, since $\Spin(V)_{w}$ acts on $A_X$
by algebra automorphisms fixing $w$. Invariance of $\Theta'$ follows.
Invariance of $\Theta$ with respect to $\Spin(V)_{w}$ follows from that of $\Theta'$ 
and the invariance of the bilinear pairing $(\bullet,\bullet)_{V}$. If $h\in w^\perp$ satisfies $(h.h)_{S^+}=\pm 2$,
then $m_h$ restricts to an isomorphism from $V$ to $S^-$, so $\Theta$ is indivisible and spans a saturated subgroup
of $\Hom(w^\perp,\wedge^2V^*)$.

We claim that the homomorphism $\Theta''$ spans a rank one $G(S^+)^{even}_w$-invariant sublattice of 
$\Hom(S^+,\Hom(V,S^-))$ and $G(S^+)^{even}_w$ acts on it via the character $ort_{S^+}$. 
$\Spin(V)_w$ is a normal subgroup of $G(S^+)^{even}_w$, hence the latter maps any $\Spin(V)_w$-invariant submodule to a
$\Spin(V)_w$-invariant submodule. Once we prove that there exists a unique rank one $\Spin(V)_w$-invariant submodule in
$(S^+)^*\otimes V^*\otimes S^-$, it would follow that it is necessarily  $G(S^+)^{even}_w$-invariant and equal to 
${\rm span}_\Integers\{\Theta''\}$. We then need to show 
that $G(S^+)^{even}_w$ acts on it via the character $ort_{S^+}$. 
The bilinear pairing of $S^+$ is $G(S^+)$-invariant, so that $S^+$ is a self-dual $G(S^+)^{even}$-module.  $V^*$ is isomorphic to $V\otimes ort_{S^+}$ as
a $G(S^+)^{even}$-module. Hence, it suffices to prove that $S^+\otimes V\otimes S^-$ 
contains a unique rank one $\Spin(S^+)_w$-invariant submodule, which is the trivial character of  $G(S^+)^{even}_w$. 
Now $G(S^+)^{even}=JG(V)^{even}J^{-1}$ and
$J^{-1}\otimes J^{-1}\otimes J^{-1}$ maps a $G(S^+)^{even}_w$-invariant element of 
$S^+\otimes V\otimes S^-$ to a $G(V)^{even}_{J^{-1}(w)}$-invariant element of  $V\otimes S^-\otimes S^+$.
The Principle of Triality thus reduces the verification of the above claim to 
the statement that $\Spin(V)_{J^{-1}(w)}$ 
has a unique invariant submodule in $V\otimes S^-\otimes S^+$, which is the trivial character of $G(V)^{even}_{J^{-1}(w)}$.
Now $V$ is a self-dual $G(V)^{even}$-module and 
$S^-_\RationalNumbers\otimes S^+_\RationalNumbers$ decomposes as a direct sum of the representations
$V_\RationalNumbers$ and $\wedge^3V_\RationalNumbers$ as a $G(V_\RationalNumbers)^{even}$-representation,
by the third displayed formula on page 96 of \cite[Sec. 3.4]{chevalley}. 
$V_\RationalNumbers$ contains a one-dimensional trivial $\Spin(V)_{J^{-1}(w)}$-submodule, the one spanned by $J^{-1}(w)$,
which is also a trivial $G(V)^{even}_{J^{-1}(w)}$-module.
The $\Spin(V)_{J^{-1}(w)}$-invariant submodule of $\wedge^3V_\RationalNumbers$ 
vanishes\footnote{Set $u:=J^{-1}(w)$. The $28$-dimensional representation 
$\wedge^2V_\RationalNumbers$ of $\Spin(V_\RationalNumbers)_{u}\cong \Spin((u^\perp)_\RationalNumbers)$ 
 decomposes as the direct sum of the $21$-dimensional adjoint representation 
 $\LieAlg{so}((u^\perp)_\RationalNumbers)\cong\LieAlg{so}(7)$ and the $7$-dimensional fundamental representation $(u^\perp)_\RationalNumbers$,
  both irreducible. Hence, the $\Spin(V)_{u}$-invariant submodule of $\Hom((u^\perp)^*,\wedge^2V)$ has rank $1$, it is contained in the image of $\Sym^2(u^\perp)_\RationalNumbers\otimes u$, hence its image in $\wedge^3V$ vanishes.}. 
Hence, the statement about $\Theta''$ is proven.

The image $m_{w}$ of $w$ in $\Hom(V,S^-)$ via $\Theta''$ spans a one-dimensional $G(S^+)^{even}_{w}$-module
isomorphic to the restriction of $ort_{S^+}$, since $w$ spans an invariant $G(S^+)^{even}_{w}$-module in $S^+$ and $\Theta''$ 
spans a character isomorphic to $ort_{S^+}$, by the previous paragraph. A similar argument shows that $m_{w}\in \Hom(S^-,V)$
spans a one-dimensional $G(S^+)^{even}_{w}$-module
isomorphic to the restriction of $ort_{S^+}$. The inverse of the bilinear pairing $(\bullet,\bullet)_V$ spans a one-dimensional 
$G(S^+)^{even}_{w}$-submodule of $\Hom(V^*,V)$
isomorphic to the restriction of $ort_{S^+}$. The homomorphism
\[
(S^+)^*\otimes V^*\otimes S^-\rightarrow (w^\perp)^*\otimes V\otimes V,
\]
induced by the restriction on the first factor, by the inverse of the bilinear pairing $(\bullet,\bullet)_V$ on the second factor,
and by $m_{w}$ on the third factor, is thus $G(S^+)^{even}_{w}$-equivariant. The above displayed homomorphism maps $\Theta''$ to $\Theta$. Hence, $G(S^+)^{even}_{w}$ acts via the character $ort_{S^+}$ on the rank one submodule spanned by $\Theta$.
\end{proof}

\begin{rem}
The $\Spin(V)_w$-equivariant composition
\[
\Sym^2(w^\perp)\RightArrowOf{\Sym^2\Theta} \Sym^2(\wedge^2(V^*))\rightarrow \wedge^4(V^*)
\]
is injective and maps the 
rank $1$ trivial submodule of $\Sym^2(w^\perp)$ to the submodule spanned by the Cayley class (see \cite[Sec. 2.1]{munoz}).
\end{rem}

\hide{
\begin{rem}
\label{rem-modified-action-on-H-2-via-character-ort}
$G(S^+)^{even}$ maps to $SO(S^+)$. Hence, if $ort_{S^+}(g)=-1$, for $g\in G(S^+)^{even}$, then 
$g$ reverses the orientation of both the positive and negative cones in $S^+_\RealNumbers$.
Otherwise, it preserves both orientations.
Now $(w,w)_{S^+}=-2n<0$ and so the signature of $w^\perp$ with respect to $(\bullet,\bullet)_{S^+}$ is $(4,3)$
and $-id_{w^\perp}$ reverses the orientation of the negative cone in $(w^\perp)_\RealNumbers$. 
The Mukai pairing is $-(\bullet,\bullet)_{S^+}$, and so  
if we let $g\in G(S^+)^{even}_{w}$ acts on $w^\perp$ via $x\mapsto ort_{S^+}(g)(g(x))$ then this modified action
preserves the positive cone of $w^\perp$ with respect to the restriction of the Mukai pairing. 
\end{rem}

\begin{rem}
The skew-symmetric form $\Theta_h$ is non-degenerate, when $(h,h)_{S^+}\neq 0$, since $m_h$ induces an automorphism of 
$(V\oplus S^-)_{\RationalNumbers}$. 
The action of $G(S^+)^{even}_{w}$ on $\wedge^2V^*$ is induced by its action on $V$. 
The elementary divisors of non-degenerate skew-symmetric forms in $\wedge^2V^*$ are invariants of
the $GL(V)$-orbits in $\wedge^2V^*$ (see 
\cite[Ch. 2, Sec 6, Lemma on page 304]{GH}). In particular, they are invariants of the $G(S^+)^{even}_{w}$-orbits.
\end{rem}
}

\begin{new-lemma}
\label{lemma-complex-multiplication}
The endomorphism $\Theta'_h:=m_{w}\circ m_h:V\rightarrow V$ 
satisfies
\[
(\Theta'_h)^2=\frac{-(w,w)(h,h)}{4}id_{V}=\frac{n(h,h)}{2}id_{V}.
\]
Given a period $\ell\in\Omega_{w^\perp}$,
the endomorphism $\Theta'_h\in \End(V)$ is a Hodge endomorphism of the integral Hodge structure determined by $\ell$, whenever 
$h$ belongs to $\{\ell,w\}^\perp\cap S^+$. 
\end{new-lemma}
\begin{proof}
We have 
$
(m_{w}\circ m_h)^2=-(m_{w}\circ m_{w})\circ (m_h\circ m_h)=
\frac{n(h,h)}{2}id_{V},
$
where the first equality is due to the
identity $m_{w}\circ m_h=-m_h\circ m_{w}$ observed above.
It remains to prove that 
the endomorphism $\Theta'_h\in \End(V)$ is a Hodge endomorphism of the integral Hodge structure determined by $\ell$, whenever 
$h$ belongs to $\{\ell,w\}^\perp\cap S^+$.
Indeed, such $h$ is of Hodge type, the two-form $\Theta_h\in\wedge^2V^*$ is of Hodge type, since $\Theta$ is an integral homomorphism of Hodge structures, and $\Theta'_h$ is obtained from $\Theta_h$ via pairing with the class $(\bullet,\bullet)_V$, which is of Hodge type as $Z_\ell$ is isotropic with respect to $(\bullet,\bullet)_V$.
\end{proof}

\hide{
%
\subsubsection{Miscellaneous} 
In Section (???) we will identify $s_n^\perp$ with $H^2(K_X(n\!-\!1),\Integers)$  and
$V$ (??? or $V\otimes ort_{S^+}$ ???) with $H^3(K_X(n\!-\!1),\Integers)$  and show that the $G(S^+)^{even}_{s_n}$-action is a monodromy action, so 
that the homomorphism $\Theta$ is  invariant with respect to this subgroup of the monodromy group (where the action on $H^2(K_X(n\!-\!1),\Integers)$ corresponds to that in Remark \ref{rem-modified-action-on-H-2-via-character-ort}). 
Furthermore, any saturated submodule of $\Hom(s_n^\perp,\wedge^2V^*)$ invariant under some finite index subgroup of
$\Spin(V)_{s_n}$ is of rank one and is thus spanned by
$\Theta$.
Hence,  the homomorphism
\begin{eqnarray*}
H^2(K_X(n\!-\!1),\Integers)\otimes \wedge^2H^3(K_X(n\!-\!1),\Integers)&\rightarrow& \Integers
\\
h\otimes (\alpha\wedge\beta) & \mapsto & \int_{K_X(n\!-\!1)}(h\cup\alpha\cup\beta)\cup c.
\end{eqnarray*}
corresponds to an integer multiple of $\Theta$, for every monodromy invariant class $c$ in $H^{4n-12}(K_X(n\!-\!1),\Integers)$.

}

%
\subsection{Diagonal twistor lines}
\label{sec-diagonal-twistor-lines}
We have the $\Spin(V)_{w}$-equivariant injective homomorphism
$\Theta:w^\perp\rightarrow \wedge^2V^*$, given in (\ref{eq-Theta}). The explicit construction in terms of $w$ and the symmetric pairing of $V$
exhibits $\Theta$ as an integral homomorphism of Hodge structures, since $w$ and the pairing are both of Hodge type.
Given a class $h\in w^\perp\cap S^+$, satisfying $(h,h)_{S^+}<0$,
let 
\begin{equation}
\label{eq-four-dimentional-period-domain}
\Omega_{\{w,h\}^\perp} 
\end{equation}
be the hyperplane section of $\Omega_{w^\perp}$ consisting of periods orthogonal to both $w$ and $h$.
Given a period
$\ell\in \Omega_{\{w,h\}^\perp}$ we get the $(1,1)$-form $\Theta_h$ in $\wedge^2V^*_\RealNumbers$ given in (\ref{eq-Theta}).
We fix an orientation of the negative cone of $w^\perp_\RealNumbers$, which determines an orientation for
every negative definite three dimensional subspace of  $w^\perp_\RealNumbers$ (see, for example, \cite[Lemma 4.1]{markman-torelli}).
The real plane $P_\ell$, given in (\ref{eq-P-ell}),  is naturally oriented by its isomorphism with the complex line $\ell$. We will always choose the sign of $h$, so that
given a basis $\{e_1, e_2\}$ of $P_\ell$, compatible with its orientation, the basis $\{e_1, e_2, h\}$ is compatible with the orientation of the negative definite three dimensional subspace $P_\ell+\RealNumbers h\subset w^\perp_\RealNumbers$. 
The complex torus $T_\ell$ is an abelian variety, for every $\ell\in\Omega_{\{w,h\}^\perp}$, by the following result.

\begin{prop}
\label{prop-Theta-h-is-a-Kahler-form}
\begin{enumerate}
\item
\label{lemma-item-negative-definite}
$\Theta_h$ is a $(1,1)$-form on the complex torus 
\[
T_\ell:= V_\ComplexNumbers/[Z_\ell+V],
\]
whose associated symmetric pairing $g(x,y):=\Theta_h(J_\ell(x),y)$ is definite (and so one of $\Theta_h$ or $\Theta_{-h}$ is 
a K\"{a}hler form). We may choose the orientation of the negative cone of $w^\perp_\RealNumbers$, so that $\Theta_h$ is 
a K\"{a}hler form.
\item
\label{lemma-item-hyperkahler-structure}
Let $W$ be a negative definite three dimensional subspace of $w^\perp_\RealNumbers$ and $\ell, \ell'$ two points of 
$\Omega_{w^\perp}$,
such that the planes $P_\ell$ and $P_{\ell'}$ are both contained in $W$. Let $h\in P_\ell^\perp\cap W$ and
$h'\in P_{\ell'}^\perp\cap W$ be classes satisfying $(h',h')=(h,h)$ and such that the pairs $(h,\ell)$ and $(h',\ell')$  are both compatible
with the orientation of $W$. 
Then $g'(x,y):=\Theta_{h'}(J_{\ell'}(x),y)$ is the same K\"{a}hler metric on $V_\RealNumbers$ as the metric $g(x,y):=\Theta_h(J_\ell(x),y)$.
If $(h,h')=0$, then $J_\ell$ and $J_{\ell'}$
satisfy $J_\ell J_{\ell'}=-J_{\ell'}J_\ell$. 
\end{enumerate}
\end{prop}

\begin{proof}
Part (\ref{lemma-item-negative-definite}) 
We first express the bilinear form $g$ in terms of the Clifford action using the description of $J_\ell$ provided by 
Lemma \ref{lemma-J-ell-as-an-element-of-spin-V}:
\begin{equation}
\label{eq-g-in-terms-of-Clifford-product}
g(x,y)=\Theta_h(J_\ell(x),y)=(\Theta'_h(J_\ell(x)),y)_V=(m_{w}\circ m_h\circ m_{e_1}\circ m_{e_2}(x),y)_V,
\end{equation}
where $\Theta'_h$ is given in (\ref{eq-Theta-prime}).
The ordered set $\{w,h,e_1,e_2\}$ is an orthogonal basis of a negative definite subspace of $S^+_\RealNumbers$.
All negative definite subspaces of $S^+_\RealNumbers$ belong to a single $\Spin(S^+_\RealNumbers)$ orbit.
It suffices to prove the analogous statement for some orthogonal basis of a negative definite subspace of $S^+_\RealNumbers$.
The latter statement translates via the commutative diagram in Corollary \ref{cor-V-plus-S-minus-is-the-Clifford-module}
to the statement that the bilinear pairing
\[
(m_{f_1}\circ m_{f_2}\circ m_{f_3}\circ m_{f_4}(x),y)_{S^-}
\]
on $S^-_\RealNumbers$ is definite, for some orthogonal basis $\{f_1,f_2,f_3,f_4\}$ of a negative definite subspace of 
$V_\RealNumbers$.
Let $\{v_1,v_2,v_3,v_4\}$ be a basis of $H^1(X,\Integers)$ satisfying
$\int_X v_1\wedge v_2\wedge v_3\wedge v_4=1$, $\{\theta_1,\theta_2,\theta_3,\theta_4\}$ the dual basis of $H^1(X,\Integers)^*$ and set
$f_i:=v_i-\theta_i$, $1\leq i\leq 4$.
Then $(f_i,f_i)=-2$ and so $m_{f_i}^2=-1$.
Set $\gamma:=m_{f_1}\circ m_{f_2}\circ m_{f_3}\circ m_{f_4}$. Then $\gamma^2=1$. 
Being an isometry, the adjoint of $\gamma$ is $\gamma^{-1}$, and so $\gamma$ is self adjoint and the pairing 
$(\gamma(\bullet),\bullet)_{S^-}$ is symmetric. Regarding $H^1(X,\Integers)$ as a subgroup of $S^-$ and considering the action of $\gamma$ on $S^-$,
we have
\[
\gamma(v_4)=m_{f_1}\circ m_{f_2}\circ m_{f_3}(m_{f_4}(v_4))=m_{f_1}\circ m_{f_2}\circ m_{f_3}(-1)=-v_1\wedge v_2\wedge v_3,
\]
so that $\tau(\gamma(v_4))=-\gamma(v_4)$ and $(\gamma(v_4),v_4)_{S^-}=-\int_X \gamma(v_4)\wedge v_4=1$.
Given a permutation $\sigma$ of $\{1,2,3,4\}$ we similarly have
\[
\mbox{sgn}(\sigma)\gamma(v_{\sigma(4)})=
m_{f_{\sigma(1)}}\circ m_{f_{\sigma(2)}}\circ m_{f_{\sigma(3)}}\circ m_{f_{\sigma(4)}}(v_{\sigma(4)})=
-v_{\sigma(1)}\wedge v_{\sigma(2)}\wedge v_{\sigma(3)},
\]
so that $\mbox{sgn}(\sigma)\gamma(v_{\sigma(4)})\wedge v_{\sigma(4)}=-v_{\sigma(1)}\wedge v_{\sigma(2)}\wedge v_{\sigma(3)}\wedge v_{\sigma(4)}=-\mbox{sgn}(\sigma)v_1\wedge v_2\wedge v_3\wedge v_4$. We conclude that
\[
(\gamma(v_i),v_i)_{S^-}=\int_X v_1\wedge v_2\wedge v_3\wedge v_4=1,
\]
for $1\leq i\leq 4$. Hence, $(\gamma(\gamma(v_i)),\gamma(v_i))=1$, since $\gamma$ acts as an isometry.
But $\{v_i\}_{i=1}^4\cup \{\gamma(v_i)\}_{i=1}^4$ is a basis of $S^-$. Hence, the bilinear form $(\gamma(\bullet),\bullet)_{S^-}$
is positive definite.

Part (\ref{lemma-item-hyperkahler-structure}) Given a negative definite $4$-dimensional subspace $\Sigma$ of $S^+_\RealNumbers$
and an orthogonal basis $\{f_1,f_2,f_3,f_4\}$ of $\Sigma$, the element 
$m_{f_1}\circ m_{f_2}\circ m_{f_3}\circ m_{f_4}$ of $C(S^+_\RealNumbers)$
depends only on the element $f_1\wedge f_2\wedge f_3\wedge f_4$ of the line $\wedge^4\Sigma$. 
Indeed, set $\tilde{f}_i:=f_i/\sqrt{-Q(f_i)}$, so that
$m_{f_i}=\sqrt{-Q(f_i)}m_{\tilde{f}_i}$. Then
$
m_{f_1}\circ m_{f_2}\circ m_{f_3}\circ m_{f_4}=\sqrt{\prod_{i=1}^4Q(f_i)}m_{\tilde{f}_1}\circ m_{\tilde{f}_2}\circ m_{\tilde{f}_3}\circ m_{\tilde{f}_4},
$
where $m_{\tilde{f}_1}\circ m_{\tilde{f}_2}\circ m_{\tilde{f}_3}\circ m_{\tilde{f}_4}$ is an element of $\Spin(S^+_\RealNumbers)$ which 
acts on $\Sigma$ by $-1$ and it acts on the orthogonal complement $\Sigma^\perp$ in $S^+_\RealNumbers$ 
by $1$. This determines $m_{\tilde{f}_1}\circ m_{\tilde{f}_2}\circ m_{\tilde{f}_3}\circ m_{\tilde{f}_4}$ up to sign, and
the sign depends on the orientation of the basis
$\{\tilde{f}_1,\tilde{f}_2,\tilde{f}_3,\tilde{f}_4\}$.
Let $\{e_1',e_2'\}$
be an orthogonal basis of $P_{\ell'}$ satisfying $(e_i',e_i')=-2$, $i=1,2$, and such that $\ell'$ is spanned by 
$e_1'-ie_2'$. We get a second orthogonal basis $\{w, h', e_1', e_2'\}$ of the negative definite subspace $\Sigma:=W+\RealNumbers w$.
The two elements $h'\wedge e_1'\wedge e_2'$ and $h\wedge e_1\wedge e_2$ of $\wedge^3W$ are equal.
Consequently, $w\wedge h\wedge e_1\wedge e_2=w\wedge h'\wedge e_1'\wedge e_2'$ and so
\[
m_w\circ m_{h'} \circ m_{e_1'}\circ m_{e_2'}=m_{w}\circ m_h\circ m_{e_1}\circ m_{e_2}.
\]
The equality of the metrics $g$ and $g'$ follows from Equation (\ref{eq-g-in-terms-of-Clifford-product}). 

It remains to prove the equality $J_\ell J_{\ell'}=-J_{\ell'}J_\ell$ when $(h,h')=0$. 
Assume, possibly after rescaling by a positive real factor, that $(h,h)_{S^+}=(h',h')_{S^+}=-2$.
Let $f$ be an element of $\{h,h'\}^\perp\cap W$ satisfying $(f,f)=-2$ such that the ordered basis 
$\{h,h',f\}$ corresponds to the orientation of $W$. Then $\{h',f\}$ is a basis of $P_\ell$, 
$\{h,f\}$ is a basis of $P_{\ell'}$, 
$J_\ell$ or $-J_\ell$ lifts to the element
$m_{h'}\circ m_f$ of $\Spin(S^+_\RealNumbers)$ and 
$J_{\ell'}$ or $-J_{\ell'}$ lifts to $m_f\circ m_h$, by Lemma \ref{lemma-J-ell-as-an-element-of-spin-V}. We have
\[
(m_{h'}\circ m_f)\circ (m_f\circ m_h)=-m_{h'}\circ m_h=-(m_f\circ m_h)\circ (m_{h'}\circ m_f).
\]
Now, $-1\in \Spin(S^+)$ acts on $V$ via multiplication by $-1$.
\end{proof}

Let $\pi:\T\rightarrow \Omega_{w^\perp}$ be the pullback of the universal torus over $IG^+(4,V_\ComplexNumbers)$ via the restriction to $\Omega_{w^\perp}$ of the embedding $\zeta$ given in Equation (\ref{eq-embeding-zeta-of-period-domains}). 
Given a three dimensional negative definite subspace $W$
of $w^\perp_\RealNumbers$, let $\PP_W$ be the smooth conic of isotropic lines in $W_\ComplexNumbers$. 
We will refer to $\PP_W$ as a {\em twistor line}, denote by $\pi_W:\T_W\rightarrow \PP_W$ the pulled back family, and refer to it as the
{\em twistor family} associated to $W$. Proposition \ref{prop-Theta-h-is-a-Kahler-form} (\ref{lemma-item-hyperkahler-structure})
verifies that the metric $g_W(x,y):=\Theta_h(J_\ell(x),y)$, $\ell\in \PP_W$, $h\in W\cap P_\ell^\perp$, $(\ell,h)$ compatible with the orientation of $W$, and $(h,h)=-2$, is independent of $\ell$ and is indeed a hyperk\"{a}hler metric, and the twistor family 
$\pi_W$ is the one associated to this metric.
Given a point $\ell\in \PP_W$ we get the commutative diagram:
\begin{equation}
\label{diagram-twistor-family-of-torus}
\xymatrix{
T_\ell \ar[r]^\subset \ar[d] & \T_W \ar[r]^\subset \ar[d]_{\pi_W} & \T \ar[d]^\pi
\\
\{\ell\}\ar[r]_{\subset} & \PP_W \ar[r]_{\subset} & \Omega_{w^\perp}.
}
\end{equation}

\begin{rem}
Any two points $\ell$, $\ell'$ in the period domain $\Omega_{w^\perp}$ are connected by a {\em twistor path},
namely by a sequence $\ell_0, W_1,\ell_1,  W_2, \dots, \ell_{k-1}, W_k, \ell_k$, such that 
$\ell=\ell_0$, $\ell'=\ell_k$, $W_i$ is a negative definite three dimensional subspace, and both $\ell_i$ and $\ell_{i+1}$ belong to $\PP_{W_i}$, for $0\leq i\leq k$ (see \cite[Lemma 8.4]{huybrects-basic-results}).
\end{rem}

\begin{rem}
When $(w,w)=-2n$, $n\geq 3$, then $\Omega_{w^\perp}$ is the period domain of generalized Kummers of dimension $2n-2$. 
For all $n\geq 2$, the family $\pi:\T\rightarrow \Omega_{w^\perp}$ should be related to generalized (not necessarily commutative) deformations of the derived categories of coherent sheaves over abelian surfaces, in a sense similar to \cite{markman-mehrotra-generalized-deformations}. The four dimensional compact complex torus should be thought of as the identity component of the subgroup of the group of autoequevalences of the deformed triangulated category, which acts trivially on its numerical $K$-group.
\end{rem}
%
\subsection{Abelian fourfolds of Weil type}
\label{subsection-abelian-fourfolds-of-Weil-type}
The following Corollary asserts that $(T_\ell,\Theta_h)$ is a polarized abelian variety of Weil type according to \cite[Def. 4.9]{van-Geemen}.
This was first observed by O'Grady for the isogeneous intermediate Jacobians of projective irreducible holomorphic symplectic manifolds of generalized Kummer deformation type \cite{ogrady}. Given a positive integer $d$, let the {\em norm} map $Nm:\RationalNumbers[\sqrt{-d}]\rightarrow \RationalNumbers$ be given by 
$Nm(a+b\sqrt{-d}):=(a+b\sqrt{-d})(a-b\sqrt{-d})=a^2+b^2d$. Let $n\geq 1$ be an integer and $w$ a primitive element of $S^+$
satisfying $(w,w)=-2n$.

\begin{cor}
\label{cor-weil-type}
Let $h\in w^\perp$ be an integral class and $\ell\in \Omega_{w^\perp}$, such that the pair $(h,\ell)$  satisfies the assumptions of Proposition \ref{prop-Theta-h-is-a-Kahler-form}.
Then $d:=-n(h,h)/2$ is a positive integer, $T_\ell$ is an abelian variety, and the ring $\Integers[\sqrt{-d}]$ acts on $T_\ell$ via 
integral Hodge endomorphisms, such that $
\lambda^*(\Theta_h)=Nm(\lambda)\Theta_h$, for all $\lambda\in \Integers[\sqrt{-d}]$.
\end{cor}

\begin{proof}
 If $h$ is integral, then $\Theta_h$ is an ample class, by Proposition \ref{prop-Theta-h-is-a-Kahler-form} (\ref{lemma-item-negative-definite}), and so $T_\ell$ is an abelian variety. Integrality of $d$ is due to the fact that the lattice $V$ is even.
$\Integers[\sqrt{-d}]$ acts, by sending $\sqrt{-d}$ to the endomorphism $\Theta'_h$, which satisfies 
$(\Theta'_h)^2=(-d) id_V$, by Lemma \ref{lemma-complex-multiplication}. Finally, we compute
\[
\Theta_h(\Theta'_h(x),\Theta'_h(y))=((\Theta'_h)^2(x),\Theta'_h(y))=-d(x,\Theta'_h(y))=d(\Theta'_h(x),y)=d\Theta_h(x,y),
\]
where the third equality follows from the anti-self-duality of $\Theta'_h$. Set $\lambda:=a+b\sqrt{-d}$. We get
\begin{eqnarray*}
(\lambda^*\Theta_h)(x,y)&=&\Theta_h(ax+b\Theta'_h(x),ay+b\Theta'_h(y))
\\
&=&
(a^2+b^2d)\Theta_h(x,y)+ab[\Theta_h(x,\Theta'_h(y))+\Theta_h(\Theta'_h(x),y)]
\\
&=&(a^2+b^2d)\Theta_h(x,y)=Nm(\lambda)\Theta_h(x,y).
\end{eqnarray*}
\end{proof}

We recall next a discrete isogeny invariant of abelian varieties of Weil type.
Set $K:=\RationalNumbers[\sqrt{-d}]$.
Consider the map
$
H:V_\RationalNumbers\otimes V_\RationalNumbers\rightarrow K,
$
given by
\begin{equation}
\label{eq-Hermitian-form}
H(x,y):=\Theta_h(x,\Theta'_h(y))+\sqrt{-d}\Theta_h(x,y)=d(x,y)+\sqrt{-d}(\Theta'_h(x),y).
\end{equation}
$H$ is a non-degenerate Hermitian form on the $4$-dimensional $K$-vector space $V_\RationalNumbers$, by
\cite[Lemma 5.2]{van-Geemen}. Choose a $K$-basis $\beta:=\left\{x_1,x_2,x_3,x_4\right\}$ of $V_\RationalNumbers$ and denote by 
$\Psi:=\left(H(x_i,x_j)\right)$ the Hermitian matrix 
of $H$ with respect to $\beta$. 
\begin{defi}
\label{def-discriminant}
The {\em discriminant}
$
\det H
$
of $H$ is  the image of $\det(\Psi)$ in $\RationalNumbers^*/Nm(K^*)$. 
\end{defi}
The discriminant $\det H$ is independent of the choice of $\beta$,
by \cite[Lemma 5.2(3)]{van-Geemen}. 

\begin{new-lemma}
\label{lemma-trivial-discriminant}
The  Hermitian forms of the abelian fourfolds of Weil type in Corollary \ref{cor-weil-type} all have trivial discriminants.
\end{new-lemma}
\begin{proof}
Let $U$ be the rank $2$ even unimodular lattice with Gram matrix $\left(\begin{array}{cc}0&1\\1&0\end{array}\right)$.
The isometry group of the orthogonal direct sum of three or more copies of $U$ acts transitively on the set of primitive elements with a fixed self intersection, by \cite[Theorem 1.14.4]{nikulin}. $S^+$ is isometric to $U^{\oplus 4}$. Hence, 
the sublattice spanned by $\{w,h\}$ is contained in a sublattice of $S^+$ isometric to $U\oplus U$.
Consequently, the  orthogonal sublattice $\{w,h\}^\perp$ contains a sublattice $U_1\oplus U_2$ of $S^+$ isometric to $U\oplus U$.
Here and below the notation $(\bullet)^\perp$ is with respect to the bilinear parings $(\bullet,\bullet)_{S^+}$ or
$(\bullet,\bullet)_{V}$, but not with respect to $H$. In this proof $(\bullet,\bullet)_V$ will be denoted by $(\bullet,\bullet)$.
Let $e_i,f_i\in U_i$ be elements satisfying 
\[
(e_i,e_i)=2, \ (f_i,f_i)=-2, \ (e_i,f_i)=0,
\]
$i=1,2$.
We get the four isotropic classes $z_1:=e_1-f_1$, $z_2:=e_1+f_1$, 
$y_1:=e_2-f_2$, and $y_2:=e_2+f_2$.
The elements $\eta_i:=m_{e_i}\circ m_{f_i}\in G(S^+)^{even}$ commute with $\Theta'_h$ (and so are $\Integers[\sqrt{-d}]$-module automorphisms) and satisfy $\eta_i^2=1\in C(S^+)$ and
\begin{equation}
\label{eq-eta-i-reverses-the-sgn-of-pairing-on-V}
(\eta_i(x),\eta_i(x))=-(x,x), \ \forall x\in V.
\end{equation}
Let $L_{z_i}$ and $L_{y_i}$ be the $4$-dimensional isotropic subspaces of $V_\RationalNumbers$ associated to the isotropic vectors 
$z_i$ and $y_i$, $i=1,2$, by \cite[IV.1.1]{chevalley}. $L_{z_1}$ and $L_{z_2}$ are transversal, by 
\cite[III.1.10 and III.1.12]{chevalley}. 
The automorphism $\eta_i$ acts on $U_i$ via multiplication by $-1$ and on $U_i^\perp$ as the identity. The four isotropic lines, and hence also the four isotropic subspaces $L_{z_i}$ and $L_{y_i}$, $i=1,2$, are each invariant with respect to both $\eta_1$ and $\eta_2$.
The action of $\eta_1$ on $L_{z_i}$ commutes with that of the subgroup $\Spin(S^+)_{e_1,f_1}$  of $\Spin(S^+)$ stabilizing both $e_1$ and $f_1$. $L_{z_i}$ is an irreducible representation of $\Spin(S^+)_{e_1,f_1}$. Hence, $\eta_1$ acts on each $L_{z_i}$ via multiplication by a scalar, which is $1$ or $-1$, since $\eta_1^2=1$. 
The automorphism $\eta_1$ acts on one of $L_{z_1}$ or $L_{z_2}$ via $-1$ and on the other  as the identity, by Equation
(\ref{eq-eta-i-reverses-the-sgn-of-pairing-on-V}) and the transversality of $L_{z_1}$ and $L_{z_2}$.
Similarly, $\eta_2$ acts on one of $L_{y_1}$ or $L_{y_2}$ via $-1$ and on the other as the identity. 
The subspaces $L_{z_i,y_j}:=L_{z_i}\cap L_{y_j}$, $i,j\in\{1,2\}$, are two dimensional, by \cite[III.1.12]{chevalley},
since the subspace spanned by
$\{z_i,y_j\}$ is isotropic.
We conclude that each of $L_{z_i,y_j}:=L_{z_i}\cap L_{y_j}$, $i,j\in\{1,2\}$, is the direct sum of two copies of the same character of the group $G$ generated by $\eta_1$ and $\eta_2$ and the four characters are distinct. It follows that $\Theta'_h$ leaves each 
$L_{z_i,y_j}$ invariant, since it commutes with $G$. Hence, each of $L_{z_i,y_j}$ is a one-dimensional $K$ subspace of $V_\RationalNumbers$. Let $\hat{}$ be the transposition permutation of $\{1,2\}$.
Then 
$
L_{z_i,y_j}^\perp=L_{z_i,y_j}+L_{z_{\hat{i}},y_j}+L_{z_i,y_{\hat{j}}}.
$
Being $\Theta'_h$-invariant, the right hand side is also the $H$-orthogonal $K$-subspace to $L_{z_i,y_j}$. 
Furthermore, both $(\bullet,\bullet)_V$ and $H$ induce a non-degenerate bilinear pairing between $L_{z_i,y_j}$ and
$L_{z_{\hat{i}},y_{\hat{j}}}$. 

Let $a\in L_{z_1,y_1}$ and $b\in L_{z_2,y_2}$ be elements satisfying
\[
 (a,b)\not=0 \ \mbox{and}  \ (a,\Theta'_h(b))=0.
\]
Note that $(a,a)=0=(b,b)$.
Set $x_1:=a+b$ and $x_2=a-b$. Then $(x_1,x_2)=0$, $(x_1,x_1)=2(a,b)=-(x_2,x_2)$, and 
\[
H(x_1,x_2)=d(x_1,x_2)+\sqrt{-d}(\Theta'_h(x_1),x_2)=-2\sqrt{-d}(a,\Theta'_h(b))=0.
\]
Choose $a'\in L_{z_1,y_2}$ and $b'\in L_{z_2,y_1}$ satisfying
\[
(a',b')\not=0 \ \mbox{and}  \ (a',\Theta'_h(b'))=0.
\]
Then $(x_3,x_3)=2(a',b')=-(x_4,x_4)$ and $H(x_3,x_4)=0$. 
We conclude that $\beta:=\{x_1,x_2,x_3,x_4\}$ is an $H$-orthogonal $K$-basis for $V_\RationalNumbers$ and the $\beta$-matrix $\Psi$ of $H$ satisfies
\[
\det(\Psi)=\prod_{i=1}^4H(x_i,x_i)=d^4\prod_{i=1}^4(x_i,x_i)=d^4(x_1,x_1)^2(x_3,x_3)^2\in (\RationalNumbers^*)^2.
\]
The discriminant is trivial, since $(\RationalNumbers^*)^2$ is contained in $Nm(K^*)$.
\end{proof}

\begin{new-lemma}
\label{lemma-spin-w-h-preserves-H}
The subgroup $\Spin(S^+)_{w,H}$ of $\Spin(S^+)_w$ leaving invariant the Hermitian form $H$ given in (\ref{eq-Hermitian-form}) is equal to the subgroup
$\Spin(S^+)_{w,h}$ stabilizing both $w$ and $h$.
\end{new-lemma}

\begin{proof}
$\Spin(S^+)_{w,h}$ preserves the bilinear pairing $(\bullet,\bullet)_V$, 
acting as a subgroup of $\Spin(V)$ via the identification (\ref{eq-Spin-V-equals-Spin-S-plus}), and
$\Spin(S^+)_{w,h}$ commutes with the endomorphism $\Theta'_h:=m_w\circ m_h$ of $V$. Hence, $\Spin(S^+)_{w,h}$ leaves
the Hermitian form $H$ invariant. Conversely, the subgroup $\Spin(S^+)_{w,H}$ consists of elements of $\Spin(S^+)_w$, which commute with $\Theta'_h$, by definition of the Hermitian form  $H$. But $\Theta'_h$ is the element corresponding to $h$ 
in an irreducible $\Spin(S^+)_w$-subrepresentation of $\Hom(V,V)$ isomorphic to $w^\perp$. Hence,
$\Spin(S^+)_{w,H}$ is contained in $\Spin(S^+)_{w,h}$.
\end{proof}
%
\subsection{A universal deformation of a moduli space of sheaves}
\label{subsection-a-universal-deformation-of-a-moduli-space-of-sheaves}
Let $\M(w):=\M_H(w)$ be a smooth and compact moduli space of $H$-stable sheaves of primitive Mukai vector $w$ of dimension $\geq 8$ over an abelian surface $X$. Let 
$
\alb:\M(w)\rightarrow \Alb^1(\M(w))
$
be the Albanese morphism to the Albanese variety of degree $1$.
Choose a point $a\in \Alb^1(\M(w))$ and denote by $K_a(w)$ the fiber of $\alb$ over $a$. Let $\iota_a:K_a(w)\rightarrow \M(w)$ be the inclusion.
A {\em $\Lambda$-marking} for an irreducible holomorphic symplectic manifold $M$ is an isometry
$\eta:H^2(M,\Integers)\rightarrow \Lambda$ with a fixed lattice $\Lambda$. There exists a moduli space
$\fM_\Lambda$ of $\Lambda$-marked irreducible holomorphic symplectic manifolds, which is a non-Hausdorff complex manifold \cite{huybrects-basic-results}.
Let $\eta_0:H^2(K_a(w),\Integers)\rightarrow w^\perp$ be the isometry of Theorem \ref{thm-yoshioka}, with respect to the Beauville-Bogomolov-Fujiki pairing on $H^2(K_a(w),\Integers)$ and the Mukai pairing $-(\bullet,\bullet)_{S^+}$ on $S^+$.
Below we will continue to work with the pairing  $(\bullet,\bullet)_{S^+}$ rather than the Mukai pairing. 
Let $t_0\in \fM_{w^\perp}$ be the point representing the isomorphism class  of $(K_a(w),\eta_0)$ 
in the moduli space $\fM_{w^\perp}$ of $w^\perp$-marked irreducible holomorphic symplectic manifolds 
and let
$\fM^0_{w^\perp}$ be the connected component of  $\fM_{w^\perp}$ containing $t_0$. 
Denote by 
$Per:\fM^0_{w^\perp}\rightarrow \Omega_{w^\perp}$ the period map, sending a marked pair $(Y,\eta)$ to $\eta(H^{2,0}(Y))$.

Let $\underline{w}^\perp$ be the trivial local system over $\fM^0_{w^\perp}$ with fiber $w^\perp$. 
There exists a universal family 
\begin{equation}
\label{eq-p-universal-family-of-generalized-kummer-type}
p:\Y\rightarrow \fM^0_{w^\perp}
\end{equation} 
and a trivialization $\eta:R^2p_*\Integers\rightarrow \underline{w}^\perp$
with value $\eta_0$ at $t_0$, by \cite[Theorem 1.1]{markman-universal-family}. 
The  groups of  automorphisms of the fibers of $p$, which act trivially on the second cohomology, form a trivial local system
$\Aut_0(p)$ over $\fM^0_{w^\perp}$,  by \cite[Theorem 1.1]{markman-universal-family}.
The local subsystem $\Z$ of $\Aut_0(p)$, of subgroups which act trivially on the third cohomology as well, is thus a trivial local system.
We may thus extend the isomorphism of the fiber of $\Z$ over $t_0$ with the group $\Gamma_w$, given in Lemma \ref{lemma-Gamma-v},
to a trivialization $\psi:\Z\rightarrow \underline{\Gamma_w}$, where $\underline{\Gamma_w}$ is the trivial local system with fiber $\Gamma_w$. 

Let $Per^*(\pi):Per^*\T\rightarrow \fM^0_{w^\perp}$ be the pullback via the period map of the universal torus $\pi:\T\rightarrow \Omega_{w^\perp}$ given in Diagram (\ref{diagram-twistor-family-of-torus}). Ignoring the complex structure, $\pi$ is a differentiably trivial fibration with fiber the compact torus $V_\RealNumbers/V$. Hence, the local system $\underline{\Gamma_w}$ embeds naturally as a subsystem of torsion subgroups of $Per^*\T$. Let
\[
\M := Per^*\T\times_{\underline{\Gamma_w}}\Y
\]
be the quotient of the fiber product of $Per^*\T$ and $\Y$ over $\fM^0_{w^\perp}$ by the anti-diagonal action of $\underline{\Gamma_w}$ (this action is defined below Diagram \ref{eq-commutative-diagram-of-actions}).
Denote by 
\begin{equation}
\label{eq-universal-deformation-of-a-moduli-space-of-sheaves}
\Pi:\M\rightarrow \fM^0_{w^\perp}
\end{equation}
the natural projection and let $\M_t$ be the fiber of $\Pi$ over $t\in \fM^0_{w^\perp}$.
The fiber $\M_{t_0}$  is naturally isomorphic to $\M(w)$, by Lemma \ref{lemma-Gamma-v}. 
The relative Albanese map is then
\[
alb:\M\rightarrow (Per^*\T)/\underline{\Gamma_w}.
\] 

Let $(Y_t,\eta_t)$ be a fiber of $\Y$ over $t\in\fM^0_{w^\perp}$ endowed with the marking determined by $\eta$. 
Set $\ell:=Per(Y_t,\eta_t)$.
Let $\kappa$ be a K\"{a}hler class on $Y_t$ and set $h:=\eta_t(\kappa)$.
Let $W$ be the subspace of $(w^\perp)_\RealNumbers$ spanned by the negative definite plane $P_\ell$ and $h$.
Then $W$ is negative definite with respect to $(\bullet,\bullet)_{S^+}$. Again denote by $\PP_W$ the conic of isotropic lines in $W_\ComplexNumbers$.
We get a twistor family $p_W:\Y_W\rightarrow \PP_W$ of deformations of $Y_t$ \cite[1.17]{huybrects-basic-results}.
The marking $\eta_t$ extends to a trivialization $\eta_W$ of $R^2p_{W,*}\Integers$, since $\PP_W$ is simply connected. 
The pair $(\Y_W,\eta_W)$ determines an embedding $\iota_W:\PP_W\rightarrow \fM^0_{w^\perp}$, such that
$Per\circ \iota_W$ is the inclusion of $\PP_W$ in $\Omega_{w^\perp}$.
The image $\widetilde{\PP}_W:=\iota_W(\PP_W)$ is called {\em the twistor line through the point $(Y_t,\eta_t)$
associated to the K\"{a}hler class $\kappa$}. 

The twistor family $p_W:\Y_W\rightarrow \PP_W$ admits a differential geometric construction, which we now recall following 
\cite{beauville-varieties-with-zero-c-1}. Let $M$ be an irreducible holomorphic symplectic manifold, $I$ its complex structure, and $\kappa$ a K\"{a}hler class. There exists a unique Ricci-flat K\"{a}hler metric $g$ such that $\omega_I(\bullet,\bullet):=g(I(\bullet),\bullet)$ is a K\"{a}hler form with class $\kappa$. Furthermore, there exists an action of the quaternion algebra $\HH:=\RealNumbers+\RealNumbers I+\RealNumbers J+\RealNumbers K$ on the real tangent bundle of $M$ via parallel endomorphisms, such that 
$\omega_J+i\omega_K$ is a non-degenerate holomorphic $2$-form and $I$ is the original complex structure. We will refer to $g$ as the {\em hyperk\"{a}hler metric associated to the K\"{a}hler form $\kappa$}. 
Every purely imaginary unit quaternion $\lambda:=aI+bJ+cK$, $a^2+b^2+c^2=1$, yields a complex structure on $M$ and a  
K\"{a}hler form $\omega_\lambda:=g(\lambda(\bullet),\bullet)$ with class in the positive definite three dimensional subspace $W:=\RealNumbers\kappa+[H^{2,0}(M)+H^{0,2}(M)]\cap H^2(M,\RealNumbers)$, such that the line $\ell_\lambda:=H^{2,0}(M,\lambda)$ is one of the two isotropic lines in $W_\ComplexNumbers$ orthogonal to $\omega_\lambda$. The sphere of unit purely imaginary quaternions gets identified with the complex plane conic $\PP_W\subset\PP(W_\ComplexNumbers)$ of isotropic lines with respect to the Beauville-Bogomolov-Fujiki pairing, by sending $\lambda$ to $\ell_\lambda$. One gets a complex structure on the differentiable manifold $\PP_W\times M$, such that the first projection $p_W$ is holomorphic and the fiber over $\ell_\lambda\in \PP_W$ is endowed with the complex structure $\lambda$. The resulting family $p_W:\PP_W\times M\rightarrow \PP_W$
is the twistor family $p_W:\Y_W\rightarrow \PP_W$, if we let $M=Y_t$. 

Given a marked pair $(Y,\eta)\in\fM^0_{w^\perp}$ and a K\"{a}hler class $\kappa$ on $Y$, we get the negative definite $3$-dimensional subspace $W$ of $w^\perp_\RealNumbers$ containing $h:=\eta(\kappa)$ and such that $W_\ComplexNumbers$ contains $\ell:=Per(Y,\eta)$.
We get the hyperk\"{a}hler metric $g_W$ on $Y$, associated to the K\"{a}hler form $\kappa$, and
the metric $g$ on $T_\ell$ associated to K\"{a}hler form $\Theta_h$ in Proposition \ref{prop-Theta-h-is-a-Kahler-form}.
Hence, we get a hyperk\"{a}hler metric on the product $T_\ell\times Y$, which we call the {\em product hyperk\"{a}hler metric
associated to $(Y,\eta,\kappa)$}. The assignment $\lambda\mapsto [\omega_\lambda]\in W$, associating to a purely imaginary unit quaternion $\lambda$ the class in $W$ of the K\"{a}hler form, identifies the sphere in $W$ of self intersection $-(\kappa,\kappa)$ with 
the sphere of complex structures on $Y$ associated to $\kappa$. Similarly, the sphere in $W$ of self intersection $-2$ was identified with
the sphere of complex structures on $T_\ell$ associated to $\Theta_h$. Rescaling $\kappa$ so that $(\kappa,\kappa)=2$, we get an identification of the sphere of complex structures on $Y$ and $T_\ell$ and so an action of $\HH$ on the real tangent bundle of
$T_\ell\times Y$, so that the purely imaginary unit quaternions act via parallel complex structures. 

\begin{defi}
\label{def-product-hyperkahler-structure}
\begin{enumerate}
\item
The {\em product hyperk\"{a}hler  structure on $T_\ell\times Y$, associated to a marked pair $(Y,\eta)\in\fM^0_{w^\perp}$ with period $\ell$ and a K\"{a}hler class $\kappa$ on $Y$,} is the data consisting of the product hyperk\"{a}hler metric  associated to $(Y,\eta,\kappa)$ and the above action of the quaternion algebra $\HH$.
\item
\label{def-item-natural-hyperkahler-structure}
The above product hyperk\"{a}hler  structure on $T_\ell\times Y$ is equivariant with respect to the anti-diagonal action of $\Gamma_w$
and it thus descends to a hyperk\"{a}hler  structure on the quotient $[T_\ell\times Y]/\Gamma_w$, which we call the
{\em natural hyperk\"{a}hler  structure on $[T_\ell\times Y]/\Gamma_w$ associated to a marked pair $(Y,\eta)\in\fM^0_{w^\perp}$ with period $\ell$ and a K\"{a}hler class $\kappa$ on $Y$.}
\end{enumerate}
\end{defi}

Denote by
\begin{equation}
\label{eq-twistor-family-of-natural-hyperkahler-structure}
\Pi_W:\M_W\rightarrow \widetilde{\PP}_W
\end{equation}
the restriction of $\Pi:\M\rightarrow \fM^0_{w^\perp}$
to the twistor line $\widetilde{\PP}_W$ through $(Y_t,\eta_t)$ associated to the K\"{a}hler class $\kappa$.
Note that $\M_W$ is the quotient by the anti-diagonal action of $\Gamma_w$ on the fiber product of the two twistor families $\pi_W:\T_W\rightarrow \PP_W$, given in 
(\ref{diagram-twistor-family-of-torus}), and $p_W:\Y_W\rightarrow \PP_W$ over the same twistor line $\PP_W$ in $\Omega_{w^\perp}$. The fiber product is itself a twistor deformation
with respect to the product hyperk\"{a}hler structure on $T_\ell\times Y_t$. Similarly,
the twistor family  $\Pi_W$ displayed above is the one associated to the natural hyperk\"{a}hler structure on $\M_t:=[T_\ell\times Y_t]/\Gamma_w$.

Let $(Y,\eta)$ be a marked pair in $\fM^0_\Lambda$ and set $\ell:=Per(Y,\eta)$. 
Let $H^{2p}(Y\times T_\ell,\RationalNumbers)^{\Gamma_w}$ be the subspace invariant under the anti-diagonal action of
$\Gamma_w$. Note the isomorphism
$H^*(Y\times T_\ell,\RationalNumbers)^{\Gamma_w}\cong 
H^*(Y,\RationalNumbers)^{\Gamma_w}\otimes H^*(T_\ell,\RationalNumbers)$. 
Recall that the quotient $Mon(Y)/\Gamma_w$ has a canonical normal subgroup $N$, obtained by conjugating the image of
$\Spin(S^+)_w$ in $Mon(K_a(w))/\Gamma_w$ via a parallel transport operator (Lemma \ref{lemma-compatibility-with-verbitsky}).
Every local system over $\fM^0_{w^\perp}$ is trivial, by \cite[Lemma 2.1]{markman-universal-family}. 
Hence, we get an identification of $N$ with $\Spin(S^+)_w$.

\begin{new-lemma}
\label{lemma-invariance-under-diagonal-monodromy-action}
Any class in $H^{2p}(Y\times T_\ell,\RationalNumbers)^{\Gamma_w}$, which is invariant under the
diagonal $\Spin(S^+)_w$ monodromy action, is of Hodge type $(p,p)$.
\end{new-lemma}

\begin{proof}
If the Hodge operator belongs to the Lie algebra of the identity component of the Zariski closure of a group acting on the cohomology rings of two compact K\"{a}hler manifolds $M_1$ and $M_2$, then it belongs to Lie algebra of the identity component of the Zariski closure of its diagonal action on the cartesian product $M_1\times M_2$, see for example the proof of  \cite[Lemma 3.2]{markman-hodge}. Hence, the statement follows from Lemmas  \ref{lemma-invariance-under-both-translation-and-spin-actions} and \ref{lemma-sub-hodge-structures}.
\end{proof}

%
\subsection{Third intermediate Jacobians of generalized Kummers}
\label{sec-third-intermediate-jacobians}
The third Betti number of a generalized Kummer is $8$ \cite[page 50]{gottsche}.
The Hodge group $H^{0,3}(\Y_t)$ vanishes, for the fiber $\Y_t$ of $p$ over $t\in \fM^0_{w^\perp}$, 
and so the third intermediate Jacobian
$H^{1,2}(\Y_t)/H^3(\Y_t,\Integers)$ is an abelian fourfold, whenever $\Y_t$ is projective, by the Hodge-Riemann bilinear relations. 

\begin{new-lemma}
\label{lemma-intermediate-jacobians}
There exists a global  isogeny between the family of third intermediate Jacobians of
$p:\Y\rightarrow \fM^0_{w^\perp}$ and the family $Per^*\T$ over 
$\fM^0_{w^\perp}$. 
\end{new-lemma}

\begin{proof}
Let $\underline{V}$ be the trivial local system with fiber $V$ over $\fM^0_{w^\perp}$.
It is isomorphic to the weight $1$ variation of integral Hodge structures of the family $Per^*\T$.
It suffices to construct a global $\Spin(V)_w$-equivariant isogeny between the integral local systems
$\underline{V}$ and $R^3p_*\Integers$, by Lemma \ref{lemma-invariance-under-diagonal-monodromy-action}.

Let $\underline{S}^-$ be the trivial local system with fiber $S^-$ over $\fM^0_{w^\perp}$.
Clifford multiplication $m_w:V\rightarrow S^-$ by $w$ induces a $\Spin(V)_w$-equivariant isogeny 
$
m_w:\underline{V}\rightarrow \underline{S}^-.
$
It remains to construct a $\Spin(V)_w$-equivariant isogeny from $\underline{S}^-$ to $R^3p_*\Integers$.
Let $\Pi:\M\rightarrow \fM^0_{w^\perp}$ be the universal deformation of $\M(w)$ given in 
(\ref{eq-universal-deformation-of-a-moduli-space-of-sheaves}). Let $\underline{Q}^3$ be the quotient of $R^3\Pi_*\Integers$ by the cup product image of 
$R^1\Pi_*\Integers\otimes R^2\Pi_*\Integers$. 
Recall that the fiber $\Y_{t_0}$ was the Albanese fiber of the moduli space $\M(w)$, by construction.
$\underline{Q}^3$ has rank $8$, since $\M(w)=[\Y_{t_0}\times X\times \hat{X}]/\Gamma_w$ 
and $\Gamma_w$ acts trivially on $H^i(\Y_{t_0}\times X\times \hat{X},\RationalNumbers)$, for $i\leq 3$.
The homomorphism $\bar{\theta}_3:S^-\rightarrow Q^3(\M(w))$, given in (\ref{eq-double-tilde-theta-j}), is a $\Spin(V)_w$-equivariant 
isomorphism, by Lemmas \ref{lemma-tilde-theta-j-is-an-isomorphism} and \ref{lemma-two-theta-j-are-proportional} (extended to general Mukai vector $w$ via Theorem \ref{thm-Hom-G3-non-empty}).
The restriction homomorphism $H^3(\M(w),\Integers)\rightarrow H^3(\Y_{t_0},\Integers)$ factors through an 
injective\footnote{That restriction homomorphism is known to be surjective for generalized Kummer fourfolds, by \cite[Th. 6.33]{kapfer-menet}.}
$\Spin(V)_w$-equivariant homomorphism $Q^3(\M(w))\rightarrow H^3(\Y_{t_0},\Integers)$, since $H^1(\Y_{t_0},\Integers)=0.$
Composing the latter with $\bar{\theta}_3$ we get a $\Spin(V)_w$-equivariant isogeny 
$S^-\rightarrow H^3(\Y_{t_0},\Integers)$. Every local system over $\fM^0_{w^\perp}$ is trivial, by 
\cite[Lemma 2.1]{markman-universal-family}. We get a $\Spin(V)_w$-equivariant isogeny 
from $\underline{S}^-$ to $R^3p_*\Integers$.
\end{proof}

\begin{rem}
{\rm
Note that for fourfolds $Y$ of generalized Kummer type the polarization map $\Theta:w^\perp\rightarrow \wedge^2V^*$,
given in Equation (\ref{eq-Theta}), is conjugated via Mukai's isometry $H^2(Y,\Integers)\cong w^\perp$ and the
isogeny between $H^3(Y,\Integers)$ and $V$ of the above Lemma, to a map proportional to 
\[
H^2(Y,\Integers)\rightarrow \wedge^2H^3(Y,\Integers)^*,
\]
given by $h\mapsto \int_Y h\cup x\cup y$, for $x,y\in H^3(Y,\Integers)$, as both belong to the rank $1$ invariant  subgroup under the $\Spin(S^+)_w$ monodromy action on $\Hom(H^2(Y,\Integers),\wedge^2H^3(Y,\Integers)^*)$ (see Lemma \ref{lemma-Theta-spans}). O'Grady used the latter map to construct the polarization on
the intermediate Jacobians, and the positivity of the metric in Proposition \ref{prop-Theta-h-is-a-Kahler-form}(\ref{lemma-item-negative-definite}) follows in this case by the Hodge-Riemann bilinear relations. When $Y$ is of generalized Kummer type of dimension $2n\geq 6$, 
O'Grady integrates the product $h\cup x\cup y\cup\beta^{n-2}$, where $\beta\in H^{2,2}(Y,\RationalNumbers)$ is the Beauville-Bogomolov-Fujiki class \cite{ogrady}. 
}
\end{rem}

\hide{
%
\subsubsection{Miscellaneous}
Let $Y$ be an irreducible holomorphic symplectic manifold deformation equivalent to a generalized Kummer variety.
Describe the relation between the Hodge structures
$H^2(Y,\Integers)$, $H^3(Y,\Integers)$, and the quotient
$Q^4(Y,\Integers)$ of $H^4(Y,\Integers)$. 
Conclude that the period domains are isomorphic.

Fix a copy of $\Spin(V)_{s_n}$ in 
$SO[H^3(Y,\Integers)]\times SO[Q^4(Y,\Integers)]$. 
Identify $V_Y$ as a sublattice in $H^3(Y,\Integers)$
(the ``image'' of (\ref{eq-Clifford-multiplication-by-s-n})). 
Recover the Clifford multiplications  
\begin{eqnarray*}
V_Y \otimes H^2(Y,\Integers) & \longrightarrow &
H^3(Y,\Integers),
\\
V_Y \otimes Q^4(Y,\Integers) & \longrightarrow &
H^3(Y,\Integers). 
\end{eqnarray*}
Up to sign, each is the unique non-vanishing and primitive
$\Spin(V)_{s_n}$-equivariant homomorphism. 
Then the image of $V_Y\otimes H^{2,0}(Y)$ in $H^3(Y,\ComplexNumbers)$
is $H^{2,1}(Y)$. 

Let $IG(4,H^3(Y,\ComplexNumbers))$ be the grassmannian of 
maximal isotropic subspaces in $H^3(Y,\Integers)$ with respect to the
quadratic form. It has two connected components \cite[III.1.6]{chevalley}. 
The period domain of the weight $3$ Hodge structure is an open analytic
subset of a ``hyperplane section'' of one of the component of 
$IG(4,H^3(Y,\ComplexNumbers))$. 
Elements in $\widetilde{O}[H^3(Y,\Integers)]$, 
with determinant equal to $-1$, interchange the two components of 
$IG(4,H^3(Y,\ComplexNumbers))$. Conclude that $Mon^3(Y)$ 
is contained in $S\widetilde{O}[H^3(Y,\ComplexNumbers)]$.

The component of $IG(4,H^3(Y,\ComplexNumbers))$ 
containing $H^{2,1}(Y,\ComplexNumbers)$ 
is isomorphic to the quadric hypersurface 
$I\PP$ in 
$\PP Q^4(Y,\ComplexNumbers)$ of isotropic lines. 
The isomorphism is given by Clifford multiplication:
\begin{eqnarray*}
I\PP & \hookrightarrow & IG(4,H^3(Y,\ComplexNumbers))
\\
\ell & \mapsto & \ell\cdot V_Y \subset H^3(Y,\ComplexNumbers).
\end{eqnarray*}
$\ell\cdot V$
The period domain of $Q^4(Y,\ComplexNumbers)$ and $H^2(Y,\ComplexNumbers)$
is an open analytic subset of the hyperplane section 
$I\PP\cap \bar{c}^2(Y)^\perp$. 

}

%
\section{Hyperholomorphic sheaves}
\label{sec-hyperholomorphic-sheaves}

We prove in this section Theorem \ref{thm-intro-hodge-classes-of-weil-type-are-algebraic}
about the algebraicity of the Hodge-Weil classes on abelian fourfolds of Weil-type of discriminant $1$.
Let $\M(w):=\M_H(w)$ be a smooth and compact moduli space of $H$-stable sheaves of primitive Mukai vector $w$ of dimension $\geq 8$ over an abelian surface $X$. Given a class $\lambda\in S^+_X:=H^{even}(X,\Integers)$, denote by $\lambda_i$ its projection to
$H^i(X,\Integers)$.
\begin{new-lemma}
\label{lemma-order-of-brauer-class-divisible-by-g-w}
The Brauer class $\alpha\in H^2_{an}(\M(w),\StructureSheaf{\M(w)}^*)$ of the universal sheaf has order divisible by
$
g_w:=\gcd\{(w,\lambda) \ : \ \lambda\in S^+_X, \ \lambda_2\in H^{1,1}(X,\Integers) \}.
$
\end{new-lemma}

\begin{proof}
The fiber $K_a(w)$ of $\M(w)$ over $a\in \Alb^1(\M(w))$ has dimension $\geq 4$, by assumption, and so 
$H^2(K_a(w),\Integers)$ is Hodge isometric to $w^\perp$, by Yoshioka's Theorem \ref{thm-yoshioka}.
It suffices to prove that $g_w$ divides the order of the restriction of $\alpha$ to $K_a(w)$.
The proof of the latter fact is identical to that of \cite[Lemma 7.5 (2)]{markman-hodge}.
\end{proof}

Let $r$ be an even integer satisfying $r\geq 6$. 
Let $X$ be an abelian surface with a cyclic Picard group generated by an ample class $H$ with $h:=c_1(H)$ satisfying 
$(h,h)_{S^+_X}=-(2r^2+r)$ (so  $\int_X h^2=2r^2+r$). Set $w:=(r,h,r)$. Then $(w,w)_{S^+_X}=-r$ and $g_w=r$. 

\begin{new-lemma}
\label{lemma-maximally-twisted}
\begin{enumerate}
\item
\label{lemma-item-E-F-is-hyperholomorphic}
The sheaf $E_F$ over $\M(w)$ in Theorem 
\ref{thm-kappa-class-is-non-zero-and-spin-7-invariant} is $\alpha$-twisted by a Brauer class  $\alpha$ of order equal to the rank $r$ of $E_F$. Consequently, the sheaf $E_F$ does not have any non-trivial subsheaf of lower rank and
$\SheafEnd(E_F)$ is $\kappa$-slope-polystable with respect to every K\"{a}hler class $\kappa$ on $\M(w)$.
Furthermore, the first Chern class of every direct summand of $\SheafEnd(E_F)$ vanishes.
\item
The sheaf $E$ of Theorem  \ref{thm-kappa-class-is-non-zero-and-spin-7-invariant}(\ref{thm-item-c2-End-E-is-monodromy-invariant}) is $\pi_1^*(\alpha^{-1})\pi_2^*\alpha$-twisted, where $\alpha$ is a Brauer class  of order equal to the rank $r$ of $E$. Consequently, $E$ is $(\pi_1^*\kappa+\pi_2^*\kappa)$-slope-polystable with respect to every K\"{a}hler class $\kappa$ on $\M(w)$. Furthermore, the first Chern class of every direct summand of $\SheafEnd(E)$ vanishes.
\end{enumerate}
\end{new-lemma}

\begin{proof} The proofs of the two parts are identical. We prove part \ref{lemma-item-E-F-is-hyperholomorphic}.
The order of the Brauer class necessarily divides the rank of the sheaf. In our case the rank $r$ divides the order of $\alpha$ by Lemma \ref{lemma-order-of-brauer-class-divisible-by-g-w}. Hence, they are equal.
The polystability of $\SheafEnd(E)$ is proven for any torsion free reflexive sheaf $E$ twisted by a Brauer class of order equal to its 
rank in \cite[Prop. 6.6]{markman-hodge}. The vanishing of the first Chern classes of the direct summands is proven in \cite[Lemma 7.2]{markman-hodge}.
\end{proof}

\begin{thm}
\label{thm-deformability}
The sheaf $E_F$ deforms with $\M(w)$ to a reflexive sheaf, locally free on the complement of a point, over every fiber 
of the universal family (\ref{eq-universal-deformation-of-a-moduli-space-of-sheaves}). 
The sheaf $E$ deforms with $\M(w)\times \M(w)$ to a reflexive sheaf, locally free away from the diagonal, over the cartesian square of  every fiber 
of the universal family (\ref{eq-universal-deformation-of-a-moduli-space-of-sheaves}). 
\end{thm}

\begin{proof}
The sheaf $E_F$ is $\kappa$-slope-stable with respect to every K\"{a}hler class $\kappa$ on $\M(w)$, by Lemma
\ref{lemma-maximally-twisted}. The class $c_2(\SheafEnd(E_F))$ is $\Spin(S^+_X)_w$-invariant with respect to the 
monodromy representation of Theorem \ref{thm-monodromy-representation-mu}, 
by Theorem \ref{thm-kappa-class-is-non-zero-and-spin-7-invariant}. Hence, $c_2(\SheafEnd(E_F))$ remains of Hodge type $(2,2)$ 
along any flat deformation to every fiber of the family $\Pi$ in Equation (\ref{eq-universal-deformation-of-a-moduli-space-of-sheaves}),
by Lemma \ref{lemma-invariance-under-diagonal-monodromy-action}.
Let $\eta:H^2(K_a(w),\Integers)\rightarrow w^\perp$ 
be the inverse of Mukai's Hodge isometry.
It follows that the sheaf $E_F$ deforms as a twisted sheaf along the twistor family 
(\ref{eq-twistor-family-of-natural-hyperkahler-structure}) of the natural hyper-K\"{a}hler structure on $\M(w)$ 
(Definition \ref{def-product-hyperkahler-structure}(\ref{def-item-natural-hyperkahler-structure})) associated
to any K\"{a}hler class $\kappa$ on the generalized Kummer $K_a(w)$ and the marking $\eta$,
by \cite[Theorem 3.19]{kaledin-verbitski-book}, which is generalized to the case of twisted sheaves in \cite[Cor. 6.12]{markman-hodge}.
The sheaf $E_F$ deforms, furthermore, along every generic twistor path in $\fM_{w^\perp}^0$, by  
\cite[Prop. 6.17]{markman-hodge}. The statement follows from the fact that every point in $\fM_{w^\perp}^0$ is connected to $(\M(w),\eta)$ via a generic twistor path, by
\cite[Theorems 3.2 and 5.2.e]{verbitsky-announcement}.
The proof of the statement for the sheaf $E$ is identical.
\end{proof}

Let $T_\ell$, $\ell\in \Omega_{w^\perp},$ be an abelian fourfold of Weil type with ample class $\Theta_h$, $h\in w^\perp$, as in Corollary \ref{cor-weil-type}. It admits complex multiplication by $K:=\RationalNumbers[\sqrt{-d}]$, where $d=(w,w)(h,h)/4$.
Let $\Spin(S^+)_{w,h}$ be the subgroup of $\Spin(S^+)$ stabilizing both $w$ and $h$. The group $\Spin(S^+)_{w,h}$ is an arithmetic subgroup of  $\Spin(S^+_\RealNumbers)_{w,h}\cong\Spin(4,2,\RealNumbers)$ and $\Spin(4,2,\RealNumbers)$ is isomorphic to $SU(2,2)$, see \cite[IX.4.3 B (vi)]{helgason}. $SU(2,2)$ is the special Mumford-Tate group of polarized abelian fourfolds of Weil type
\cite[Theorem 6.11]{weil,van-Geemen}.

\begin{thm}
\label{thm-hodge-classes-of-weil-type-are-algebraic}
The subspace $H^4(T_\ell,\RationalNumbers)^{\Spin(S^+)_{w,h}}$, consisting of classes invariant under $\Spin(S^+)_{w,h}$,  is three dimensional consisting of algebraic classes.
\end{thm}

\begin{proof}
The complexification $\Spin(S^+_\ComplexNumbers)_{w,h}$ of $\Spin(S^+)_{w,h}$ is isomorphic to $SL(4,\ComplexNumbers)$. 
The invariant subspace $H^4(T_\ell,\RationalNumbers)^{\Spin(S^+)_{w,h}}$ is three dimensional, by \cite[Prop. 2]{munoz}, and it consists of Hodge type $(2,2)$ classes, by Lemma
\ref{lemma-sub-hodge-structures}(\ref{lemma-item-spin-6-invariant}). 
This agrees with Weil's observation that $H^{2,2}(A,\RationalNumbers)$ is three dimensional 
for the general polarized abelian fourfold $A$ of Weil type with
complex multiplication by the field $K=\RationalNumbers[\sqrt{-d}]$ in each complete family \cite[Theorems 4.11 and 6.12]{weil,van-Geemen}. Regarding $H^1(A,\RationalNumbers)$ as a $4$-dimensional $K$ vector space, we get that 
$\wedge^4_KH^1(A,\RationalNumbers)$ is a one-dimensional $K$ vector space, which is a $2$-dimensional $\RationalNumbers$-subspace of $\wedge^4_{\RationalNumbers}H^1(A,\RationalNumbers)$.
Weil proved that $H^{2,2}(A,\RationalNumbers)$ contains 
$\wedge^4_KH^1(A,\RationalNumbers)$ 
and for a generic $A$ of Weil type the equality 
\[
H^{2,2}(A,\RationalNumbers)=\mbox{span}_\RationalNumbers\{\Theta_h^2\}+\wedge^4_KH^1(A,\RationalNumbers)
\]
holds \cite[Theorems 4.11 and 6.12]{weil,van-Geemen}.
It follows that $\wedge^4_KH^1(T_\ell,\RationalNumbers)$ is contained in $H^4(T_\ell,\RationalNumbers)^{\Spin(S^+)_{w,h}}$,
for a generic $\ell\in\Omega_{w^\perp}$, such that $(\ell,h)=0$. The inclusion
\begin{equation}
\label{eq-inclusion-of-wedge-4-K-in-invariant-subspace}
\wedge^4_KH^1(T_\ell,\RationalNumbers) \subset H^4(T_\ell,\RationalNumbers)^{\Spin(S^+)_{w,h}}
\end{equation}
must thus hold for all $\ell\in\Omega_{w^\perp}$, such that $(\ell,h)=0$, as it is a closed condition.

A class $\alpha\in \wedge^4_{\RationalNumbers}H^1(T_\ell,\RationalNumbers)$ belongs to $\wedge^4_KH^1(T_\ell,\RationalNumbers)$,
if and only if 
\begin{equation}
\label{eq-membership-in-wedge-4-K}
(\alpha,\lambda(v_1)\wedge v_2\wedge v_3 \wedge v_4)=(\alpha,v_1\wedge \cdots \wedge\lambda(v_i)\wedge \cdots \wedge v_4),
\end{equation}
for $2\leq i\leq 4$, for all $v_i\in H_1(T_\ell,\RationalNumbers)$, and for all $\lambda\in K$. The structure of a one-dimensional $K$-vector space on $\wedge^4_KH^1(T_\ell,\RationalNumbers)$
is given by
\[
(\lambda\alpha,v_1\wedge v_2\wedge v_3 \wedge v_4):=(\alpha,\lambda(v_1)\wedge v_2\wedge v_3 \wedge v_4).
\]

The inclusion (\ref{eq-inclusion-of-wedge-4-K-in-invariant-subspace}) can be seen more directly using
the following description of the subspace $\wedge^4_KH^1(T_\ell,\RationalNumbers)$ of 
$H^4(T_\ell,\RationalNumbers)^{\Spin(S^+)_{w,h}}$. 
Let $Z_1$ and $Z_2$ be the two maximal isotropic subspace of $H^1(T_\ell,\ComplexNumbers)\cong V_\ComplexNumbers$ corresponding to the two isotropic lines in the plane $\mbox{span}_\ComplexNumbers\{w,h\}$ in $S^+_\ComplexNumbers$.
Explicitly, $Z_i$ is the kernel of $m_{\lambda_i}:V_\ComplexNumbers\rightarrow S^-_\ComplexNumbers$, where
$\lambda_i=w\pm \frac{2\sqrt{-d}}{(h,h)}h$. 
Note that $Z_i$ is defined over $K$. The $\Spin(S^+_\ComplexNumbers)_{w,h}$ action on $Z_i$ factors through $SL(Z_i)$
and so it acts trivially on $\wedge^4Z_i$. Each of $\wedge^4Z_i$, $i=1,2$, is defined over $K$ and the non-trivial element in $Gal(K/\RationalNumbers)$ interchanges the two, so their direct sum is defined over $\RationalNumbers$. 
Now, $Z_1$ and $Z_2$ are the two eigenspaces of the endomorphism $\Theta'_h$ of $V_\ComplexNumbers$ in 
Lemma \ref{lemma-complex-multiplication}.
If $\{z_1,z_2,z_3,z_4\}$ is a basis for $Z_i$, then
\[
z_1\wedge \cdots \wedge (a+b\Theta'_h)(z_i)\wedge \cdots \wedge z_4 = 
(a+bc_i)(z_1\wedge z_2\wedge z_3\wedge z_4),
\]
for all $a,b\in\RationalNumbers$,
where the eigenvalue $c_i$ of $\Theta'_h$ is $\pm\sqrt{-d}$. Hence, Equation (\ref{eq-membership-in-wedge-4-K}) holds
for $\alpha:=z_1\wedge z_2\wedge z_3\wedge z_4$ and for $v_i\in V^*\otimes_\RationalNumbers K.$
The subspace $\wedge^4_KH^1(T_\ell,\RationalNumbers)$ thus contains 
the intersection of $\wedge^4_{\RationalNumbers}H^1(T_\ell,\RationalNumbers)$ with the two dimensional complex subspace  $\wedge^4Z_1+\wedge^4Z_2$ of 
$\wedge^4_{\ComplexNumbers}H^1(T_\ell,\RationalNumbers)$. Both subspaces are two dimensional over $\RationalNumbers$, hence equal.

Let $\varphi:K\rightarrow \End_\RationalNumbers(H^1(T_\ell,\RationalNumbers))$ be the homomorphism sending $\sqrt{-d}$ to 
the endomorphism $\Theta'_h$ given in (\ref{eq-Theta-prime}). 
We get the degree $4$ polynomial map $\varphi_4:K\rightarrow \End_\RationalNumbers(\wedge^4_{\RationalNumbers}H^1(T_\ell,\RationalNumbers))$ sending $\lambda$ to $\wedge^4\varphi(\lambda)$ and the latter restricts to the one-dimensional $K$-vector space
$\wedge^4_KH^1(T_\ell,\RationalNumbers)$ as scalar multiplication by $\lambda^4\in K$.
The image of the map $\lambda\mapsto\lambda^4$, from $K$ to $K$, spans $K$ as a $\RationalNumbers$-vector space. 
Hence, given any non-zero element $\alpha\in \wedge^4_KH^1(T_\ell,\RationalNumbers)$,
the set $\varphi_4(\lambda)(\alpha)$, $\lambda\in K$, spans $\wedge^4_KH^1(T_\ell,\RationalNumbers)$.
In contrast, $\varphi_4(\lambda)(\Theta_h^2)=Nm(\lambda)^2\Theta_h^2$, by Corollary \ref{cor-weil-type}, and so 
the one-dimensional subspace $\mbox{span}_\RationalNumbers\{\Theta_h^2\}$
is invariant under $\varphi_4(K)$.

The Cayley class $C$, associated to the $\Spin(S^+)_w$-action on $H^4(T_\ell,\Integers)$, and $\Theta_h^2$ are linearly independent, by \cite[Prop. 2]{munoz}. Hence, the two-dimensional 
$\RationalNumbers$-subspaces 
$\mbox{span}_\RationalNumbers\{C,\Theta_h^2\}$ and $\wedge^4_KH^1(T_\ell,\RationalNumbers)$
of the three-dimensional $H^4(T_\ell,\RationalNumbers)^{\Spin(S^+)_{w,h}}$
intersect non-trivially along a one-dimensional $\RationalNumbers$-subspace. Choose a non-zero class $\alpha$
in their intersection. 
It follows that 
$H^4(T_\ell,\RationalNumbers)^{\Spin(S^+)_{w,h}}$
is spanned  by $\Theta_h^2$ and the two-dimensional $\RationalNumbers$-subspace spanned by the $\varphi_4(K)$-translates of 
$\alpha$

Let $\M_t$, $t\in \fM^0_{w^\perp}$, be any fiber of  the universal family $\Pi:\M\rightarrow \fM^0_{w^\perp}$ given in Equation 
(\ref{eq-universal-deformation-of-a-moduli-space-of-sheaves}) and let $(\Y_t,\eta_t)$ be the marked fiber of the universal family $p:\Y\rightarrow \fM^0_{w^\perp}$ of generalized Kummer type
given in Equation (\ref{eq-p-universal-family-of-generalized-kummer-type}). Let $\ell$ be the period of $(\Y_t,\eta_t)$.
Let
$\iota:T_\ell\rightarrow \M_t$ be the inclusion of a general fiber of $\M_t\rightarrow \Y_t/\Gamma_w$.
Let $E_t$ be a deformation of $E_F$ to the fiber $\M_t$ as in  Theorem \ref{thm-deformability}. 
Then $c_2(\SheafEnd(E_t))$ restricts to $T_\ell$  as a non-zero 
$\Spin(S^+)_w$ invariant class $\iota^*c_2(\SheafEnd(E_t))$, by Theorem \ref{thm-kappa-class-is-non-zero-and-spin-7-invariant}, hence the restriction is a non-zero integral multiple of the Cayley class. 
The Cayley class is thus algebraic. Hence, so is the class $\alpha$ above.
 The ring $\Integers[\sqrt{-d}]$ acts on $T_\ell$ via holomorphic group endomorphisms, which are necessarily algebraic, by Corollary
 \ref{cor-weil-type}. This algebraic action induces the cohomological action on $H^4(T_\ell,\RationalNumbers)$ by
 $\varphi_4(\Integers[\sqrt{-d}])$.
Hence, the two dimensional subspace spanned by the $\varphi_4(K)$-translates of 
$\alpha$
consists of algebraic classes as well.
\end{proof}

\begin{proof} (of Theorem \ref{thm-intro-hodge-classes-of-weil-type-are-algebraic}).
The two discrete invariants $K$ and the discriminant, of a polarized abelian fourfold of Weil type $(A,K,h)$, determine a four dimensional connected period domain of all polarized abelian fourfold of Weil type with these two invariants, up to an isogeny compatible with the subspaces of Hodge-Weil classes, by \cite[Lemma 4, Sec. 6, and Sec. 7]{schoen}.
Every polarized abelian fourfold of Weil type $(A,K,h)$ with discriminant $1$ and imaginary quadratic field 
$K:=\RationalNumbers[\sqrt{-d}]$
is thus isogenous to $T_\ell$, for some period $\ell$ in the period domain $\Omega_{\{w,h'\}^\perp}$ given in 
(\ref{eq-four-dimentional-period-domain}), for some integral
classes $w\in S^+$ and  $h'\in w^\perp$ of negative self-intersection, such that $(h',h')(w,w)/4=d$, by 
Corollary \ref{cor-weil-type} and Lemma \ref{lemma-trivial-discriminant}.
 The push forward of an algebraic class, via an isogeny of abelian varieties, is algebraic. 
Theorem  \ref{thm-intro-hodge-classes-of-weil-type-are-algebraic} thus follows from Theorem \ref{thm-hodge-classes-of-weil-type-are-algebraic}.
\end{proof}

%
\section{The generalized Hodge conjecture for co-dimension $2$ cycles on IHSM's of kummer type.}
\label{sec-generalized-Hodge-conjecture}
We prove Theorem \ref{thm-generalized-Hodge} in this section verifying the generalized Hodge conjecture for codimension $2$ algebraic cycles homologous to $0$ on every projective irreducible holomorphic symplectic manifold of generalized Kummer deformation type. 
We will need a few preparatory results. 
Let $X$ be an abelian surface, $H$ a polarization on $X$, $w\in H^{even}(X,\Integers)$ a primitive Mukai vector, and assume that the moduli space  $\M$ of $H$-stable sheaves on $X$ with Chern character $w$ is smooth and projective of dimension $\geq 8$. Assume that there exists a universal sheaf $\U$ over $X\times \M$ (untwisted). The latter assumption is equivalent to the equality
$\gcd\{(w,\lambda) \ : \ \lambda\in S^+_X, \ \lambda_2\in H^{1,1}(X,\Integers)\}=1$, by \cite[Appendix 2]{mukai-hodge}.
Let $\pi_{ij}$ be the projection from $\M\times X\times \M$ onto the product of the $i$-th and $j$-th factors. Let 
$
E:=\SheafExt^1_{\pi_{13}}\left(\pi_{12}^*\U,\pi_{23}^*\U\right)
$
be the relative first extension sheaf over $\M\times \M$. 
Let $alb:\M\rightarrow Alb(\M)$ be the Albanese map. 
Denote by $K_t(w)$, $t\in Alb(\M)$,  the fiber $alb^{-1}(t)$. Let $e_t:K_t(w)\hookrightarrow \M$ be the inclusion. 
Set $k:=\frac{1}{2}\dim_\ComplexNumbers(K_t(w))$. 


Given $F\in \M$, let $E_F$ be the restriction of $E$ to $\{F\}\times \M$. 
Fix $F_0\in \M$ and consider the map
\begin{equation}
\label{eq-AJ-E}
AJ_E: \M \rightarrow J^2(K_0(w))
\end{equation}
sending $F$ to the Abel-Jacobi image of an algebraic cycle representing  the Chow class
$e_t^*\left[c_2(E_F^\vee\stackrel{L}{\otimes}E_F)-c_2(E_{F_0}^\vee\stackrel{L}{\otimes}E_{F_0})\right].$
The latter class is the same as the one 
given in Equation (\ref{eq-AJ-t}). In the introduction the sheaf $E$ depended on a parameter $b$ and $AJ_{E_b}$ was denoted by $AJ_b$ for short. The  proof of the surjectivity of $AJ_E$ requires a few lemmas.



Set $\hat{X}:=\Pic^0(X)$. The group $X\times \hat{X}$ acts on $\M$ via
$(x,L)F=\tau_{x,*}(F)\otimes L$, where $\tau_x:X\rightarrow X$ is the translation by the point $x\in X$.
Given $g:=(x,L)\in X\times \hat{X}$, denote by $\tilde{g}:\M\rightarrow \M$ the automorphism of $\M$. 

\begin{new-lemma}
The isomorphism $E_{\tilde{g}(F)}\cong \tilde{g}_*(E_F)$ holds for all $F\in \M$ and $g\in X\times \hat{X}$.
\end{new-lemma}
\begin{proof}
The sheaf $E$ is homogeneous with respect to the diagonal action of $X\times \hat{X}$ on $\M\times \M$, i.e.,  we have an isomorphism $(\tilde{g}\times\tilde{g})^*(E)\cong E$, for every $g\in X\times \hat{X}$.
Indeed, given $x\in X$ 
set $\tilde{\tau}_x:=\tau_x\times id_\M:X\times \M\rightarrow X\times \M$ and observe that 
for every $L\in \Pic^0(X)$ 
we have the natural isomorphism
\[
\SheafExt^1_{\pi_{13}}(\pi_2^*L\otimes\pi_{12}^*\tilde{\tau}_{x,*}\U,\pi_2^*L\otimes\pi_{23}^*\tilde{\tau}_{x,*}\U)
\cong \SheafExt^1_{\pi_{13}}(\pi_{12}^*\tilde{\tau}_{x,*}\U,\pi_{23}^*\tilde{\tau}_{x,*}\U)
\RightArrowOf{\tau_x^*} \SheafExt^1_{\pi_{13}}(\pi_{12}^*\U,\pi_{23}^*\U).
\]
The isomorphism $(\tilde{g}\times\tilde{g})^*(E)\cong E$ yields
$(\tilde{g}\times id_\M)^*(E)\cong (id_\M\times \tilde{g})_*(E)$ explaining the second isomorphism below: 
\begin{equation}
\label{eq-E-is-homogeneous}
E_{\tilde{g}(F)}\cong ((\tilde{g}\times id_\M)^*(E))_F\cong ((id_\M\times \tilde{g})_*(E))_F \cong \tilde{g}_*(E_F).
\end{equation}
\end{proof}

Set $\tilde{\M}:=X\times \hat{X}\times K_0(w)$. Let 
\begin{equation}
\label{eq-a}
a: \tilde{\M} \rightarrow \M
\end{equation}
be the restriction of the action morphism $X\times \hat{X}\times \M \rightarrow \M$, given by $a(x,L,F)=\tau_{x,*}(F)\otimes L$.
Then $a$ is a surjective \'{e}tale morphism. Denote by $\U_0$ the pullback of $\U$ to $X\times K_0(w)$ via $id_X\times e_0$. 
Let $\pi_{ij}$ be the projection from $X\times X\times \hat{X}\times K_0(w)$ onto the product of the $i$-th and $j$-th factors. Set $\tilde{\U}:=(id_X\times a)^*\U$.
The restriction of $\tilde{\U}$ to $X\times\{(x,L,F)\}$ is isomorphic to $\tau_{x,*}(F)\otimes L$.
The restriction of $\pi_{14}^*\U_0$ to $X\times\{(x,L,F)\}$ is isomorphic to $F$. 
Define $\eta:X\times X \rightarrow X\times X$ by $\eta(x_1,x_2)=(x_1+x_2,x_2)$. 
Let $\tilde{\eta}$ be the automorphism of $X\times X \times\hat{X}\times K_0(w)$ given by
$\eta\times id_X\times id_{K_0(w)}$. We conclude that there is a line bundle $N$ over $X\times \hat{X}\times K_0(w)$ and an isomorphism
\begin{equation}
\label{eq-factorization-of-U-tilde}
\tilde{\U}\cong \pi_{234}^*N\otimes \tilde{\eta}^*\pi_{14}^*\U_0\otimes \pi_{13}^*\P.
\end{equation}
Let $p_{ij}$ be the projection from $\tilde{\M}=X\times \hat{X}\times K_0(w)$ onto the product of the $i$-th and $j$-th factors.
Let $\ComplexNumbers_0$ be the sky-scraper sheaf supported on the origin in $X$.
Let $\Phi_{\tilde{\U}}:D^b(X)\rightarrow D^b(\tilde{\M})$ be the integral functor with kernel $\tilde{\U}$.

\begin{new-lemma}
\label{lemma-Fourier-Mukai-of-sky-scraper-sheaf-of-the-origin-of-X}
$\Phi_{\tilde{\U}}(\ComplexNumbers_0)\cong N\otimes p_{13}^*\U_0.$
\end{new-lemma}
\begin{proof}
Note that $\eta$ restricts to $\{0\}\times X$ as the diagonal embedding of $X$  in $X\times X$.
Hence, the restriction of  $\pi_{14}\circ \tilde{\eta}$  to $\{0\}\times X\times \hat{X}\times K_0(w)$ is equal to that of  $\pi_{24}$.
The restriction of $\pi_{13}^*\P$ to $\{0\}\times \tilde{\M}$ is the trivial line-bundle. 
The statement follows from Equation (\ref{eq-factorization-of-U-tilde}).
\end{proof}

Let $\gamma:X\times K_0(w)\rightarrow \tilde{\M}$ be given by 
$\gamma(x,F)=(x,\hat{0},F)$, where $\hat{0}\in \hat{X}$ represents  $\StructureSheaf{X}$.
Let $\pi_{K_0(w)}$ be the projection from $X\times K_0(w)$ to $K_0(w)$.

\begin{new-lemma}
\label{lemma-Fourier-Mukai-of-structure-sheaf}
$L\gamma^*\Phi_{\tilde{\U}}(\StructureSheaf{X})\cong 
(\gamma^*N)\otimes L\pi_{K_0(w)}^*R\pi_{K_0(w),*}\U_0$.
\end{new-lemma}

\begin{proof}
Let $q_{ij}$ be the projection from $X\times X\times K_0(w)$ to the product of the $i$-th and $j$-th factors.
Set $\tilde{\gamma}:=(id_X\times \gamma):X\times X\times K_0(w)\rightarrow X\times \tilde{\M}$.
We have $\Phi_{\tilde{\U}}(\StructureSheaf{X})\cong
R\pi_{234,*}\left[\pi_{234}^*N\otimes \tilde{\eta}^*\pi_{14}^*\U_0\otimes\pi_{13}^*\P\right]
\cong
N\otimes R\pi_{234,*}\left[\tilde{\eta}^*\pi_{14}^*\U_0\otimes\pi_{13}^*\P\right]$.
Hence, $L\gamma^*\Phi_{\tilde{\U}}(\StructureSheaf{X})$ is isomorphic to 
$\gamma^*N\otimes Rq_{23,*}\left[L\tilde{\gamma}^*\left\{\tilde{\eta}^*\pi_{14}^*\U_0\right\}\right]$, 
by the triviality of $(\pi_{13}\circ\tilde{\gamma})^*\P$ and 
cohomology and base change for the right square in the cartesian diagram below.
\[
\xymatrix{
X\times K_0(w) \ar[d]_{\pi_{K_0(w)}} & 
X\times X\times K_0(w) \ar[r]^{\tilde{\gamma}} \ar[l]_{q_{13}} \ar[d]_{q_{23}} & 
X\times X\times \hat{X}\times K_0(w)  \ar[d]_{\pi_{234}}
\\
K_0(w) & 
X\times K_0(w) \ar[r]_{\gamma} \ar[l]^{\pi_{K_0(w)} }& 
X\times \hat{X}\times K_0(w).
}
\]
Set $\hat{\eta}:=\eta\times id_{K_0(w)}$.
The isomorphism
$
L\gamma^*\Phi_{\tilde{\U}}(\StructureSheaf{X})\cong 
(\gamma^*N)\otimes Rq_{23,*}[\hat{\eta}^*q_{13}^*\U_0]
$
thus follow from the equality $\pi_{14}\circ\tilde{\eta}\circ \tilde{\gamma}=q_{13}\circ\hat{\eta}$.
The isomorphism
\[
L\gamma^*\Phi_{\tilde{\U}}(\StructureSheaf{X})\cong 
(\gamma^*N)\otimes Rq_{23,*}[q_{13}^*\U_0]
\]
follows from the equality
$q_{23}=q_{23}\circ \hat{\eta}$ and the fact that $R\hat{\eta}_*L\hat{\eta}^*$ is the identity.
The statement follows by cohomology and base change with respect to the left square in the above diagram.
\end{proof}

\begin{proof}[Proof of Theorem \ref{thm-generalized-Hodge}]
Up to translation, the morphism $\overline{AJ}_b$, given in (\ref{eq-overline-AJ}),  
is determined by the homomorphism 
$\overline{AJ}_{b,*}:H_1(\Alb(\M_b),\Integers)\rightarrow H_1(J^2(Y_b),\Integers)$. The latter depends continuously on $b$ and $E_b$. Any two points $b_1, b_2$ of $\fM^0_{\omega^\perp}$, with projective $Y_{b_1}$ and $Y_{b_2}$, can be connected by a subfamily of $\Pi$ with projective fibers. 
It thus suffices to prove Theorem \ref{thm-generalized-Hodge} for one point in $\fM^0_{\omega^\perp}$.
We will prove it for a moduli space of sheaves $\M$ as in Equation (\ref{eq-AJ-E}).

Let $\CH^i(\M)$ be the group of codimension $i$ algebraic cycles in $\M$ and $\CH^i(\M)_0$ its subgroup of cycles homologous to zero. 
Given a point $[F]$ in $\M$ representing the isomorphism class of a sheaf F, let
$
\iota_F:X\times \hat{X}\rightarrow \M
$
be the map onto the orbit of $[F]$ under the $X\times\hat{X}$-action. 
Set $J^2(\M):=H^3(\M,\ComplexNumbers)/[F^2H^3(\M,\ComplexNumbers)+H^3(\M,\Integers)]$,
where $F^2H^3(\M,\ComplexNumbers):=H^{3,0}(\M)\oplus H^{2,1}(\M)$ is the second subspace in the Hodge filtration.
We have the commutative diagram of Abel-Jacobi maps
\[
\xymatrix{
X\times \hat{X}\ar[r]^{\iota_{F_0}} &\M\ar[r]^{\psi} \ar[dr] & \CH^2(\M)_0 \ar[r]^{AJ_\M} \ar[d]_{e_0^*} & J^2(\M) \ar[d]^{r}
\\
& &\CH^2(K_0(w))_0\ar[r] _{AJ_K} & J^2(K_0(w)),
}
\]
where the horizontal map $\psi$ sends $F$ to 
$c_2(E^\vee_F\stackrel{L}{\otimes} E_F)-c_2(E^\vee_{F_0}\stackrel{L}{\otimes} E_{F_0})$, 
with both the dual $E^\vee_F$ and the tensor product taken in the derived category, and
the right vertical homomorphism $r$ is induced by the restriction homomorphism
$e_0^*:H^3(\M,\ComplexNumbers)\rightarrow H^3(K_0(w),\ComplexNumbers).$

It suffices to prove the surjectivity of $r\circ AJ_\M\circ \psi\circ \iota_{F_0}$,  as $r\circ AJ_\M\circ \psi$ is equal to $AJ_E$ given in 
(\ref{eq-AJ-E}). Being a morphism of complex tori, the composition is induced by a linear homomorphism 
(its differential) from $H_1(X\times \hat{X},\RealNumbers)\cong V^*_\RealNumbers$ to $H^3(K_0(w),\RealNumbers)$.
Both are irreducible $\Spin(V)_w$-representations. Hence, it suffices to prove that
the differential of $r\circ AJ_\M\circ \psi\circ \iota_{F_0}$ is $\Spin(V)_w$-equivariant and it does not vanish.

Let $Z_0$ be an algebraic cycle representing the Chow class $c_2(E_{F_0}^\vee\stackrel{L}{\otimes}E_{F_0})$. Then
$\tilde{g}_*(Z_0)$ represents $c_2(E_{\tilde{g}(F_0)}^\vee\stackrel{L}{\otimes}E_{\tilde{g}(F_0)})$, by
Equation (\ref{eq-E-is-homogeneous}). Given a smooth path $\gamma$ from $0$ to $g_1\in X\times \hat{X}$
we get the co-chain $\Gamma:=\cup_{g\in \gamma}\tilde{g}(Z_0)$ with boundary $\tilde{g}_1(Z_0)-Z_0$. The point  $(AJ_\M\circ\psi)(\tilde{g}_1(F_0))$ is the projection to
$J^2(\M)$ of the class in $H^{1,2}(\M)\oplus H^{0,3}(\M)$, which corresponds to the linear functional  sending a class $\phi$ in $H^{2k+3,2k+2}(\M)\oplus H^{2k+4,2k+1}(\M)$ to
$
\int_\Gamma \phi.
$
Let $\xi$ be a tangent vector to $\gamma$ at $0$. 
Let 
\[
da:T_0[X\times\hat{X}]\rightarrow H^0(T\M)
\] 
be the homomorphism induced by the action of $X\times\hat{X}$ on $\M$. 
Then the restriction of $da(\xi)$  to $Z_0$ maps to a global section of the real normal bundle of $Z_0$ in $\Gamma$.
The differential of $(AJ_\M\circ\psi)$ maps $da(\xi)$ to 
$H^{1,2}(\M)\oplus H^{0,3}(\M)$, hence to a linear functional on
$H^{2k+3,2k+2}(\M)\oplus H^{2k+4,2k+1}(\M)$, whose value at a cohomology class $\phi$
is equal to 
$\int_{Z_0}(\phi,da(\xi))$,
where $(\bullet,da(\xi)):H^q(\M,\Omega^p_\M)\rightarrow H^q(\M,\Omega^{p-1}_\M)$ is induced by contraction with $da(\xi)$ (see \cite[Lecture 6]{green}).
The differential of $AJ_\M\circ \psi\circ \iota_{F_0}$ at $0\in X\times \hat{X}$ thus 
maps $\xi$ to the linear functional
\[
\phi \mapsto \int_{\M}(\phi,da(\xi))\cup c_2(E_{F_0}^\vee\stackrel{L}{\otimes}E_{F_0}),
\]
since the homology class of $Z_0$ is Poincar\'{e} dual to the cohomology class of  
$c_2(E_{F_0}^\vee\stackrel{L}{\otimes}E_{F_0})$.
The value at $\xi\otimes \phi$ of the differential of $AJ_E\circ \iota_{F_0}$ at $0$ is thus given by 
\begin{equation}
\label{eq-differential-of-AJ-E}
d_0(AJ_E\circ \iota_{F_0})(\xi\otimes \phi)=\int_{\M}(e_{0,*}(\phi),da(\xi))\cup c_2(E_{F_0}^\vee\stackrel{L}{\otimes}E_{F_0}),
\end{equation}
where $e_{0,*}:H^{4k-3}(K_0(w),\ComplexNumbers)\rightarrow H^{4k+5}(\M,\ComplexNumbers)$
is the Gysin homomorphism, since $dr^*$ is induced by $(e_0^*)^*=e_{0,*}$.

Set $\tilde{\M}:= X\times \hat{X}\times K_0(w)$. 
Let $[pt]\in H^8(X\times\hat{X},\Integers)$ be the class Poincar\'{e} dual to a point in $X\times \hat{X}$.
We prove next that the following equality holds: 
\begin{equation}
\label{eq-push-forward-commutes-with-contraction}
(e_{0,*}(\phi),da(\xi))=a_*(\phi\boxtimes ([pt],\xi)), 
\end{equation}
for all $\phi\in H^{4k-3}(K_0(w),\ComplexNumbers)$, where $\boxtimes$ denotes the outer product
and we consider $H^{4k-3}(K_0(w),\ComplexNumbers)\otimes H^7(X\times\hat{X},\ComplexNumbers)$ as a subspace of $H^{4k+4}(\tilde{\M},\ComplexNumbers)$ via the K\"{u}nneth decomposition.
Let $\tilde{e}_0:K_0(w)\rightarrow \tilde{\M}$ be given by $t\mapsto (0,0,t)$.
We have $e_0=a\circ \tilde{e}_0$, where $a$ is given in (\ref{eq-a}), and so the Gysin map $e_{0,*}$ is the composition 
$a_*\circ \tilde{e}_{0,*}$. Let $\Gamma_w$ be the Galois group of $a:\tilde{\M}\rightarrow \M$. A class $\beta$ in $H^*(\tilde{\M},\ComplexNumbers)$  decomposes as $\beta=a^*(\beta')+\beta''$, where $\beta''$ belongs to the direct sum of non-trivial $\Gamma_w$-representations. Then
\[
\int_\M \alpha\cup a_*(\beta)=\int_{\tilde{\M}}a^*(\alpha)\cup\beta=\int_{\tilde{\M}}a^*(\alpha\cup\beta')=\deg(a)\int_\M\alpha\cup\beta',
\]
for all $\alpha\in H^*(\M)$. Hence, $a_*(\beta)=\deg(a)\beta'$. Now, $\tilde{e}_{0,*}(\phi)=\phi\boxtimes [pt]$, 
the outer product of $\phi$ with $[pt]$. 
The group $\Gamma_w$ acts on $X\times \hat{X}\times K_0(w)$ via its translation action on $X\times \hat{X}$ and its action on $K_0(w)$ as automorphisms acting trivially on $H^i(K_0(w),\ComplexNumbers)$, for $i\leq 3$ and so also for $i\geq 4k-3$, by Lemma \ref{lemma-Gamma-v}(\ref{lemma-item-Gamma-v-embedds-in-Mon}).
Hence, given $\phi\in H^{2k-1,2k-2}(K_0(w))$, the class $\phi\boxtimes [pt]$
is $\Gamma_w$ invariant and is equal to $a^*\phi'$, for some $\phi'\in H^{2k+3,2k+2}(\M)$, and $e_{0,*}(\phi)=\deg(a)\phi'=(k+1)^4\phi'$.  Let $\tilde{\xi}$ be the global tangent vector of 
$\tilde{\M}=X\times \hat{X}\times K_0(w)$ corresponding to $\xi$ via the natural isomorphism $H^0(T[X\times \hat{X}])=H^0(T\tilde{\M})$. Then $a^*(da(\xi))=\tilde{\xi}$ via the isomorphism $a^*T\M\cong T\tilde{\M}$. We have 
\[
\frac{1}{\deg(a)}a^*(e_{0,*}(\phi),da(\xi))=a^*(\phi',da(\xi))=(a^*\phi',\tilde{\xi})=(\phi\boxtimes[pt],\tilde{\xi})=\phi\boxtimes ([pt],\xi).
\] 
Applying $a_*$ to both sides in the latter displayed equation
we get Equation (\ref{eq-push-forward-commutes-with-contraction}).

Combining Equations 
(\ref{eq-differential-of-AJ-E}) and (\ref{eq-push-forward-commutes-with-contraction}) we get
\[
d_0(AJ_E\circ \iota_{F_0})(\xi\otimes \phi)=\int_{\tilde{\M}}(\phi\boxtimes ([pt],\xi))\cup a^*c_2(E_{F_0}^\vee\stackrel{L}{\otimes}E_{F_0}).
\]
The differential $d_0(AJ_E\circ \iota_{F_0})$ is $\Spin(V)_w$-equivariant, by the 
$\Spin(V)_w$-invariance of $c_2(E_{F_0}^\vee\stackrel{L}{\otimes}E_{F_0})$ established in Theorem 
\ref{thm-kappa-class-is-non-zero-and-spin-7-invariant}. 
The right hand side in the  above displayed equation does not vanish, for some $\xi\otimes\phi\in H^0(T[X\times\hat{X}])\otimes H^{2k-1,2k-2}(K_0(w))$, if and only if the K\"{u}nneth direct summand of $a^*c_2(E_{F_0}^\vee\stackrel{L}{\otimes}E_{F_0})$ in 
$H^1(X\times\hat{X})\otimes H^3(K_0(w))$ does not vanish.
This is the case, if and only if the K\"{u}nneth direct summand of $a^*c_2(E_{F_0})$ in 
$H^1(X\times\hat{X})\otimes H^3(K_0(w))$ does not vanish, since the corresponding direct summand of $a^*c_1(E_{F_0})^2$ vanishes, as $H^2(X\times\hat{X}\times K_0(w))$ decomposes as
$H^2(X\times\hat{X})\oplus H^2(K_0(w))$. For the same reason, the direct summand of $a^*c_2(E_{F_0})$ in 
$H^1(X\times\hat{X})\otimes H^3(K_0(w))$ does not vanish, if and only if that of $a^*ch_2(E_{F_0})$ does not vanish.

Assume next that $w$ is the Mukai vector $(1,0,-1-k)$ of the ideal sheaf of a length $k+1$ subscheme. 
The class $a^*ch_2(E_{F_0})$ is equal to $-ch_2(\Phi_{\tilde{\U}}(F_0^\vee))$, as $a^*E_{F_0}$ is the first sheaf cohomology of
$\Phi_{\tilde{\U}}(F_0^\vee)$, the second sheaf cohomology is the direct sum of the sky-scraper sheaves of the points of $\tilde{\M}$ over $[F_0]$, and all other sheaf cohomologies vanish. We have
\[
ch_2[\Phi_{\tilde{\U}}(F_0^\vee)]=ch_2[\Phi_{\tilde{\U}}(\StructureSheaf{X})]-
(k+1)ch_2[\Phi_{\tilde{\U}}(\ComplexNumbers_0)].
\]
Let $\hat{0}$ be the point of $\hat{X}$ representing the isomorphism class of the trivial line bundle.
It suffices to prove that the K\"{u}nneth direct summand of $ch_2[\Phi_{\tilde{\U}}(\ComplexNumbers_0)]$
in $H^1(X\times\hat{X})\otimes H^3(K_0(w))$ restricts non-trivially to $X\times\{\hat{0}\}\times K_0(w)$, while 
the K\"{u}nneth direct summand of $ch_2[\Phi_{\tilde{\U}}(\StructureSheaf{X})]$
in $H^1(X\times\hat{X})\otimes H^3(K_0(w))$ restricts to zero
in $X\times\{\hat{0}\}\times K_0(w)$. 

The object
$(id_X\times \tilde{e}_0)^*\Phi_{\tilde{\U}}(\ComplexNumbers_0)$  is isomorphic to the tensor product of a line bundle with $\U_0$, by Lemma \ref{lemma-Fourier-Mukai-of-sky-scraper-sheaf-of-the-origin-of-X}. Hence, the K\"{u}nneth component of $(id_X\times \tilde{e}_0)^*ch_2(\Phi_{\tilde{\U}}(\ComplexNumbers_0))$ in $H^1(X)\otimes H^3(K_0(w))$ is equal to that of $ch_2(\U_0)$.
We have the equality
\begin{equation}
\label{two-correspondences-and-projection-formula}
e_0^*\pi_{\M,*}[\pi_X^*(\lambda)\cup ch(\U)]=\pi_{K_o(w),*}[\pi_X^*(\lambda)\cup ch(\U_0)],
\end{equation}
by the projection formula applied to the cartesian diagram
\[
\xymatrix{
X\times K_0(w) \ar[r]^{id_X\times e_0} \ar[d]_{\pi_{K_0(w)}}&
X\times \M \ar[d]^{\pi_\M}
\\
K_0(w)\ar[r]_{e_0} & \M.
}
\]
The only graded summand of $ch(\U_0)$ (resp. $ch(\U)$), which contributes to the homomorphism 
 (\ref{two-correspondences-and-projection-formula}) from
$H^3(X)$ to $H^3(K_0(w))$ (resp. to $H^3(\M)$), is
$ch_2(\U_0)$ (resp. $ch_2(\U)$). 
The left hand side of (\ref{two-correspondences-and-projection-formula}) induces an  isomomorphism from $H^{odd}(X,\RationalNumbers)$ onto $H^3(K_0(w),\RationalNumbers)$, by Lemma  \ref{lemma-tilde-theta-j-is-an-isomorphism} and the surjectivity of $h_3$, given in Equation (\ref{eq-h-i}),  established in the paragraph preceding Lemma \ref{lemma-Gamma-v}. Hence, so does the right hand side, and the K\"{u}nneth direct summand of $ch_2[\Phi_{\tilde{\U}}(\ComplexNumbers_0)]$
in $H^1(X\times\hat{X})\otimes H^3(K_0(w))$ restricts non-trivially to $X\times\{\hat{0}\}\times K_0(w)$.

The K\"{u}nneth component in $H^1(X)\otimes H^3(K_0(w))$ of 
$\gamma^*ch_2[\Phi_{\tilde{\U}}(\StructureSheaf{X})]$ is equal to that of 
$(\gamma^*N)\otimes L\pi_{K_0(w)}^*R\pi_{K_0(w),*}\U_0$, 
by Lemma \ref{lemma-Fourier-Mukai-of-structure-sheaf}, and hence also to that of 
$L\pi_{K_0(w)}^*R\pi_{K_0(w),*}\U_0$, as observed above.
The K\"{u}nneth component  in $H^1(X)\otimes H^3(K_0(w))$ of the latter clearly vanishes.
This completes the proof of Theorem \ref{thm-generalized-Hodge}.
\end{proof}

\hide{
%
\section{Characters and automorphisms of $G(S^+)^{even}_{s_n}$}

The center of $G(S^+)^{even}_{s_n}$ is generated by $\tilde{\alpha}$,
given in (\ref{eq-central-element-tilde-alpha}), and is 
contained in the kernel $\Spin(V)_{s_n}$ of 
the orientation character (\ref{eq-ort}) (??? need to use instead $ort_{S^+}$ as in Section \ref{sec-polarizations}
since $G(S^+)^{even}_{s_n}$ is not contained in the image of $G(V)$ in $GL(A_X)$ ???) $G(S^+)^{even}_{s_n}$. 
Consider the following automorphism of $G(V)^{even}_{s_n}$:
\begin{equation}
\label{eq-automorphism-of-G-V-s-n}
g \ \ \ \mapsto \ \ \ \tilde{\alpha}^{ort(g)}\cdot g.
\end{equation}
Note that $ort$ is invariant under this automorphism. 

The orientation character of $G(S^+)^{even}_{s_n}$
is related to the composition of the 
determinant character with the homomorphism 
(\ref{eq-homomorphism-from-stabilizer-in-G-S-plus-even-to-GL}) (??? why). 
The orientation character (\ref{eq-ort})
of $G(S^+)^{even}_{s_n}$ is the pullback of that of $SO[H^2(Y,\Integers)]$.
Note, however, that the center of 
$SO[H^2(Y,\Integers)]$ is not contained in the kernel of its
orientation character. The homomorphism of $SO[H^2(Y,\Integers)]$, given by
\begin{equation}
\label{eq-automorphism-of-SO-H-2}
g \ \ \ \mapsto \ \ \ (-1)^{ort^2(g)}\cdot g,
\end{equation}
has a kernel generated by $-1$. This homomorphism maps 
$SO[H^2(Y,\Integers)]$ onto
$SO_+[H^2(Y,\Integers)]$. 
The image of $G(S^+)^{even}_{s_n}$ in $SO[H^2(Y,\Integers)]$ 
is the reflection group $\Reflection:=\Reflection(s_n^\perp)$, 
consisting of words of even length
in reflections by $+2$ and $-2$ vectors in $s_n^\perp$. 
The homomorphism (\ref{eq-automorphism-of-SO-H-2})
maps the image $\Reflection$ 
of $G(S^+)^{even}_{s_n}$ onto $\W^{\det\cdot \chi}$. 
The homomorphism (\ref{eq-automorphism-of-SO-H-2}) pulls back the 
character $\det$ (or equivalently $\chi$) of $\W^{\det\cdot \chi}$
to the orientation character of  $\Reflection$.

(??? work with $\widetilde{O}(\bullet)$) 
Let $ort^3$ be the orientation character of $O[H^3(Y,\Integers)]$ 
associated to the positive cone in $H^3(Y,\RealNumbers)$. If we choose the
negative cone instead, we get a different character, but the
two restrict to the same character of $SO[H^3(Y,\Integers)]$, and hence of
$Mon^3(Y)$. 
The pullback of  $ort^3$ to $G(S^+)^{even}_{s_n}$ 
is equal to the orientation character (\ref{eq-ort}) of the latter.
Consider the  automorphism of $SO[H^3(Y,\Integers)]$
analogous to (\ref{eq-automorphism-of-G-V-s-n})
\begin{equation}
g \ \ \ \mapsto \ \ \ (-1)^{ort^3(g)}\cdot g.
\end{equation}
Note that $ort^3(-1)=1$, so $ort^3$ is invariant under this automorphism. 


\section{Arithmetic constraints on the monodromy groups}
\label{sec-arithmetic-constraints}

\subsection{Constraints on $Mon(\M(s_n))$}

Define $Q^i(\M(s_n),\Integers)$ and prove the isomorphisms\footnote{
The second isomorphism holds for $n=3$, by the paper of Kapfer-Menet. For $n\geq 4$ use instead that the composition
$S^-_X\rightarrow Q^3(\M(s_n),\Integers)\rightarrow H^3(K_X(n-1),\Integers)$ is injective and its image is monodromy invariant of finite index.
}
\[
Q^2(\M(s_n),\Integers) \ \cong \ H^2(K_X(n\!-\!1),\Integers) \ \cong \ s_n^\perp
\subset S^+_X := H^{even}(X,\Integers).
\]
\[
Q^3(\M(s_n),\Integers) \ \cong \ H^3(K_X(n\!-\!1),\Integers)  \ \cong \ 
S^-_X := H^{odd}(X,\Integers).
\]

\begin{new-lemma}
\begin{enumerate}
\item
$Q^3(\M(s_n),\Integers)$ is an irreducible $Mon^3(\M(s_n))$ representation.
\item
$Q^3(\M(s_n),\Integers)$ admits a natural, up to sign, unimodular
symmetric bilinear form, with respect to which the isomorphism
\[
\theta \ : \ H^{odd}(X,\Integers) \ \ \ \longrightarrow \ \ \ 
Q^3(\M(s_n),\Integers),
\]
induced by the universal sheaf, is an isometry.
\end{enumerate}
\end{new-lemma}

\begin{proof}
Use the invariance of the universal sheaf under $\Spin(V_X)_{s_n}$.
\end{proof}

\medskip
Let $Mon^i(\M(s_n))$ be the image of $Mon(\M(s_n))$ in 
$\Aut[Q^i(\M(s_n),\Integers)]$ and let $K^i(\M(s_n))$ be the kernel. 
\[
0 \rightarrow K^i(\M(s_n)) \rightarrow Mon(\M(s_n)) \rightarrow 
Mon^i(\M(s_n)) \rightarrow 0.
\]

\begin{new-lemma}
The monodromy representation 
(\ref{eq-homomorphism-gamma-from-G-S-plus-even-s-n-to-Mon}) 
of $G(S^+)^{even}_{s_n}$ 
surjects (??? impossible\footnote{$G(S^+)^{even}_{s_n}$ leaves invariant the homomorphism $\Theta:Q^2(\M(s_n),\Integers)\rightarrow \wedge^2Q^3(\M(s_n),\Integers)$, where $\Theta$ is given in (\ref{eq-Theta}). } ???) onto $S\widetilde{O}Q^3(\M(s_n),\Integers)$ and is hence equal to
$Mon^3(\M(s_n))$.
\end{new-lemma}

Let $Mon^{2,3}(\M(s_n))$ be the image on $Mon(\M(s_n))$ in 
$\Aut[Q^2(\M(s_n),\Integers)\oplus Q^3(\M(s_n),\Integers)]$.
Let $K^{2,3}(\M(s_n))$ be the kernel.

\begin{new-lemma}
\label{lemma-Mon-of-moduli-splits}
$K^3(\M(s_n))$ is contained in the center of $Mon(\M(s_n))$.
Consequently, the image of $G(S^+)^{even}_{s_n}$
under the monodromy representation
(\ref{eq-homomorphism-gamma-from-G-S-plus-even-s-n-to-Mon})
is a normal subgroup of $Mon(\M(s_n))$ and 
$Mon(\M(s_n))$ is isomorphic to the product 
\[
Mon(\M(s_n)) \ \ \ \cong \ \ \ 
K^{2,3}(\M(s_n))\times G(S^+)^{even}_{s_n}.
\]
\end{new-lemma}

\begin{proof}
Use the fact that the images of $H^*(X,\Integers)$ in 
$H^i(\M(s_n),\Integers)$ generate the cohomology and the 
compatibility of the monodromy representation 
(\ref{eq-homomorphism-gamma-from-G-S-plus-even-s-n-to-Mon}) with Verbitsky's 
representation of $\Spin(S^+)_{s_n}$ on $H^*(K_X(n\!-\!1))$. 
In order to apply Verbitsky's result, use the fact that
$Q^i(\M(s_n),\RationalNumbers)$ inject into 
$Q^i(K_X(n\!-\!1),\RationalNumbers)$, for $i\geq 2$ (???),
and the image is monodromy invariant.
\end{proof}

\begin{cor}
The monodromy representation
(\ref{eq-homomorphism-gamma-from-G-S-plus-even-s-n-to-Mon}) 
of $G(S^+)^{even}_{s_n}$ surjects onto $Mon^2(\M(s_n))$. 
\end{cor}

\begin{proof}
The centralizer of the image of $G(S^+)^{even}_{s_n}$ in
$O[Q^2(\M(s_n),\Integers)]$ is trivial.
\end{proof}

\medskip
Lemma \ref{lemma-Mon-of-moduli-splits} implies, that any representation of 
$G(S^+)^{even}_{s_n}$ (and hence of $G(S^+)^{even}$)
pulls back to a representation of $Mon(\M(s_n))$. 
Set $Q_-^3:=Q^3(\M(s_n),\Integers)$ and let 
$Q_+^3$ be the pullback of the other
integral half-spin representation of $G(S^+)^{even}$.
(In terms of $G(S^+)^{even}$ these are $S^-$ and $V_X$, however it is not
clear which among these two should be considered $Q_-^3$).

\begin{new-lemma}
There exists a $Mon(\M(s_n))$-invariant pair
$\{e,-e\}$ of primitive integral isometric embeddings (up to sign) of
$Q^2(\M(s_n),\Integers)$ in $Q_+^3\otimes Q_-^3$. 
\end{new-lemma}

\begin{cor}
Clifford multiplication by $s_n$
\[
s_n \ : \ Q_+^3 \ \ \longrightarrow \ \  Q_-^3
\]
is $Mon(\M(s_n))$-equivariant.
\end{cor}

\medskip
We conclude, that we have a representation of $Mon(\M(s_n))$ in the
co-kernel of $s_n$. We get also a $Mon(\M(s_n))$-invariant subgroup 
$\Gamma'$ of 
$Q^3(\M(s_n),\Integers/n\Integers)$. 

\begin{new-lemma}
$\Gamma'$ is $Mon(K_X(n\!-\!1))$ invariant as well.
\end{new-lemma}

As a corollary we can construct a deformation of $\M(s_n)$ from any
deformation $\Y\rightarrow B$ of $K_X(n\!-\!1)$ as follows.
Take the quotient, of the fiber product over $B$ of 
$\Y$ and the relative third intermediate jacobian of $\Y$, 
by the diagonal action of the group scheme of relative $\Gamma'$s.

\begin{cor}
There is a surjective homomorphism
\[
Mon(K_X(n\!-\!1)) \ \ \ \longrightarrow \ \ \ Mon(\M(s_n)), 
\]
for $n\geq 3$.
\end{cor}

\subsection{Constraints on $Mon(K_X(n\!-\!1))$}

Characterize the group $\W$ in equation (\ref{eq-W})
as the group of orientation preserving isometries of $s_n^\perp$,
which extend to the Mukai lattice $S^+$. 

Two possible proofs that $Mon^2(Y)$ is contained in $\W$.
One uses $Q^4(Y,\Integers)$, as in the case of Hilbert schemes of 
zero-dimensional subschemes of a $K3$ surface.

The second proof uses the fact, that the spin representation 
$H^3(Y,\Integers)$ of $SO[H^2(Y,\Integers)]$ is irreducible, and is the 
restriction of the half-spin representation
$S^-$ of $SO(V)$ (or $SO(S^+)$). 

\bigskip
Prove, that $Mon^2(Y)$ is contained in the kernel of the character 
$\det\cdot \chi$ of $\W$. The group $\W^{\det\cdot \chi}$ 
is the image of the
subgroup $G(S^+)^{even}_{s_n}$, while $G(S^+)_{s_n}$ surjects onto $\W$.
Use the fact, that $G(S^+)_{s_n}$ does not have a representation,
which restricts to  $G(H^2(X,\Integers))^{even}$ as the standard
rank $4$ representation $H^2(X,\Integers/(n+1)\Integers)$ of the latter,
modulo $n+1$.

}

%

{\bf Acknowledgements:}
This work was partially supported by a grant  from the Simons Foundation (\#427110). 
I am grateful to Kieran O'Grady for sharing with me an early draft of his insightful paper \cite{ogrady}, the insight it provided was crucial in the proof of 
Theorem \ref{thm-intro-hodge-classes-of-weil-type-are-algebraic}. I thank Misha Verbitsky for reference \cite{munoz}. I am grateful to Claire Voisin for her suggestion of Theorem \ref{thm-generalized-Hodge}. I thank the referees for their insightful comments and suggestions.

\section{Glossary of Notation}
\label{sec-glossary-of-notation}
\begin{longtable}{l l l}
\\
$V$ & the lattice $H^1(X,\Integers)\oplus H^1(\hat{X},\Integers)$ & Eq. (\ref{eq-pairing-on-V})
\\
$\Spin(V)$ & the spin  group of a lattice or a vector space $V$ & Eq. (\ref{eq-Spin-Pin-and-G})
\\
$\Spin(V)_w$ & the stabilizer of $w$ in $\Spin(V)$ & Sec. \ref{sec-monodromy-intro}
\\
$\Pin(V)$ & the pin group of a lattice or a vector space $V$ & Eq. (\ref{eq-Spin-Pin-and-G})
\\
$G(V)$ & one of the Clifford groups & Eq. (\ref{eq-Spin-Pin-and-G})
\\
$G_0(V)$ & one of the Clifford groups & Eq. (\ref{eq-Spin-Pin-and-G})
\\
$G(V)^{even}$ &the even Clifford group & Eq. (\ref{eq-Spin-Pin-and-G})
\\
$G(V)^{even}_{w}$ &the stabilizer of $w$ in $G(V)^{even}$ & Sec. \ref{sec-monodromy-intro}
\\
$C(V)$ & the Clifford algebra of a lattice or a vector space $V$ & Sec. \ref{sec-Clifford-groups}
\\
$C(V)^{even}$ &  the even  direct summand of $C(V)$ & Sec. \ref{sec-Clifford-groups}
\\
$C(V)^{odd}$ &  the odd direct summand of $C(V)$ & Sec. \ref{sec-Clifford-groups}
\\
$S$  & the cohomology $H^*(X,\Integers)$  as the spin representation & Sec. \ref{sec-Clifford-groups}
\\
$S^+$ & $H^{even}(X,\Integers)$  as the half spin representation & Sec. \ref{sec-Clifford-groups}
\\
$S^-$ &  $H^{odd}(X,\Integers)$  as the half spin representation & Sec. \ref{sec-Clifford-groups}
\\
$A_X$ & the algebra $V\oplus S^+\oplus S^-$ & Sec. \ref{sec-triality}
\\
$m$ & spin representation of the Clifford algebra  $C(V)$ & Eq. (\ref{eq-m-from-C-V}), 
\\
& or $C(S^+)$ & Cor. \ref{cor-V-plus-S-minus-is-the-Clifford-module}
\\
$m_w$ & the value of $m$ on $w\in S^+$ & Eq. (\ref{eq-composition-of-m-y-1-and-m-y-2})
\\
$m$ & spin representation of the Clifford group &  Eq. (\ref{eq-Cl})
\\
$\tilde{m}$ & embedding of $G(S^+)$ in $GL(A_X)$ & Eq. (\ref{eq-tilde-m})
\\
$\widetilde{O}(S^+)$ & subgroup of $GL(S^+)$ preserving the pairing up to sign & Sec. \ref{sec-Clifford-groups}
\\
$S\widetilde{O}(S^+)$ & subgroup of $SL(S^+)$ preserving the pairing up to sign & Sec. \ref{sec-Clifford-groups}
\\
$(\bullet,\bullet)_V$ & the pairing on $V$ & Eq. (\ref{eq-pairing-on-V-introduction})
\\
$\langle\bullet,\bullet\rangle$ & the Mukai pairing on $H^{even}(X,\Integers)$ & Eq. (\ref{eq-Mukai-pairing-with-sign-as-in-Mukai})
\\
$(\bullet,\bullet)_S$ & the pairing on $S$ & Eq. (\ref{eq-Mukai-pairing})
\\
$\tau$ & main anti-automorphism of $C(V)$ & Eq. (\ref{eq-tau})
\\
$\tilde{\tau}$ & lift of $\tau$ to $G(S^+)^{even}$ & Eq. (\ref{eq-tau-is-in-G-S-plus-even})
\\
$\tau_X$ & the automorphism of $A_X$ induced by $\tilde{\tau}$ & Eq. (\ref{eq-tau-X})
\\
$\alpha$ & main involution of $C(V)$ & Eq. (\ref{eq-main-involution})
\\
$\tilde{\alpha}$ & element in the center of $\Spin(V)$ & Eq. (\ref{eq-central-element-tilde-alpha})
\\
$\rho$ & the homomorphism $G(V)\rightarrow O(V)$ & Eq. (\ref{eq-standard-representation-of-G-V})
\\
$\tilde{\mu}$ & the homomorphism $G(V)\rightarrow GL(A_X)$ & Eq. (\ref{eq-representation-of-G-V-on-A-X})
\\
$ort$ & the orientation character & Eq. (\ref{eq-orientation-character}), (\ref{eq-ort-S+})
\\
$SO_+(V)$ & the kernel of the orientation character in $SO(V)$ & Sec. \ref{sec-Clifford-groups}
\\
$O_+(V)$ & the kernel of the orientation character in $O(V)$ & Sec. \ref{sec-Clifford-groups}
\\
$L_w$ & the endomorphism $w\wedge\bullet:H^*(X,\Integers)\rightarrow H^*(X,\Integers)$ & Eq. (\ref{eq-left-wedge-by-w}) 
\\
$D_\theta$ & endomorphism of $H^*(X,\Integers)$ of contraction with $\theta$  & Sec. \ref{sec-Clifford-groups}
\\
$PD$ & Poincare Duality homomorphism & Eq. (\ref{eq-Poincare-Duality})
\\
$\hat{X}$ & $\Pic^0(X)$ of an abelian surface $X$ & Sec. \ref{sec-notation}
\\
$X^{[n]}$ & Hilbert scheme of length $n$ subschemes of a surface $X$ & Sec. \ref{sec-monodromy-intro}
\\
$X^{(n)}$ & the $n$-th symmetric product of an abelian surface $X$ & Sec. \ref{sec-monodromy-intro}
\\
$K_X(m)$ & generalized kummer variety of an abelian surface $X$ & Sec. \ref{sec-monodromy-intro}
\\
$\M_H(v)$ & moduli space of stable sheaves with Mukai vector $v$ & Sec. \ref{subsection-mukai-lattice}
\\
$K_a(v)$ & fiber of the albanese map $\M_H(v)\rightarrow X\times \hat{X}$ over $a$ & Sec. \ref{subsection-mukai-lattice}
\\
$Mon(Y)$ & monodromy group of a compact K\"{a}hler manifold $Y$ & Def. \ref{def-monodromy}
\\
$\mon$ & the monodromy representation on $H^*(\M(w),\Integers)$ & 
Eq. (\ref{eq-homomorphism-gamma-from-G-S-plus-even-s-n-to-Mon})
\\
$\overline{\mon}$ & the monodromy representation 
on $H^*(K_a(v),\Integers)$ & 
Prop. \ref{prop-overline-mon}
\\
$\Gamma_X$ & group of points of order $n$ on the abelian surface $X$ & Sec. \ref{sec-monodromy-intro}
\\
$\Gamma_w$ & subgroup of torsion points in the torus $V_\RealNumbers/V$ & Rem. \ref{rem-Z-w}
\\
$\gamma_{g,\epsilon}(\E_1,\E_2)$ & correspondence in $H^{2m}(\M(w_1)\times \M(w_2),\RationalNumbers)$ & Eq. (\ref{eq-gamma-delta})
\\
$D_M$ & homomorphism acting by $(-1)^i$ on $H^{2i}(M)$ & Sec. \ref{sec-equivariance-of-the-universal-sheaf}
\\
$d_X$, $d_{\M(w)}$ & a choice of factorization $D_{X\times \M(w)}=d_X\otimes d_{\M(w)}$ & Eq. (\ref{eq-factorization-of-D})
\\
$\tilde{\theta}$ & homomorphism $S\rightarrow H^*(\M(s_n),\RationalNumbers)$ & Eq. (\ref{eq-tilde-theta-homomorphism})
\\
$Q^d(\M(w))$ & quotient of $H^d(\M(w),\Integers)$ & Sec. \ref{sec-half-spin}
\\
$S^j_X$ & $\left\{\begin{array}{ccl}
S^+_X & \mbox{if} & j \ \mbox{is even}, \ j\neq 2,
\\
S^+_X\cap s_n^{\perp} & \mbox{if}& j=2,
\\
S^-_X & \mbox{if} & j \ \mbox{is odd}
\end{array}\right.$ 
&Eq. (\ref{eq-tilde-theta-j-rational})

\\
$\tilde{\theta}_j$ & homomorphism $S^j_X\rightarrow Q^j(\M(w))\otimes_\Integers\RationalNumbers$ & Eq. (\ref{eq-tilde-theta-j-rational})
\\
$h_i$ & homomorphism $Q^i(\M(w))\otimes_\Integers\RationalNumbers\rightarrow H^i(K_a(v),\RationalNumbers)$ & Eq. (\ref{eq-h-i})
\\
$\iota_F$ & embedding of $X\times \hat{X}$ in the orbit of $F$ in $\M_H(w)$ & Eq. (\ref{eq-iota-F})
\\
$q_w$ & $\iota_F^*:H^*(\M_H(w),\Integers)\rightarrow H^*(X\times \hat{X},\Integers)$  & Eq. (\ref{eq-q-w})
\\
$\Omega_{w^\perp}$ & period domain & Eq. (\ref{eq-Omega-w-perp})
\\
$\Omega_{\{w,h\}^\perp}$ & period domain & Eq. (\ref{eq-four-dimentional-period-domain})
\\
$\fM_{w^\perp}$ & moduli space of marked hyperk\"{a}hler manifolds & Sec. \ref{subsection-a-universal-deformation-of-a-moduli-space-of-sheaves}
\\
$\fM_{w^\perp}^0$ & a connected component of $\fM_{w^\perp}$ & Sec. \ref{subsection-a-universal-deformation-of-a-moduli-space-of-sheaves}
\\
$\Theta'$ & homomorphism $w^\perp\rightarrow \Hom(V,V)$ & Eq. (\ref{eq-Theta-prime})
\\
$\Theta$ & homomorphism $w^\perp\rightarrow \wedge^2V^*$ & Eq. (\ref{eq-Theta})
\\
$J_\ell$ & complex structure on $V_\ComplexNumbers$ associated to $\ell\in \Omega_{w^\perp}$ & Sec. \ref{sec-two-isomorphic-period-domains}
\\
$T_\ell$ & a complex torus associated to a period $\ell\in \Omega_{w^\perp}$ & Sec. \ref{sec-two-isomorphic-period-domains}
\\
$Per$ & the period map   $\fM_{w^\perp}\rightarrow \Omega_{w^\perp}$ & Sec. \ref{subsection-a-universal-deformation-of-a-moduli-space-of-sheaves}
\end{longtable}




\end{document}